\documentclass[12pt]{amsart}
\usepackage{amsfonts,amssymb,amsmath,euscript}
\newcommand*{\abs}[1]{\left\lvert#1\right\rvert}   
\newcommand*{\set}[1]{\left\{#1\right\}}           
\newcommand*{\brs}[1]{\left(#1\right)}             
\newcommand*{\norm}[1]{\left\Vert#1\right\Vert}    
\newcommand*{\hilb}{\mathcal H}                     

\newcommand*{\mbC}{{\mathbb C}}
\newcommand*{\mbR}{{\mathbb R}}
\newcommand*{\mbZ}{{\mathbb Z}}
\newcommand*{\mbN}{{\mathbb N}}

\newcommand*{\Tr}{\operatorname{Tr}}        
\renewcommand*{\det}{\operatorname{det}}     

    \newtheorem{thm}{Theorem}                     [section]
    \newtheorem{thm*}{Theorem}
    \newtheorem{prop}[thm]{Proposition}
    \newtheorem{lemma}[thm]{Lemma}
    \newtheorem{cor}[thm]{Corollary}

    \newtheorem{lemma*}{Lemma}    
    
    \newtheorem{assump}[thm]{Assumption}

    \newtheorem{defn}[thm]{Definition}                 
    \newtheorem{example}[thm]{Example}                 
    \newtheorem{rems}[thm]{Remark}                     
    
    \newtheorem{rems*}{Remark}   

\newcommand{\ndef}{\newcommand*}
\def\rndef{\renewcommand}

\ndef{\myaddress}[1]{\begin{center} \it\small #1 \end{center}}

\ndef{\clA}{{\mathcal A}} \ndef{\rmA}{{\mathrm A}} \ndef{\mbA}{{\mathbb A}} \ndef{\bfA}{{\mathbf A}} \ndef{\euA}{{\EuScript A}} \ndef{\frA}{{\mathfrak A}}
\ndef{\clB}{{\mathcal B}} \ndef{\rmB}{{\mathrm B}} \ndef{\mbB}{{\mathbb B}} \ndef{\bfB}{{\mathbf B}} \ndef{\euB}{{\EuScript B}} \ndef{\frB}{{\mathfrak B}}
\ndef{\clC}{{\mathcal C}} \ndef{\rmC}{{\mathrm C}}                          \ndef{\bfC}{{\mathbf C}} \ndef{\euC}{{\EuScript C}} \ndef{\frC}{{\mathfrak C}}
\ndef{\clD}{{\mathcal D}} \ndef{\rmD}{{\mathrm D}} \ndef{\mbD}{{\mathbb D}} \ndef{\bfD}{{\mathbf D}} \ndef{\euD}{{\EuScript D}} \ndef{\frD}{{\mathfrak D}}
\ndef{\clE}{{\mathcal E}} \ndef{\rmE}{{\mathrm E}} \ndef{\mbE}{{\mathbb E}} \ndef{\bfE}{{\mathbf E}} \ndef{\euE}{{\EuScript E}} \ndef{\frE}{{\mathfrak E}}
\ndef{\clF}{{\mathcal F}} \ndef{\rmF}{{\mathrm F}} \ndef{\mbF}{{\mathbb F}} \ndef{\bfF}{{\mathbf F}} \ndef{\euF}{{\EuScript F}} \ndef{\frF}{{\mathfrak F}}
\ndef{\clG}{{\mathcal G}} \ndef{\rmG}{{\mathrm G}} \ndef{\mbG}{{\mathbb G}} \ndef{\bfG}{{\mathbf G}} \ndef{\euG}{{\EuScript G}} \ndef{\frG}{{\mathfrak G}}
\ndef{\clH}{{\mathcal H}} \ndef{\rmH}{{\mathrm H}} \ndef{\mbH}{{\mathbb H}} \ndef{\bfH}{{\mathbf H}} \ndef{\euH}{{\EuScript H}} \ndef{\frH}{{\mathfrak H}}
\ndef{\clI}{{\mathcal I}} \ndef{\rmI}{{\mathrm I}} \ndef{\mbI}{{\mathbb I}} \ndef{\bfI}{{\mathbf I}} \ndef{\euI}{{\EuScript I}} \ndef{\frI}{{\mathfrak I}}
\ndef{\clJ}{{\mathcal J}} \ndef{\rmJ}{{\mathrm J}} \ndef{\mbJ}{{\mathbb J}} \ndef{\bfJ}{{\mathbf J}} \ndef{\euJ}{{\EuScript J}} \ndef{\frJ}{{\mathfrak J}}
\ndef{\clK}{{\mathcal K}} \ndef{\rmK}{{\mathrm K}} \ndef{\mbK}{{\mathbb K}} \ndef{\bfK}{{\mathbf K}} \ndef{\euK}{{\EuScript K}} \ndef{\frK}{{\mathfrak K}}
\ndef{\clL}{{\mathcal L}} \ndef{\rmL}{{\mathrm L}} \ndef{\mbL}{{\mathbb L}} \ndef{\bfL}{{\mathbf L}} \ndef{\euL}{{\EuScript L}} \ndef{\frL}{{\mathfrak L}}
\ndef{\clM}{{\mathcal M}} \ndef{\rmM}{{\mathrm M}} \ndef{\mbM}{{\mathbb M}} \ndef{\bfM}{{\mathbf M}} \ndef{\euM}{{\EuScript M}} \ndef{\frM}{{\mathfrak M}}
\ndef{\clN}{{\mathcal N}} \ndef{\rmN}{{\mathrm N}}                          \ndef{\bfN}{{\mathbf N}} \ndef{\euN}{{\EuScript N}} \ndef{\frN}{{\mathfrak N}}
\ndef{\clO}{{\mathcal O}} \ndef{\rmO}{{\mathrm O}} \ndef{\mbO}{{\mathbb O}} \ndef{\bfO}{{\mathbf O}} \ndef{\euO}{{\EuScript O}} \ndef{\frO}{{\mathfrak O}}
\ndef{\clP}{{\mathcal P}} \ndef{\rmP}{{\mathrm P}} \ndef{\mbP}{{\mathbb P}} \ndef{\bfP}{{\mathbf P}} \ndef{\euP}{{\EuScript P}} \ndef{\frP}{{\mathfrak P}}
\ndef{\clQ}{{\mathcal Q}} \ndef{\rmQ}{{\mathrm Q}}                          \ndef{\bfQ}{{\mathbf Q}} \ndef{\euQ}{{\EuScript Q}} \ndef{\frQ}{{\mathfrak Q}}
\ndef{\clR}{{\mathcal R}} \ndef{\rmR}{{\mathrm R}}                          \ndef{\bfR}{{\mathbf R}} \ndef{\euR}{{\EuScript R}} \ndef{\frR}{{\mathfrak R}}
\ndef{\clS}{{\mathcal S}} \ndef{\rmS}{{\mathrm S}} \ndef{\mbS}{{\mathbb S}} \ndef{\bfS}{{\mathbf S}} \ndef{\euS}{{\EuScript S}} \ndef{\frS}{{\mathfrak S}}
\ndef{\clT}{{\mathcal T}} \ndef{\rmT}{{\mathrm T}} \ndef{\mbT}{{\mathbb T}} \ndef{\bfT}{{\mathbf T}} \ndef{\euT}{{\EuScript T}} \ndef{\frT}{{\mathfrak T}}
\ndef{\clU}{{\mathcal U}} \ndef{\rmU}{{\mathrm U}} \ndef{\mbU}{{\mathbb U}} \ndef{\bfU}{{\mathbf U}} \ndef{\euU}{{\EuScript U}} \ndef{\frU}{{\mathfrak U}}
\ndef{\clV}{{\mathcal V}} \ndef{\rmV}{{\mathrm V}} \ndef{\mbV}{{\mathbb V}} \ndef{\bfV}{{\mathbf V}} \ndef{\euV}{{\EuScript V}} \ndef{\frV}{{\mathfrak V}}
\ndef{\clW}{{\mathcal W}} \ndef{\rmW}{{\mathrm W}} \ndef{\mbW}{{\mathbb W}} \ndef{\bfW}{{\mathbf W}} \ndef{\euW}{{\EuScript W}} \ndef{\frW}{{\mathfrak W}}
\ndef{\clX}{{\mathcal X}} \ndef{\rmX}{{\mathrm X}} \ndef{\mbX}{{\mathbb X}} \ndef{\bfX}{{\mathbf X}} \ndef{\euX}{{\EuScript X}} \ndef{\frX}{{\mathfrak X}}
\ndef{\clY}{{\mathcal Y}} \ndef{\rmY}{{\mathrm Y}} \ndef{\mbY}{{\mathbb Y}} \ndef{\bfY}{{\mathbf Y}} \ndef{\euY}{{\EuScript Y}} \ndef{\frY}{{\mathfrak Y}}
\ndef{\clZ}{{\mathcal Z}} \ndef{\rmZ}{{\mathrm Z}}                          \ndef{\bfZ}{{\mathbf Z}} \ndef{\euZ}{{\EuScript Z}} \ndef{\frZ}{{\mathfrak Z}}

\ndef{\tA}{{\widetilde A}} \ndef{\tcA}{{\widetilde\clA}} \ndef{\ttcA}{\widetilde{\tcA}} \ndef{\sfA}{{\textsf A}} \ndef{\ttA}{\widetilde{\tA}} \ndef{\dzA}{{A^\sharp}}
\ndef{\tB}{{\widetilde B}} \ndef{\tcB}{{\widetilde\clB}} \ndef{\ttcB}{\widetilde{\tcB}} \ndef{\sfB}{{\textsf B}} \ndef{\ttB}{\widetilde{\tB}} \ndef{\dzB}{{B^\sharp}}
\ndef{\tC}{{\widetilde C}} \ndef{\tcC}{{\widetilde\clC}} \ndef{\ttcC}{\widetilde{\tcC}} \ndef{\sfC}{{\textsf C}} \ndef{\ttC}{\widetilde{\tC}} \ndef{\dzC}{{C^\sharp}}
\ndef{\tD}{{\widetilde D}} \ndef{\tcD}{{\widetilde\clD}} \ndef{\ttcD}{\widetilde{\tcD}} \ndef{\sfD}{{\textsf D}} \ndef{\ttD}{\widetilde{\tD}} \ndef{\dzD}{{D^\sharp}}
\ndef{\tE}{{\widetilde E}} \ndef{\tcE}{{\widetilde\clE}} \ndef{\ttcE}{\widetilde{\tcE}} \ndef{\sfE}{{\textsf E}} \ndef{\ttE}{\widetilde{\tE}} \ndef{\dzE}{{E^\sharp}}
\ndef{\tF}{{\widetilde F}} \ndef{\tcF}{{\widetilde\clF}} \ndef{\ttcF}{\widetilde{\tcF}} \ndef{\sfF}{{\textsf F}} \ndef{\ttF}{\widetilde{\tF}} \ndef{\dzF}{{F^\sharp}}
\ndef{\tG}{{\widetilde G}} \ndef{\tcG}{{\widetilde\clG}} \ndef{\ttcG}{\widetilde{\tcG}} \ndef{\sfG}{{\textsf G}} \ndef{\ttG}{\widetilde{\tG}} \ndef{\dzG}{{G^\sharp}}
\ndef{\tH}{{\widetilde H}} \ndef{\tcH}{{\widetilde\clH}} \ndef{\ttcH}{\widetilde{\tcH}} \ndef{\sfH}{{\textsf H}} \ndef{\ttH}{\widetilde{\tH}} \ndef{\dzH}{{H^\sharp}}
\ndef{\tI}{{\widetilde I}} \ndef{\tcI}{{\widetilde\clI}} \ndef{\ttcI}{\widetilde{\tcI}} \ndef{\sfI}{{\textsf I}} \ndef{\ttI}{\widetilde{\tI}} \ndef{\dzI}{{I^\sharp}}
\ndef{\tJ}{{\widetilde J}} \ndef{\tcJ}{{\widetilde\clJ}} \ndef{\ttcJ}{\widetilde{\tcJ}} \ndef{\sfJ}{{\textsf J}} \ndef{\ttJ}{\widetilde{\tJ}} \ndef{\dzJ}{{J^\sharp}}
\ndef{\tK}{{\widetilde K}} \ndef{\tcK}{{\widetilde\clK}} \ndef{\ttcK}{\widetilde{\tcK}} \ndef{\sfK}{{\textsf K}} \ndef{\ttK}{\widetilde{\tK}} \ndef{\dzK}{{K^\sharp}}
\ndef{\tL}{{\widetilde L}} \ndef{\tcL}{{\widetilde\clL}} \ndef{\ttcL}{\widetilde{\tcL}} \ndef{\sfL}{{\textsf L}} \ndef{\ttL}{\widetilde{\tL}} \ndef{\dzL}{{L^\sharp}}
\ndef{\tM}{{\widetilde M}} \ndef{\tcM}{{\widetilde\clM}} \ndef{\ttcM}{\widetilde{\tcM}} \ndef{\sfM}{{\textsf M}} \ndef{\ttM}{\widetilde{\tM}} \ndef{\dzM}{{M^\sharp}}
\ndef{\tN}{{\widetilde N}} \ndef{\tcN}{{\widetilde\clN}} \ndef{\ttcN}{\widetilde{\tcN}} \ndef{\sfN}{{\textsf N}} \ndef{\ttN}{\widetilde{\tN}} \ndef{\dzN}{{N^\sharp}}
\ndef{\tO}{{\widetilde O}} \ndef{\tcO}{{\widetilde\clO}} \ndef{\ttcO}{\widetilde{\tcO}} \ndef{\sfO}{{\textsf O}} \ndef{\ttO}{\widetilde{\tO}} \ndef{\dzO}{{O^\sharp}}
\ndef{\tP}{{\widetilde P}} \ndef{\tcP}{{\widetilde\clP}} \ndef{\ttcP}{\widetilde{\tcP}} \ndef{\sfP}{{\textsf P}} \ndef{\ttP}{\widetilde{\tP}} \ndef{\dzP}{{P^\sharp}}
\ndef{\tQ}{{\widetilde Q}} \ndef{\tcQ}{{\widetilde\clQ}} \ndef{\ttcQ}{\widetilde{\tcQ}} \ndef{\sfQ}{{\textsf Q}} \ndef{\ttQ}{\widetilde{\tQ}} \ndef{\dzQ}{{Q^\sharp}}
\ndef{\tR}{{\widetilde R}} \ndef{\tcR}{{\widetilde\clR}} \ndef{\ttcR}{\widetilde{\tcR}} \ndef{\sfR}{{\textsf R}} \ndef{\ttR}{\widetilde{\tR}} \ndef{\dzR}{{R^\sharp}}
\ndef{\tS}{{\widetilde S}} \ndef{\tcS}{{\widetilde\clS}} \ndef{\ttcS}{\widetilde{\tcS}} \ndef{\sfS}{{\textsf S}} \ndef{\ttS}{\widetilde{\tS}} \ndef{\dzS}{{S^\sharp}}
\ndef{\tT}{{\widetilde T}} \ndef{\tcT}{{\widetilde\clT}} \ndef{\ttcT}{\widetilde{\tcT}} \ndef{\sfT}{{\textsf T}} \ndef{\ttT}{\widetilde{\tT}} \ndef{\dzT}{{T^\sharp}}
\ndef{\tU}{{\widetilde U}} \ndef{\tcU}{{\widetilde\clU}} \ndef{\ttcU}{\widetilde{\tcU}} \ndef{\sfU}{{\textsf U}} \ndef{\ttU}{\widetilde{\tU}} \ndef{\dzU}{{U^\sharp}}
\ndef{\tV}{{\widetilde V}} \ndef{\tcV}{{\widetilde\clV}} \ndef{\ttcV}{\widetilde{\tcV}} \ndef{\sfV}{{\textsf V}} \ndef{\ttV}{\widetilde{\tV}} \ndef{\dzV}{{V^\sharp}}
\ndef{\tW}{{\widetilde W}} \ndef{\tcW}{{\widetilde\clW}} \ndef{\ttcW}{\widetilde{\tcW}} \ndef{\sfW}{{\textsf W}} \ndef{\ttW}{\widetilde{\tW}} \ndef{\dzW}{{W^\sharp}}
\ndef{\tX}{{\widetilde X}} \ndef{\tcX}{{\widetilde\clX}} \ndef{\ttcX}{\widetilde{\tcX}} \ndef{\sfX}{{\textsf X}} \ndef{\ttX}{\widetilde{\tX}} \ndef{\dzX}{{X^\sharp}}
\ndef{\tY}{{\widetilde Y}} \ndef{\tcY}{{\widetilde\clY}} \ndef{\ttcY}{\widetilde{\tcY}} \ndef{\sfY}{{\textsf Y}} \ndef{\ttY}{\widetilde{\tY}} \ndef{\dzY}{{Y^\sharp}}
\ndef{\tZ}{{\widetilde Z}} \ndef{\tcZ}{{\widetilde\clZ}} \ndef{\ttcZ}{\widetilde{\tcZ}} \ndef{\sfZ}{{\textsf Z}} \ndef{\ttZ}{\widetilde{\tZ}} \ndef{\dzZ}{{Z^\sharp}}

\ndef{\bfa}{{\mathbf a}}
\ndef{\bfb}{{\mathbf b}}
\ndef{\bfc}{{\mathbf c}}
\ndef{\bfd}{{\mathbf d}}

\ndef{\euu}{{\EuScript u}}

  \ndef{\eps}{\varepsilon}

\let\geq\geqslant
\let\leq\leqslant

\ndef{\lims}[1]{\lim\limits_{#1}}
\ndef{\sums}[1]{\sum\limits_{#1}}
\ndef{\ints}[1]{\int_{#1}}
\ndef{\sups}[1]{\sup\limits_{#1}}
\ndef{\liminfty}[1]{\lims{#1\to\infty}}
\ndef{\suminf}[1]{\sums{#1=1}^\infty}

\ndef{\limo}[1]{\omega\mbox{-}\!\!\!\lims{#1\to\infty}}          
\ndef{\limL}[1]{\rmL\mbox{-}\!\!\!\lims{#1\to\infty}}            
\ndef{\limLOne}[1]{\clL_1\mbox{-}\!\lims{#1}}
\ndef{\tildelimo}[1]{\tilde\omega\mbox{-}\!\!\!\lims{#1\to\infty}}
\ndef{\slim}{\mathrm{s}\mbox{-}\!\!\lim}          
\ndef{\wlim}{\mathrm{w}\mbox{-}\!\lim}          

\ndef{\Aut}{\operatorname{Aut}}      
\ndef{\Ch}{\operatorname{ch}}        
\ndef{\End}{\operatorname{End}}      
\ndef{\Hom}{\operatorname{Hom}}      
\rndef{\ker}{\operatorname{ker}}      
\ndef{\coker}{\operatorname{coker}}      
\ndef{\im}{\operatorname{im}}        
\ndef{\Log}{\operatorname{Log}}      
\ndef{\OP}{\operatorname{OP}}        
\ndef{\Op}{\operatorname{Op}}        
\ndef{\Symb}{\operatorname{Symb}}    
\ndef{\Wres}{\operatorname{Wres}}    
\ndef{\cl}{\operatorname{cl}}        
\ndef{\com}{\operatorname{com}}
\ndef{\const}{\operatorname{const}}  
\ndef{\conv}{\operatorname{conv}}    
\ndef{\Var}{\operatorname{Var}}
\ndef{\Cov}{\operatorname{Cov}}

\ndef{\detFK}[1]{\Delta\brs{#1}} 
\ndef{\detFKrel}[2]{\Delta_{#2}\brs{#1}} 

\ndef{\adj}{\operatorname{adj}}    
\ndef{\diag}{\operatorname{diag}}    
\ndef{\dist}{\operatorname{dist}}    
\ndef{\dom}{\operatorname{dom}}      
\ndef{\ec}{\operatorname{ec}}        
\ndef{\id}{\mathrm{Id}}                        
\ndef{\ind}{\operatorname{ind}}      
\ndef{\mydeg}{\operatorname{deg}}    
\ndef{\op}{\operatorname{op}}
\ndef{\rank}{\operatorname{rank}}
\ndef{\res}{\operatorname{res}}      
\ndef{\ran}{\operatorname{ran}}      
\ndef{\sflow}{\operatorname{sf}}     
\ndef{\isf}{\operatorname{isf}}      
\ndef{\sign}{\operatorname{sign}}    
\ndef{\sgn}{\operatorname{sgn}}      
\ndef{\sing}{\operatorname{sing}}    
\ndef{\supp}{\operatorname{supp}}    
\ndef{\tr}{\operatorname{tr}}        
\ndef{\var}{\operatorname{var}}      
\ndef{\vol}{\operatorname{vol}}      
\ndef{\wn}{\operatorname{wn}}        
\ndef{\wres}{\operatorname{wres}}    
\rndef{\Im}{\operatorname{Im}}       
\rndef{\Re}{\operatorname{Re}}       

\ndef{\prng}[1]{\mathrm R_{#1}} 
\ndef{\pker}[1]{\mathrm N_{#1}} 
\ndef{\rprng}[2]{\mathrm R_{#1}^{#2}}           
\ndef{\rpker}[2]{\mathrm N_{#1}^{#2}}           
\ndef{\rsupp}[1]{\supp_r(#1)}
\ndef{\lsupp}[1]{\supp_l(#1)}
\ndef{\rslv}[1]{R_z(#1)}      
\ndef{\HH}{H}                 
\ndef{\tHH}{\tilde \HH}       
\ndef{\VV}{V}                 
\ndef{\Rz}{R_z}               
\ndef{\tRz}{\tR_z}            
\ndef{\psif}[1]{#1^{[1]}} 
\ndef{\WPlus}[1]{W_{#1}(\mbR)} 
\newcommand{\Cc}{C_\csupp^\infty(\mbR)}
\newcommand{\Texp}{\mathrm{T}\!\exp}
\newcommand{\Phia}{\Phi^{(a)}}
\newcommand{\Phis}{\Phi^{(s)}}
\newcommand{\xia}{\xi^{(a)}}
\newcommand{\xis}{\xi^{(s)}}
\newcommand{\mua}{\mu^{(a)}}
\newcommand{\mus}{\mu^{(s)}}
\newcommand{\phia}{\phi^{(a)}}

\ndef{\bndl}{\xi}                         
\ndef{\bndlA}{\eta}                       
\ndef{\GlueMap}{\varphi}                  
\ndef{\ChartMap}{h}                       
\ndef{\chern}{\ensuremath{\mathrm{ch}}}
\ndef{\hilba}{\clH^{(a)}}                    
\ndef{\hilbs}{\clH^{(s)}}                    
   \ndef{\hilbasargument}{(\hilb)} 
\ndef{\LpH}[1]{\clL_{#1}\hilbasargument}       
\ndef{\saLpH}[1]{\clL_{sa}^{#1}\hilbasargument}       
\ndef{\clBH}{\clB\hilbasargument}              
\ndef{\ubBH}{\clB_1\hilbasargument}            
\ndef{\clCH}{\clC\hilbasargument}              
\ndef{\clKH}{\clK\hilbasargument}              
\ndef{\clFH}{\clF\hilbasargument}              
\ndef{\clUH}{\clU\hilbasargument}              
\ndef{\clCFH}{{\clC\clF}\hilbasargument}       
\ndef{\saBH}{\clB_{sa}\hilbasargument}         
\ndef{\saCH}{\clC_{sa}\hilbasargument}         
\ndef{\saFH}{\clF_{sa}\hilbasargument}         
\ndef{\saKH}{\clK_{sa}\hilbasargument}         
\ndef{\saCFH}{\clC\clF_{sa}\hilbasargument}    
\ndef{\clUFH}{\clU\clF\hilbasargument}         
\ndef{\Uinj}{\clU_{inj}\hilbasargument}        
\ndef{\UFinj}{\clU\clF_{inj}\hilbasargument}   

\ndef{\spproj}[2]{E^{#1}_{#2}}                      
\ndef{\spprojb}[2]{E^{#2}_{#1}}                     

\ndef{\LpN}[1]{\clL^{#1}(\clN,\tau)}     
\ndef{\saLpN}[1]{\clL^{#1}_{sa}(\clN,\tau)} 
\ndef{\rLpN}[1]{L^{#1}(\clN,\tau)}       
\ndef{\clAND}{(\clA,\clN,D)}             
\ndef{\clBA}{{\clB(\clA)}}
\ndef{\saKN}{{\clK_{sa}(\clN,\tau)}}          
\ndef{\clKN}{{\clK(\clN,\tau)}}          
\ndef{\clKtN}{{\clK(\tilde\clN,\tau)}}   
\ndef{\clFN}{{\clF(\clN,\tau)}}          
\ndef{\saFN}{{\clF_{sa}(\clN,\tau)}}     
\ndef{\clPN}{\clP(\clN)}                 
\ndef{\clQN}{\clQ(\clN,\tau)}            
\ndef{\infPN}{{\clP_\tau^\infty(\clN)}}  
\ndef{\clOF}[2]{\clF_{#1\mbox{-}#2}(\clN,\tau)}         
\ndef{\oind}[2]{{\rm \tau\mbox{-}ind}_{#1\mbox{-}#2}}   
\ndef{\tind}{\tau\mbox{-}\ind}                  
\ndef{\DInd}{\ind_{\clD,\tau}}           
\ndef{\BF}{Breuer-Fredholm}              
\ndef{\skewfred}[2]{$(#1\cdot #2)$ $\tau$\tire Fredholm}   
\ndef{\affl}{\eta}                       
\ndef{\vNa}{von Neumann algebra}         
\ndef{\nsf}{faithful normal semifinite } 
\ndef{\taubrs}[1]{\tau\brackets{#1}}     
\ndef{\sqbrs}[1]{[#1]}        
\ndef{\Sqbrs}[1]{\big[#1\big]}        
\ndef{\SqBrs}[1]{\Big[#1\Big]}        

\ndef{\domd}{\bigcap\limits_{n\ge 0} \dom\;\delta^n}         
\ndef{\DiffOP}{{\rm \clD}}
\ndef{\ADA}{\clA \cup [\clD,\clA]}
\ndef{\DixIdeal}[1]{\LpH{#1,\infty}}               
\ndef{\dixideal}{\ell^{1,\infty}}                  
\ndef{\WDixIdeal}{\LpH{1,\mathrm w}}               
\ndef{\DixIdealPos}[1]{\DixIdeal{#1}_+}            
\ndef{\DixIdealN}[1]{\LpN{#1,\infty}}              
\ndef{\DixIdealNPar}[2]{\clL^{#1,\infty}_{#2}(\clN,\tau)}    
\ndef{\DixIdealNPos}[1]{\LpN{#1,\infty}_+}                   
\ndef{\TrD}{\Tr_\omega}                                      
\ndef{\tauD}{{\tau_\omega}}                                  
\ndef{\ILogN}{\frac 1{\log(1+N)}}
\ndef{\DixNorm}[1]{\norm{#1}_{(1,\infty)}}                   
\ndef{\DixInt}[1]{\ints 0^t \mu_s(#1)\,ds}
\ndef{\DixIntL}[1]{\ints 0^{\lambda_{1/t}(#1)}\mu_s(#1)\,ds}
    \ndef{\SmallIdeal}{{\clL^{1, \mathrm w}}}
    \ndef{\SmallIdealMeas}{{\clL^{1, \mathrm w}_m}}
    \ndef{\DixIntII}[1]{\int_0^t \mu_s(#1)\,ds}
    \ndef{\DixIntf}[1]{\Phi_t(#1)}
    \ndef{\DixIntg}[1]{\Psi_t(#1)}

\ndef{\lpi}{\clL^{1,\pi}(\clN,\tau)}

\ndef{\strl}[1]{\stackrel \longrightarrow {#1}}
\ndef{\IIinfty}{$\mathrm{II}_\infty$\ }

\ndef{\fourier}[1]{\clF(#1)}          
\ndef{\HaarMeasBohrs}{\nu}            
\ndef{\BrownsMeas}{\mu}               
\ndef{\BohrCont}[1]{\tilde{#1}}       
\ndef{\APMean}{{M}}                   
\ndef{\CDSS}{{\clA_B}}                
\ndef{\matr}{{\rm Mat}}               
\ndef{\seque}[1]{\ensuremath{\{#1_n\}_{n=1}^\infty}}    
\ndef{\sequen}[2]{\ensuremath{\{#1_#2\}_{#2=1}^\infty}}    
\ndef{\Seque}[1]{\ensuremath{\left(#1_0,#1_1,#1_2,\dots\right)}}    
\ndef{\Cesaro}{H}                           
\ndef{\CesaroRPlus}{M}                      
\ndef{\Dilation}{D}                         
\ndef{\Shift}{T}                            

\ndef{\TrNorm}[1]{\norm{#1}_1}              
\ndef{\HSNorm}[1]{\norm{#1}_2}              
\ndef{\InftyNorm}[1]{\norm{#1}_\infty}      
\ndef{\normQN}[1]{\norm{#1}_{\clQN}}        
\ndef{\clLpnorm}[2]{\norm{#2}_{\clL^{#1}}}    
\ndef{\clLnorm}[1]{\clLpnorm{1}{#1}}    

\ndef{\ccurve}{\gamma}                      

\ndef{\Brs}[1]{\big(#1\big)}                
\ndef{\BRS}[1]{\Big(#1\Big)}                
\ndef{\scal}[2]{\left\langle #1,#2\right\rangle}               
\ndef{\Scal}[1]{\left\langle #1\right\rangle}               
\ndef{\precprec}{\prec\!\!\!\prec}
\ndef{\qeq}{\stackrel?=}
\ndef{\spectrum}[1]{\sigma_{#1}} 
\ndef{\spectruma}[1]{\sigma^{(a)}_{#1}} 
\rndef{\emptyset}{\varnothing}                              
\ndef{\csupp}{c}                           
\ndef{\closure}[1]{\overline{#1}}
\ndef{\linspan}[1]{\mathrm{span}\,{#1}}
\ndef{\bddborel}[1]{B(#1)}                 
\ndef{\charfunc}{\chi}
\ndef{\FrDer}{\euD}                        
\ndef{\LieDer}[1]{\pounds_{#1}\,}          
\ndef{\dds}{\left.\frac d{ds} \right|_{s = 0}}
\ndef{\ortcmp}[1]{#1^{\scriptscriptstyle \perp}}            
\ndef{\Laplace}{\Delta}                    

\ndef{\matrPQ}[3]
{
    \left(
      \begin{array}{cc}
        #1_{11} & #1_{12} \\
        #1_{21} & #1_{22}
      \end{array}
    \right)_{[#2,#3]}
}

\ndef{\margOK}{\marginpar{\bf \small OK}}

\newcounter{margcomcount}
\ndef{\margcom}[1]{\marginpar{\bf \small #1} \addtocounter{margcomcount}{1}
   \index{\indexcom{{\bf COMMENT: #1}}}}

\ndef{\mytimes}{\!\times\!}
\ndef{\sss}[1]{\subsubsection{}\label{#1}}

\rndef{\phi}{\varphi} \ndef{\OpenUnitDisk}{D}
\ndef{\RHS}{RHS}                            
\ndef{\LHS}{LHS} 
\ndef{\ttt}{\Leftrightarrow}
\ndef{\then}{\Rightarrow}
\ndef{\tto}{\longrightarrow}
\ndef{\nno}{\nonumber\\}
\ndef{\newn}[1]{\index{#1} {\bfseries #1}}       
\ndef{\la}{\langle}
\ndef{\ra}{\rangle}
\ndef{\dbar}{{\;\bar{\phantom{o}} \!\!\!\! d}}
\ndef{\stl}[1]{\stackrel{\vbox to 0pt{\vss\hbox{$\scriptstyle #1$}}}}
\ndef{\mathcomment}[1]{{\hfill \qquad\qquad\qquad\text{by (#1)}}}        
\ndef{\mathcomm}[1]{{\hfill \qquad\qquad\qquad\qquad\qquad\text{#1}}}        
\ndef{\details}[1]{\smallskip\begin{center} {\bf Here:}
#1\end{center}\medskip} \ndef{\indexcom}[1]{ --- #1}
\ndef{\longsim}{\ \sim \ }              
\ndef{\tire}{-}              
\ndef{\intinfinf}{\int_{-\infty}^\infty}
     \ndef{\npartial}{\slash\!\!\!\partial}
     \ndef{\Heis}{\operatorname{Heis}}
     \ndef{\Solv}{\operatorname{Solv}}
     \ndef{\Spin}{\operatorname{Spin}}
     \ndef{\SO}{\operatorname{SO}}
     \ndef{\Index}{\operatorname{index}}

             \ndef{\p}{\partial}
             \ndef{\dd}{|\clD|}
             \ndef{\n}{\parallel}

\let\LatexCite=\cite  

\let\ifnumref\iffalse 

\ndef{\ifuncited}[4]{\expandafter\ifx\csname used#4\endcsname\relax}

\ndef{\ifcited}[4]{\expandafter\ifx\csname used#4\endcsname\relax\else}

  \ndef{\papertitle}[1]{ \emph{#1}, }
  \ndef{\paperauthor}[2]{#2}  
  \ndef{\pbbi}[9]{%
      \ifcited{#1}{#2}{#3}{#5}%
        \ifnumref%
          \bibitem{#5}\paperauthor{#1}{#6},\papertitle{#7}#8.%
        \else%
          \advance #9 by 1%
          \ifnum#9<1%
            \bibitem[#4]{#5}\paperauthor{#1}{#6}, \papertitle{#7}#8.%
          \else%
            \bibitem[#4{\the#9}]{#5}\paperauthor{#1}{#6},\papertitle{#7}#8.%
          \fi%
        \fi%
      \fi%
  }
  \ndef{\mbbi}[8]{%
     \ifcited{#1}{#2}{#3}{#5}%
        \ifnumref%
          \bibitem{#5}\paperauthor{#1}{#6},\papertitle{#7}#8.%
        \else%
          \bibitem[#4]{#5}\paperauthor{#1}{#6},\papertitle{#7}#8.%
        \fi%
     \fi%
  }

\ndef{\AddCite}[1]{%
   \ifuncited{0}{0}{0}{#1}%
     \expandafter\gdef\csname used#1\endcsname {}%
   \fi%
}

\def\ProcessCite#1,{%
     \ifx\relax#1%
         \let\next=\relax%
     \else%
         \AddCite{#1}%
         \let\next=\ProcessCite%
     \fi%
     \next%
}

\ndef{\AddCites}[1]{\ProcessCite#1,\relax,}

\ndef{\CiteWithoutExtension}[1]{%
   \AddCites{#1}%
   \LatexCite{#1}%
}

\def\CiteWithExtension[#1]#2{%
   \AddCites{#2}%
   \LatexCite[#1]{#2}%
}

\ndef{\CleverCite}{%
    \ifx\NChar[ %
       \let\MyCite=\CiteWithExtension %
    \else %
       \let\MyCite=\CiteWithoutExtension %
    \fi %
    \MyCite%
}

\renewcommand{\cite}{\futurelet\NChar\CleverCite}

      \ndef{\volume}[1]{{\bf #1}}
      \ndef{\VolYearPP}[3]{\ifnum#2=0 (to appear)\else\volume{#1} (#2), #3\fi}
      \ndef{\VolNoYearPP}[4]{\ifnum#3=0 (to appear)\else\volume{#1} #2 (#3), #4\fi}
      \ndef{\libcode}[1]{}

\ndef{\jnActaMath}[3]{Acta Math. \VolYearPP{#1}{#2}{#3}}                       
\ndef{\jnAdvMath}[3]{Adv.\,in~Math. \VolYearPP{#1}{#2}{#3}}                     
\ndef{\jnAlgAnal}[3]{Algebra i~Analiz \VolYearPP{#1}{#2}{#3}}
\ndef{\jnAmerJMath}[3]{Amer.\,J.\,Math. \VolYearPP{#1}{#2}{#3}}                  
\ndef{\jnAmerMathMonth}[3]{Amer.\,Math.\,Monthly \VolYearPP{#1}{#2}{#3}}         
\ndef{\jnAnnMath}[4]{Ann. of~Math. \VolNoYearPP{#1}{#2}{#3}{#4}}               
\ndef{\jnAnalMath}[3]{J. Anal. Math. \VolYearPP{#1}{#2}{#3}}                   
\ndef{\jnArchRatMechAnal}[3]{Arch. Rational Mech. Anal. \VolYearPP{#1}{#2}{#3}}                   
\ndef{\jnBullLondMathSoc}[3]{Bull. London Math. Soc. \VolYearPP{#1}{#2}{#3}}   
\ndef{\jnBullAMS}[3]{Bull. Amer. Math. Soc. \VolYearPP{#1}{#2}{#3}}   
\ndef{\jnCanMathBull}[3]{Canad. Math. Bull. \VolYearPP{#1}{#2}{#3}}            
\ndef{\jnCanMath}[3]{Canad. J.~Math. \VolYearPP{#1}{#2}{#3}}             
\ndef{\jnCommMathPhys}[3]{Comm. Math. Phys. \VolYearPP{#1}{#2}{#3}}             
\ndef{\jnCommPDE}[3]{Comm. Partial Differential Equations \VolYearPP{#1}{#2}{#3}}             
\ndef{\jnComptRendue}[3]{C.\,R.~Acad. Sci. Paris S\'er. A-B \VolYearPP{#1}{#2}{#3}}      
\ndef{\jnContMath}[3]{Contemporary Math. \VolYearPP{#1}{#2}{#3}}               %
\ndef{\jnDukeMJ}[3]{Duke Math. J. \VolYearPP{#1}{#2}{#3}}
\ndef{\jnDiffGeom}[3]{J.~Diff. Geom. \VolYearPP{#1}{#2}{#3}}                   
\ndef{\jnErgodicTheory}[3]{Ergodic Theory and Dynamical Systems \VolYearPP{#1}{#2}{#3}} 
\ndef{\jnFuncAnal}[3]{J.~Functional Analysis \VolYearPP{#1}{#2}{#3}}           
\ndef{\jnFunkAnalPril}[4]{Funct. Anal. Appl. \VolNoYearPP{#1}{#2}{#3}{#4}}  
\ndef{\jnGAFA}[3]{GAFA \VolYearPP{#1}{#2}{#3}}                                 
\ndef{\jnIHES}[3]{IHES Publ. Math. (Paris) \VolYearPP{#1}{#2}{#3}}             
\ndef{\jnIEOT}[3]{Integral Equations Operator Theory   \VolYearPP{#1}{#2}{#3}} 
\ndef{\jnIsrMath}[3]{Israel J.~Math. \VolYearPP{#1}{#2}{#3}}                   
\ndef{\jnKTheory}[3]{K-Theory \VolYearPP{#1}{#2}{#3}}                          
\ndef{\jnLetMathPhys}[3]{Lett. Math. Phys. \VolYearPP{#1}{#2}{#3}}             
\ndef{\jnMathAnn}[3]{Math. Ann. \VolYearPP{#1}{#2}{#3}}                        
\ndef{\jnMathAnalAppl}[3]{J.~Math. Anal. and Appl. \VolYearPP{#1}{#2}{#3}}     
\ndef{\jnMathNachr}[3]{Math.\,Nachr. \VolYearPP{#1}{#2}{#3}}
\ndef{\jnMathPhys}[3]{J. Math. Phys. \VolYearPP{#1}{#2}{#3}}
\ndef{\jnMathSocJap}[3]{J. Math. Soc. Japan \VolYearPP{#1}{#2}{#3}}
\ndef{\jnOperTheory}[3]{J.~Operator Theory \VolYearPP{#1}{#2}{#3}}             
\ndef{\jnPacJMath}[3]{Pacific J.~Math. \VolYearPP{#1}{#2}{#3}}                  
\ndef{\jnPositivity}[3]{Positivity \VolYearPP{#1}{#2}{#3}}
\ndef{\jnProcAmerMS}[3]{Proc. Amer. Math. Soc. \VolYearPP{#1}{#2}{#3}}         
\ndef{\jnProcCambPhilSoc}[3]{Math. Proc. Camb. Phil. Soc. \VolYearPP{#1}{#2}{#3}}
\ndef{\jnReineAngew}[3]{J.~Reine Angew. Math. \VolYearPP{#1}{#2}{#3}}          
\ndef{\jnTokyoMath}[3]{Tokyo J.~Math. \VolYearPP{#1}{#2}{#3}}
\ndef{\jnTopology}[3]{Topology \VolYearPP{#1}{#2}{#3}}
\ndef{\jnTransAmerMathSoc}[3]{Trans. Amer. Math. Soc. \VolYearPP{#1}{#2}{#3}}
\ndef{\jnIzvANSSSR}[3]{Izv. Akad. Nauk SSSR, Ser. Mat. \VolYearPP{#1}{#2}{#3}}
\ndef{\jnIzvVyshUchZav}[3]{Izv. Vyssh. Uch. Zav., Mat. \VolYearPP{#1}{#2}{#3} (Russian)}
\ndef{\jnIzdatLenUniv}[2]{Izdat. Leningrad. Univ., Leningrad, (#1), #2 (Russian)}
\ndef{\jnFieldsInsComm}[3]{Fields Inst. Comm. \VolYearPP{#1}{#2}{#3}}
\ndef{\jnDoklANSSSR}[3]{Dokl. Akad. Nauk SSSR \VolYearPP{#1}{#2}{#3}}
\ndef{\jnMatZametki}[3]{Matem. zametki \VolYearPP{#1}{#2}{#3}}
\ndef{\jnRussMathSurvey}[3]{Russian Math. Surveys \VolYearPP{#1}{#2}{#3}}
\ndef{\jnSibMathJ}[3]{Sib. Math.~J. \VolYearPP{#1}{#2}{#3}}
\ndef{\jnSovMath}[3]{J.~Soviet math. \VolYearPP{#1}{#2}{#3}}
\ndef{\jnTransMoscMathSoc}[3]{Trans. Moscow Math. Soc. \VolYearPP{#1}{#2}{#3}}
\ndef{\jnUMN}[3]{Uspekhi Mat. Nauk \VolYearPP{#1}{#2}{#3}}

\ndef{\bkTransMathMon}[2]{Trans. Math. Monographs, AMS, \volume{#1}, #2}

\ndef{\pbBirkhauser}[1]{Birkh\"auser, Boston, #1}
\ndef{\pbFactorial}[1]{Moscow, Factorial, #1}
\ndef{\pbGauthier}[1]{Gauthier-Villars, Paris, #1}
\ndef{\pbNauka}[1]{Moscow, Nauka, #1 (Russian)}
\ndef{\pbNaukaR}[1]{Москва, Наука, #1}
\ndef{\pbPrinceton}[1]{Princeton University Press, Princeton, New Jersey, #1}
\ndef{\pbPublPerish}[1]{Publish or Perish Inc., Berkeley, #1}
\ndef{\pbSpringer}[1]{Springer-Verlag, #1}

\ndef{\myauthor}[1]{\mbox{#1}}

\ndef{\Agmon}{\myauthor{Sh.\,Agmon}}
\ndef{\Ahiezer}{\myauthor{N.\,I.\,Ahiezer}}
\ndef{\Arazy}{\myauthor{J.\,Arazy}}
\ndef{\Aronszajn}{\myauthor{N.\,Aronszajn}}
\ndef{\Astashkin}{\myauthor{S.\,V.\,Astashkin}}
\ndef{\Atiyah}{\myauthor{M.\,Atiyah}}
\ndef{\Avron}{\myauthor{J.\,E.\,Avron}}
\ndef{\Azamov}{\myauthor{N.\,A.\,Azamov}}
\ndef{\Banach}{\myauthor{S.\,Banach}}
\ndef{\Benameur}{\myauthor{M-T.\,Benameur}}
\ndef{\Bennett}{\myauthor{C.\,Bennett}}
\ndef{\Berezin}{\myauthor{F.\,A.\,Berezin}}
\ndef{\Berline}{\myauthor{N.\,Berline}}
\ndef{\Birman}{\myauthor{M.\,Sh.\,Birman}}
\ndef{\Blackadar}{\myauthor{B.\,Blackadar}}
\ndef{\Bogolyubov}{\myauthor{N.\,N.\,Bogolyubov}}
\ndef{\Bonsall}{\myauthor{F.\,F.\,Bonsall}}
\ndef{\Bony}{\myauthor{J.\,F.\,Bony}}
\ndef{\BoosBavnbek}{\myauthor{B.\,Boo$\beta$-Bavnbek}}
\ndef{\Bott}{\myauthor{R.\,Bott}}
\ndef{\Branges}{\myauthor{L.\,de Branges}}
\ndef{\Bratteli}{\myauthor{O.\,Bratteli}}
\ndef{\Bredon}{\myauthor{G.\,E.\,Bredon}}
\ndef{\Breuer}{\myauthor{M.\,Breuer}}
\ndef{\Brown}{\myauthor{L.\,G.\,Brown}}
\ndef{\Bruneau}{\myauthor{V.\,Bruneau}}
\ndef{\Buslaev}{\myauthor{V.\,S.\,Buslaev}}
\ndef{\Carey}{\myauthor{A.\,L.\,Carey}}
\ndef{\CareyRW}{\myauthor{R.\,W.\,Carey}} 
\ndef{\Cartan}{\myauthor{H.\,Cartan}}
\ndef{\Chilin}{\myauthor{V.\,I.\,Chilin}}
\ndef{\Coburn}{\myauthor{L.\,A.\,Coburn}}
\ndef{\Connes}{\myauthor{A.\,Connes}}
\ndef{\Cornfeld}{\myauthor{I.\,P.\,Cornfeld}}
\ndef{\Daletskii}{\myauthor{Yu.\,L.\,Daletski\u\i}}   
\ndef{\Dixmier}{\myauthor{J.\,Dixmier}}
\ndef{\DoddsPG}{\myauthor{P.\,G.\,Dodds}}
\ndef{\DoddsTK}{\myauthor{T.\,K.\,Dodds}}
\ndef{\Douglas}{\myauthor{R.\,G.\,Douglas}}
\ndef{\Dubrovin}{\myauthor{B.\,A.\,Dubrovin}}
\ndef{\Dugundji}{\myauthor{J.\,Dugundji}}
\ndef{\Duncan}{\myauthor{J.\,Duncan}}
\ndef{\Dunford}{\myauthor{N.\,Dunford}}
\ndef{\Dykema}{\myauthor{K.\,J.\,Dykema}}
\ndef{\Edwards}{\myauthor{R.\,E.\,Edwards}}
\ndef{\Eilenberg}{\myauthor{S.\,Eilenberg}}
\ndef{\Entina}{\myauthor{S.\,B.\,\`Entina}}
\ndef{\Fack}{\myauthor{T.\,Fack}} 
\ndef{\Faddeev}{\myauthor{L.\,D.\,Faddeev}}
\ndef{\Farber}{\myauthor{M.\,Farber}}
\ndef{\Farforovskaya}{\myauthor{Yu.\,B.\,Farforovskaya}}
\ndef{\Federer}{\myauthor{H.\,Federer}}
\ndef{\Fedosov}{\myauthor{B.\,V.\,Fedosov}}
\ndef{\Figiel}{\myauthor{T.\,Figiel}} 
\ndef{\Figueroa}{\myauthor{H.\,Figueroa}}
\ndef{\Fillmore}{\myauthor{P.\,A.\,Fillmore}}
\ndef{\Fomenko}{\myauthor{A.\,T.\,Fomenko}} 
\ndef{\Fomin}{\myauthor{S.\,V.\,Fomin}}
\ndef{\Frohlich}{\myauthor{J.\,Fr\"ohlich}}
\ndef{\Fuglede}{\myauthor{B.\,Fuglede}}
\ndef{\Furutani}{\myauthor{K.\,Furutani}}
\ndef{\Gelfand}{\myauthor{I.\,M.\,Gelfand}}
\ndef{\Gesztesy}{\myauthor{F.\,Gesztesy}}     
\ndef{\Getzler}{\myauthor{E.\,Getzler}} 
\ndef{\Gilkey}{\myauthor{P.\,B.\,Gilkey}}
\ndef{\Gitler}{\myauthor{S.\,Gitler}}
\ndef{\Glazman}{\myauthor{I.\,M.\,Glazman}}
\ndef{\Glimm}{\myauthor{J.\,Glimm}}
\ndef{\Gohberg}{\myauthor{I.\,C.\,Gohberg}}
\ndef{\Goldshtein}{\myauthor{Ya.\,Goldshtein}}
\ndef{\Golze}{\myauthor{F.\,Golze}}
\ndef{\GraciaBondia}{\myauthor{J.\,M.\,Gracia-Bond\'{i}a}}
\ndef{\Greenleaf}{\myauthor{F.\,P.\,Greenleaf}}
\ndef{\Gromov}{\myauthor{M.\,Gromov}}
\ndef{\Gunning}{\myauthor{R.\,C.\,Gunning}}
\ndef{\Haagerup}{\myauthor{U.\,Haagerup}}
\ndef{\Haag}{\myauthor{R.\,Haag}}
\ndef{\Halmos}{\myauthor{P.\,R.\,Halmos}}
\ndef{\Hardy}{\myauthor{G.\,H.\,Hardy}}
\ndef{\Herbst}{\myauthor{I.\,W.\,Herbst}}
\ndef{\Higson}{\myauthor{N.\,Higson}}  
\ndef{\Hoermander}{\myauthor{L.\,H\"ormander}} 
\ndef{\Hoffman}{\myauthor{K.\,Hoffman}} 
\ndef{\Ito}{\myauthor{K.\,Ito}}
\ndef{\Ikebe}{\myauthor{T.\,Ikebe}}
\ndef{\Jaffe}{\myauthor{A.\,Jaffe}}
\ndef{\James}{\myauthor{I.\,M.\,James}}
\ndef{\Javrjan}{\myauthor{V.\,A.\,Javrjan}}
\ndef{\Jitomirskaya}{\myauthor{S.\,Jitomirskaya}}
\ndef{\Kadison}{\myauthor{R.\,V.\,Kadison}}
\ndef{\Kalton}{\myauthor{N.\,J.\,Kalton}} 
\ndef{\Kato}{\myauthor{T.\,Kato}} 
\ndef{\Kobayashi}{\myauthor{S.\,Kobayashi}}
\ndef{\Koplienko}{\myauthor{L.\,S.\,Koplienko}}
\ndef{\Korotyaev}{\myauthor{E.\,Korotyaev}}
\ndef{\Kosaki}{\myauthor{H.\,Kosaki}}
\ndef{\Kostrykin}{\myauthor{V.\,Kostrykin}}
\ndef{\Kotani}{\myauthor{S.\,Kotani}}
\ndef{\Krein}{\myauthor{Kre\u\i n}}
\ndef{\KreinMG}{\myauthor{M.\,G.\,Kre\u\i n}}
\ndef{\KreinSG}{\myauthor{S.\,G.\,Kre\u\i n}}
\ndef{\Kuroda}{\myauthor{S.\,T.\,Kuroda}}
\ndef{\Leichtnam}{\myauthor{E.\,Leichtnam}}
\ndef{\Lesch}{\myauthor{M.\,Lesch}}
\ndef{\Lesniewski}{\myauthor{A.\,Lesniewski}}
\ndef{\Levitan}{\myauthor{B.\,M.\,Levitan}}
\ndef{\Lidskii}{\myauthor{V.\,B.\,Lidskii}}
\ndef{\Lifshitz}{\myauthor{I.\,M.\,Lifshitz}}
\ndef{\Lindenstrauss}{\myauthor{J.\,Lindenstrauss}}
\ndef{\Loday}{\myauthor{J.-L.\,Loday}}
\ndef{\Lord}{\myauthor{S.\,Lord}}      
\ndef{\Lorentz}{\myauthor{G.\,Lorentz}}
\ndef{\Magnus}{\myauthor{W.\,Magnus}}
\ndef{\Makarov}{\myauthor{K.\,A.\,Makarov}}
\ndef{\MakarovN}{\myauthor{N.\,Makarov}}
\ndef{\Mathai}{\myauthor{V.\,Mathai}}         
\ndef{\McKean}{\myauthor{H.\,P.\,McKean}}
\ndef{\Mishchenko}{\myauthor{A.\,S.\,Mishchenko}}
\ndef{\Molchanov}{\myauthor{S.\,A.\,Molchanov}}
\ndef{\Moore}{\myauthor{C.\,C.\,Moore}}
\ndef{\Moscovici}{\myauthor{H.\,Moscovici}}  
\ndef{\Motovilov}{\myauthor{A.\,K.\,Motovilov}}
\ndef{\Moyer}{\myauthor{R.\,D.\,Moyer}}
\ndef{\Naboko}{\myauthor{S.\,N.\,Naboko}}
\ndef{\Narasimhan}{\myauthor{R.\,Narasimhan}}
\ndef{\Nomizu}{\myauthor{K.\,Nomizu}}
\ndef{\Novikov}{\myauthor{S.\,P.\,Novikov}}
\ndef{\Osterwalder}{\myauthor{K.\,Osterwalder}}
\ndef{\Patodi}{\myauthor{V.\,Patodi}}
\ndef{\Pagter}{\myauthor{B.\,de~Pagter}}  
\ndef{\Pastur}{\myauthor{L.\,A.\,Pastur}}  
\ndef{\Pavlov}{\myauthor{B.\,S.\,Pavlov}}
\ndef{\Pedersen}{\myauthor{G.\,K.\,Pedersen}}
\ndef{\Peller}{\myauthor{V.\,V.\,Peller}}
\ndef{\Perera}{\myauthor{V.\,S.\,Perera}}
\ndef{\Petunin}{\myauthor{Ju.\,I.\,Petunin}}
\ndef{\Phillips}{\myauthor{J.\,Phillips}}  
\ndef{\Piazza}{\myauthor{P.\,Piazza}}   
\ndef{\Pincus}{\myauthor{J.\,D.\,Pincus}}   
\ndef{\Poincare}{Poincar\'e}
\ndef{\Postnikov}{\myauthor{M.\,M.\,Postnikov}} 
\ndef{\Povzner}{\myauthor{A.\,Ya.\,Povzner}}
\ndef{\Prinzis}{\myauthor{R.\,Prinzis}}
\ndef{\Privalov}{\myauthor{I.\,I.\,Privalov}}
\ndef{\Pushnitski}{\myauthor{A.\,B.\,Pushnitski}} 
\ndef{\Raeburn}{\myauthor{I.\,Raeburn}}
\ndef{\Raikov}{\myauthor{G.\,Raikov}}
\ndef{\Reed}{\myauthor{M.\,Reed}}
\ndef{\Rennie}{\myauthor{A.\,Rennie}}
\ndef{\Rickart}{\myauthor{C.\,E.\,Rickart}}
\ndef{\Riesz}{\myauthor{F.\,Riesz}}
\ndef{\Ringrose}{\myauthor{J.\,Ringrose}}
\ndef{\Rio}{\myauthor{R.\,del Rio}}
\ndef{\Robinson}{\myauthor{D.\,Robinson}}
\ndef{\Rossi}{\myauthor{H.\,Rossi}}
\ndef{\Rudin}{\myauthor{W.\,Rudin}}
\ndef{\Ruelle}{\myauthor{D.\,Ruelle}}
\ndef{\Ruzhansky}{\myauthor{M.\,Ruzhansky}}
\ndef{\Sakai}{\myauthor{Sh.\,Sakai}}
\ndef{\Sargsjan}{\myauthor{I.\,S.\,Sargsjan}}
\ndef{\Sato}{\myauthor{H.\,Sato}}
\ndef{\Schaeffer}{\myauthor{D.\,G.\,Schaeffer}}
\ndef{\Schluchtermann}{\myauthor{G.\,Schluchtermann}}
\ndef{\Schochet}{\myauthor{C.\,Schochet}}
\ndef{\SchroedingerE}{\myauthor{E.\,Schr\"odinger}}
\ndef{\Schroedinger}{\myauthor{Schr\"odinger}}
\ndef{\Schrohe}{\myauthor{E.\,Schrohe}}
\ndef{\Schwartz}{\myauthor{J.\,T.\,Schwartz}}
\ndef{\Sedaev}{\myauthor{A.\,A.\,Sedaev}}
\ndef{\Seiler}{\myauthor{R.\,Seiler}}
\ndef{\Semenov}{\myauthor{E.\,M.\,Semenov}}
\ndef{\Shabat}{\myauthor{B.\,V.\,Shabat}}
\ndef{\Shafarevich}{\myauthor{I.\,R.\,Shafarevich}}
\ndef{\Sharpley}{\myauthor{R.\,Sharpley}}
\ndef{\Shilov}{\myauthor{G.\,E.\,Shilov}}
\ndef{\Shirkov}{\myauthor{D.\,V.\,Shirkov}}
\ndef{\Shubin}{\myauthor{M.\,A.\,Shubin}}
\ndef{\Silverman}{\myauthor{H.\,Silverman}}
\ndef{\Simon}{\myauthor{B.\,Simon}}
\ndef{\Sinai}{\myauthor{Ya.\,G.\,Sinai}}
\ndef{\Singer}{\myauthor{I.\,M.\,Singer}}
\ndef{\Solomyak}{\myauthor{M.\,Z.\,Solomyak}}
\ndef{\Soloviev}{\myauthor{Yu.\,P.\,Soloviev}}
\ndef{\Spivak}{\myauthor{M.\,Spivak}}
\ndef{\Stein}{\myauthor{E.\,M.\,Stein}}
\ndef{\Stenkin}{\myauthor{V.\,V.\,Sten'kin}}
\ndef{\Stratila}{\myauthor{S.\,Stratila}}
\ndef{\Sucheston}{\myauthor{L.\,Sucheston}}
\ndef{\Sukochev}{\myauthor{F.\,A.\,Sukochev}}
\ndef{\Switzer}{\myauthor{R.\,M.\,Switzer}}
\ndef{\SzNagy}{\myauthor{B.\,Sz.-Nagy}}
\ndef{\Takesaki}{\myauthor{M.\,Takesaki}}
\ndef{\Taylor}{\myauthor{M.\,E.\,Taylor}}
\ndef{\Treves}{\myauthor{F.\,Treves}}
\ndef{\Troitsky}{\myauthor{E.\,V.\,Troitsky}}
\ndef{\Tzafriri}{\myauthor{L.\,Tzafriri}}
\ndef{\Varilly}{\myauthor{J.\,C.\,V\'{a}rilly}}
\ndef{\Vergne}{\myauthor{M.\,Vergne}}
\ndef{\Vladimirov}{\myauthor{V.\,S.\,Vladimirov}}
\ndef{\Voiculescu}{\myauthor{D.\,Voiculescu}}
\ndef{\Weiss}{\myauthor{G.\,Weiss}}
\ndef{\Wells}{\myauthor{R.\,O.\,Wells}}
\ndef{\Williams}{\myauthor{J.\,P.\,Williams}}
\ndef{\Winkler}{\myauthor{S.\,Winkler}}
\ndef{\Witten}{\myauthor{E.\,Witten}}
\ndef{\Wodzicki}{\myauthor{M.\,Wodzicki}}
\ndef{\Wojciechowski}{\myauthor{K.\,P.\,Wojciechowski}}
\ndef{\Yafaev}{\myauthor{D.\,R.\,Yafaev}}
\ndef{\Yosida}{\myauthor{K.\,Yosida}}
\ndef{\Zsido}{\myauthor{L.\,Zsido}}

\usepackage{graphicx}

\ndef{\mbCz}{\mbC^{(z)}} \ndef{\mbCr}{\mbC^{(r)}} \ndef{\yy}{y}

\newcommand{\LambHF}[1]{\Lambda(#1;F)}
\ndef{\ells}{{\ell_2}} \ndef{\hlambda}{{\mathfrak h_\lambda}}
\ndef{\hlambdao}{{\mathfrak h_\lambda^{(0)}}}
\ndef{\hlambdas}{{\mathfrak h_\lambda^{(s)}}}
\ndef{\hlambdar}{{\mathfrak h_\lambda^{(r)}}}
\begin{document}
\title[A.c. and singular spectral shift functions]{Absolutely continuous and singular \\ spectral shift functions}
\author{Nurulla Azamov}
\address{School of Computer Science, Engineering and Mathematics
   \\ Flinders University
   \\ Bedford Park, 5042, SA Australia.}
\email{nurulla.azamov@flinders.edu.au}
 \keywords{Spectral shift function, scattering matrix, absolutely continuous spectral shift function,
 singular spectral shift function, Birman-Krein formula, Pushnitski $\mu$-invariant, infinitesimal spectral flow, infinitesimal scattering matrix}
 \subjclass[2000]{ 
     Primary 47A55; 
     Secondary 47A11
 }
\begin{abstract} Given a self-adjoint operator~$H_0,$ a self-adjoint trace-class operator~$V$
and a fixed Hilbert-Schmidt operator~$F$
with trivial kernel and co-kernel, using the limiting absorption principle an explicit set
of full Lebesgue measure~$\LambHF{H_0} \subset \mbR$ is defined, 
such that for all points~$\lambda$ of the set~$\LambHF{H_0+rV} \cap \LambHF{H_0},$ where~$r \in \mbR,$
the wave~$w_\pm(\lambda; H_0+rV,H_0)$ and the scattering matrices
$S(\lambda; H_0+rV,H_0)$ can be defined unambiguously. Many well-known properties of the wave and scattering matrices
and operators are proved, including the stationary formula for the scattering matrix.
This version of abstract scattering theory allows, in particular, to prove that
$$
  \det S(\lambda;H_0+V,H_0) = e^{-2\pi i \xia(\lambda)}, \ \ \text{a.e.} \ \lambda \in \mbR,
$$
where~$\xia(\lambda) = \xia_{H_0+V,H_0}(\lambda)$ is the so called absolutely continuous part of the spectral shift function
defined by
$$
  \xia_{H_0+V,H_0}(\lambda) := \frac d{d\lambda} \int_0^1 \Tr(V E^{(a)}_{H_0+rV}(\lambda)) \,dr
$$
and where~$E_H^{(a)}(\lambda)=E^{(a)}_{(-\infty,\lambda)}(H)$ denotes the absolutely continuous part of the spectral projection.
Combined with the Birman-Krein formula, this implies that the singular part of the spectral shift function
$$
  \xis_{H_0+V,H_0}(\lambda) := \frac d{d\lambda} \int_0^1 \Tr(V E^{(s)}_{H_0+rV}(\lambda)) \,dr
$$
is an almost everywhere integer-valued function, where~$E_H^{(s)}(\lambda)=E^{(s)}_{(-\infty,\lambda)}(H)$ denotes the singular part of the spectral projection.
%
%
\end{abstract}

\maketitle

%
%
%
%
%
%
%
%
%
%
%
%
%
%
%

\newpage
\tableofcontents


%
%
%
%
%

\section{Introduction}
\subsection{Short summary}
In this paper a new approach is given to abstract scattering theory.
This approach is constructive and allows to prove new results in perturbation theory
of continuous spectra of self-adjoint operators which the conventional scattering theory
is not able to achieve.

Among the results of this paper are: for trace-class perturbations
of arbitrary self-adjoint operators:

$\bullet$ A new approach to the spectral theorem for self-adjoint
operators (without singular continuous spectrum) via a special
constructive representation of the absolutely continuous part
(with respect to a fixed self-adjoint operator) of the Hilbert
space as a direct integral of fiber Hilbert spaces.

$\bullet$ A new and constructive proof of existence of the wave matrices and of the wave operators.

$\bullet$ A new proof of the multiplicativity property of the wave matrices and of the wave operators.

$\bullet$ A new and constructive proof of the existence of the scattering matrix and of the scattering operator.

$\bullet$ A new proof of the stationary formula for the scattering matrix.

$\bullet$ A new proof of the Kato-Rosenblum theorem.

\bigskip

This paper does not contain only new proofs of existing theorems.

$\bullet$ A new formula (to the best knowledge of the author) for the scattering matrix in terms of chronological exponential.

The main result of this paper is the following

{\bf Theorem}. Let $H_0$ be a self-adjoint operator and let $V$ be a trace-class self-adjoint operator
in a Hilbert space $\hilb.$ Define a generalized function
$$
  \xis(\phi) = \int_0^1 \Tr\brs{V \phi(H_r^{(s)})}\,dr, \quad \phi \in C_c^\infty(\mbR),
$$
where $H_r:=H_0+rV,$ and $H_r^{(s)}$ is the singular part of the self-adjoint operator $H_r.$
Then $\xis$ is an absolutely continuous measure and its density $\xis(\lambda)$ (denoted by the same symbol!)
is a.e. integer-valued.

Note that in the case of operators with compact resolvent this theorem is well known, and the
function $\xis(\lambda)$ in this case coincides with spectral flow \cite{APS75,APS76,Ge93Top,Ph96CMB,Ph97FIC,CP98CJM,CP2,ACDS,ACS,Azbook}.
Spectral flow is integer-valued just by definition as a total Fredholm index of a path of operators.
In the case of operators with compact resolvent instead of $H_r^{(s)}$ one writes $H_r,$ since in this case
the continuous spectrum is absent, so that $H_r^{(s)} = H_r.$

The above theorem strongly suggests that the function $\xis(\lambda),$ which I call the singular part of the spectral shift function,
calculates the spectral flow of the singular spectrum even in the presence and inside of the absolutely continuous spectrum.

Finally, it is worth to stress that the new approach to abstract scattering theory given in this paper has been invented
with the sole purpose to prove the above theorem. Existing versions of scattering theory turned out to be insufficient for this purpose.
At the same time, this approach seems to have a value of its own.
In particular, I believe that, properly adjusted, this approach may allow to unify
the trace-class and smooth scattering theories, --- a long-standing problem mentioned in the introduction of D.\,Yafaev's book \cite{Ya}.

%

\subsection{Introduction}
%
Let~$H_0$ be a self-adjoint operator,~$V$ be a self-adjoint trace-class operator and let~$H_1 = H_0+V.$
The Lifshits-\Krein\ spectral shift function~\cite{Li52UMN,Kr53MS} is the unique~$L_1$-function
$\xi(\cdot) = \xi_{H_1,H_0}(\cdot),$ such that for all~$f \in \Cc$ the trace formula
$$
  \Tr(f(H_1) - f(H_0)) = \int f'(\lambda)\xi_{H_1,H_0}(\lambda)\, d\lambda
$$
holds. The Birman-Solomyak formula for the spectral shift function~\cite{BS75SM} asserts that
\begin{equation} \label{Int: BS formula}
  \xi_{H_1,H_0}(\lambda) = \frac d{d\lambda} \int_0^1 \Tr(V E^{H_r}_\lambda)\,dr, \quad \text{a.e.} \  \lambda,
\end{equation}
where
$$
  H_r = H_0 + rV,
$$
and~$E^{H_r}_\lambda$ is the spectral resolution of~$H_r.$
This formula was established by \Javrjan\ in~\cite{Jav} in the case of perturbations of
the boundary condition of a Sturm-Liouville operator on~$[0,\infty),$ which corresponds to
rank-one perturbation of~$H_0.$
The Birman-Solomyak formula is also called the spectral averaging formula.
A simple proof of this formula was found in~\cite{Si98PAMS}.
There is enormous literature on the subject of spectral averaging, cf. e.g.
\cite{GM03AA,GM00JAnalM,Ko86CM} and references therein. A survey on the spectral shift function
can be found in~\cite{BP98IEOT}.

Let~$S(\lambda; H_1, H_0)$ be the scattering matrix of the pair~$H_0,H_1 = H_0+V$ (cf.~\cite{BE}, see also \cite{Ya}).
In~\cite{BK62DAN} \Birman\ and \KreinMG\ established the following formula
\begin{equation} \label{Int: BK formula}
  \det S(\lambda; H_1, H_0) = e^{-2 \pi i \xi(\lambda)} \quad \text{a.e.} \ \lambda \in \mbR
\end{equation}
for trace-class perturbations~$V = H_1-H_0$ and arbitrary self-adjoint operators~$H_0.$
This formula is a generalization of a similar result of \Buslaev\ and \Faddeev\ ~\cite{BF60DAN}
for Sturm-Liouville operators on~$[0,\infty).$

In~\cite{Az} I introduced the absolutely continuous and singular
spectral shift functions by the formulae
\begin{equation} \label{Int: def of xia}
  \xia_{H_1,H_0}(\lambda) = \frac d{d\lambda} \int_0^1 \Tr\brs{V E^{H_r}_\lambda P^{(a)}(H_r)}\,dr, \quad \text{a.e.} \  \lambda,
\end{equation}
\begin{equation} \label{Int: def of xis}
  \xis_{H_1,H_0}(\lambda) = \frac d{d\lambda} \int_0^1 \Tr\brs{VE^{H_r}_\lambda P^{(s)}(H_r)}\,dr, \quad \text{a.e.} \  \lambda,
\end{equation}
where~$P^{(a)}(H_r)$ (respectively,~$P^{(s)}(H_r)$) is the projection onto the absolutely continuous
(respectively, singular) subspace of~$H_r.$
These formulae are obvious modifications of the Birman-Solomyak spectral averaging formula, and one can see that
$$
  \xi = \xis + \xia.
$$
In~\cite{Az} it was observed that for $n$-dimensional Schr\"odinger operators~$H_r = -\Delta + rV$
with quickly decreasing potentials~$V$ the scattering matrix~$S(\lambda;H_r,H_0)$ is a continuous operator-valued
function of~$r$ and
it was shown that 
\begin{equation} \label{Int: xia = -2pi i det S}
  -2\pi i \xia_{H_r,H_0}(\lambda) = \log \det S(\lambda;H_r,H_0),
\end{equation}
where the logarithm is defined in such a way that the function
$$
  [0,r] \ni s \mapsto \log \det S(\lambda;H_s,H_0)
$$
is continuous. It was natural to conjecture that some variant of this formula
should hold in general case. In particular, this formula, compared with the Birman-Krein formula (\ref{Int: BK formula}), has naturally led
to a conjecture that the singular part of the spectral shift function is an a.e. integer-valued function.
In case of $n$-dimensional Schr\"odinger operators with quickly decreasing potentials
this is an obvious result, since these operators do not have singular spectrum on the positive semi-axis.
In~\cite{Az2} it was observed that even in the case of operators which admit embedded eigenvalues
the singular part of the spectral shift function is also either equal to zero on the positive semi-axis
or in any case it is integer-valued.

In this paper I give positive solution of this conjecture for
trace-class perturbations of arbitrary self-adjoint operators.

The proof of (\ref{Int: xia = -2pi i det S}) is based on the following formula for the scattering matrix
\begin{equation} \label{Int: S = T exp...}
  S(\lambda;H_r,H_0) = \Texp\brs{-2\pi i \int_0^r w_+(\lambda;H_0,H_s)\Pi_{H_s}(V)(\lambda)w_+(\lambda;H_s,H_0)\,ds},
\end{equation}
where $\Pi_{H_s}(V)(\lambda)$ is the so-called infinitesimal scattering matrix (see (\ref{F: Pi = ZJZ*})).
If~$\lambda$ is fixed, then for this formula to make sense,
the wave matrix $w_+(\lambda;H_s,H_0)$ has to be defined for all~$s \in [0,r],$
except possibly a discrete set. In the case of Schr\"odinger operators
$$
  H=-\Delta + V
$$
in~$\mbR^n$ with short range potentials (in the sense
of~\cite{Agm}), the wave matrices $w_\pm(\lambda; H_s,H_0)$ are
well-defined, since there are explicit formulae for them, cf.
e.g.~\cite{Agm,BYa92AA2,Kur,KuJMSJ73I,KuJMSJ73II}. For example,
if~$\lambda$ does not belong to the discrete set~$e_+(H)$ of
embedded eigenvalues of~$H,$ then the scattering
matrix~$S(\lambda)$ exists as an operator from~$L_2(\Sigma)$
to~$L_2(\Sigma),$ where $\Sigma = \set{\omega \in \mbR^n\colon
\abs{\omega} = 1}$ (cf. e.g.~\cite[Theorem 7.2]{Agm}).

The situation is quite different in the case of the main setting of abstract scattering
theory~\cite{BW,BE,RS3,Ya}, which considers trace-class perturbations of arbitrary self-adjoint operators.
A careful reading of proofs in~\cite{BE,Ya} shows that
one takes an arbitrary core of the spectrum
of the initial operator~$H_0$ and during the proofs
one throws away from a core of the spectrum
several finite and even countable families of null sets. Furthermore, the nature of the initial core of the spectrum
and the nature of the null sets being thrown away are not clarified. They depend on arbitrarily chosen objects.
This is in sharp contrast with potential scattering theory, where non-existence of the wave matrix or the scattering matrix
at some point~$\lambda$ of the absolutely continuous spectrum means that~$\lambda$ is an embedded eigenvalue, cf. e.g.~\cite{Agm}.

So, in the case of trace-class perturbations of arbitrary
self-adjoint operators, given a fixed~$\lambda$ (from some
predefined full set~$\Lambda$) the existence of the wave matrix
for all~$r \in [0,1],$ except possibly a discrete set, cannot be
established by usual means. In order to make the argument of the
proof of (\ref{Int: S = T exp...}), given in~\cite{Az}, work for
trace-class (to begin with) perturbations of arbitrary
self-adjoint operators, one at least needs to give an explicit set
of full measure~$\Lambda,$ such that for all~$\lambda$
from~$\Lambda$ all the necessary ingredients of scattering theory,
such as~$w_\pm(\lambda; H_r,H_0),$ $S(\lambda;H_r,H_0)$
and~$Z(\lambda;G)$ exist. One of the difficulties here is that the
spectrum of an arbitrary self-adjoint operator, unlike the
spectrum of Schr\"odinger operators, can be very bad: it can, say,
have everywhere dense pure point spectrum, or a singular
continuous spectrum, or even both.

To the best knowledge of the author, abstract scattering theory in its present form (cf.~\cite{BW,BE,RS3,Ya})
does not allow to resolve this problem. In the present paper a new abstract scattering theory is developed (to the best knowledge of the author).

\medskip

In this theory, given a self-adjoint operator~$H_0$ on a Hilbert space~$\hilb$ with the so-called frame
$F$ and a trace-class perturbation~$V,$ an explicit set of full measure
$\LambHF{H_0}$ is defined in a canonical (constructive) way via the data~$(H_0,F),$ such that
for all~$\lambda \in \LambHF{H_0} \cap \LambHF{H_r}$ the wave matrices~$w_\pm(\lambda; H_r,H_0)$
exist, and moreover, explicitly constructed.
\begin{defn} \label{D: def of frame from intro-n}
A \emph{frame}~$F$ in a Hilbert space~$\hilb$ is a
sequence $$\brs{(\phi_1,\kappa_1), (\phi_2,\kappa_2), (\phi_3,\kappa_3),\ldots,},$$
where~$(\kappa_j)_{j=1}^\infty$ is an~$\ell_2$-sequence of positive numbers,
and~$(\phi_j)_{j=1}^\infty$ is an orthonormal basis of~$\hilb.$
\end{defn}
In other words, a frame is a fixed orthonormal basis such that norms of the basis vectors form an $\ell_2$-sequence.
It is convenient to encode the information about a frame in a
Hilbert-Schmidt operator with trivial kernel and co-kernel
$$
  F \colon \hilb \to \clK, \ \ F = \sum\limits_{j =1}^\infty \kappa_j \scal{\phi_j}{\cdot}\psi_j,
$$
where~$\clK$ is another Hilbert space and~$(\psi_j)_{j=1}^\infty$ is an orthonormal basis in~$\clK.$
The nature of the Hilbert space~$\clK$ and of the basis~$(\psi_j)_{j=1}^\infty$ is immaterial,
so that one can actually take~$\clK=\hilb$ and~$(\psi_j)_{j=1}^\infty = (\phi_j)_{j=1}^\infty.$

Once a frame (operator)~$F$ is fixed in~$\hilb,$ given a self-adjoint operator~$H_0$ on~$\hilb,$
the frame enables to construct explicitly:
\begin{enumerate}
 \item an explicit set of full measure~$\LambHF{H_0},$ which depends only on $H_0$ and $F;$
 \item for every~$\lambda \in \LambHF{H_0},$ an explicit (to be fiber) Hilbert space~$\hlambda \subset \ell_2;$
 \item a measurability base~$\set{\phi_j(\cdot)},$~$j=1,2,\ldots,$ where all functions~$\phi_j(\lambda) \in \hlambda,
              j=1,2,\ldots,$ are explicitly defined for \emph{all}~$\lambda \in \LambHF{H_0};$
 \item (as a consequence) a direct integral of Hilbert spaces
$$
  \euH := \int_{\LambHF{H_0}}^\oplus \mathfrak h_\lambda\,d\lambda,
$$
where the case of~$\dim \hlambda = 0$ is \emph{not excluded}.
 \item Further, considered as a rigging, a frame~$F$ generates a triple of Hilbert spaces
 $\hilb_1 \subset \hilb=\hilb_0 \subset \hilb_{-1}$ with scalar products
 $$
   \scal{f}{g}_{\hilb_\alpha} = \scal{\abs{F}^{-\alpha}f}{\abs{F}^{-\alpha}g}, \ \ \alpha = -1,0,1
 $$
 and natural isomorphisms
 $$
   \hilb_{-1} \stackrel {\abs{F}}{\longrightarrow} \hilb \stackrel {\abs{F}}{\longrightarrow} \hilb_1.
 $$
 \item For any $\lambda \in \LambHF{H_0}$ we have an \emph{evaluation} operator
  \begin{gather*}
     \euE_\lambda = \euE_{\lambda+i0} \colon \hilb_1 \to \hlambda;
    \\
     \euE \colon \hilb_1 \to \euH.
  \end{gather*}
   The operator $\euE_\lambda \colon \hilb_1 \to \hlambda$ is a Hilbert-Schmidt operator, and the operator $\euE,$ considered as an operator~$\hilb \to \euH,$
   extends continuously to a \emph{unitary isomorphism} of the absolutely continuous part (with respect to~$H_0$) of~$\hilb$ to $\euH,$
   and, moreover, the operator $\euE$ \emph{diagonalizes} the absolutely continuous part of~$H_0.$
\end{enumerate}
Here is a quick description of this construction.
\begin{defn} \label{D: def of Lambda from intro-n} A point $\lambda \in \mbR$ belongs to $\LambHF{H_0}$ if and only if
  \begin{itemize}
    \item[(i)]  the operator~$F R_{\lambda + iy}(H_0) F^*$ has a limit in the uniform (norm) topology as $y \to 0^+,$ and
    \item[(ii)]  the operator~$F \Im R_{\lambda + iy}(H_0) F^*$ has a limit in the trace-class norm as $y \to 0^+.$
  \end{itemize}
\end{defn}
\noindent
It follows from the limiting absorption principle (cf.~\cite{deBranges,BE} and~\cite[Theorems 6.1.5, 6.1.9]{Ya}),
that $\LambHF{H_0}$ has full Lebesgue measure, and that for all $\lambda \in \LambHF{H_0}$
the matrix
$$
  \phi(\lambda) := (\phi_{ij}(\lambda)) = \frac 1\pi \brs{\kappa_i\kappa_j \scal{\phi_i}{\Im R_{\lambda+i0}(H_0) \phi_j}}
$$
exists and is a non-negative trace-class operator
on $\ell_2$ (Proposition \ref{P: Lambda subset Lambda1}).
%
%
The value $\phi_j(\lambda)$ of the vector $\phi_j$ at $\lambda \in \LambHF{H_0}$
is defined as the~$j$-th column $\eta_j(\lambda)$ of the Hilbert-Schmidt operator $\sqrt{\phi(\lambda)}$ over the weight $\kappa_j$ of $\phi_j:$
$$
  \phi_j(\lambda) = \kappa_j^{-1} \eta_j(\lambda).
$$
It is not difficult to see that if $f \in \hilb_1(F),$ so that
$$
  f = \sum\limits_{j=1}^\infty \kappa_j \beta_j \phi_j,
$$
where $(\beta_j) \in \ell_2,$ then the series
$$
  f(\lambda) := \sum_{j=1}^\infty \kappa_j \beta_j \phi_j(\lambda) = \sum_{j=1}^\infty \beta_j \eta_j(\lambda)
$$
absolutely converges in $\ell_2.$ The fiber Hilbert space~$\hlambda$ is by definition
the closure of the image of~$\hilb_1$ under the map
$$
  \euE_\lambda \colon \hilb_1 \ni f \mapsto f(\lambda) \in \ell_2.
$$
The image of the set of frame vectors $\phi_j$ under the map $\euE$ form
a measurability base of a direct integral of Hilbert spaces
$$
  \euH := \int_{\LambHF{H_0}}^\oplus \mathfrak h_\lambda\,d\lambda,
$$
and the operator
$$
  \euE \colon \hilb_1 \to \euH
$$
is bounded as an operator from~$\hilb$ to $\euH,$ vanishes on the
singular subspace~$\hilb^{(s)}(H_0)$ of~$H_0,$ is isometric on the
absolutely continuous subspace~$\hilb^{(a)}(H_0)$ of~$H_0$ with
the range $\euH$ (Propositions \ref{P: euE is an isometry},
\ref{P: euE is unitary}) and is diagonalizing for~$H_0$ (Theorem
\ref{T: H0 is diagonal}), that is,
$$
  \euE_\lambda(H_0 f) = \lambda \euE_\lambda f \ \ \text{for all}
  \ \lambda \in \Lambda(H_0,F).
$$

\bigskip
\begin{center}
  * \quad * \quad *
\end{center}
\bigskip

So far, we have had one self-adjoint operator~$H_0$ acting in~$\hilb.$
Let $V$ be a self-adjoint trace-class operator. Let a frame~$F$ be such that
$V = F^*JF,$ where $J \colon \clK \to \clK$ is a self-adjoint bounded operator. Clearly, for any trace-class
operator such a frame exists. This means that the operator $V$ can be considered as a bounded operator
$$
  V \colon \hilb_{-1} \to \hilb_1.
$$
By definition, if $\lambda \in \LambHF{H_0},$ then the limit
$$
  R_{\lambda + i0}(H_0) \colon \hilb_1 \to \hilb_{-1}
$$
exists in the Hilbert-Schmidt norm and the limit
$$
  \Im R_{\lambda + i0}(H_0) \colon \hilb_1 \to \hilb_{-1}
$$
exists in the trace-class norm. So, if $\lambda \in \LambHF{H_0} \cap \LambHF{H_1},$
then the following operator is a well-defined trace-class operator (from~$\hilb_1$ to~$\hilb_{-1}$)
$$
  \mathfrak a_\pm(\lambda;H_1,H_0) := \SqBrs{1- R_{\lambda \mp i0}(H_1)V} \cdot \frac 1\pi \Im R_{\lambda + i0}(H_0).
$$
Let $\lambda \in \LambHF{H_0} \cap \LambHF{H_1},$ where~$H_1 = H_0+V,$
so that both fiber Hilbert spaces~$\hlambdao$ and $\mathfrak h_\lambda^{(1)}$ are well-defined.
Then there exists a unique (for each sign $\pm$) operator
$$
  w_\pm(\lambda; H_1,H_0) \colon \hlambdao \to \mathfrak h_\lambda^{(1)}
$$
such that for any $f,g \in \hilb_1$ the equality
$$
  \scal{\euE_\lambda(H_1) f}{w_\pm(\lambda; H_1,H_0) \euE_\lambda(H_0) g} = \scal{f}{\mathfrak a_\pm(\lambda;H_1,H_0)g}_{1,-1}
$$
holds, where $\scal{\cdot}{\cdot}_{1,-1}$ is the pairing of the rigged Hilbert space $(\hilb_1,\hilb,\hilb_{-1}).$
The operator
$$
  w_\pm(\lambda; H_1,H_0)
$$
is correctly defined, and, moreover, it is unitary and satisfies multiplicative property.
The operator $w_\pm(\lambda; H_1,H_0)$ is actually the wave matrix, which is thus explicitly constructed for \emph{all}
$\lambda$ from the set of full Lebesgue measure $\LambHF{H_0} \cap \LambHF{H_1}.$

So far we considered a pair of operators~$H_0$ and~$H_1.$ But if the aim is to prove the formula (\ref{Int: S = T exp...}),
then one needs to make sure that the wave matrix $w_\pm(\lambda; H_r,H_0)$ exists for all, with a possible exception of some small
set, values of $r \in [0,1].$ It turns out that, indeed, the wave matrix $w_\pm(\lambda; H_r,H_0)$ is defined for all $r$ except a discrete set,
as follows from the following simple but important property of the set $\LambHF{H_0}$ (Theorem \ref{T: R(H0,G) is discrete}):
$$
  \text{if} \ \ \lambda \in \LambHF{H_0}, \ \ \text{then} \  \lambda \in \LambHF{H_r} \ \ \text{for all} \ r \notin R(\lambda, H_0,V;F),
$$
where $R(\lambda, H_0,V;F)$ is a discrete set of special importance called resonance set (see the picture below).

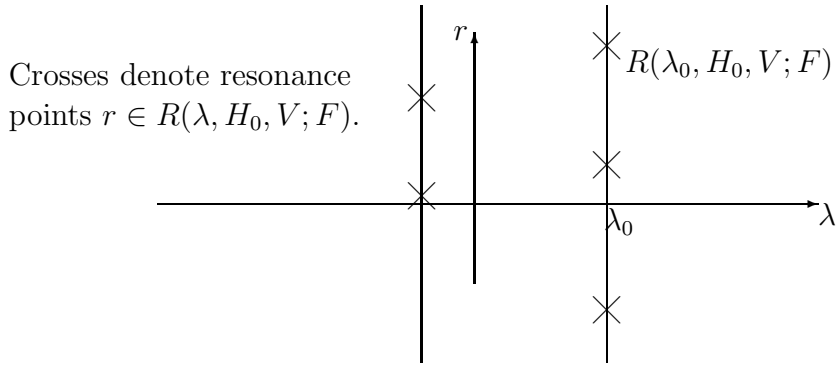
\begin{picture}(400,180)
\put(80,80){\vector(1,0){250}}  
\put(200,50){\vector(0,1){95}}  
\put(330,72){$\lambda$} \put(192,142){$r$}
\put(20,125){{ Crosses denote resonance }}
\put(20,110){{ points $r \in R(\lambda, H_0,V;F).$}}

\put(253,130){{ $R(\lambda_0,H_0,V;F)$ }}
\put(245,70){{ $\lambda_0$ }}

\put(250,20){\line(0,1){135}}
\put(245,35){\line(1,1){10}}
\put(255,35){\line(-1,1){10}}
\put(245,90){\line(1,1){10}}
\put(255,90){\line(-1,1){10}}
\put(245,135){\line(1,1){10}}
\put(255,135){\line(-1,1){10}}

\put(180,20){\line(0,1){135}}
\put(175,78){\line(1,1){10}}
\put(185,78){\line(-1,1){10}}
\put(175,115){\line(1,1){10}}
\put(185,115){\line(-1,1){10}}
\end{picture}

\noindent If $\lambda$ is an eigenvalue of~$H_r = H_0+rV,$ then $r \in R(\lambda,H_0,V;F)$
for any~$F.$ But $R(\lambda,H_0,V;F)$ may contain other points as well, which may depend on~$F.$
This partly justifies the terminology ``resonance points'' and gives a basis for classification of resonance points into two different types.

So, the set $\set{(\lambda,r) \colon \lambda \in \LambHF{H_r}}$ behaves very regularly with respect to $r,$ but it does not do so with respect to~$\lambda:$
while for fixed $r_0 \in \mbR$ and $\lambda_0 \in \mbR$
the set $\set{\lambda \in \mbR \colon \lambda \notin \LambHF{H_{r_0}}}$
can be a more or less arbitrary null set, the set $\set{r \in \mbR \colon \lambda_0 \notin \LambHF{H_r}}$
is a discrete set, i.e. a set with no finite accumulation points.

Further, the multiplicative property of the wave
matrix
$$
  w_\pm(\lambda; H_{r_2},H_{r_0}) = w_\pm(\lambda; H_{r_2},H_{r_1})w_\pm(\lambda; H_{r_1},H_{r_0})
$$
is proved (Theorem \ref{T: mult-ive property of w pm}), where
$r_2,r_1,r_0$ do not belong to the above mentioned discrete
resonance set $R(\lambda, H_0,V;F).$ As is known
(cf.~\cite[Subsection 2.7.3]{Ya}), the proof of this property for
the wave operator $W_\pm(H_1,H_0)$ composes the main difficulty of
the stationary approach to the abstract scattering theory. A bulk
of this paper is devoted to the definition of $w_\pm(\lambda;
H_r,H_0)$ for all $\lambda \in \LambHF{H_r}  \cap \LambHF{H_0}$
and to the proof of the multiplicative property. This is the main
feature of the new scattering theory given in this paper. Further,
for all $\lambda \in \LambHF{H_r} \cap \LambHF{H_0}$ the
scattering matrix~$S(\lambda;H_r,H_0)$ is \emph{defined} as an
operator from $\hlambdao$ to $\hlambdao$ by the formula
$$
  S(\lambda;H_r,H_0) = w^*_+(\lambda;H_r,H_0)w_-(\lambda;H_r,H_0).
$$
The scattering operator $\bfS(H_r,H_0)\colon \hilba(H_0) \to \hilba(H_0)$ is defined as the direct integral of scattering matrices:
$$
  \bfS(H_r,H_0) := \int_{\Lambda(H_0;F) \cap \Lambda(H_r;F)} S(\lambda;H_r,H_0)\,d\lambda.
$$
For all $\lambda \in \LambHF{H_r} \cap \LambHF{H_0}$ the stationary formula for the scattering matrix
$$
  S(\lambda;H_r,H_0) = 1_\lambda - 2\pi i \euE_\lambda rV(1 + rR_{\lambda+i0}(H_0)V)^{-1} \euE^\diamondsuit_\lambda
$$
is proved (Theorem \ref{T: stationary rep-n for SM}).
Though the scattering matrix~$S(\lambda;H_r,H_0)$ does not exist for resonance points $r \in R(\lambda, H_0,V;F),$
a simple but important property of the scattering matrix is that it admits analytic continuation to the resonance points
(Proposition \ref{P: S(r) is continuous}).
The stationary formula enables to show that
for all $\lambda \in \LambHF{H_0}$ and all $r$ not in the resonance set $R(\lambda, H_0,V;F),$
the formula (\ref{Int: S = T exp...}) holds (Theorem \ref{T: S = T exp...}),
where for all non-resonance points $r$ the infinitesimal scattering matrix is defined as
$$
  \Pi_{H_r}(V)(\lambda) = \euE_\lambda(H_r) V \euE_\lambda^\diamondsuit(H_r) \colon \hlambdar \to \hlambdar
$$
and where $\euE_\lambda^\diamondsuit = \abs{F}^{-2}\euE_\lambda^*.$

The main object of the abstract scattering theory
given in~\cite{BE,Ya}, the wave operator $W_\pm(H_r,H_0),$ is defined
as the direct integral of the wave matrices
$$
  W_\pm(H_r,H_0) = \int^\oplus_{\LambHF{H_r} \cap \LambHF{H_0}} w_\pm(\lambda; H_r,H_0)\,d\lambda.
$$
The usual definition
$$
  W_\pm(H_r,H_0) = \mbox{s-}\lim\limits_{t \to \pm \infty} e^{itH_r}e^{-itH_0}P^{(a)}_0
$$
of the wave operator becomes a theorem (Theorem \ref{T: time dep-t WO}).
The formula
$$
  \bfS(H_r,H_0) = W_+^*(H_r,H_0)W_-(H_r,H_0),
$$
which is usually considered as definition of the scattering
operator, obviously holds.


This new scattering theory has allowed to prove (\ref{Int: xia = -2pi i det S}) for all $\lambda \in \LambHF{H_0}$
(Theorem \ref{T: xia = -2pi i log det S}). Combined with the Birman-Krein formula (\ref{Int: BK formula}) this implies
that the singular part of the spectral shift function
is an a.e. integer-valued function for arbitrary trace-class perturbations of arbitrary self-adjoint
operators (Theorem \ref{T: xis is int valued}):
$$
  \xis_{H_1,H_0}(\lambda) \in \mbZ \ \ \text{for a.e. } \ \lambda \in \mbR.
$$
Theorem \ref{T: xis is int valued} is the main result of this paper.
This result is to be considered as unexpected, since the definition (\ref{Int: def of xis})
of the singular part of the spectral shift function does not suggest anything like this.

In section \ref{S: Push...} another proof of Theorem \ref{T: xis is int valued} is given which does not use the Birman-Krein formula (\ref{Int: BK formula}),
so that the Birman-Krein formula itself becomes a corollary of Theorem \ref{T: xis is int valued} and (\ref{Int: xia = -2pi i det S}).
This proof uses the so-called $\mu$-invariant introduced by Alexander Pushnitski in \cite{Pu01FA}. Pushnitski $\mu$-invariant measures spectral flow
of scattering phases (eigenvalues of the scattering matrix) through a given point $e^{i\theta}$ on the unit circle $\mbT.$
In section \ref{S: Push...} it is shown that there is another natural way to measure spectral flow of scattering phases. It is shown that the difference
of these two $\mu$-invariants does not depend on the angle variable $\theta$ and is equal (up to a sign) to the singular part of the spectral shift function.

I would like to stress that even though the scattering theory presented in this paper is different in its
nature from the conventional scattering theory given in~\cite{BE,Ya}, many essential ideas
are taken from~\cite{BE,Ya} (cf. also~\cite{BW,RS3}), and essentially no new results appear until subsection
\ref{SS: InfinScatM}, though most of the proofs are original (to the best of the author's knowledge). At the same time,
this new approach to abstract scattering theory is simpler than that of given in~\cite{Ya}, and it is this new approach
which allows one to prove main results of this paper.

\section{Preliminaries}
In these preliminaries I follow mainly~\cite{GK,RS1,SimTrId,Ya}.
Details and (omitted) proofs can be found in these references.
A partial purpose of these preliminaries is to fix notation and terminology.
\subsection{Notation}
$\mbR$ is the set of real numbers.
$\mbC$ is the set of complex numbers.
$\mbC_+$ is the open upper half-plane of the complex plane~$\mbC.$

\subsection{Functions holomorphic in $\mbC_+$}
Proof of the following theorem can be found in \cite{Priv} (see also \cite[\S 1.2]{Ya}).

\begin{thm} \label{T: Luzin-Privalov} (a) If $f \colon \mbC_+ \to \mbC$ is a bounded holomorphic function,
then for a.e. $\lambda \in \mbR$ the angular limit $f(\lambda+i0)$ exists.
\\ (b) If the function $f(\lambda+i0)$ is equal to zero on a set of positive Lebesgue measure then $f = 0.$
\end{thm}
This theorem has a much stronger version, but it is all we need.

\subsection{Measure theory} Here we collect some definitions from measure theory. Details can be found in D.\,Yafaev's book
\cite{Ya}.

The $\sigma$-algebra $B(\mbR)$ of Borel sets is generated by open subsets of~$\mbR.$
By a measure on~$\mbR$ we mean a locally-finite non-negative countably additive function $m$ defined on the $\sigma$-algebra
of Borel sets. Locally-finite means that the measure of any compact set is finite.
By a \emph{Borel support} of a measure $m$ we mean any Borel set~$X$ whose complement has zero $m$-measure:
$m(\mbR \setminus X) = 0.$ By the \emph{closed support} of a measure $m$ we mean the smallest closed Borel support of $m.$
The closed support exists and is unique.

By $|X|$ we denote the Lebesgue measure of a Borel set $X.$
A Borel set $Z$ is called a \emph{null set}, if it has zero Lebesgue measure: $|Z| = 0.$ A Borel set $\Lambda$ is called \emph{full set},
if the complement of~$X$ is a null set: $|\mbR \setminus \Lambda| = 0.$
Full sets will usually be denoted by $\Lambda,$ with indices and arguments, if necessary.

A Borel support $X$ of a measure $m$ is called \emph{minimal}, if for any other Borel support~$X'$
the equality $|X \setminus X'| = 0$ holds. Note that the closed support of a measure is not necessarily minimal.
A minimal Borel support exists, but it is not unique. Two minimal supports can differ by not more than a null set.

A measure $m$ is called \emph{absolutely continuous},
if for any null set $Z$ the equality $m(Z) = 0$ holds. The Radon-Nikodym theorem asserts that a measure $m$
is absolutely continuous if and only if there exists a locally-summable non-negative function $f$ such that for any Borel set~$X$
$$
  m(X) = \int_X f(\lambda)\,d\lambda.
$$
A measure $m$ is called \emph{singular}, if there exists a null Borel support of $m,$
that is, a Borel support of zero Lebesgue measure. Any measure $m$ admits a unique decomposition
$$
  m = m^{(a)}+m^{(s)}
$$
into the sum of an absolutely continuous measure $m^{(a)}$ and a singular measure $m^{(s)}.$

Two measures $m_1$ and $m_2$ have the same spectral type, if they are absolutely continuous
with respect to each other, that is, if $m_1(X) = 0$ for some Borel set $X,$ then $m_2(X) = 0,$
and vice versa.

The abbreviation a.e. will always refer to the Lebesgue measure.

Two measures are mutually singular, if they have non-intersecting Borel supports.

A signed measure is a locally finite countably-additive function $m$ defined on bounded Borel sets.
Every signed measure $m$ admits a unique Hahn decomposition:
$$
  m = m_+-m_-,
$$
where non-negative measures $m_-$ and $m_+$ are mutually singular. The measure $\abs{m}:=m_++m_-$ is called
total variation of $m.$

\subsubsection{Vitali's theorem}
Apart of Lebesgue's dominated convergence theorem, we shall need
Vitali's theorem. This is the following theorem (for a proof see
\cite{Nat}).
\begin{thm} \label{T: Vitali} Let $X$ be a Borel subset of $\mbR.$
Suppose for functions $f_y \in L_1(\mbR),$ $y>0,$ the integrals
$$
  \int_X f_y(\lambda)\,d\lambda
$$
tend to zero uniformly with respect to $y$ as $\abs{X} \to 0.$
Suppose also the same for $X = (-\infty,N) \cup (N,\infty)$ as $N \to \infty.$
If for a.e. $\lambda \in \mbR$
$$
  \lim_{y \to 0} f_y(\lambda) = f(\lambda),
$$
then the function $f$ is summable and
$$
  \lim_{y \to 0} \int_{-\infty}^\infty f_y(\lambda)\,d\lambda = \int_{-\infty}^\infty f(\lambda)\,d\lambda.
$$
\end{thm}

\subsubsection{Poisson integral}
Let~$F$ be a function of bounded variation on $\mbR.$
Poisson integral $\euP_F$ of $F$ is the following function of two variables:
$$
  \euP_F(x,y) = \frac y\pi \int_{-\infty}^{\infty} \frac {dF(t)}{(x-t)^2+y^2}.
$$
The function
\begin{equation} \label{F: Poisson kernel}
  P_y(x) = \frac 1\pi \frac y{x^2+y^2}
\end{equation}
is the kernel of the Poisson integral and
$$
  \euP_F(x,y) = P_y * dF(x).
$$
The family $\set{P_y(x), y>0}$ form an approximate unit for the delta-function,
that is, all these functions are non-negative, an integral of each of the functions $P_y$
is equal to $1$ and $P_y$ converge in distributions sense to the Dirac's delta function $\delta.$

In case when $F$ is the distribution function of a summable function $f \in L_1(\mbR),$
allowing an abuse of terminology, we also say that $P_y* f(x)$ is the Poisson integral of the function~$f.$

\begin{lemma} \label{L: Zet} Let $g \in L_1(\mbR)$ and let $g_y$ be the Poisson integral of $g.$
If $X$ is a Borel set, then the integral
$$
  \int_X \abs{g_y(\lambda)}\,d\lambda
$$
converges to zero as $\abs{X} \to 0$ and as $N \to \infty$ in $X = (-\infty,-N) \cup (N,\infty)$ uniformly
with respect to $y \in(0,1).$
\end{lemma}
%

\subsubsection{Fatou's theorem}
The following Fatou's theorem plays an important role in this paper.
For a discussion of this theorem see~\cite{Ya}.
\begin{thm} \label{T: Fatou} Let~$F$ be a function of bounded variation on $\mbR.$
If at some point $x_0 \in \mbR$ the function~$F$ has the symmetric derivative
$$
  F'_{sym}(x_0) := \lim_{h \to 0^+} \frac{F(x_0+h) - F(x_0-h)}{2h},
$$
then the limit of the Poisson integral of $F$
$$
  \lim_{y\to 0^+} \euP_F(x_0,y)
$$
exists and is equal to~$F'_{sym}(x_0).$ In particular, the limit exists for a.e. $x_0.$
\end{thm}

\subsubsection{Privalov's theorem}
Let~$F \colon \mbR \to \mbC$ be a function of bounded variation.
The Cauchy-Stieltjes transform of~$F$ is a function holomorphic in both the upper and the lower complex half-planes $\mbC_\pm;$
this function is defined by the formula
$$
  \clC_F(z) = \int_{-\infty}^\infty (x-z)^{-1}\,dF(x).
$$
The following theorem is known as Privalov's theorem (cf.~\cite{Priv},
\cite[Theorem 1.2.5]{Ya}). This theorem can be formulated for an upper half-plane or, equivalently,
for a unit disk. Proof of the theorem can also be found in~\cite[Chapter VI, \S 59, Theorem 1]{AG}.
\begin{thm} \label{T: Ya 1.2.5}
Let~$F \colon \mbR \to \mbC$ be a function of bounded variation. The limit values
$$
  \clC_F(\lambda \pm i0) := \lim_{y \to 0^+} \clC_F(\lambda \pm iy)
$$
of the Cauchy-Stieltjes transform $\clC_F(z)$ of~$F$ exist for a.e. $\lambda \in \mbR,$ and
for a.e. $\lambda \in \mbR$ the equality
\begin{gather} \label{F: Privalov's formula}
  \clC_F(\lambda \pm i0) = \pm \pi i \frac {dF(\lambda)}{d\lambda} + {\rm p.v.} \int_{-\infty}^\infty (\mu-\lambda)^{-1}\,dF(\mu)
\end{gather}
holds, where the principal value integral on the right-hand side also exists for a.e. $\lambda \in \mbR.$
\end{thm}
Since the imaginary part of the Cauchy-Stieltjes transform of~$F$ is the Poisson kernel of~$F:$
\begin{equation} \label{F: Im C = Poisson}
  \frac 1\pi \Im \clC_F(\lambda+iy) =\euP_F(\lambda,y),
\end{equation}
the convergence of $\frac 1 \pi \Im \clC_F(\lambda\pm iy)$ to $\pm F'(\lambda)$ for a.e. $\lambda$ follows from Fatou's theorem \ref{T: Fatou}.

\subsubsection{The set $\Lambda(f)$} \label{SS: L(F)}
It is customary to consider a summable function $f \in L_1(\mbR)$ as a class of equivalent functions, where two functions are considered to be
equivalent if they coincide everywhere except a null set.
So, a summable function is defined up to a set of Lebesgue measure zero. In this way,
in general one cannot ask what is the value of a summable function $f$
at, say, $\sqrt 2.$ In this paper we take a different approach. By a summable function we mean a complex-valued summable function $f$
which is \emph{explicitly} defined on some explicit set of full Lebesgue measure.

Given a summable function $f \in L_1(\mbR)$ there are two (among many other) natural ways to assign to the function
a canonical set of full Lebesgue measure $\Lambda,$ so that $f$ is in some natural way defined at \emph{every} point
of $\Lambda$ (see the first paragraph of \cite[p.\,384]{AD56}).

The first way is this. If $f \in L_1(\mbR),$ then one can define a set of full Lebesgue measure $\Lambda'(f)$
as the set of all those numbers $x$ at which the function
$$
  \int_0^x f(t)\,dt
$$
is differentiable. Lebesgue's differentiation theorem says this set is a full one.
If $x \in \Lambda'(f),$ then one can define $f(x)$ by the formula
$$
  f(\lambda) := \frac d{d\lambda} \int_0^\lambda f(x)\,dx.
$$

However, there is another canonical set of full Lebesgue measure,
associated with $f:$
$$
  \Lambda(f) := \set{\lambda \in \mbR \colon \ \lim_{y\to 0^+} \Im \clC_F(\lambda+iy) \ \text{exists}},
$$
where~$F(\lambda)=\int_0^\lambda f(x)\,dx$ and $\clC_F(z)$
is the Cauchy-Stieltjes transform of~$F.$ That $\Lambda(f)$ is a full set follows from Theorem \ref{T: Fatou}.
For any $\lambda \in \Lambda(f)$
one can define the value $f(\lambda)$ of the function $f$ at $\lambda$ by the formula
\begin{equation} \label{F: f(l)=1/pi Im C(l+i0)}
  f(\lambda) := \frac 1\pi \Im \clC_F(\lambda+i0) := \frac 1\pi \lim_{y\to 0^+} \Im \clC_F(\lambda+iy) = \lim_{y\to 0^+} f * P_y(\lambda).
\end{equation}
Since $\frac 1\pi \Im \clC_F(\lambda+iy)$ is the Poisson kernel of $F$ (see (\ref{F: Im C = Poisson})),
it follows from Theorem \ref{T: Fatou} that the two explicit summable functions defined in this way are equivalent.

It is clear that two elements $f$ and $g$ of $L_1(\mbR)$ (as equivalence classes) coincide if and only if $\Lambda(f) = \Lambda(g)$
and $f(\lambda) = g(\lambda)$ for all $\lambda \in \Lambda(f).$

So, from now on, all summable functions $f$ are understood in this sense (if not stated otherwise): $f$ is a function on the full set $\Lambda(f)$
defined by (\ref{F: f(l)=1/pi Im C(l+i0)}). Probably, it is worth to stress again that in this definition by a function we mean a function.

\subsubsection{De la Vall\'ee Poussin decomposition theorem}
This is the following theorem (see e.g. \cite[Theorem IV.9.6]{Saks}, \cite{RudRCA}):
\begin{thm} \label{T: Vallee Poussin} Let $m$ be a finite signed measure. Let $\abs{m}$ be the total variation of $m.$
Let $E_{-\infty}$ (respectively, $E_{+\infty}$) be the set where the derivative of the distribution function~$F_m$ of $m$ is $-\infty$
(respectively, $+\infty$).
If $X$ is a Borel subset of $\mbR,$ 
then
$$
  m(X) = m(X \cap E_{-\infty}) + m(X \cap E_{+\infty}) + \int_X F_m'(t)\,dt
$$
and
$$
  \abs{m}(X) = \abs{m(X \cap E_{-\infty})} + m(X \cap E_{+\infty}) + \int_X \abs{F_m'(t)}\,dt.
$$
\end{thm}

Remark. The formulation of \cite[Theorem IV.9.6]{Saks} contains an additional condition that~$F_m$ is continuous at every point of $X.$
This condition is obviously redundant.

\subsubsection{Standard supports of measures} \label{SS: L(m)}
If $m$ is a finite signed measure, then its Cauchy-Stieltjes transform~$\clC_m(z)$
is defined as the Cauchy-Stieltjes transform of its distribution function
$$
  F_m(x) = m((-\infty,x)).
$$
That is,
$$
  \clC_m(z) := \int_{-\infty}^{\infty} \frac {m(dx)}{x - z}.
$$
A finite signed measure has a natural decomposition
$$
  m = m^{(a)}+m^{(s)}
$$
into the sum of an absolutely continuous measure $m^{(a)}$ and a singular measure $m^{(s)}.$
The signed measures $m^{(a)}$ and $m^{(s)}$ are mutually singular.
It is desirable to split the set of real numbers $\mbR$ in some natural way, such that the first
set is a Borel support of the absolutely continuous part $m^{(a)},$ while the second set is a Borel support
of the singular part $m^{(s)}$ of the measure $m.$

It is possible to do so in several ways. The choice which suits our needs in the best way is the following.
To every finite signed measure $m$ we assign the set
$$
  \Lambda(m) := \set{\lambda \in \mbR \colon \ \text{a finite limit} \ \Im \clC_m(\lambda+i0) \in \mbR \ \text{exists}},
$$
This set was introduced by Aronszajn in~\cite{Ar57}.

The following theorem belongs to Aronszajn~\cite{Ar57}.
\begin{thm} \label{T: complement of Lambda(F)} Let $m$ be a finite signed measure.
The set $\Lambda(m)$ is a full set. The complement of the set $\Lambda(m)$
is a minimal Borel support of the singular part of $m.$
\end{thm}
%
The main point of this theorem is that it gives a natural splitting of the set of real numbers $\mbR$
into two parts such that the first part supports $m^{(a)}$ and the second part supports $m^{(s)}.$
Actually, the support of the singular part $\mbR \setminus \Lambda(m)$ can be made smaller.
Namely, the set of all points $\lambda \in \mbR$ for which $\Im \clC_m(\lambda+i0)$ equals $+\infty$ or $-\infty$
is a Borel support of the singular part of $m.$

The function $\Im \clC_m(\lambda+iy)$ cannot grow to infinity faster than $C/y.$ If it grows as $C/y,$
then the point $\lambda$ has a non-zero measure equal to $C.$ The set of points where
$\Im \clC_m(\lambda+iy)$ grows as $C/y$ form a Borel support of the discrete part of $m.$ The set of points
where $\Im \clC_m(\lambda+iy)$ grows to infinity slower than $C/y$ form a Borel support of the singular continuous part of $m.$
These Borel supports were also introduced in \cite{Ar57}. Though these supports of the singular part(s) of $m$ are more natural
and finer than $\mbR \setminus \Lambda(m),$ for the purposes of this paper the last support suffices.

Also, imposing different growth conditions on $\Im \clC_m(\lambda+iy),$ such as $\Im \clC_m(\lambda+iy) \sim C/y^\rho,$ where $\rho \in (0,1),$
one can get further finer classification of the singular continuous spectrum, see \cite{Ro} for details.

The set $\Lambda(m)$ is not a minimal Borel support of $m^{(a)},$ but it is not difficult to indicate a natural minimal Borel
support of $m^{(a)}$ (see \cite{Ar57}):
\begin{equation} \label{F: clA(m)=...}
  \clA(m) = \set{\lambda \in \Lambda(m): \Im \clC_m(\lambda+i0) \neq 0}.
\end{equation}
This follows from the fact that for a.e. $\lambda \in \Lambda(m)$
$$
  F_m'(\lambda) = \frac 1\pi \Im \clC_m(\lambda+i0)
$$
and form the fact that the function $\lambda \mapsto F_m'(\lambda)$ is a density of the absolutely continuous part of $m.$
The number $F_m'(\lambda)$ will be considered as a standard value of the density function at points of $\Lambda(m).$
\begin{cor} \label{C: int Delta dF = int Delta F'dl}
  Let~$F$ be a function of bounded variation on~$\mbR$ and let $m$ be the corresponding (signed) measure.
For any Borel subset $\Delta$ of $\Lambda(m)$ the equalities
$$
  \int_\Delta dF(\lambda) = \int _\Delta F'(\lambda)\,d\lambda = \frac 1 \pi \int _\Delta \Im \clC_F(\lambda+i0)\,d\lambda
   = \frac 1 \pi \int _\Delta \Im \clC_{F^{(a)}} (\lambda+i0)\,d\lambda
$$
hold.
\end{cor}
%

There is another canonical full set associated with a function of bounded variation,
namely, the Lebesgue set of all points where~$F$ is differentiable. But the set $\Lambda(F)$
is easier to deal with, and it seems to be more natural in the context of scattering theory.

\subsection{Bounded operators}
Let~$\hilb$ be a separable Hilbert space with scalar product $\scal{\cdot}{\cdot},$ anti-linear
in the first variable (all Hilbert spaces in this paper are complex and separable).
Let $T$ be a bounded operator on $\hilb.$
The (uniform) norm $\norm{T}$ of a bounded operator $T$ is defined as
$$
  \norm{T} = \sup\limits_{f \in \hilb, \, \norm{f} = 1} \norm{Tf}.
$$
A bounded operator $T$ in $\hilb$ is non-negative, if
$\scal{Tf}{f} \geq 0$ for any $f \in \hilb.$

The algebra of all bounded operators in $\hilb$ is denoted by $\clBH.$
Let $\alpha$ run through some net of indices $I.$

A net of operators $T_\alpha \in \clBH$ converges to $T \in \clBH$
in the strong operator topology, if for any $f \in \hilb$
the net of vectors $T_\alpha f$ converges to $Tf.$ In other words,
the strong operator topology is generated by seminorms $T \mapsto \norm{Tf},$
where $f \in \hilb.$

A net of operators $T_\alpha \in \clBH$ converges to $T \in \clBH$
in the weak operator topology, if for any $f,g \in \hilb$
the net $\scal{T_\alpha f}{g}$ converges to $\scal{Tf}{g}.$
In other words, the weak operator topology is generated by
seminorms $T \mapsto \abs{\scal{Tf}{g}},$ where $f,g \in \hilb.$

The adjoint $T^*$ of a bounded operator $T$ is the unique operator which for all $f,g \in \hilb$ satisfies the equality
$\scal{T^*f}{g} = \scal{f}{Tg}.$ A bounded operator $T$ is self-adjoint if $T^*=T.$

If $T$ is a bounded self-adjoint operator, then for any bounded Borel function $f$ there is a bounded self-adjoint operator $f(T)$
(the Spectral Theorem), such that, in particular, the map $f \mapsto f(T)$ is a homomorphism.

The real $\Re(T)$ and the imaginary $\Im(T)$ parts of an operator $T \in \clB(\hilb)$ are defined by
$$
  \Re(T) = \frac{T+T^*}2 \quad \text{and} \quad \Im(T) = \frac{T-T^*}{2i}.
$$
The real and imaginary parts are self-adjoint operators.

The absolute value $\abs{T}$ of a bounded operator $T$ is the operator
$$
  \abs{T} = \sqrt{T^*T}.
$$

An operator $T \in \clBH$ is Fredholm, if (1) the kernel of $T$
$$
  \ker(T) := \set{f \in \hilb \colon Tf = 0}
$$
is finite-dimensional, (2) the image of $T$
$$
  \im(T) := \set{f \in \hilb \colon \ \exists g \in \hilb \ f = Tg}
$$
is closed and (3) the orthogonal complement (that is, co-kernel $\coker(T)$) of $\im(T)$ is finite-dimensional.
If $T$ is Fredholm, then the index $\ind(T)$ of $T$ is the number
$$
  \ind(T) := \dim \ker(T) - \dim \coker(T) = \dim \ker(T) - \dim \ker(T^*).
$$
\begin{thm} \label{T: Fredholm alternative} (Fredholm alternative) If $K$ is a compact operator, then $1+K$ is Fredholm and $\ind(1+K) = 0.$
\end{thm}
In particular, if $K$ is compact and if $1+K$ has trivial kernel, then $1+K$ is invertible.

\subsection{Self-adjoint operators} For details regarding the material of this subsection see~\cite{RS1}.

Let~$\hilb$ be a separable Hilbert space with scalar product $\scal{\cdot}{\cdot},$ anti-linear
in the first variable.

By a linear operator $T$ on $\hilb$ one means a linear operator from some linear manifold $\euD(T) \subset \hilb$
to $\hilb.$ The set $\euD(T)$ is called the domain of $T.$
A linear operator $T$ is symmetric if its domain $\euD(T)$ is dense and if for any $f,g \in \euD(T)$
the equality \ $\scal{Tf}{g} = \scal{f}{Tg}$ \ holds.
A linear operator $S$ is an extension of a linear operator $T,$ if $\euD(T) \subset \euD(S)$ and $Sf = Tf$
for all $f \in \euD(T).$ In this case one also writes $T \subset S$ (this inclusion can be considered as inclusion of sets,
if one identifies an operator with its graph).
A linear operator $T$ is closed if $f_1, f_2,\ldots \in \euD(T),$
$f_n \to f$ and $Tf_n \to g$ as $n\to \infty$ imply that $f \in \euD(T)$ and $Tf = g.$
An operator $T$ is closable, if it has a closed extension.
For every closable operator $T$ there exists a minimal (with respect to order $\subset$) closed extension $\overline T.$
The adjoint $T^*$ of a densely defined operator $T$ is a linear operator with domain
$$
  \euD(T^*) := \set{ g \in\hilb \colon \exists h \in \hilb \ \forall f \in \euD(T) \ \scal{Tf}{g} = \scal fh};
$$
such a vector $h$ is unique and by definition $T^*g = h.$ For every densely defined closable operator $T$ its adjoint $T^*$ is closed.
For every densely defined operator $T$ the inclusion $\overline T \subset T^{**}$ holds.
A symmetric operator $T$ satisfies $\overline T \subset T^*.$
A symmetric operator $T$ is called self-adjoint if $T = T^*.$ So, self-adjoint operator is automatically closed.

The resolvent set $\rho(H)$ of an operator~$H$ in $\hilb$ consists of all those complex numbers $z \in \mbC,$
for which the operator~$H-z$ has a bounded inverse with domain dense in $\hilb.$
The resolvent of an operator~$H$ is the operator
$$R_z(H) = (H-z)^{-1}, \ \ z \in \rho(H).$$
The spectrum $\sigma(H)$ of~$H$ is the complement of the resolvent set $\rho(H),$
i.e. $\sigma(H) = \mbC \setminus \rho(H).$

A closed symmetric operator $H$ is self-adjoint if and only if $\ker\brs{H-z} = \set{0}$ for any non-real $z \in \mbC.$
The spectrum of a self-adjoint operator is a subset of $\mbR.$

Let~$H_0$ be a self-adjoint operator with domain $\euD(H_0)$ in~$\hilb.$
By $E_X = E_X^{H_0}$ we denote the spectral projection of the operator~$H_0,$
corresponding to a Borel set $X \subset \mbR$ (cf.~\cite{RS1}).
Usually, dependence on the operator~$H_0$ will be omitted in the
notation of the spectral projection.
If $X = (-\infty,\lambda),$ then we also write
$E_\lambda = E_{(-\infty,\lambda)}.$

By a subspace of a Hilbert space~$\hilb$ we mean a closed linear subspace of~$\hilb.$

If $f,g \in \hilb,$ then the spectral measure associated with $f$ and $g$ is the (signed) measure
$$
  m_{f,g}(X) = \scal{f}{E_X g}.
$$
We also write $m_f = m_{f,f}.$

A vector $f$ is called absolutely continuous (respectively, singular) with respect to~$H_0,$
if the spectral measure $m_f(X) = \scal{E_X f}{f}$ is absolutely continuous (respectively, singular).
The set of all vectors, absolutely continuous with respect to~$H_0,$
form a (closed) subspace of~$\hilb,$ denoted by~$\hilb^{(a)}(H_0).$
The subspace~$\hilb^{(a)}(H_0)$ is called the absolutely continuous subspace (with respect to~$H_0$).
Similarly, the set of all vectors, singular with respect to~$H_0,$
form a subspace of~$\hilb,$ denoted by~$\hilb^{(s)}(H_0).$
The subspace~$\hilb^{(s)}(H_0)$ is called the singular subspace (with respect to~$H_0$).
If there is no danger of confusion, dependence on the self-adjoint operator~$H_0$ is usually omitted.

The absolutely continuous and singular subspaces of~$H_0$ are invariant subspaces of~$H_0.$
That is, if $f \in \hilb^{(a)}(H_0) \cap \euD(H_0)$ then~$H_0f \in \hilb^{(a)}(H_0);$
similarly, if $f \in \hilb^{(s)}(H_0) \cap \euD(H_0)$ then~$H_0f \in \hilb^{(s)}(H_0).$
Also, $\hilb^{(a)}(H_0)$ and $\hilb^{(s)}(H_0)$ are orthogonal, and their direct sum is the whole~$\hilb:$
$$
  \hilb^{(a)} \perp \hilb^{(s)}
$$
and
$$
  \hilb^{(a)} \oplus \hilb^{(s)} = \hilb.
$$
The absolutely continuous (respectively, singular) spectrum $\sigma^{(a)}(H_0)$ (respectively, $\sigma^{(s)}(H_0)$)
of~$H_0$ is the spectrum of the restriction of~$H_0$ to $\hilb^{(a)}(H_0)$
(respectively, to $\hilb^{(s)}(H_0)$).

By $P^{(a)}(H_0)$ we denote the orthogonal projection onto the absolutely continuous subspace of~$H_0.$
If $f \in \hilb,$ then by $f^{(a)}$ we denote the absolutely continuous part
of $f$ with respect to~$H_0,$ i.e. $f^{(a)} = P^{(a)} f.$

The set of all densely defined closed operators on~$\hilb$ will be denoted by~$\clC(\hilb).$

\subsection{Trace-class and Hilbert-Schmidt operators}
\subsubsection{Schatten ideals}
Let $\hilb$ and $\clK$ be Hilbert spaces. A bounded operator $T \colon \hilb \to \clK$ is finite-dimensional,
if its image $\im(T)$ is finite-dimensional.
A bounded operator $T \colon \hilb \to \clK$ is compact, if one of the following equivalent conditions hold:
(1) $T$ is the uniform limit of a sequence of finite-dimensional operators; (2) the closure of the image $T(B_1)$ of the unit ball
$B_1 := \set{f \in \hilb\colon \norm{f} \leq 1}$ is compact in $\clK.$

By $\clL_\infty\brs{\hilb,\clK}$ we denote the set of all compact operators from a Hilbert space~$\hilb$
to a possibly another Hilbert space~$\clK.$ If~$\clK=\hilb,$ then we write $\clL_\infty\brs{\hilb}.$
The same agreement is used in relation to other classes of operators.

The set of compact operators $\clL_\infty\brs{\hilb}$ is an involutive norm-closed two-sided ideal of the algebra $\clB(\hilb).$

Let $T$ be a compact operator in $\hilb.$ If $\lambda \in \mbC$ is an eigenvalue of $T,$ then the root space
of this eigenvalue is the vector space of all those vectors $f$ for which there exists an integer $k =1,2,\ldots,$
such that $(T-\lambda)^k f = 0.$ Root space of any non-zero eigenvalue of a compact operator is finite-dimensional.
The dimension of this root space is called (algebraic) multiplicity of the corresponding eigenvalue. \emph{Spectral measure} $\nu_T$ of a compact operator
$T$
is a measure in $\mbC\setminus\set{0}$ which to every subset $X$ of $\mbC\setminus\set{0}$ assigns
the sum of algebraic multiplicities of all eigenvalues $\lambda$ from the set $X.$ If two bounded
operators $A\colon \hilb \to \clK$ and $B\colon \clK \to \hilb$ are such that the operators $AB$
and $BA$ are compact, then
\begin{equation} \label{F: spec mes(AB)=spec mes(BA)}
  \nu_{AB} = \nu_{BA}.
\end{equation}
Also,
\begin{equation} \label{F: spec mes(A*)=spec mes(A)*}
  \nu_{T^*} = \overline{\nu}_{T}.
\end{equation}

Let~$T$ be a compact operator from~$\hilb$ to~$\clK.$
The absolute value of~$T$ is the self-adjoint compact operator
$$
  \abs{T} := \sqrt{T^*T}.
$$
Singular numbers (or $s$-numbers)
$$
  s_1(T), s_2(T), s_3(T), \ldots
$$
of the operator~$T$ are eigenvalues of $\abs{T},$ listed as a non-increasing sequence,
and such that the number of appearances of each eigenvalue
is equal to the multiplicity of that eigenvalue. Every compact operator $T \in \clL_\infty(\hilb,\clK)$
can be written in the Schmidt representation:
$$
  T = \sum_{n=1}^\infty s_n(T) \scal{\phi_n}{\cdot}\psi_n,
$$
where $(\phi_n)$ is an orthonormal basis in $\hilb,$ and $(\psi_n)$ is an orthonormal basis in~$\clK.$

Singular numbers of a compact operator $T$ have the following property: for any $A,B \in \clB(\hilb)$
\begin{equation} \label{F: Ya (1.6.6)}
  s_n(ATB) \leq \norm{A}\norm{B}s_n(T).
\end{equation}
Also, $s_n(A) = s_n(A^*).$

Let $p \in [1,\infty).$ By $\LpH{p}$ we denote the set of all compact operators~$T$ in~$\hilb,$ such that
$$
  \norm{T}_p := \brs{\sum_{n=1}^\infty s_n^p(T)}^{1/p} < \infty.
$$
The space $(\LpH{p}, \norm{\cdot}_p)$ is an \emph{invariant operator ideal}; this means that
\begin{enumerate}
\item $\LpH{p}$ is a Banach space, 
\item $\LpH{p}$ is a $*$-ideal, that is, if $T \in \LpH{p}$ and $A,B \in \clBH,$ then $T^*, AT, TA \in \LpH{p},$
\item for any $T \in \LpH{p}$ and $A,B \in \clBH$ the following inequalities hold:
$$
  \norm{T}_p \geq \norm{T}, \ \norm{T^*}_p = \norm{T}_p \ \text{and} \ \norm{ATB}_p \leq \norm{A} \norm{T}_p\norm{B}.
$$
\end{enumerate}
A norm which satisfies the above three conditions is called unitarily invariant norm.
The ideal $\LpH{p}$ is called the Schatten ideal of $p$-summable operators.

Note that for the definition of the singular numbers $s_1(T), s_2(T), \ldots$ of an operator $T$ it is immaterial whether $T$ acts from $\hilb$
to $\hilb,$ or maybe from $\hilb$ to another Hilbert space~$\clK.$ In the latter case we write $T \in \clL_p(\hilb,\clK).$

Proofs of the following lemmas can be found in \cite[\S 6.1]{Ya}.

\begin{lemma} \label{L: If A is Lp then A=BT} If $A \in \clL_p(\hilb),$ then $A = BT$ (or $A = TB$) for some $B \in \clL_p(\hilb)$
and some compact operator $T.$
\end{lemma}

\begin{lemma} \label{L: An to A so then AnV to AV} Let $A_1, A_2, \ldots$ be a sequence of bounded operators converging to $A$ in the strong operator topology
and let $p \in [1,\infty].$ If $V \in \clL_p(\hilb),$ then $A_nV \to AV$ and $VA_n \to VA$ in $\clL_p(\hilb).$
\end{lemma}

\begin{lemma} \label{L: if An to A weakly and ... then An to A in Lp} Let $A_1, A_2, \ldots$ be a sequence of operators from $\clL_p(\hilb)$
converging to $A$ in the weak operator topology and such that $\norm{A_n} \leq C < \infty.$ Then $A \in \clL_p(\hilb)$
and for any compact operators $T,Y$
$$
  \lim_{n \to \infty} \norm{T(A_n-A)Y}_p = 0.
$$
\end{lemma}

\subsubsection{Trace-class operators}

Operators from $\LpH{1}$ are called trace-class operators. For trace-class operators~$T$ one defines the trace
$\Tr(T)$ by the formula
\begin{equation} \label{F: Tr(T)=sum (T fj,fj)}
  \Tr(T) = \sum_{j=1}^\infty \scal{T\phi_j}{\phi_j},
\end{equation}
where $\set{\phi_j}_{j=1}^\infty$ is an arbitrary orthonormal basis of~$\hilb.$
Sometimes we write $\Tr_\hilb(T)$ instead of $\Tr(T)$ to indicate the Hilbert space which $T$ acts on.
For a trace-class operator~$T$ the series above is absolutely convergent and
is independent from the choice of the basis $\set{\phi_j}_{j=1}^\infty.$ The trace $\Tr \colon \LpH{1} \to \mbC$
is a continuous linear functional, which satisfies the equality
$$
  \Tr(AB) = \Tr(BA),
$$
whenever both products $AB$ and $BA$ are trace-class. In particular, the above equality holds,
if $A$ is trace-class and $B$ is a bounded operator.

The norm $\norm{\cdot}_1$ is called trace-class norm. For any trace-class operator~$T$ the following equality holds:
$$
  \norm{T}_1 = \Tr(\abs{T}).
$$
More generally,
$$
  \norm{T}_p = \brs{\Tr(\abs{T}^p)}^{1/p}.
$$
The Lidskii theorem asserts that for any trace-class operator~$T$
\begin{equation} \label{F: Lidskii thm}
  \Tr(T) = \sum_{j=1}^\infty \lambda_j,
\end{equation}
where $\lambda_1, \lambda_2, \lambda_3, \ldots$ is the list of eigenvalues of~$T$ counting
multiplicities.\footnote{By multiplicity of an eigenvalue $\lambda_j$ of $T$ we always mean algebraic multiplicity; that is,
the dimension of the vector space $\set{f \in \hilb\colon \exists\, k=1,2,\ldots \ (T-\lambda_j)^kf=0}$}

The dual of the Banach space $\LpH{1}$ is the algebra of all bounded operators~$\clBH$ with uniform norm $\norm{\cdot}:$
every continuous linear functional on $\LpH{1}$ has the form
$$
  T \mapsto \Tr(AT)
$$
for some bounded operator $A \in \clBH,$ and, vice versa, any functional of this form is continuous.

\subsubsection{Hilbert-Schmidt operators}
Operators from $\LpH{2}$ are called Hilbert-Schmidt operators. The norm
$$
  \norm{T}_2 = \sqrt{\Tr(\abs{T}^2)}
$$
is also called Hilbert-Schmidt norm.
For a Hilbert-Schmidt operator $T \in \LpH{2}$ and any orthonormal basis $(\phi_j)$ of~$\hilb$
the following equality holds:
\begin{equation} \label{F: HS-norm of T = sum (Tfj)2}
  \norm{T}_2^2 = \sum\limits_{j=1}^\infty \norm{T\phi_j}^2.
\end{equation}
If~$S,T$ are Hilbert-Schmidt operators, then the product~$ST$ is trace-class
and the following inequality holds:
\begin{equation} \label{F: Holder inequality}
  \norm{ST}_1 \leq \norm{S}_2\norm{T}_2.
\end{equation}
This assertion is a particular case of the more general H\"older inequality which follows.
Let $p, q \in [1,+\infty]$ such that $\frac 1p + \frac 1q = 1.$
If~$S \in \LpH{p}$ and $T \in \LpH{q},$ then~$ST$ is trace-class and
$$
  \norm{ST}_1 \leq \norm{S}_p\norm{T}_q,
$$
where $\norm{\cdot}_\infty$ means the usual operator norm.
This inequality implies that
\begin{equation} \label{F: Sn to S and Tn to T then SnTn to ST}
  \text{if} \ \norm{S_n-S}_p \to 0 \ \text{and} \ \norm{T_n-T}_q \to 0 \ \text{then} \ \norm{S_nT_n-ST}_1 \to 0.
\end{equation}

The ideal $\LpH{2}$ is actually a Hilbert space with scalar product
$$
  \scal{S}{T} = \Tr(S^*T).
$$
So, the dual of $\LpH{2}$ is $\LpH{2}$ itself.

\subsubsection{Fredholm determinant}
\label{SS: Fredholm det}
Let $(\phi_j)$ be an orthonormal basis in~$\hilb.$
If~$T$ is a trace-class operator, then one can define the determinant
$\det(1+T)$ of $1+T$ by the formula
$$
  \det(1+T) = \lim_{n \to \infty} \det \big(\scal{(1+T)\phi_i}{\phi_j}\big)_{i,j=1}^n,
$$
where the determinant in the right hand side is the usual finite-dimensional determinant.
For any trace-class operator $T$ the limit in the right hand side exists
and it does not depend on the choice of the orthonormal basis $(\phi_j).$

We list some properties of the determinant.

The determinant has the product property: for any trace-class operators~$S,T$
the equality holds:
\begin{equation} \label{F: det(AB)=det(A)det(B)}
  \det\Big((1+S)(1+T)\Big) = \det(1+S)\det(1+T).
\end{equation}
If $0 \leq S \leq T \in \clL_1(\hilb),$ then
\begin{equation} \label{F: if S<T, then det(S)<det(T)}
  \det(1+S) \leq \det(1+T).
\end{equation}
Also,
\begin{equation} \label{F: det(1+T*)=det(1+T)*}
  \det(1+T^*) = \overline{\det(1+T)}.
\end{equation}
If $0\leq T \in \clL_1(\hilb),$ then
\begin{equation} \label{F: Tr(T)<det(1+T)}
  \Tr(T) \leq \det(1+T).
\end{equation}

The non-linear functional
\begin{equation} \label{F: det is cont-s}
  \clL_1(\hilb) \ni T \mapsto \det(1+T) \ \ \text{is continuous}.
\end{equation}
The following Lidskii formula holds:
\begin{equation} \label{F: Lidskii for det}
  \det(1+T) = \prod_{j=1}^\infty (1+\lambda_j),
\end{equation}
where $\lambda_1, \lambda_2, \lambda_3, \ldots$ is the list of eigenvalues of~$T$ counting multiplicities.

\subsubsection{The Birman-Koplienko-Solomyak inequality}
The following assertion is called the Birman-Koplienko-Solomyak
inequality\footnote{I thank Prof. P.\,G.\,Dodds for pointing out to this inequality} (cf.~\cite{BKS}).

\begin{thm} \label{T: BKS inequality} If $A$ and $B$ are two non-negative trace-class operators, then
$$
  \norm{\sqrt{A} - \sqrt{B}}_2 \leq \norm{\sqrt{\abs{A-B}}}_2.
$$
\end{thm}
In \cite{BKS} a more general inequality is proved:
$$
  \norm{A^p-B^p}_{\mathfrak S} \leq \norm{\abs{A-B}^p}_{\mathfrak S},
$$
where $p \in (0,1]$ and $\norm{\cdot}_{\mathfrak S}$ is any unitarily invariant norm.

In \cite{Ando}, T.\,Ando (who was not aware of the paper \cite{BKS} at the time of writing \cite{Ando})
proved the following inequality
$$
  \norm{f(A)-f(B)}_{\mathfrak S} \leq \norm{f(\abs{A-B})}_{\mathfrak S},
$$
where $f \colon [0,\infty) \to [0,\infty)$ is any operator-monotone function, that is, a function with property:
if $A \geq B \geq 0,$ then $f(A) \geq f(B) \geq 0.$ Ando's inequality implies the Birman-Koplienko-Solomyak inequality,
since $f(x) = x^p$ with $p \in (0,1]$ is operator-monotone. Ando's inequality was generalized to the setting of semifinite von Neumann algebras
in \cite{DD}.

%
%
%

\begin{lemma} \label{L: lemma Y} If  $A_n \geq 0,$ $A_n \in \clL_1$ for all $n=1,2\ldots,$ and if $A_n \to A$ in~$\clL_1,$
  then $\sqrt{A_n} \to \sqrt{A}$ in~$\clL_2.$
\end{lemma}
\begin{proof} 
  It follows from Theorem \ref{T: BKS inequality}, that 
  $$
    \norm{\sqrt{A_n} - \sqrt{A}}_2 \leq \norm{\sqrt{\abs{A_n-A}}}_2 = \sqrt{\norm{A_n - A}_1} \ \to 0,
  $$
  as $n\to\infty.$ The proof is complete.
\end{proof}

%

\subsection{Direct integral of Hilbert spaces}
\label{SS: direct integral: def}
In this subsection I follow~\cite[Chapter 7]{BSbook}.

Let $\Lambda$ be a Borel subset of~$\mbR$ with a Borel measure $\rho$
(we do not need more general measure spaces here), and let
$$
  \set{\hlambda, \ \lambda \in \Lambda}
$$
be a family of Hilbert spaces,
such that the dimension function
$$
  \Lambda \ni \lambda \mapsto \dim\hlambda \in \set{0,1,2\ldots,\infty}
$$
is measurable.
Let $\Omega_0$ be a countable family of vector-functions (or sections)
$f_1,f_2,\ldots$ such that to each $\lambda \in \Lambda$
$f_j$ assigns a vector $f_j(\lambda) \in \hlambda.$
\begin{defn} \label{D: meas. base}
A family $\Omega_0 = \set{f_1,f_2,\ldots}$ of vector-functions is called a \emph{measurability base},
if it satisfies the following two conditions:
\begin{enumerate}
 \item for a.e. $\lambda \in \Lambda$ the set $\set{f_j(\lambda) \colon j \in \mbN}$
     generates the Hilbert space~$\hlambda;$
 \item the scalar product $\scal{f_i(\lambda)}{f_j(\lambda)}$ is $\rho$-measurable
     for all $i,j=1,2,\ldots$
\end{enumerate}
\end{defn}
A vector-function $\Lambda \ni \lambda \mapsto f(\lambda) \in \hlambda$ is called \emph{measurable},
if $\scal{f(\lambda)}{f_j(\lambda)}$ is measurable for all~$j=1,2,\ldots.$
The set of all measurable vector-functions is denoted by $\hat \Omega_0.$

A measurability base $\set{e_j(\cdot)}$ is called \emph{orthonormal},
if for $\rho$-a.e. $\lambda$ the system $\set{e_j(\lambda)}$ --- after throwing out zero vectors out of it ---
forms an orthonormal base of the fiber Hilbert space $\hlambda.$ (This definition of an orthonormal
measurability base slightly differs from the one given in \cite{BSbook}).

If we have a sequence $f_1,f_2, \ldots$ of vectors in a Hilbert space, then by Gram-Schmidt orthogonalization process
we mean the following procedure: for $n=1,2,\ldots$ we replace the function $f_n$ by zero vector if $f_n$
is a linear combination (in particular, if $f_n=0$) of $f_1,\ldots,f_{n-1},$
otherwise, we replace $f_n$ by the unit vector $e_n$ which is a linear combination of $f_1,\ldots,f_{n},$
which is orthogonal to all $f_1, \ldots, f_{n-1}$ and which satisfies the inequality $\scal{e_n}{f_n}>0.$
Obviously, the systems $\set{f_j}$ and $\set{e_j}$ generate the same linear subspace of the Hilbert space.

\begin{lemma}~\cite[Lemma 7.1.1]{BSbook} \label{L: BS Lemma 7.1.1}
  If $\Omega_0$ is a measurability base, then there exists an orthonormal measurability base $\Omega_1$ such that
  $\hat \Omega_0 = \hat \Omega_1,$ that is, sets of measurable vector-functions generated by $\Omega_0$ and $\Omega_1$ coincide.
\end{lemma}
%

\begin{lemma}~\cite[Corollary 7.1.2]{BSbook}  (i) If $f(\cdot)$
and $g(\cdot)$ are measurable vector-functions, then the function
\ $\Lambda \ni \lambda \mapsto
\scal{f(\lambda)}{g(\lambda)}_\hlambda$ \ is also measurable.
\\ (ii) If $f(\cdot)$ is a measurable vector-function, then the function \ $\Lambda \ni \lambda \mapsto
\norm{f(\lambda)}_\hlambda$ \ is measurable.
\end{lemma}

Two measurable functions are \emph{equivalent}, if they coincide
for $\rho$-a.e. $\lambda \in \Lambda.$ d The direct integral of
Hilbert spaces
\begin{equation} \label{F: def of direct integral}
  \euH = \int_\Lambda^\oplus \hlambda\,\rho(d\lambda)
\end{equation}
consists of all (equivalence classes of) measurable vector-functions $f(\lambda),$ such that
$$
  \norm{f}_\euH^2 := \int_\Lambda \norm{f(\lambda)}_\hlambda^2\,\rho(d\lambda) < \infty.
$$
The scalar product of $f,g \in \euH$ is given by the formula
$$
  \scal{f}{g}_\euH = \int_\Lambda \scal{f(\lambda)}{g(\lambda)}_\hlambda\,\rho(d\lambda).
$$
The set of square-summable vector-functions with this scalar product is a Hilbert space.

\begin{lemma} \label{L: BS Lemma 7.1.5} \cite[Lemma 7.1.5]{BSbook}
Let $\set{\hlambda \colon \lambda \in \Lambda}$ be a family of Hilbert spaces
with an orthogonal measurability base $\set{e_j(\cdot)},$ let $f_0 \in L_2(\Lambda, d\rho)$
be a fixed function, such that $f_0 \neq 0$ for $\rho$-a.e. $\lambda.$ Then the linear span of the set
of functions
$$
  \set{f_0(\lambda) \chi_\Delta(\lambda) e_j(\lambda) \colon j=1,2,\ldots, \Delta \ \text{is a Borel subset of } \ \Lambda}
$$
is dense in the Hilbert space (\ref{F: def of direct integral}).
\end{lemma}

There is an example of the direct integral of Hilbert spaces
relevant to this paper (cf. e.g.~\cite[Chapter 7]{BSbook}). Let
$\mathfrak h$ be a fixed Hilbert space, let $\set{\hlambda, \
\lambda \in \Lambda}$ be a family of subspaces of~$\mathfrak h$
and let $P_\lambda$ be the orthogonal projection onto~$\hlambda.$
Let the operator-function $P_\lambda,\,\lambda \in \Lambda,$ be
weakly measurable. Let $(\omega_j)$ be an orthonormal basis in
$\mathfrak h.$ The family of vector-functions $f_j(\lambda) =
\set{P_\lambda \omega_j}$ is a measurability base for the family
of Hilbert spaces $\set{\hlambda, \ \lambda \in \Lambda}.$ The
direct integral of Hilbert spaces (\ref{F: def of direct
integral}) corresponding to this family is naturally isomorphic
(in an obvious way) to the subspace of $L_2(\Lambda, \mathfrak
h),$ which consists of all measurable square integrable
vector-functions $f(\cdot),$ such that $f(\lambda) \in \hlambda$
for a.e. $\lambda \in \Lambda$~\cite[Chapter 7]{BSbook}.

One of the versions of the Spectral Theorem says that for any self-adjoint operator~$H$
in~$\hilb$ there exists a direct integral of Hilbert spaces~(\ref{F: def of direct integral})
and an isomorphism
$$
  \euF \colon \hilb \to \euH,
$$
such that~$H_0$ is diagonalized in this representation:
$$
  \euF(Hf)(\lambda) = \lambda \euF(f)(\lambda), \ f \in \dom(H),
$$
for $\rho$-a.e. $\lambda \in \Lambda.$

\subsection{Operator-valued holomorphic functions}
In this subsection I follow mainly Kato's book~\cite{Kato}.
Proofs and details can be found in this book of Kato.
See also \cite[Chapter III]{HPh}.

Let~$X$ be a Banach space. Let~$G$ be a region (open connected subset) of the complex plane~$\mbC.$
A vector-function $f \colon G \to X$ is called \emph{holomorphic} (or strongly holomorphic), if for every $z \in G$
the limit
$$
  f'(z) := \lim_{h\to 0} \frac{f(z+h)-f(z)}h
$$
exists. A vector-function $f\colon G \to X $ is holomorphic is and only if it is weakly holomorphic; that is, if for any continuous linear
functional $l$ on~$X$ the function $l(f(z))$ is holomorphic in~$G.$ The proof can be found in
\cite[Theorem III.1.37]{Kato} (see also~\cite[Theorem III.3.12]{Kato},~\cite{RS1}).

A vector-function $f \colon G \to X$ is holomorphic at $z_0 \in G$ if and only if $f$ is analytic
at $z_0,$ that is, if $f$ admits a power series representation
$$
  f(z) = f_0 + (z-z_0)f_1 + (z-z_0)^2f_2 + \ldots + (z-z_0)^n f_n + \ldots
$$
with a non-zero radius of convergence, where $f_0, f_1, \ldots \in X.$

In this paper we consider only holomorphic families of compact operators on a Hilbert space.
In one occasion we consider also a holomorphic family of operators of the form $1+T(z),$ where $T(z)$
is a holomorphic family of compact operators.

Let $T \colon G \to \clL_\infty(\hilb)$ be a holomorphic family of compact operators.
Let $z \in G$ and let $\Gamma$ be a piecewise smooth contour in the resolvent set $\rho(T(z))$ of $T(z).$
Assume that there is only a finite number of eigenvalues (counting multiplicities) $\lambda_1(z), \lambda_2(z), \ldots, \lambda_h(z)$
of $T(z)$ inside of $\Gamma.$
The operator
\begin{equation} \label{F: Riesz integral}
  P(z) = \frac 1{2\pi i} \int_\Gamma \brs{\zeta - T(z)}^{-1}\,d\zeta.
\end{equation}
is an idempotent operator\footnote{We do not use the word projection here, since by projection we mean an orthogonal idempotent.}
(an idempotent operator is a bounded operator $E$ which satisfies the equality $E^2=E$),
corresponding to the set of eigenvalues $\lambda_1(z), \lambda_2(z), \ldots, \lambda_h(z).$
The idempotent $P(z)$ is called the Riesz idempotent operator. By the Cauchy theorem, $P(z)$ does not change, if $\Gamma$ is changed continuously
inside the resolvent set of $T.$ The range of $P(z)$ is the direct sum of root spaces
of eigenvalues $\lambda_1(z), \lambda_2(z), \ldots, \lambda_h(z).$

Let $z \in G.$ If $\lambda_j(z)$ is a simple (that is, of algebraic multiplicity 1) non-zero eigenvalue of $T(z),$ then in some neighbourhood of $z$
it depends holomorphically on $z$ and remains to be simple. So does the idempotent operator $P_j(z)$ associated with the eigenvalue $\lambda_j(z).$
In particular, the eigenvector $v_j(z),$ corresponding to $\lambda_j(z),$ is also a holomorphic function in a neighbourhood of $z.$

The situation is not so simple, if the eigenvalue $\lambda_j(z)$ is not simple at some point $z_0 \in G.$ In this case in a neighbourhood
of $z_0$ the eigenvalue $\lambda_j(z)$ splits (more exactly, may split and most likely does split) into several different eigenvalues
$\lambda_{z_0,1}(z), \lambda_{z_0,2}(z), \ldots, \lambda_{z_0,p}(z),$ where $p$ is the multiplicity of $\lambda_j(z_0).$
The functions $\lambda_{z_0,1}(z), \lambda_{z_0,2}(z), \ldots, \lambda_{z_0,p}(z)$ represent branches of a multi-valued holomorphic function
with branch point $z_0.$ So, they can have an algebraic singularity at $z_0,$ though they are still continuous at $z_0.$
The idempotent of the whole group of eigenvalues $\lambda_{z_0,1}(z), \lambda_{z_0,2}(z), \ldots, \lambda_{z_0,p}(z)$ is holomorphic
in a neighbourhood of $z_0;$ but the idempotent of a subgroup of the group firstly is not defined at $z_0$ and secondly as $z\to z_0$ it (more exactly, its norm) may go to infinity
--- that is, it can have a pole at $z_0$ (see e.g.~\cite[Theorem II.1.9]{Kato}).
Note that this is possible since an idempotent is not necessarily self-adjoint.

All these potentially ``horrible'' things cannot happen, if the holomorphic family of operators $T(z)$ is \emph{symmetric}.
This means that the region~$G$ has a non-empty intersection with the real-axis $\mbR$ and for $\Im z = 0$ the operator $T(z)$
is self-adjoint, or --- at the very least --- normal. Fortunately, in this paper we shall deal only with such symmetric families of holomorphic functions.
Namely, if the family $T(z)$ is symmetric, then \ (1) \ eigenvalues $\lambda_1(z), \lambda_2(z), \lambda_3(z), \ldots$
of $T(z)$ are analytic functions for real values of $z$
(more exactly, they can be enumerated at every point $z$ in such a way that they become analytic) \ (2) \ the eigenvectors $v_1(z),
v_2(z), v_3(z), \ldots$ of $T(z)$ corresponding to those eigenvalues are analytic as well.
The eigenvectors admit analytic continuation to any real point $z_0,$ where some eigenvalue is not simple, since
in this case all Riesz idempotents of the group of isolated eigenvalues are orthogonal,
and --- as a consequence --- bounded. So, the Riesz idempotents cannot have a singularity at $z_0$
and thus are analytic at $z_0.$ It follows that the eigenvalues are also analytic.

For details see Kato's book.

%
\begin{lemma} \label{L: lemma YY}  Let $A \colon [0,1) \ni y \mapsto A_y \in \clL_1(\hilb),$ $A_y \geq 0.$
\begin{enumerate}
  \item[(i)] If $A_y$ is a real-analytic function for $y>0$ with values in~$\clL_1,$
    then $\sqrt{A_y}$ is a real-analytic function for $y>0$ with values in~$\clL_2.$
  \item[(ii)] If, moreover, $A_y$ is continuous at $y = 0$ in~$\clL_1,$ then $\sqrt{A_y}$ is
    continuous at $y = 0$ in~$\clL_2.$
\end{enumerate}
\end{lemma}
\begin{thm} \label{T: lemma X} Let $A_y,$ $y \in [0,1),$ be a family of non-negative Hilbert-Schmidt (respectively, compact) operators,
real-analytic in~$\clL_2$ (respectively, in $\norm{\cdot}$) for
$y>0.$ Then there exists a family $\set{e_j(y)}$ of orthonormal
bases, consisting of eigenvectors of $A_y,$ such that all vector-functions $(0,1)\ni y \mapsto e_j(y),
j=1,2,\ldots,$ are real-analytic functions, as well as the
corresponding eigenvalue functions $\alpha_j(y).$ Moreover, if
$A_y$ is continuous at $y=0$ in the Hilbert-Schmidt norm, then all
eigenvalue functions $\alpha_j(y)$ are also continuous at $y = 0,$
and if $\alpha_j(0)>0,$ then the corresponding eigenvector
function $e_j(y)$ can also be chosen to be continuous at $y=0.$
\end{thm}

\subsubsection{Operator-valued meromorphic functions}
Let $G$ be a region in $\mbC.$ Let $z_0 \in G$ and let $T \colon G \setminus \set{z_0} \to \clBH$
be a holomorphic family of bounded operators in a deleted neighbourhood of $z_0.$ Then $T$ admits a Laurent expansion:
$$
  T(z) = \sum_{n=-\infty}^\infty (z-z_0)^n T_n,
$$
where $T_n$ are bounded operators.

%

Let $N = \min\set{n \colon T_n \neq 0}.$ If $N > -\infty,$ then $T(z)$ is said to have a pole of order $N$
at $z_0.$

\subsubsection{Analytic Fredholm alternative}
This is the following theorem (see e.g. \cite[Theorem VI.14]{RS1}, \cite[Theorem 1.8.2]{Ya}).

\begin{thm} \label{T: analytic Fredholm alternative}
Let $G$ be an open connected subset of $\mbC.$ Let $T \colon G \to \clL_\infty(\hilb)$
be a holomorphic family of compact operators in $G.$ If the family of operators $1+T(z)$
is invertible at some point $z_1 \in G,$ then it is invertible at all points of $G$ except the discrete set
$$
  \euN := \set{z \in G \colon 1 \in \sigma(T(z))}.
$$
Further, the operator-function~$F(z) := (1+T(z))^{-1}$ is meromorphic and the set of its poles is $\euN.$ Moreover, in the expansion of
$F(z)$ in a Laurent series in a neighbourhood of any point $z_0 \in \euN$ the coefficients of negative powers of $z-z_0$ are finite dimensional
operators.
\end{thm}

\subsection{The limiting absorption principle}
\label{SS: lim absorb principle}
We recall two theorems from~\cite{Ya} (cf. also~\cite{BW}), which are absolutely crucial for this paper.
They were established by \Branges\ ~\cite{deBranges} and \Birman\ and \Entina\ ~\cite{BE}.

Because of importance of these two theorems for what follows, we shall give their proofs,
even though they follow verbatim those in \cite{Ya}.
\begin{thm} \label{T: Ya thm 6.1.5} { \rm~\cite[Theorem 6.1.5]{Ya}}
\ Let~$\hilb$ and~$\clK$ be two Hilbert spaces.
Suppose~$H_0$ is a self-adjoint operator in the Hilbert space~$\hilb$
and~$F \colon \hilb \to \clK$ is a Hilbert-Schmidt operator. Then for a.e. $\lambda \in \mbR$ the operator-valued function
$F E^{H_0}_\lambda F^* \in \clL_1(\clK)$ is differentiable in the trace-class norm, the operator-valued function
$F \Im R_{\lambda+i\yy}(H_0)F^*$ has a limit in the trace-class norm as $\yy \to 0,$ and
\begin{equation} \label{F: Ya thm 6.1.5}
  \frac 1\pi \lim_{\yy \to 0} F \Im R_{\lambda+i\yy}(H_0)F^* = \frac {d}{d\lambda} (F E_\lambda F^*),
\end{equation}
where the limit and the derivative are taken in the trace-class norm.
\end{thm}
\begin{proof} Let $T_z = F R_{z}(H_0)F^*$ and let $\Im z > 0.$

(A) Let $D$ be a dense set of linear combinations of some basis in $\hilb.$
Let $f,g \in D.$ The function
$$
  \scal{f}{\frac 1\pi \Im T_z g}
$$
is the Poisson integral of the measure $\Delta \mapsto \scal{f}{E_\Delta^{H_0} g}.$
By Theorem \ref{T: Fatou}, there exists a set of full measure $\Lambda_1$
such that for any $f,g \in D$ and for any $\lambda \in \Lambda_1$ the limit of $\scal{f}{\frac 1\pi \Im T_{\lambda+iy} g}$ as $y \to 0^+$ exists.
For any

(B) The function
$\Tr \brs{\frac 1\pi \Im T_{\lambda+iy}}$
is the Poisson integral of the measure $\Delta \mapsto \Tr(F E_\Delta^{H_0} F^*).$
By Theorem \ref{T: Fatou}, there exists a set $\Lambda_2$ of full measure, such that for all $\lambda \in \Lambda_2$
there exists a limit of $\Tr (\frac 1\pi \Im T_{\lambda+iy})$ as $y \to 0^+.$ Since the operator $\Im T_{\lambda+iy}$
is non-negative, it follows that for any $\lambda \in \Lambda_2$ there exists numbers $C(\lambda), y_0(\lambda) > 0,$ such that
$$
  \norm{\Im T_{\lambda+iy}}_1 \leq C(\lambda)
$$
for all $y < y_0(\lambda).$

(C) It follows from (A) and (B) that for all $\lambda$ from the full set $\Lambda = \Lambda_1 \cap \Lambda_2,$
the operator $\Im T_{\lambda+iy}$ has weak limit as $y \to 0^+.$

(D)
By Lemma \ref{L: If A is Lp then A=BT}, the operator $F$ can be written in the form $F = TG,$ where $T$ is a compact operator
and $G \in \clL_2(\hilb).$ By (B), for a.e. $\lambda \in \mbR$ \ the operator $\norm{G \Im R_{\lambda+iy}(H_0)G^*}_1 \leq C(\lambda)$ as $y \to 0^+,$ and by (C)
for a.e. $\lambda \in \mbR$ the operator $G\Im R_{\lambda+iy}G^*$ weakly converges as $y \to 0^+.$ Combining this with Lemma \ref{L: if An to A weakly and ... then An to A in Lp}, it follows
that $F\Im R_{\lambda+iy}F^* = T(GR_{\lambda+iy}G^*)T^*$ converges in $\clL_2(\hilb)$ for a.e. $\lambda \in \mbR.$

(E) Proof of $\clL_1$-differentiability of the function $\lambda \mapsto F E_\lambda F^*$ and of (\ref{F: Ya thm 6.1.5}) is similar and we omit the details
which can be found in \cite[\S 6.1]{Ya}.
\end{proof}
Another reason for omitting the second part of the proof of this theorem is that,
while for this paper it is crucial that the $\clL_1$-limit of $F \Im R_{\lambda+i\yy}(H_0)F^*$ exists for a.e. $\lambda,$ differentiability
of the function $\lambda \mapsto F E_\lambda F^*$ and the equality (\ref{F: Ya thm 6.1.5}) are not so important.
In fact, as Fatou's Theorem \ref{T: Fatou} shows, the derivative of a function and the limit value of its Poisson integral are in some sense
identical notions, that is, the limit of Poisson integral can be considered as a modified definition of the derivative, and one can choose to work with either of them.
In the framework of scattering theory, the limit of Poisson integral is much more convenient. On the other hand, theorems of analysis are proved for usual derivative,
and Fatou's theorem allows to exploit properties of the usual derivative.

\begin{thm} \label{T: Ya thm 6.1.9} { \rm~\cite[Theorem 6.1.9]{Ya}}
\ Suppose~$H_0$ is a self-adjoint operator in a Hilbert space~$\hilb$
and~$F \in \clL_2(\hilb,\clK).$ Then for a.e. $\lambda \in \mbR$ the operator-valued function
$F R_{\lambda\pm i\yy}(H_0)F^*$ has a limit in~$\clL_2(\clK)$ as $\yy \to 0.$
\end{thm}
\begin{proof} Let $T_z = F R_{z}(H_0)F^*$ and let $\Im z > 0.$

(A) Claim: $\abs{\det\brs{1-iT_{z}}} \geq 1.$

Proof. We have, using (\ref{F: det(1+T*)=det(1+T)*}) and (\ref{F: det(AB)=det(A)det(B)}),
\begin{equation*}
  \begin{split}
    \abs{\det\brs{1-iT_{z}}}^2 & = \det\brs{(1-iT_{z})^*}\det\brs{1-iT_{z}}
     \\ & = \det\brs{1+iT_{\bar z}}\det\brs{1-iT_{z}}
     \\ & = \det\big[\brs{1+iT_{\bar z}}\brs{1-iT_{z}}\big]
     \\ & = \det\brs{1+iT_{\bar z} -iT_{z} + T_{\bar z} T_{z}}.
  \end{split}
\end{equation*}
Since $iT_{\bar z} -iT_{z} = 2\Im T_{z} \geq 0$ and $T_{\bar z} T_{z} \geq 0,$
it follows from (\ref{F: if S<T, then det(S)<det(T)}) that $\det\brs{1+iT_{\bar z} -iT_{z} + T_{\bar z} T_{z}} \geq 1.$
Hence, $\abs{\det\brs{1-iT_{z}}} \geq 1.$

(B) Claim: for a.e. $\lambda \in \mbR$ the limit $\lim_{y \to 0^+} \det\brs{1-iT_{\lambda+iy}}$ exists.

Proof. Let $f(z) = \det\brs{1-iT_{z}}.$ The function $g(z) = 1/f(z)$ is holomorphic in the upper-half plane $\mbC_+,$
and by (A) it is bounded. It follows from Theorem \ref{T: Luzin-Privalov}(a) that $g(\lambda+i0)$ exists for a.e. $\lambda$ and by Theorem \ref{T: Luzin-Privalov}(b)
this limit is non-zero for a.e. $\lambda.$ It follows that $f(\lambda+i0)$ exists and is finite for a.e. $\lambda.$

(C) Claim: for a.e. $\lambda \in \mbR$ \ $\norm{T_{\lambda+iy}}_2 \leq C(\lambda)$ as $y \to 0^+.$

Proof. Using (\ref{F: Tr(T)<det(1+T)}), we have
\begin{equation*}
    \norm{T_{z}}_2^2 = \Tr(T_{\bar z}T_{z}) \leq \det\brs{1+T_{\bar z}T_{z}}.
\end{equation*}
Since $iT_{\bar z} -iT_{z} = 2\Im T_{z} \geq 0,$ it follows from (\ref{F: if S<T, then det(S)<det(T)}) that
$$
  \norm{T_{z}}_2^2 \leq \det\brs{1+T_{\bar z}T_{z}} \leq \det\brs{1+iT_{\bar z} -iT_{z} + T_{\bar z} T_{z}} = \abs{\det\brs{1-iT_{z}}}^2.
$$
Now (B) completes the proof.

(D) Claim: for a.e. $\lambda$ the operator $T_{\lambda+iy}$ weakly converges as $y \to 0^+.$

Proof. Let $D$ be a dense set in $\hilb$ of linear combinations of some basis. By Theorem \ref{T: Ya 1.2.5}, there exists a set of full measure $\Lambda$
such that for any $f,g \in D$ and for any $\lambda \in \Lambda$ the limit of $\scal{f}{T_{\lambda+iy} g}$ as $y \to 0^+$ exists.
It follows from this and (C) that $T_{\lambda+iy}$ weakly converges as $y \to 0^+.$

(E) By Lemma \ref{L: If A is Lp then A=BT}, the operator $F$ can be written in the form $F = TG,$ where $T$ is a compact operator
and $G \in \clL_2(\hilb).$ By (C), for a.e. $\lambda \in \mbR$ \ the operator $\norm{GR_{\lambda+iy}G^*}_2 \leq C(\lambda)$ as $y \to 0^+$ and by (D)
for a.e. $\lambda \in \mbR$ the operator $GR_{\lambda+iy}G^*$ weakly converges as $y \to 0^+.$ Combining this with Lemma \ref{L: if An to A weakly and ... then An to A in Lp}, it follows
that $FR_{\lambda+iy}F^* = T(GR_{\lambda+iy}G^*)T^*$ converges in $\clL_2(\hilb)$ for a.e. $\lambda \in \mbR.$
\end{proof}

\Naboko\ has shown that in this theorem the convergence in~$\clL_2(\clK)$ can be replaced by
the convergence in~$\clL_p(\clK)$ with any $p>1.$ In general, the convergence in~$\clL_1(\clK)$ does not hold
(cf. ~\cite{Nab87,Nab89,Nab90}).

Theorem \ref{T: Ya thm 6.1.5} plays a more key role in this paper compared to Theorem \ref{T: Ya thm 6.1.9}.
Moreover, existence of the limit of $F R_{\lambda\pm i\yy}(H_0)F^*$ in the Hilbert-Schmidt norm is not necessary at all
for what follows, --- existence of the limit in the usual norm will suffice.
For the purposes of this paper, in the second condition of Definition \ref{D: def of Lambda from intro-n}
of the full set $\Lambda(H_0,F)$ one can replace norm convergence by $\clL_p$-convergence with any $p \in (1,\infty],$
--- the set $\Lambda(H_0,F)$ will still have the full Lebesgue measure.
Since the norm-topology is weaker than the Hilbert-Schmidt topology, the set $\Lambda(H_0,F)$ becomes larger, if we use norm convergence
in Definition \ref{D: def of Lambda from intro-n}(ii), but this is not a point. It turns out that
for generalization of the results of this paper to the case of non-compact perturbations norm convergence is more
preferable: this allows to enlarge the set of non-compact perturbations covered by the theory.

%

\section{Framed Hilbert space}
\subsection{Definition} \label{SS: def of frame}
In this section we introduce the so called framed Hilbert space and study
several objects associated with it. Before giving formal definition, I would like
to explain the idea which led to this notion.

Let~$H_0$ be a self-adjoint operator on a Hilbert space~$\hilb,$ and let~$H_1$ be its trace-class perturbation.
Our ultimate purpose is to explicitly define the wave matrix $w_{\pm}(\lambda; H_1,H_0)$
at a fixed point $\lambda$ of the spectral line. The wave matrix $w_{\pm}(\lambda; H_1,H_0)$ acts between fiber Hilbert spaces
$\hlambda(H_0)$ and~$\hlambda(H_1)$ from the direct integrals of Hilbert spaces
$$
  \int_{\hat \sigma(H_0)}^\oplus \mathfrak h_\lambda(H_0)\,d\lambda \quad \text{and} \quad \int_{\hat \sigma(H_1)}^\oplus \mathfrak h_\lambda(H_1)\,d\lambda,
$$
diagonalizing absolutely continuous parts of the operators~$H_0$ and~$H_1,$ where $\hat \sigma(H_j)$ is a core of the spectrum of~$H_j.$
Before defining $w_{\pm}(\lambda; H_1,H_0),$ one should first define explicitly the fiber Hilbert spaces
$\hlambda(H_0)$ and~$\hlambda(H_1).$
Moreover, given a vector $f \in \hilb,$ it is necessary to be able to assign an explicit value $f(\lambda) \in \hlambda$
of the vector $f$ at a single point $\lambda \in \mbR.$
Obviously, the vectors $f(\lambda)$ generate the fiber Hilbert space $\hlambda.$
So, one of the first important questions to ask is:
\begin{equation} \label{F: What is f(l)}
  \text{What is} \quad f(\lambda)?
\end{equation}
Actually, since the measure $d\lambda$ in the direct integral decomposition of the Hilbert space can be replaced
by any other measure $\rho(d\lambda)$ with the same spectral type, it is not difficult to see, that $f(\lambda)$
does not make sense, as it is. Indeed, let us consider an operator of multiplication by a continuous function $f(x)$
on the Hilbert space $L_2([-\pi,\pi]).$
The Hilbert space $L_2([-\pi,\pi])$ can be represented as a direct integral of one-dimensional
Hilbert spaces $\mathfrak h_\lambda \simeq \mbC:$
$$
  L_2([-\pi,\pi]) = \int_{[-\pi,\pi]}^\oplus \mbC\,dx.
$$
(As a measurability base one can take here the system which consists of only one function, say, $e^{inx},$
where $n$ is any integer; in particular, a non-zero constant function will do).
Since $f(x)$ is continuous we can certainly say what is, say, $f(0).$ But the measure $d\theta$ can be replaced by any
other measure of the same spectral type; for example, by
$$
  d\rho(x) = \brs{2+\sin \frac1{x}}dx.
$$
The Spectral Theorem says, that the operator of multiplication $M_f$ by $f(x)$ does not notice this change of measure;
that is, the operator $M_f$ will stay in the same unitary equivalence class. At the same time, now it is difficult to say
what $f(0)$ is. That is, the value $f(\lambda) \in \hlambda$ of a vector $f$ at a point $\lambda$ of the spectral line is affected
by the choice of a measure in its spectral type. As a consequence, the expression $f(\lambda)$ does not make sense.
The measure $\rho$ defined by the above formula is far from being the worst scenario: instead of $\sin \frac1{x}$
one can take, say, any $L_\infty$-function bounded below by $-1.$ In this case, we have a difficulty to define the value of $f$
at any point.

In order to give meaning to $f(\lambda),$ one needs to introduce some additional
structure. (One can see that fixing a measure $d\rho$ in the spectral type does not help).
There are different approaches to this problem. Firstly, if we try to
single out what enables to give meaning to $f(x)$ for all $x$ in the case of the measure $dx,$
we see that this additional structure is of geometric character: it is the (Riemannian) metric.
The problem is that in the setting of arbitrary self-adjoint operators we don't have a metric.
But the metric is fully encoded
in the Dirac operator $\frac 1i \frac d{d\theta}$ (see \cite[Chapter VI]{CoNG}).
The operator $\frac 1i \frac d{d\theta}$ on $L_2(\mbT)$ has discrete spectrum
and so it is identified by a sequence of its eigenvalues and by the orthonormal basis of its eigenvectors.
This type of data consisting of numbers and vectors of the Hilbert space can be easily dealt with in the abstract situation.

So, to see in another way what kind of additional structure can allow to define~$f(\lambda),$ let us assume, to begin with,
that there is a fixed unit vector $\phi_1 \in \hilb.$ In this case, it is possible to define the number
$$
  \scal{f(\lambda)}{\phi_1(\lambda)}
$$
for a.e. $\lambda,$ by formula (\ref{F: f(l)=1/pi Im C(l+i0)}), since the above scalar product is a summable
function of $\lambda.$
Note, that neither $f(\lambda),$ nor $\phi_1(\lambda)$ are yet defined,
but their scalar product is defined.

If there are many enough (unit) vectors $\phi_1,\phi_2,\ldots,$ then one can hope
that the knowledge of all the scalar products $\scal{f(\lambda)}{\phi_j(\lambda)}$
will allow to recover the vector $f(\lambda) \in \hlambda.$ (Note, that we don't know yet
what exactly~$\hlambda$ is). But this is still not the case. Note that the scalar product
$\scal{\phi_j(\lambda)}{\phi_k(\lambda)}$ should satisfy the formal equality
\begin{equation} \label{F: key equality}
  \scal{\phi_j(\lambda)}{\phi_k(\lambda)} = \la \phi_j | \delta(H_0-\lambda) | \phi_k \ra
  = \frac 1\pi \la \phi_j | \Im (H_0-\lambda-i0)^{-1} | \phi_k \ra,
\end{equation}
where $\la \phi | A | \psi \ra$ is physicists' (Dirac's) notation for $\scal{\phi}{A\psi}.$
That this equality must hold for the absolutely continuous part $H_0^{(a)}$ can be seen from
$$
  \scal{\phi_j}{j_\eps(H_0^{(a)}-\lambda)\phi_k} = \int_\mbR j_\eps(\mu-\lambda) \scal{\phi_j(\mu)}{\phi_k(\mu)}\,d\mu,
$$
where $j_\eps$ is an approximate unit for the Dirac $\delta$-function.
In order to satisfy this key equality, we use an artificial trick. We assign to each vector
$\phi_j$ a weight $\kappa_j>0$ such that $(\kappa_j) = (\kappa_1,\kappa_2,\ldots) \in \ell_2.$
Now, we form the matrix
$$
  \phi(\lambda) := \brs{\kappa_j\kappa_k \frac 1\pi \la \phi_j | \Im (H_0-\lambda-i0)^{-1} | \phi_k \ra}.
$$
Using the limiting absorption principle (Theorem \ref{T: Ya thm 6.1.5}),
it can be easily shown that this matrix is a non-negative trace-class matrix. Now,
if we define $\phi_j(\lambda)$ as the~$j$th column of the square root of the matrix $\phi(\lambda)$ over $\kappa_j,$
then $\phi_j(\lambda)$ will become an element of $\ell_2$ and
the equality (\ref{F: key equality}) will be satisfied. For all $\lambda$ from some explicit set of full Lebesgue measure,
which depends only on~$H_0$ and the data $(\phi_j,\kappa_j),$
this allows to define the value $f(\lambda)$
at $\lambda$ for each $f = \phi_j, \ j = 1,2,\ldots$
and, consequently, for any vector from the dense manifold of finite linear combinations of $\phi_j.$
Finally, the fiber Hilbert space~$\hlambda$ can be defined as a linear subspace of $\ell_2$ generated
by $\phi_j(\lambda)$'s.

Evidently, the data $(\phi_j,\kappa_j)$ can be encoded in a single Hilbert-Schmidt operator
\ $F = \sum\limits_{j =1}^\infty \kappa_j \scal{\phi_j}{\cdot}\psi_j,$ \
where $(\psi_j)$ is an arbitrary orthonormal system in a possibly another Hilbert space.
Actually, in the case of $\hilb = L_2(M)$ discussed above, where $M$ is a Riemannian manifold, $F$
can be chosen to be the appropriate negative power of the Laplace-Beltrami operator $\Delta.$

This justifies introduction of the following
\begin{defn} A \emph{frame} in a Hilbert space~$\hilb$ is a Hilbert-Schmidt operator \
$F \colon \hilb \to \clK,$ \
with trivial kernel and co-kernel, of the following form
\begin{equation} \label{F: Frame=...}
  F = \sum\limits_{j =1}^\infty \kappa_j \scal{\phi_j}{\cdot}\psi_j,
\end{equation}
where~$\clK$ is another Hilbert space,
and where $(\kappa_j) \in \ell_2$ is a fixed decreasing sequence of $s$-numbers of~$F,$
all of which are non-zero, $(\phi_j)$ is a fixed orthonormal basis in~$\hilb,$ and
$(\psi_j)$ is an orthonormal basis in~$\clK.$

A \emph{framed Hilbert space} is a pair $(\hilb,F),$ consisting of a Hilbert space~$\hilb$
and a frame~$F$ in~$\hilb.$
\end{defn}
Throughout this paper we shall work with only one frame~$F,$ with some restrictions imposed later on it, and
$\kappa_j,$ $\phi_j$ and $\psi_j$ will be as in the formula (\ref{F: Frame=...}).

What is important in the definition of a frame is the orthonormal basis $(\phi_j)$
and the $\ell_2$-sequence of weights $(\kappa_j)$ of the basis vectors.
The Hilbert space~$\clK$ is of little importance, if any.
For the most part of this paper,
one can take~$\clK = \hilb$ and~$F$ to be self-adjoint, but later we shall see
that the more general definition given above is more useful.

A frame introduces rigidity into the Hilbert space. In particular, a frame fixes a measure
on the spectrum of a self-adjoint operator by the formula $\mu(\Delta) = \Tr(FE_\Delta^H F^*).$
In other words, a frame fixes a measure in its spectral type.

%

For further use, we note trivial relations
\begin{equation} \label{F: F phi j=...}
  F\phi_j = \kappa_j \psi_j, \quad F^*\psi_j = \kappa_j \phi_j.
\end{equation}

\subsection{Spectral triple associated with an operator on a framed Hilbert space}

To the pair $(H_0,F)$ consisting of a self-adjoint operator $H_0$
and a frame operator $F$ one can assign a spectral triple
\cite{CoNG}.
The involutive algebra~$\clA$ of a spectral triple
$(\clA,\hilb,\abs{F}^{-1})$ is given by
\begin{equation*}
  \clA = \set{\phi(H) \colon \phi \in C_b(\mbR), \ [\abs{F}^{-1},\phi(H)] \in \clBH}.
\end{equation*}
Here the class $C_b$ of all continuous bounded functions on $\mbR$ can be replaced by $L_\infty.$
Let us check that $\clA$ is an algebra. If $\phi_1(H), \, \phi_2(H) \in \clA$ and $\alpha_1, \, \alpha_2 \in \mbC,$
then obviously $\phi_1(H)^* = \bar \phi(H) \in \clA$ and
$\alpha_1 \phi_1(H) + \alpha_2 \phi_2(H)\in \clA.$ Now, if $\phi_1(H), \phi_2(H) \in \clA,$
then the operator
$$
  [\abs{F}^{-1},\phi_1(H)\phi_2(H)] = [\abs{F}^{-1},\phi_1(H)]\phi_2(H) + \phi_1(H)[\abs{F}^{-1},\phi_2(H)]
$$
is also bounded, so that $\phi_1(H)\phi_2(H) \in \clA.$
Consequently, $\clA$ is an involutive algebra. The second axiom of the spectral triple is satisfied obviously,
that is the resolvent $(\abs{F}^{-1}-z)^{-1}$ of the operator $\abs{F}^{-1}$ is compact for non-real $z.$

\subsection{Non-compact frames}
In the pair $(H_0,F),$ consisting of a self-adjoint operator $H_0$
and a frame operator $F,$ $H_0$ and $F$ are independent of each
other. One can consider more general pairs $(H_0,F)$ which, I
believe, may be useful in generalizing the present work to the
case of non-trace-class (non-compact) perturbations $V.$

A generalized frame operator $F$ for a self-adjoint operator $H_0$
is an operator $F \colon \hilb \to \clK,$ such that \ (1) domain
of $F$ contains all subspaces $E^{H_0}_\Delta \hilb$ with bounded
Borel $\Delta,$ \ (2) for any bounded Borel $\Delta$ the operator
$F E^{H_0}_\Delta$ is Hilbert-Schmidt, and (3) the kernel of $F
E^{H_0}_\Delta$ as an operator on $E^{H_0}_\Delta \hilb$ is trivial.

For such pairs one can construct a sheaf of Hilbert spaces
over $\mbR$ which diagonalizes $H_0.$ Details of this construction
and its applications to scattering theory for non-compact
perturbations appear elsewhere.

\subsection{The set $\LambHF{H_0}$ and the matrix $\phi(\lambda)$}
Let~$H_0$ be a self-adjoint operator in a framed Hilbert space $(\hilb,F).$

By $E_\lambda = E_\lambda^{H_0},$ $\lambda \in \mbR,$ we denote
the family of spectral projections of~$H_0.$ For any (ordered)
pair of indices $(i,j)$ one can consider a finite (signed) measure
\begin{equation} \label{F: def of m(ij)}
  m_{ij}(\Delta) := \la \phi_i, E_\Delta^{H_0}\phi_j \ra.
\end{equation}
We denote by
\begin{equation} \label{F: def of Lambda0}
  \Lambda_0(H_0,F)
\end{equation}
the intersection of all the sets $\Lambda(m_{ij}),$ $i,j \in
\mbN$ \ (see subsection \ref{SS: L(m)}), even though it depends only on~$H_0$ and the vectors
$\phi_1, \phi_2, \phi_3, \ldots$ \
So, for any $\lambda \in \Lambda_0(H_0,F)$ the limit
$$
  \phi_{ij}(\lambda) := \frac 1\pi \kappa_i\kappa_j\la \phi_i, \Im R_{\lambda+ i0}(H_0)\phi_j \ra
$$
exists. 
It follows that, for any $\lambda \in \Lambda_0(H_0,F),$ one can form an infinite matrix
$$
  \phi(\lambda) = \brs{\phi_{ij}(\lambda)}_{i,j=1}^\infty.   
$$

Our aim is to consider $\phi(\lambda)$ as an operator on $\ells.$
Evidently, the matrix $\phi(\lambda)$ is symmetric in the sense that for any $i,j = 1,2,\ldots$
$$
  \overline{\phi_{ij}(\lambda)} = \phi_{ji}(\lambda).
$$
But it may turn out that $\phi(\lambda)$ is not a matrix of a bounded,
or even of an unbounded, operator on $\ells.$ So, we have to investigate the set of points,
where $\phi(\lambda)$ determines a bounded self-adjoint operator on $\ells.$
As is shown below,
it turns out that $\phi(\lambda)$ is a trace-class operator on a set of full measure.


In the following definition one of the central notions of this paper is introduced.
\begin{defn} \label{D: Lambda(H0,F)}
The standard set of full Lebesgue measure $\LambHF{H_0},$ associated with
a self-adjoint operator~$H_0$ acting on a framed Hilbert space $(\hilb,F),$
consists of those points $\lambda \in \mbR,$ at which the limit
of~$FR_{\lambda+iy}(H_0)F^*$ (as $y \to 0^+$) exists in the uniform norm and the limit of
$F \Im R_{\lambda+iy}(H_0)F^*$ exists in~$\clL_1$-norm.

In other words, a number $\lambda$ belongs to $\LambHF{H_0}$ if and only if it belongs
to both sets of full measure from Theorems \ref{T: Ya thm 6.1.5} and \ref{T: Ya thm 6.1.9}.
\end{defn}

\begin{prop} \label{P: Lambda(H0) has full meas}
  For any self-adjoint operator~$H_0$ on a framed Hilbert space $(\hilb,F)$
  the set $\LambHF{H_0}$ has full Lebesgue measure.
\end{prop}
\begin{proof} This follows from Theorems \ref{T: Ya thm 6.1.5} and \ref{T: Ya thm 6.1.9}.
\end{proof}
The following proposition gives one of the two main properties of the set $\LambHF{H_0}.$
\begin{prop} \label{P: Lambda subset Lambda1}
Let~$H_0$ be a self-adjoint operator acting on a framed Hilbert space $(\hilb,F).$
If $\lambda \in \LambHF{H_0},$ then the matrix $\phi(\lambda)$ exists, is non-negative and is trace-class.
\end{prop}
\begin{proof} Let $\lambda \in \LambHF{H_0}.$
Since for $\lambda \in \LambHF{H_0}$ the limit
$$
  F R_{\lambda\pm i0}(H_0)F^* = \lim\limits_{y\to 0^+} F R_{\lambda\pm iy}(H_0)F^*
$$
exists in the Hilbert-Schmidt norm, it follows that for any pair $(i,j)$ the limit
$$
  P^*_iF R_{\lambda\pm i0}(H_0) F^*P_j = \lim\limits_{y\to 0^+} P^*_iF R_{\lambda\pm iy}(H_0) F^*P_j
$$
also exists in the Hilbert-Schmidt norm, where $P_j = \scal{\phi_j}{\cdot}\psi_j.$ This is equivalent to the existence of the limit
$$
  \la \phi_i, R_{\lambda\pm i0}(H_0)\phi_j \ra = \lim\limits_{y\to 0^+} \la \phi_i, R_{\lambda\pm iy}(H_0)\phi_j \ra.
$$
Hence, $\LambHF{H_0} \subset \Lambda_0(H_0,F);$ so $\phi(\lambda)$ exists for any $\lambda \in \LambHF{H_0}.$

The matrix $\phi(\lambda)$ is unitarily equivalent to the non-negative trace-class operator $F \frac 1\pi
\Im R_{\lambda+iy}(H_0)F^*.$ Hence, $\phi(\lambda)$ is also non-negative and trace-class.

\end{proof}
%

\begin{lemma} \label{L: phi(lambda) is meas}
  The operator function $\LambHF{H_0} \ni \lambda \mapsto \phi(\lambda) \in \clL_1(\ells)$
  is measurable.
\end{lemma}
Indeed, $\phi(\lambda)$ is an a.e. point-wise limit of matrices $\phi(\lambda+iy)$ with continuous matrix elements.

\subsection{A core of the singular spectrum $\mbR \setminus \Lambda(H_0,F)$}

We call a null set $X \subset \mbR$ a \emph{core} of the singular spectrum of~$H_0,$
if the operator $E^{H_0}_{\mbR \setminus X} H_0$ is absolutely continuous.
Evidently, any core of the singular spectrum contains the pure point spectrum.
Apart of it, a core of the singular spectrum contains a null Borel
support of the singular continuous spectrum.

\begin{lemma} \label{L: if the complement of Lambda is not a core ...}
Let~$H_0$ be a self-adjoint operator on~$\hilb$
and let $\Lambda$ be a full set. If~$\mbR \setminus \Lambda$ is not a core of
the singular spectrum of~$H_0,$ then there exists a null set $X \subset \Lambda,$ such that $E_X \neq 0.$
\end{lemma}
\begin{proof} Let $Z_a$ be a full set such that $E_{Z_a}$ is the projection onto the absolutely continuous
subspace of~$H_0 E_\Lambda.$ Such a set exists by~\cite[Lemma 1.3.6]{Ya}.
If~$\mbR \setminus \Lambda$ is not a core of the singular spectrum, then the operator~$H_0 E_\Lambda$
is not absolutely continuous. So, the set $X := \Lambda \setminus Z_a$ is a null set and $E_X \neq 0.$
\end{proof}
\begin{prop} \label{P: Lambda 0 is a core of sing. sp}
For any self-adjoint operator~$H_0$ on a framed Hilbert space $(\hilb,F),$
the set~$\mbR \setminus \Lambda_0(H_0,F)$ is a core of the singular spectrum of~$H_0.$
\end{prop}
\begin{proof}
Assume the contrary. Then by Lemma \ref{L: if the complement of Lambda is not a core ...}
there exists a null subset $X$ of $\Lambda_0(H_0,F)$
such that $E_X \neq 0.$ 
Since $\brs{\phi_j}$ is a basis, there exists $\phi_j,$ such that $E_X \phi_j \neq 0.$
Hence, $\scal{E_X \phi_j}{\phi_j} \neq 0,$ that is,
$$m_{jj}^{(s)}(X) = m_{jj}(X) \neq 0,$$
where $m_{jj}$ is the spectral measure of $\phi_j$ (see (\ref{F:
def of m(ij)})). Since $X \subset \Lambda(m_{jj}),$ this
contradicts the fact that the complement of $\Lambda(m_{jj})$ is a
Borel support of $m_{jj}^{(s)}$ (see Theorem \ref{T: complement of
Lambda(F)}).
\end{proof}
Since $\LambHF{H_0} \subset \Lambda_0(H_0,F),$ it follows that
\begin{cor} \label{C: Lambda is a core of sing. sp}
For any self-adjoint operator~$H_0$ on a framed Hilbert space $(\hilb,F),$
the set~$\mbR \setminus \LambHF{H_0}$ is a core of the singular spectrum of~$H_0.$
\end{cor}
Since $\LambHF{H_0}$ has full measure,
this corollary means that the set $\LambHF{H_0}$ cuts out the singular spectrum of~$H_0$
from~$\mbR.$ Given a frame operator~$F \in \clL_2(\hilb,\clK),$ we consider the set
$\mbR \setminus \LambHF{H_0}$ as a standard core of the singular spectrum of~$H_0,$
associated with the given frame~$F.$


\subsection{The Hilbert spaces~$\hilb_\alpha(F)$\!}
Let $\alpha \in \mbR.$ In analogy with Sobolev spaces $W_\alpha^2$ (see e.g.~\cite[\S IX.6]{RS2},~\cite{Co95LMPh}),
given a framed Hilbert space $(\hilb,F),$ 
we introduce the Hilbert spaces~$\hilb_\alpha(F).$ By  definition,
$\hilb_\alpha(F)$ is the completion of the linear manifold
\begin{equation} \label{F: euD - def-n}
  \euD = \euD(F) := \set{f \in \hilb \colon \ f = \sum\limits_{j=1}^N \beta_j \phi_j, \ N < \infty}
\end{equation}
in the norm
$$
  \norm{f}_{\hilb_\alpha(F)} = \norm{\abs{F}^{-\alpha} f},
$$
with the scalar product
$$
  \scal{f}{g}_{\hilb_\alpha(F)} = \scal{\abs{F}^{-\alpha} f}{\abs{F}^{-\alpha} g}.
$$
That is, if $f = \sum\limits_{j=1}^N \beta_j\phi_j,$ then
\begin{equation} \label{F: norm(f)alpha}
  \norm{f}_{\hilb_\alpha(F)} = \brs{\sum\limits_{j=1}^N \abs{\beta_j}^2 \kappa_j^{-2\alpha} }^{1/2}.
\end{equation}
Since~$F$ has trivial kernel, $\norm{\cdot}_{\hilb_\alpha(F)}$ is indeed a norm.
The scalar product of vectors $f = \sum\limits_{j=1}^N \alpha_j \phi_j$ and
$g = \sum\limits_{j=1}^N \beta_j \phi_j$ in~$\hilb_\alpha(F)$ is given by the formula
$$
  \scal{f}{g}_{\hilb_\alpha(F)} = \sum\limits_{j=1}^N \bar{\alpha}_j \beta_j \kappa_j^{-2\alpha}.
$$

The Hilbert space~$\hilb_\alpha(F)$ has a natural orthonormal basis $(\kappa_j^{\alpha}\phi_j).$
Since $$\abs{F}^\gamma (\kappa_j^{\alpha}\phi_j) = \kappa_j^{\alpha+\gamma}\phi_j,$$ it follows that
\begin{lemma} \label{L: F maps H(A) to H(a+1)}
  For any $\alpha, \gamma \in \mbR$
  the operator $\abs{F}^{\gamma} \colon \euD \to \euD$ is unitary as
  an operator from~$\hilb_{\alpha}(F)$ to~$\hilb_{\alpha+\gamma}(F).$
\end{lemma}
It follows that all Hilbert spaces~$\hilb_\alpha(F)$ are naturally isomorphic,
the natural isomorphism being the appropriate power of $\abs{F}.$

Plainly,~$\hilb_0(F) = \hilb.$ Let $\alpha, \beta \in \mbR.$ If $\alpha < \beta,$ then \
$\hilb_\beta(F) \subset \hilb_\alpha(F).$
The inclusion operator
$$
  i_{\alpha,\beta}\colon \hilb_\beta(F) \hookrightarrow \hilb_\alpha(F)
$$
is compact with Schmidt representation
$$
  i_{\alpha,\beta} = \sum_{j=1}^\infty \kappa_j^{\beta-\alpha} \scal{\kappa_j^\beta\phi_j}{\cdot}_{\hilb_\beta}\kappa_j^\alpha \phi_j.
$$
It follows that the $s$-numbers of the inclusion operator $i$ are $s_j(i) = \kappa_j^{\beta-\alpha}.$
In particular, the inclusion operator
$$
  i_{\alpha,\alpha+1}\colon \hilb_{\alpha+1}(F) \hookrightarrow \hilb_\alpha(F)
$$
is Hilbert-Schmidt with $s$-numbers $s_j=\kappa_j.$


Since we shall work in a fixed framed Hilbert space $(\hilb,F),$
the argument~$F$ of the Hilbert spaces~$\hilb_\alpha(F)$ will be often omitted.


\begin{prop} \label{P: if GAnG converges ...} 
Let $\set{A_\iota \in \clBH, \iota \in I}$ be a net of bounded operators on a Hilbert space
with frame~$F.$
The net of operators
$$
  \abs{F} A_\iota \abs{F} \colon \hilb \to \hilb
$$
converges in~$\clBH$ (respectively, in~$\clL_p(\hilb)$) if and only if the net of operators
$$
  A_\iota \colon \hilb_1 \to \hilb_{-1}
$$
converges in~$\clB(\hilb_1(F), \hilb_{-1}(F))$
(respectively, in~$\clL_p(\hilb_1, \hilb_{-1})$).
\end{prop}
Elements of~$\hilb_1$ are regular (smooth), while elements of~$\hilb_{-1}$ are non-regular.
In this sense, the frame operator~$F$ increases smoothness of vectors.
\begin{rems*} If $\alpha > 0,$ then the triple $(\hilb_\alpha,\hilb,\hilb_{-\alpha})$
forms a rigged Hilbert space. So, a frame in a Hilbert space generates a natural rigging.
At the same time, a frame evidently contains essentially more information, than a rigging.
\end{rems*}

\subsubsection{Diamond conjugate}
\label{SSS: diamond conjugate}
Let $\alpha \in \mbR.$
On the product~$\hilb_\alpha \times \hilb_{-\alpha}$ there exists a unique bounded
form $\scal{\cdot}{\cdot}_{\alpha,-\alpha}$ such that for any $f,g \in \hilb_{\abs{\alpha}}$
$$
  \scal{f}{g}_{\alpha,-\alpha} = \scal{f}{g}.
$$
Let~$\clK$ be a Hilbert space.
For any bounded operator $A \colon \hilb_\alpha \to \clK,$
there exists a unique bounded operator $A^\diamondsuit \colon \clK \to \hilb_{-\alpha}$
such that for any $f \in \clK$ and $g \in \hilb_{\alpha}$ the equality
$$
  \scal{A^\diamondsuit f}{g}_{-\alpha,\alpha} = \scal{f}{A g}_{\clK}
$$
holds.
In particular, if $A \colon \hilb_1 \to \clK$ and $f,g \in \hilb_1,$ then
\begin{equation} \label{F: (f,A*Ag)=(Af,Ag)}
  \scal{f}{A^\diamondsuit A g}_{1,-1} = \scal{Af}{A g}_{\clK}.
\end{equation}

There is a connection between the diamond conjugate and usual conjugate
$$
  A^\diamondsuit = \abs{F}^{-2\alpha}A^*  
$$
where $A^* \colon \clK \to \hilb_\alpha$ and $\abs{F}^{-2\alpha} \colon \hilb_\alpha \to \hilb_{-\alpha}.$
It follows from Lemma \ref{L: F maps H(A) to H(a+1)},
that if $A$ belongs to~$\clL_p(\hilb_\alpha,\clK),$ 
then $A^\diamondsuit$ belongs to~$\clL_p(\clK,\hilb_{-\alpha}).$

\subsection{The trace-class matrix $\phi(\lambda+i\yy)$}
\label{SS: phi(lambda+i eps)}

Let~$H_0$ be a self-adjoint operator on a framed Hilbert space $(\hilb,F).$
Let $\lambda$ be a fixed point of $\LambHF{H_0}.$
For any $\yy \geq 0,$ we introduce the matrix
\begin{equation} \label{F: phi(l+i eps)=(si sj phi i Im R ...)}
  \phi(\lambda+i\yy) = \frac 1\pi \brs{\kappa_i\kappa_j \scal{\phi_i}{\Im R_{\lambda+i\yy}(H_0)\phi_j}}
\end{equation}
and consider it as an operator on $\ell_2.$

We note several elementary properties of $\phi(\lambda+i\yy).$
\begin{itemize}
  \item[(i)] For all $\yy \geq 0,$ \ $\phi(\lambda+i\yy)$ is a non-negative trace-class operator on $\ell_2$
  and its trace is equal to
    the trace of $\frac 1 \pi F \Im R_{\lambda+i\yy}(H_0)F^*.$

  This follows from Theorem \ref{T: Ya thm 6.1.5} and the fact that $\phi(\lambda+i\yy)$
  is unitarily equivalent to $\frac 1 \pi F \Im R_{\lambda+i\yy}(H_0)F^*.$
  \item[(ii)] For all $\yy > 0,$ the kernel of $\phi(\lambda+i\yy)$ is trivial. \\
    This follows from the fact that the kernel of
   ~$F \Im R_{\lambda+i\yy}(H_0)F^*$ is trivial. Indeed, otherwise for some non-zero $f \in \clK$
   ~$F \Im R_{\lambda+i\yy}(H_0)F^* f = 0$ $\then$ $\ker R_{\lambda+i\yy}(H_0) \ni F^*f \neq 0,$ which is impossible.
  \item[(iii)] The matrix $\phi(\lambda+i\yy)$ is a real-analytic function of the parameter $\yy>0$ with values in~$\clL_1(\ell_2),$
    and it is continuous in~$\clL_1(\ell_2)$ up to $\yy = 0,$ as
    it follows from Theorem \ref{T: Ya thm 6.1.5}.
  \item[(iv)] The estimate \ $s_n(\phi(\lambda+i\yy)) \leq \yy^{-1}
  \kappa_n^2$ \
   holds. This follows from the equality $s_n(A^*A) = s_n(AA^*)$ and the estimate (\ref{F: Ya (1.6.6)}). 
\end{itemize}

\subsection{The Hilbert-Schmidt matrix $\eta(\lambda+i\yy)$}
\label{SS: matrix eta}

Let $\lambda \in \LambHF{H_0}.$ For any $\yy \geq 0,$ we also introduce the matrix
\begin{equation} \label{F: eta = sqrt phi}
  \eta(\lambda+i\yy) = \sqrt {\phi(\lambda+i\yy)}.
\end{equation}
We list elementary properties of $\eta(\lambda+i\yy).$
\begin{enumerate}
  \item[(i)] For all $y\geq 0,$ $\eta(\lambda+i\yy)$ is a non-negative Hilbert-Schmidt operator on $\ell_2.$
  \item[(ii)] If $\yy > 0,$ then the kernel of $\eta(\lambda+i\yy)$ is trivial. \\
  \item[(iii)] The matrix $\eta(\lambda+i\yy)$ 
      is a real-analytic function of the parameter $\yy>0$ with values in $\LpH{2}.$
  \item[(iv)] The matrix $\eta(\lambda+i\yy)$ is continuous in $\LpH{2}$ up to $\yy = 0.$
  \item[(v)] The estimate
   $s_n(\eta(\lambda+i\yy)) \leq \yy^{-1/2} \kappa_n 
   $
   holds.
\end{enumerate}

\subsection{Eigenvalues $\alpha_j(\lambda+i\yy)$ of $\eta(\lambda+i\yy)$}
\label{SS: alpha j}
Let $\lambda \in \LambHF{H_0}.$

We denote by $\alpha_j(\lambda+i\yy)$ the~$j$-th eigenvalue of $\eta(\lambda+i\yy)$ (counting multiplicities).

We list elementary properties of $\alpha_j(\lambda+i\yy).$
\begin{enumerate}
  \item[(i)] For $y>0,$ all eigenvalues $\alpha_j(\lambda+i\yy)$ are strictly positive. 
  \item[(ii)] For $y\geq 0,$ the sequence $(\alpha_j(\lambda+i\yy))$ belongs to $\ell_2.$ 
  \item[(iii)] The functions $(0,\infty) \ni \yy \mapsto \alpha_j(\lambda+i\yy)$ can be chosen
           to be real-analytic (after proper enumeration).
      This follows from Theorem \ref{T: lemma X} and \ref{SS: matrix eta}(iii). 
  \item[(iv)] All $\alpha_j(\lambda+i\yy)$ converge as $\yy \to 0.$
      This follows from Theorem \ref{T: lemma X} and \ref{SS: matrix eta}(iv). 
\end{enumerate}

\subsection{Zero and non-zero type indices}
\label{SS: index of non-zero type}
Let $\lambda \in \LambHF{H_0}.$
While the eigenvalues $\alpha_j(\lambda+i\yy)$ of the matrix $\eta(\lambda+iy)$ are strictly
positive for $y>0,$ the limit values $\alpha_j(\lambda)$ of some of them can
be equal to zero. We say that the eigenvalue function
$\alpha_j(\lambda+i\yy)$ is of \emph{non-zero type}, if its limit
is not equal to zero. Otherwise we say that it is of \emph{zero
type}. We denote the set of non-zero type indices
by~$\clZ_\lambda.$


Though it is not necessary, we agree to enumerate functions $\alpha_j(\lambda+i\yy)$
in such a way, that the sequence $\set{\alpha_j(\lambda+i0)}$ is decreasing.

\subsection{Vectors $e_j(\lambda+i\yy)$}
\label{SS: e j}

For any $\lambda \in \LambHF{H_0}$ we consider the sequence
of normalized eigenvectors
$$
  e_j(\lambda+i\yy) \in \ell_2, \ j=1,2,\ldots
$$
of the non-negative Hilbert-Schmidt matrix $\eta(\lambda+i\yy).$
These vectors are also eigenvectors of $\phi(\lambda+i\yy).$
We enumerate the functions $e_j(\lambda+i\yy)$ in such a way that
\begin{equation} \label{F: eta e j = alpha j eta j}
  \eta(\lambda+i\yy) e_j(\lambda+i\yy) = \alpha_j(\lambda+i\yy) e_j(\lambda+i\yy), \qquad y > 0,
\end{equation}
where enumeration of $\alpha_j(\lambda+i\yy)$ is given in subsection \ref{SS: index of non-zero type}.

We list elementary properties of $e_j(\lambda+i\yy)$'s.
\begin{enumerate}
  \item[(i)] If $y>0,$ then the sequence $e_j(\lambda+i\yy) \in \ell_2, \ j=1,2,\ldots$ is an orthonormal basis of $\ell_2.$
  \item[(ii)] The functions $(0,\infty) \ni \yy \mapsto e_j(\lambda+i\yy) \in \ell_2$ can be chosen to be real-analytic.
     This follows from Theorem \ref{T: lemma X} and the item \ref{SS: matrix eta}(iii).
  \item[(iii)] For indices~$j$ of non-zero type, the functions
      $[0,\infty) \ni \yy \mapsto e_j(\lambda+i\yy) \in \ell_2$ are continuous up to $\yy = 0.$
     This follows from Theorem \ref{T: lemma X} and \ref{SS: matrix eta}(iv).
  \item[(iv)]  We say that $e_j(\lambda+iy)$ is of (non-)zero type, if the corresponding eigenvalue function $\alpha_j(\lambda+iy)$
  is of (non-)zero type. Non-zero type vectors $e_j(\lambda+iy)$ have limit values $e_j(\lambda+i0),$
  which form an orthonormal system in $\ell_2,$ as it follows from Theorem \ref{T: lemma X}.

  Note that zero-type vectors $e_j(\lambda+iy)$ may not converge as $y \to 0.$

  \item[(v)] For non-zero type indices~$j$ the vectors $e_j(\lambda+i0)$ are measurable.
\end{enumerate}

\subsection{Vectors $\eta_j(\lambda+i\yy)$}
\label{SS: eta j}
Let $\lambda \in \LambHF{H_0}.$
We introduce the vector $\eta_j(\lambda+i\yy)$ as the~$j$-th column of the Hilbert-Schmidt matrix
$\eta(\lambda+i\yy)$ (see (\ref{F: eta = sqrt phi})).
This definition implies that
\begin{equation} \label{F: (eta j,eta k)=phi(jk)}
  \scal{\eta_j(\lambda+i\yy)}{\eta_k(\lambda+i\yy)} = \phi_{jk}(\lambda+i\yy).
\end{equation}

We list elementary properties of $\eta_j(\lambda+i\yy)$'s.
\begin{enumerate}
  \item[(i)] For all $\yy \geq 0,$ all vectors $\eta_j(\lambda+i\yy)$ belong to $\ell_2.$ \\
  \item[(ii)] For all $\yy \geq 0,$ the norms of vectors $\eta_j(\lambda+i\yy)$ constitute a sequence
  $$
    (\norm{\eta_1(\lambda+i\yy)}, \norm{\eta_2(\lambda+i\yy)}, \norm{\eta_3(\lambda+i\yy)}, \ldots),
  $$
  which belongs to $\ell_2.$
  This follows from the fact that $\eta(\lambda+iy)$ is a Hilbert-Schmidt operator for all $y\geq 0$
  \item[(iii)] \label{I: eta j complete} If $\yy > 0,$ then the set of vectors $\set{\eta_j(\lambda+i\yy)}$ is complete in $\ell_2.$

  Proof. We have,
      \begin{equation} \label{F: e j (l+i eps)=alpha j sum ... phi k}
        e_j(\lambda+i\yy) = \alpha_j^{-1}(\lambda+i\yy) \sum\limits_{k=1}^\infty e_{kj}(\lambda+i\yy) \eta_k(\lambda+i\yy), \qquad y \geq
        0.
      \end{equation}
  Indeed, it follows from (\ref{F: eta e j = alpha j eta j}) that
  \begin{equation*} 
   \begin{split}
    e_j(\lambda+i\yy) & = \alpha_j^{-1}(\lambda+i\yy) \eta(\lambda+i\yy)e_j(\lambda+i\yy)
       \\ & = \alpha_j^{-1}
         \left(\begin{array}{ccc}
           \eta_{11} & \eta_{12} & \ldots \\
           \eta_{21} & \eta_{22} & \ldots \\
           \ldots & \ldots & \ldots \\
         \end{array} \right)
         \left(\begin{array}{c}
           e_{1j}  \\
           e_{2j}  \\
           \ldots
         \end{array} \right)
        = \alpha_j^{-1} \left(\begin{array}{c}
           \eta_{11}e_{1j}+\eta_{12}e_{2j}+\ldots \\
           \eta_{21}e_{1j}+\eta_{22}e_{2j}+\ldots \\
           \ldots
         \end{array} \right)
       \\ & = \alpha_j^{-1}(\lambda+i\yy) \sum\limits_{k=1}^\infty e_{kj}(\lambda+i\yy) \eta_k(\lambda+i\yy), \qquad y \geq 0,
   \end{split}
  \end{equation*}
  where, in case of $y=0,$ the summation is over indices~$j$ of non-zero type.
  Hence, the set of vectors $\set{\eta_1(\lambda+i\yy), \eta_2(\lambda+i\yy), \ldots}$ is complete.
  Note also, that the linear combination above is absolutely convergent,
  according to (ii).
  \item[(iv)] Let $y>0.$
     If for some $\beta = (\beta_j) \in \ell_2$ the equality
     $$ \sum\limits_{j=1}^\infty \beta_j \eta_j(\lambda+i\yy) = 0$$
     holds, then $(\beta_j) = 0.$

     Assume the contrary. We have
      \begin{equation*}
       \begin{split}
        \eta(\lambda+i\yy) \beta & = \left[\begin{array}{c} \beta_1 \eta_{11}(\lambda+i\yy) + \beta_2 \eta_{12}(\lambda+i\yy) + \ldots
            \\ \ldots \\ \beta_1 \eta_{i1}(\lambda+i\yy) + \beta_2 \eta_{i2}(\lambda+i\yy) + \ldots \\ \ldots \end{array}\right]
         \\ & = \beta_1 \eta_1(\lambda+i\yy) + \beta_2 \eta_2(\lambda+i\yy) + \ldots
         \\ & = 0,
       \end{split}
      \end{equation*}
      where the second equality makes sense, since the series $\sum\limits_{j=1}^\infty \beta_j \eta_j(\lambda+i\yy)$
      is absolutely convergent by \ref{SS: eta j}(ii). It follows that $\beta$ is an eigenvector of $\eta(\lambda+i\yy)$ corresponding to a zero eigenvalue.
      Since, by \ref{SS: matrix eta}(ii), for $y>0$ the matrix $\eta(\lambda+i\yy)$ does not have zero eigenvalues, we get a contradiction.

   \item[(v)] Vectors $\eta_j(\lambda+i\yy)$ converge to $\eta_j(\lambda)$ in $\ell_2$ as $\yy \to 0.$
      This follows from property \ref{SS: matrix eta}(iv) of $\eta(\lambda+i\yy).$
\end{enumerate}

\subsection{Unitary matrix $e(\lambda+i\yy)$}
\label{SS: matrix e(lambda+i eps)}
Let $\lambda \in \LambHF{H_0}.$
We can form a matrix
$$
  e(\lambda+i\yy) = (e_{jk}(\lambda+i\yy)),
$$
whose columns are $e_j(\lambda+i\yy),$~$j=1,2,\ldots.$
Since vectors $e_j(\lambda+i\yy),$~$j=1,2,\ldots,$ form an orthonormal basis of $\ell_2,$
this matrix is unitary and it diagonalizes the matrix $\eta(\lambda+iy):$
$$
  e(\lambda+i\yy)^*\eta(\lambda+i\yy)e(\lambda+i\yy)
    = \diag (\alpha_1(\lambda+i\yy), \alpha_2(\lambda+i\yy), \ldots),
$$
where $(\alpha_j(\lambda+i\yy)) \in \ell_2$ are eigenvalues of $\eta(\lambda+i\yy),$
see subsection \ref{SS: alpha j}.

\subsection{Vectors $\phi_j(\lambda+i\yy)$}
\label{SS: phi j}
Let $\lambda \in \LambHF{H_0}.$
Now we introduce vectors
\begin{equation} \label{F: phi j(l+iy)=k(-1)eta j(l+iy)}
  \phi_j(\lambda+i\yy) = \kappa_j^{-1}\eta_j(\lambda+i\yy) \in \ell_2.
\end{equation}
It may seem to be more consistent to denote by $\phi_j(\lambda+i\yy)$ the~$j$-th column of
the matrix $\phi(\lambda+i\yy).$ But, firstly, we don't need columns of $\phi(\lambda+i\yy),$
secondly, there is an advantage of this notation. Namely, $\phi_j(\lambda)$ can be considered
as the value of the vector $\phi_j \in \hilb$ at $\lambda \in \LambHF{H_0},$ as we shall
see later (see Section \ref{S: direct integral}).

Some properties of $\phi_j(\lambda+i\yy).$
\begin{enumerate}
  \item[(i)] All vectors $\phi_j(\lambda+i\yy)$ belong to $\ell_2.$
    This follows from $\eta_j(\lambda+i\yy) \in \ell_2,$ see \ref{SS: eta j}(i).
  \item[(ii)] If $y>0,$ then the set of vectors $\set{\phi_j(\lambda+i\yy)}$ is complete in $\ell_2.$
    This follows from a similar property of $\set{\eta_j(\lambda+i\yy)},$ see \ref{SS: eta j}(iii).
  \item[(iii)] \label{phi j's are lin independent} 
    Let $y>0.$ If $(\kappa^{-1}_j\beta_j) \in \ell_2$ and
     $$ \sum\limits_j \beta_j \phi_j(\lambda+i\yy) = 0,$$
     then $(\beta_j) = 0.$
  \item[(iv)] The following equality holds
  \begin{equation} \label{F: scal phi j(l+ieps), phi k(l+ieps) = ...}
    \scal{\phi_j(\lambda+i\yy)}{\phi_k(\lambda+i\yy)} = \frac 1 \pi \scal{\phi_j}{\Im R_{\lambda+i\yy}(H_0) \phi_k}.
  \end{equation}
 This immediately follows from the definition of $\phi_j(\lambda+i\yy)$'s.
  Indeed, using (\ref{F: phi j(l+iy)=k(-1)eta j(l+iy)}), (\ref{F: (eta j,eta k)=phi(jk)}) and (\ref{F: phi(l+i eps)=(si sj phi i Im R ...)}),
  \begin{equation*}
    \begin{split}
      \scal{\phi_j(\lambda+i\yy)}{\phi_k(\lambda+i\yy)} & = \kappa^{-1}_j \kappa^{-1}_k \scal{\eta_j(\lambda+i\yy)}{\eta_k(\lambda+i\yy)}
        \\ & = \kappa^{-1}_j\kappa^{-1}_k \phi_{jk}(\lambda+i\yy)
        \\ & = \frac 1 \pi \scal{\phi_j}{\Im R_{\lambda+i\yy}(H_0) \phi_k}.
    \end{split}
  \end{equation*}

  \item[(v)] It follows from (\ref{F: e j (l+i eps)=alpha j sum ... phi k}) and (\ref{F: phi j(l+iy)=k(-1)eta j(l+iy)}),
      that each $e_j(\lambda+i\yy)$ can be written as a linear combination of $\phi_j(\lambda+i\yy)$'s
      with coefficients of the form $\kappa_j\beta_j,$ where $(\beta_j) \in \ell_2:$
      \begin{equation} \label{F: e j (l+i eps)=alpha j sum ... eta k}
        e_j(\lambda+i\yy) = \alpha_j^{-1}(\lambda+i\yy) \sum\limits_{k=1}^\infty \kappa_k e_{kj}(\lambda+i\yy) \phi_k(\lambda+i\yy).
      \end{equation}
      Moreover, this representation is unique, according to (iii).
   \item[(vi)] For all~$j = 1,2,\ldots$ \ \ $\norm{\phi_j(\lambda+i\yy)}_{\ell_2} \leq (\yy\pi)^{-1/2}.$
   \item[(vii)] \ $\phi_j(\lambda+i\yy)$ converges to $\phi_j(\lambda)$ in $\ell_2,$ as $\yy \to
   0$ (recall that $\lambda \in \Lambda(H_0,F)$).
      This follows from \ref{SS: eta j}(v).
   \item[(viii)]  \label{F: (phi(l) j,phi(l) k) = (phi j,Im R phi k)} The equality
    $$
     \scal{\phi_j(\lambda)}{\phi_k(\lambda)}_{\ells} = \frac 1 \pi \scal{\phi_j}{\Im R_{\lambda+i0}(H_0)\phi_k}_{\hilb}
    $$
    holds.

    Proof. Since for $\lambda \in \Lambda(H_0;F)$ the limit on the right hand side exists by \ref{SS: phi(lambda+i eps)}(iii),
    this follows from (vii) and (iv).
\end{enumerate}

\subsection{The operator $\euE_{\lambda+i\yy}$}
\label{SS: euE's}

Let $\lambda \in \LambHF{H_0}.$
Let
 $$
   \euE_{\lambda+i\yy} \colon \hilb_1 \to \ell_2
 $$
be a linear operator defined on the frame vectors by the formula
 \begin{equation} \label{F: euE(l+iy)phi=phi(l+iy)}
   \euE_{\lambda+i\yy} \phi_j = \phi_j(\lambda+i\yy).
 \end{equation}

We list some properties of $\euE_{\lambda+i\yy},$ which more or
less immediately follow from the definition.
\begin{enumerate}
  \item[(i)] For $y>0,$ the equality
  $$
    \scal{\euE_{\lambda+i\yy} \phi_j}{\euE_{\lambda+i\yy} \phi_k}_{\ell_2}
    = \frac 1 \pi \scal{\phi_j}{\Im R_{\lambda+i\yy}(H_0)\phi_k}_\hilb
  $$
  holds. 
  It follows that
  \begin{equation} \label{F: euE(*)euE=1/pi Im R(z)}
       \euE^*_{\lambda+i\yy}\euE_{\lambda+i\yy} = \frac 1\pi \Im R_{\lambda+i\yy}(H_0).
  \end{equation}

  \item[(ii)] Let $y>0.$ The operator $\euE_{\lambda+i\yy}$ is a Hilbert-Schmidt operator as an operator
  from~$\hilb_1$ to $\ell_2.$ Moreover,
  $$
    \norm{\euE_{\lambda+i\yy}}_{\clL_2(\hilb_1,\ell_2)}^2 = \frac 1 \pi \Tr_\clK\brs{F \Im R_{\lambda+i\yy}(H_0)F^*}.
  $$
  \\ Indeed, evaluating the trace of $\euE_{\lambda+i\yy}^*\euE_{\lambda+i\yy}$
  in the orthonormal basis $\set{\kappa_j\phi_j}$ of~$\hilb_1,$ we get, using (i) and (\ref{F: F phi j=...}),
  \begin{equation*}
   \begin{split}
     \sum\limits_{j=1}^\infty & \scal{\euE_{\lambda+i\yy}^*\euE_{\lambda+i\yy}\kappa_j\phi_j}{\kappa_j\phi_j}_{\hilb_1}
     \\ & \qquad = \sum\limits_{j=1}^\infty \kappa_j^2 \scal{\euE_{\lambda+i\yy}\phi_j}{\euE_{\lambda+i\yy}\phi_j}_{\ell_2}
     \\ & \qquad = \frac 1 \pi \sum\limits_{j=1}^\infty \kappa_j^2 \scal{\phi_j}{\Im R_{\lambda+i\yy}(H_0) \phi_j}_{\hilb}      \mathcomment{i}
     \\ & \qquad = \frac 1 \pi \sum\limits_{j=1}^\infty \scal{F^*\psi_j}{\Im R_{\lambda+i\yy}(H_0) F^*\psi_j}_{\hilb}           \mathcomment{\ref{F: F phi j=...}}
     \\ & \qquad = \frac 1 \pi \Tr_\clK(F\Im R_{\lambda+i\yy}(H_0)F^*).                \mathcomment{\ref{F: Tr(T)=sum (T fj,fj)}}
   \end{split}
  \end{equation*}
  \item[(iii)] The norm of $\euE_{\lambda+i\yy} \colon \hilb_1 \to \ell_2$ is $\leq \norm{\eta\brs{\lambda+iy}}_2.$
  Indeed, if $\beta = (\beta_j) \in \ell_2,$ then \ $f:= \sum\limits_{j=1}^\infty \kappa_j\beta_j\phi_j \in \hilb_1$
  with $\norm{f}_{\hilb_1} = \norm{\beta},$ and, using (\ref{F: euE(l+iy)phi=phi(l+iy)}), (\ref{F: phi j(l+iy)=k(-1)eta j(l+iy)})
  and Schwarz inequality, one gets
  \begin{equation*}
   \begin{split}
    \norm{\euE_{\lambda+i\yy} f} & = \norm{\sum\limits_{j=1}^\infty \kappa_j\beta_j\phi_j(\lambda+iy)}
         = \norm{\sum\limits_{j=1}^\infty \beta_j\eta_j(\lambda+iy)}
         \\ & \leq \norm{\beta} \cdot \brs{\sum_{j=1}^\infty \norm{\eta_j(\lambda+iy)}^2}^{1/2}
            = \norm{f}_{\hilb_1} \cdot \norm{\eta(\lambda+iy)}_2.
   \end{split}
  \end{equation*}

  \item[(iv)] For all $\yy > 0,$ the operator $\euE_{\lambda+i\yy} \colon \hilb_1 \to \ell_2$ has trivial kernel.
  \\ Indeed, otherwise for some non-zero vector $f \in \hilb_1,$
  $$
    0 = \scal{\euE_{\lambda+i\yy}f}{\euE_{\lambda+i\yy}f} = \frac 1 \pi \scal{f}{\Im R_{\lambda+i\yy}(H_0)f}.
  $$
  Combining this equality with the formula
    $$
      \Im R_{\lambda+i\yy}(H_0) = \yy R_{\lambda-i\yy}(H_0)R_{\lambda+i\yy}(H_0),
    $$
    one infers that $R_{\lambda+i\yy}(H_0)$ has non-trivial kernel.
    But this is impossible.
  \item[(v)] The operator $\euE_{\lambda+i\yy} \colon \hilb_1 \to \ell_2$ as a function of $\yy > 0$
  is real-analytic in~$\clL_2(\hilb_1, \ell_2).$
  \item[(vi)] The operator $\euE_{\lambda+i\yy} \colon \hilb_1 \to \ell_2$ converges in the Hilbert-Schmidt norm
    to $\euE_{\lambda},$ as $y \to 0.$
    \\ Proof. We have, in the orthonormal basis $\set{\kappa_j\phi_j}$ of~$\hilb_1,$
       \begin{equation*}
         \begin{split}
           \norm{\euE_{\lambda+i\yy}- \euE_{\lambda}}_{\clL_2(\hilb_1)}^2 & = \sum_{j=1}^\infty \norm{\brs{\euE_{\lambda+i\yy}- \euE_{\lambda}}(\kappa_j\phi_j)}^2 \mathcomment{\ref{F: HS-norm of T = sum (Tfj)2}}
            \\ & = \sum_{j=1}^\infty \norm{\kappa_j\phi_j\brs{\lambda+i\yy}- \kappa_j\phi_j\brs{\lambda}}^2 \mathcomment{\ref{F: euE(l+iy)phi=phi(l+iy)}}
            \\ & = \sum_{j=1}^\infty \norm{\eta_j\brs{\lambda+i\yy}- \eta_j\brs{\lambda}}^2  \mathcomment{\ref{F: phi j(l+iy)=k(-1)eta j(l+iy)}}
            \\ & = \norm{\eta\brs{\lambda+i\yy} - \eta\brs{\lambda}}_2^2 \to 0,        \mathcomment{\ref{F: HS-norm of T = sum (Tfj)2}}
         \end{split}
       \end{equation*}
       where the convergence holds by \ref{SS: matrix eta}(iv).
   \item[(vii)] It follows that the equality in (i) holds for $y=0$ as well
  $$
    \scal{\euE_{\lambda} \phi_j}{\euE_{\lambda} \phi_k}_{\ell_2}
    = \frac 1 \pi \scal{\phi_j}{\Im R_{\lambda+i0}(H_0)\phi_k}_\hilb.
  $$
  Moreover, the operator $\euE_{\lambda} \colon \hilb_1 \to \ell_2$ is also Hilbert-Schmidt and
  $$
    \norm{\euE_{\lambda}}_{\clL_2(\hilb_1,\ell_2)}^2 = \frac 1 \pi \Tr_\clK\brs{F \Im R_{\lambda+i0}(H_0)F^*}.
  $$
\end{enumerate}

\subsection{Vectors $b_j(\lambda+i\yy) \in \hilb_1$}
\label{SS: b j}

Let $y>0$ and $\lambda \in \Lambda(H_0,F).$
For each~$j=1,2,3\ldots$ we introduce the vector $b_j(\lambda+i\yy) \in \hilb_1$
as a unique vector from the Hilbert space~$\hilb_1$ with property
\begin{equation} \label{F: def of bj}
  \euE_{\lambda+i\yy} b_j(\lambda+i\yy) = e_j(\lambda+i\yy).
\end{equation}

Property \ref{SS: phi j}(v) of $\phi_j(\lambda+i\yy) = \euE_{\lambda+i\yy}\phi_j$ implies that the vector
\begin{equation} \label{F: b j(l+ieps) = alpha sum ...}
   b_j(\lambda+i\yy) = \alpha_j^{-1}(\lambda+i\yy) \sum\limits_{k=1}^\infty \kappa_k e_{kj}(\lambda+i\yy) \phi_k
\end{equation}
satisfies the above equation,
where $e_{kj}(\lambda+i\yy)$ is the $k$'s coordinate of $e_{j}(\lambda+i\yy).$
Property \ref{SS: phi j}(iii) of $\phi_j(\lambda+i\yy)$ implies that such representation is unique.

The representation (\ref{F: b j(l+ieps) = alpha sum ...}) shows that
the functions $(0,\infty) \ni y \mapsto b_j(\lambda+i\yy) \in \hilb_1$ are continuous,
since, by Schwarz inequality and $\norm{e_j(\lambda+iy)} = 1,$
the series in the right hand side of (\ref{F: b j(l+ieps) = alpha sum ...}) absolutely converges
locally uniformly with respect to $y>0.$

We list some properties of the vectors $b_j(\lambda+i\yy).$
\begin{enumerate}
  \item[(i)] The relations
   \begin{gather*}
     \norm{b_j(\lambda+i\yy)}_\hilb \leq \alpha_j^{-1}(\lambda+i\yy) \norm{F}_2,
     \\ \norm{b_j(\lambda+i\yy)}_{\hilb_1} = \alpha_j^{-1}(\lambda+i\yy)
   \end{gather*}
   hold.
  \item[(ii)] Vectors $b_j(\lambda+i\yy),$ $j=1,2,\ldots,$ are linearly independent. 
  \item[(iii)] The system of vectors $\set{b_j(\lambda+i\yy)}$ is complete in~$\hilb_1$ (and, consequently, in~$\hilb$ as well).

  Proof. This follows from the equality
  \begin{equation} \label{F: phi(l)=alpha(l)(-1)sum(j)...e(j)}
    \phi_l = \kappa_l^{-1} \sum_{j=1}^\infty \bar e_{lj}(\lambda+iy)\alpha_j(\lambda+iy) b_j(\lambda+iy), \quad l=1,2,\ldots
  \end{equation}
  This equality itself follows from (\ref{F: b j(l+ieps) = alpha sum ...}) and from the unitarity of the matrix $(e_{jk}(\lambda+iy)).$
  \item[(iv)] The equality \
   \begin{equation} \label{F: (euE b_j,euE b_k)=delta(jk)}
     \scal{\euE_{\lambda+i\yy} b_j(\lambda+i\yy)} {\euE_{\lambda+i\yy} b_k(\lambda+i\yy)} = \delta_{jk}
   \end{equation}
   \ holds. 
  \item[(v)] The equality \
   $$
     \frac \yy \pi \scal{R_{\lambda\pm i\yy}(H_0) b_j(\lambda+i\yy)} {R_{\lambda\pm i\yy}(H_0) b_k(\lambda+i\yy)} = \delta_{jk}
   $$
   \ holds.

   Proof. We have
   \begin{equation*}
     \begin{split}
       \frac \yy \pi & \scal{R_{\lambda+i\yy}(H_0) b_j(\lambda+i\yy)} {R_{\lambda + i\yy}(H_0) b_k(\lambda+i\yy)}
          \\ & \quad = \frac \yy \pi \scal{b_j(\lambda+i\yy)} {R_{\lambda - i\yy}(H_0)R_{\lambda + i\yy}(H_0) b_k(\lambda+i\yy)}
          \\ & \quad = \scal{b_j(\lambda+i\yy)} {\frac 1\pi \Im R_{\lambda + i\yy}(H_0) b_k(\lambda+i\yy)}
          \\ & \quad = \scal{b_j(\lambda+i\yy)} {\euE^*_{\lambda+i\yy}\euE_{\lambda+i\yy} b_k(\lambda+i\yy)}     \mathcomment{\ref{F: euE(*)euE=1/pi Im R(z)}}
          \\ & \quad = \scal{\euE_{\lambda+i\yy}b_j(\lambda+i\yy)} {\euE_{\lambda+i\yy} b_k(\lambda+i\yy)}
          \\ & \quad = \delta_{jk}.       \mathcomment{\ref{F: (euE b_j,euE b_k)=delta(jk)}}
     \end{split}
   \end{equation*}

   \item[(vi)] The set of vectors $\sqrt{\frac \yy\pi}\set{R_{\lambda+i\yy}(H_0) b_j(\lambda+i\yy)}$ is an orthonormal basis in~$\hilb.$\\
   Proof. By (v), it is enough to show that this set is complete.
   If for a non-zero vector $g$
   $$
     \scal{R_{\lambda+i\yy}(H_0) b_j(\lambda+i\yy)}{g} = 0
   $$ for all~$j,$ then
   $$
     \scal{b_j(\lambda+i\yy)}{R_{\lambda-i\yy}(H_0) g} = 0
   $$
   for all~$j.$ By completeness (iii) of the set
   $\set{b_j(\lambda+i\yy)},$ one infers from this that $R_{\lambda-i\yy}(H_0) g = 0.$ This is impossible,
    since $R_{\lambda-i\yy}(H_0)$ has trivial kernel.
    \item[(vii)] The set of vectors $\sqrt{\frac \yy\pi}\set{R_{\lambda+i\yy}(H_0) b_j(\lambda+i\yy)}$ form an orthonormal basis of~$\hilb$
    for each choice of the sign $\pm.$ This follows from previous items.
   \item[(vii)] If~$j$ is of non-zero type, then
      $b_j(\lambda+i\yy) \in \hilb_1$ converges in~$\hilb_1$ to $b_j(\lambda+i0) \in \hilb_1.$
      This follows from the convergence of $e_j(\lambda+i\yy)$ in $\ell_2$ (see item \ref{SS: e j}(iii))
     and (\ref{F: b j(l+ieps) = alpha sum ...}).
\end{enumerate}

\section{The evaluation operator $\euE_\lambda$}
\label{S: direct integral}

As it was mentioned before, a frame in a Hilbert space~$\hilb,$ on which a self-adjoint operator
$H_0$ acts, allows to define explicitly the fiber Hilbert space $\mathfrak h_\lambda$
of the direct integral of Hilbert spaces diagonalizing~$H_0,$ with the purpose to define
$f(\lambda)$ as an element of $\mathfrak h_\lambda$ for a dense set~$\hilb_1$ of vectors and any $\lambda$
from a fixed set of full Lebesgue measure $\LambHF{H_0}.$ In this section we give this construction.

Let~$H_0$ be a self-adjoint operator on a fixed framed Hilbert space $(\hilb,F),$
where the frame~$F$ is given by (\ref{F: Frame=...}).
For $\lambda \in \LambHF{H_0}$ (see Definition \ref{D: Lambda(H0,F)}),
we have a Hilbert-Schmidt operator (see item \ref{SS: euE's}(vi))
$$
  \euE_\lambda \colon \hilb_1 \to \ell_2,
$$
defined by the formula
\begin{equation} \label{F: def of E(l) f}
  \euE_\lambda f = \sum\limits_{j=1}^\infty \beta_j \eta_j(\lambda),
\end{equation}
where $f = \sum\limits_{j=1}^\infty \beta_j \kappa_j \phi_j \in \hilb_1,$ $(\beta_j) \in \ell_2$
(see item (v) of subsection \ref{SS: eta j} for definition of $\eta_j(\lambda)$).
(Remark: the formula (\ref{F: def of E(l) f}) is one of the most important definitions in this paper).
Since, by \ref{SS: eta j}(ii), $(\norm{\eta_j(\lambda)}) \in \ell_2,$ the series above converges absolutely:
by the Schwarz inequality
$$
  \sum\limits_{j=1}^\infty \norm{\beta_j \eta_j(\lambda)}_{\ell_2} \leq \norm{\beta}_{\ell_2} \brs{\sum\limits_{j=1}^\infty \norm{\eta_j(\lambda)}_{\ell_2}^2}^{1/2}
    = \norm{\beta}_{\ell_2} \norm{\eta(\lambda)}_2.
$$
The set $\euE_\lambda \hilb_1$ is a pre-Hilbert space.
We denote the closure of this set in $\ell_2$ by~$\hlambda:$
\begin{equation} \label{F: def of hlambda}
  \hlambda := \overline{\euE_\lambda \hilb_1}.
\end{equation}
It is clear that the dimension function $\LambHF{H_0} \ni \lambda \mapsto \dim\hlambda$ is Borel measurable,
since, by definition,
$$
  \dim\hlambda = \rank(\eta(\lambda)) \in \set{0,1,2,\ldots, \infty},
$$
and it's clear that the matrix $\eta(\lambda)$ is Borel measurable.
Since the matrix $\phi(\lambda)$ is self-adjoint, it is also clear that
$$
  \dim\hlambda = \rank(\phi(\lambda)).
$$
One can give one more formula for $\dim(\lambda)$
$$
  \mathrm{Card} \set{j\colon j \ \text{is of non-zero type}} = \dim\hlambda.
$$
\begin{lemma} The system of vector-functions $\set{\phi_j(\lambda), j=1,2,\ldots}$
  satisfies the axioms of the measurability base (Definition \ref{D: meas. base}) for the family
  of Hilbert spaces $\set{\hlambda}_{\lambda \in \LambHF{H_0}},$ given by (\ref{F: def of hlambda}).
\end{lemma}
\begin{proof} For any fixed $\lambda \in \LambHF{H_0},$ vectors $\phi_1(\lambda),\phi_2(\lambda),\ldots$
generate~$\hlambda$ by definition.
Measurability of functions $\LambHF{H_0} \ni \lambda \mapsto \scal{\phi_j(\lambda)}{\phi_k(\lambda)}$
follows from \ref{SS: phi j}(viii). So, both axioms of the measurability base hold.
\end{proof}
The field of Hilbert spaces
$$
  \set{\hlambda \colon \lambda \in \LambHF{H_0}}
$$
with measurability base
\begin{equation} \label{F: measurability base}
  \lambda \mapsto \euE_\lambda \phi_j = \phi_j(\lambda), \ \ j = 1,2,\ldots
\end{equation}
determines a direct integral of Hilbert spaces (see subsection \ref{SS: direct integral: def})
\begin{equation} \label{F: direct integral}
  \euH := \int_{\LambHF{H_0}}^\oplus \hlambda \,d\lambda.
\end{equation}


The vector $\phi_j(\lambda)$ is to be interpreted as the value of the vector $\phi_j$
at $\lambda,$ as we shall see later.
Note that though the vectors $\phi_j(\lambda) \in \hlambda, \ j = 1,2,\ldots$ depend on the sequence $(\kappa_j)$
of weights of the frame~$F,$ their norms and scalar products
$$
  \norm{\phi_j(\lambda)}_\hlambda, \ \ \scal{\phi_j(\lambda)}{\phi_k(\lambda)}
$$
are independent of weights, as directly follows from \ref{SS: phi j}(viii).
This also means that if two frames $F_1$ and $F_2$ have different weights, but the same frame vectors,
and if $\lambda$ belongs to both full sets $\Lambda(H_0,F_1)$ and $\Lambda(H_0,F_2),$ then the Hilbert
spaces $\hlambda(H_0,F_1)$ and $\hlambda(H_0,F_2)$ are naturally isomorphic. The isomorphism is given
by the correspondence
$$
  \phi_j^{(1)}(\lambda) \longleftrightarrow \phi_j^{(2)}(\lambda), \ j=1,2,\ldots,
$$
where $\phi_j^{(k)}(\lambda),$ $k=1,2,$ is the vector constructed using the frame $F_k.$

\begin{example} \rm
Let $\lambda \in [0,2\pi).$
Let~$\hilb = L_2(\mbT) \ominus \set{\mathrm{constants}}$ and let
$$
  F = \sum_{j\in \mbZ^*} \abs{j}^{-1} \scal{\phi_j}{\cdot}\phi_j,
$$
where $\phi_j = e^{-ij\lambda}$ and $\mbZ^* = \set{\pm 1,\pm 2,\ldots}.$ Let~$H_0$ be
the multiplication by $\lambda$ on $[0,2\pi) \equiv \mbT.$ In this case
$$
  \phi(\lambda) = \brs{\abs{jk}^{-1} e^{i(j-k)\lambda}}_{j,k \in \mbZ^*}
$$
and $\LambHF{H_0} = \mbR.$
For all $\lambda \in [0,2\pi),$ this matrix has rank one, so that there is only one index of non-zero type and $\dim \hlambda = 1.$
This corresponds to the fact that~$H_0$ has simple spectrum.
Vectors $f$ from~$\hilb_1$ are absolutely continuous functions with $L_2$ derivative.
The value of $\phi_j$ at $\lambda$ should be interpreted as the~$j$th column of $\eta(\lambda) = \sqrt{\phi(\lambda)}$
over $\abs{j}.$
For the only non-zero type index $1$ we have $$\alpha_1(0)^2 = 2\sum\limits_{n=1}^\infty n^{-2}.$$
The matrix $\eta(\lambda)$ is usually difficult to calculate, but in this case it can be easily calculated. Since $\phi(\lambda)$
is one-dimensional, it follows that $$\eta(\lambda) = \alpha_1^{-1}(0)\phi(\lambda).$$
So, in this case it is possible to write down an explicit formula for the evaluation operator~$\euE_\lambda.$
If $f \in \hilb_1,$ then the Fourier series of $f$ is
$$
  f = \sum_{j \in \mbZ^*} {\abs{j}}^{-1}{\beta_j} e^{-ij\lambda},
$$
where $(\beta_j) \in \ell_2(\mbZ^*)$ and by (\ref{F: def of E(l) f})
$$
  \euE_\lambda f = \eta(\lambda)\beta = \alpha_1(0)^{-1} \sum_{k\in\mbZ^*} \abs{k}^{-1}\beta_k e^{-ik\lambda} \brs{\abs{j}^{-1}e^{ij\lambda}}_{j \in \mbZ^*}
  = f(\lambda) \psi(\lambda),
$$
where $\psi(\lambda) = \alpha_1(0)^{-1} \brs{\abs{j}^{-1}e^{ij\lambda}}_{j \in \mbZ^*}$ is a normalized vector from $\ell_2(\mbZ^*).$
The one-dimensional Hilbert space $\hlambda$ is spanned by $\psi(\lambda).$
\end{example}

\begin{lemma} \label{L: norm of phi j(l) less 1} For any~$j = 1,2,\ldots,$ the function $\euE\phi_j$ belongs to $\euH$ and $\norm{\euE\phi_j}_\euH \leq 1.$
\end{lemma}
\begin{proof} We only need to show that $\phi_j(\lambda)=\euE_\lambda\phi_j$ is square summable and that the estimate holds.
It follows from \ref{SS: phi j}(viii) 
that
\begin{equation*}
  \begin{split}
     \scal{\euE\phi_j}{\euE\phi_j}_\euH & = \int_{\LambHF{H_0}} \scal{\phi_j(\lambda)}{\phi_j(\lambda)}\,d\lambda
     \\ & = \frac 1\pi \int_{\LambHF{H_0}} \scal{\phi_j}{\Im R_{\lambda+i0}(H_0) \phi_j}\,d\lambda = :(E).
  \end{split}
\end{equation*}
Since $\frac 1\pi \scal{\phi_j}{\Im R_{\lambda+iy}(H_0) \phi_j}$ is the Poisson integral of the function $\scal{\phi_j}{E_\lambda^{H_0}\phi_j},$
it follows from Theorem \ref{T: Fatou} that $$\frac 1\pi \scal{\phi_j}{\Im R_{\lambda+i0}(H_0) \phi_j} = \frac d{d\lambda} \scal{\phi_j}{E_\lambda^{H_0}\phi_j}$$
for a.e. $\lambda.$ Consequently,
 \begin{equation*}
   \begin{split}
      (E) = \int_{\LambHF{H_0}} \frac {d}{d\lambda} \scal{\phi_j}{E^{H_0}_\lambda \phi_j}\,d\lambda \leq 1.
   \end{split}
 \end{equation*}
\end{proof}
\begin{cor} \label{C: phi j(l),phi k(l) is summable} For any pair of indices~$j$ and $k$ the function
$$
  \LambHF{H_0} \ni \lambda \mapsto \scal{\phi_j(\lambda)}{\phi_k(\lambda)}_{\hlambda}
$$
is summable and its $L_1$-norm is $\leq 1.$
\end{cor}
\begin{proof} This follows from Schwarz inequality and Lemma \ref{L: norm of phi j(l) less 1}.
\end{proof}

%
%
%

A function $\LambHF{H_0}\ni \lambda \to f(\lambda) \in \ells$
will be called \emph{$\euH$-measurable,} if $f(\lambda) \in \hlambda$ for a.e. $\lambda \in \LambHF{H_0},$
if $f(\cdot)$ is measurable with respect to the measurability base (\ref{F: measurability base})
and if $f \in \euH$ (i.e. if $f$ is square summable).

We can define a linear operator \ $\euE \colon \hilb_1 \to \euH$
with dense domain $\euD,$ by the formula
\begin{equation} \label{F: def of euE}
  (\euE \phi_j)(\lambda) = \phi_j(\lambda),
\end{equation}
where $\euD$ is defined by (\ref{F: euD - def-n}).

One can define a standard minimal core~$\clA(H_0,F)$
of the absolutely continuous spectrum of~$H_0,$
acting on a framed Hilbert space, by the formula
$$
  \clA(H_0,F) = \bigcup\limits_{i,j=1}^\infty \clA(m_{ij}),
$$
where $m_{ij}(\Delta) = \scal{E_\Delta \phi_i}{\phi_j}$ is a (signed) spectral measure, and $\clA(m)$ is a minimal support of the absolutely continuous part of $m,$
defined by (\ref{F: clA(m)=...}).
\begin{prop} The dimension of the fiber Hilbert space~$\hlambda$ is not zero
  if and only if $\lambda \in \clA(H_0,F).$
\end{prop}
\begin{proof} $(\Leftarrow).$ If $\lambda \in \clA(H_0,F),$ then for some pair $(i,j)$ of indices
$\lambda \in \clA(m_{ij}).$ This means that the limit~$\clC_{m_{ij}}(\lambda+i0)$ exists and is not zero.
This implies that $\phi(\lambda) = (\phi_{ij}(\lambda)) \neq 0,$ as well as $\eta(\lambda) \neq 0.$
So, the Hilbert space~$\hlambda$ is generated by at least one non-zero
vector $\phi_j(\lambda).$

$(\Rightarrow).$ If $\dim \hlambda \neq 0,$ then by definition (\ref{F: def of hlambda})
of~$\hlambda$ for some index~$j$
the vector $\phi_j(\lambda) = \kappa_j^{-1}\eta_j(\lambda)$ is non-zero. Hence, the matrix $\eta(\lambda)$ is non-zero.
It follows that $\phi(\lambda)$ is non-zero. If $\phi_{ij}(\lambda) \neq 0,$ then $\lambda \in \clA(m_{ij}).$
So, $\lambda \in \clA(H_0,F).$
\end{proof}
It follows from this Proposition that the direct integral (\ref{F: direct integral}) can be rewritten as
\begin{equation} 
  \euH = \int_{\clA(H_0,F)}^\oplus \hlambda \,d\lambda.
\end{equation}
Hence, instead of the full set $\LambHF{H_0}$ one can use~$\clA(H_0,F).$
However, since the set $\LambHF{H_0}$ has full Lebesgue measure, it is more convenient to work with.

\medskip

Recall that the vectors $e_j(\lambda), j =1,2,\ldots,$ corresponding to non-zero
type indices~$j,$ are the limit values of the non-zero type eigenvectors $e_j(\lambda+iy), j =1,2,\ldots$
of $\eta(\lambda+iy) = \sqrt {\phi(\lambda+iy)}.$
\begin{lemma} \label{L: e j(l) is basis} The system of $\ell_2$-vectors $\set{e_j(\lambda) \colon j \ \text{is of non-zero type}}$
is an orthonormal basis of~$\hlambda.$
\end{lemma}
\begin{proof} Firstly, by \ref{SS: e j}(iv), the system of vectors $\set{e_j(\lambda) \colon j \ \text{is of non-zero type}}$
is orthonormal. In part (A) it is shown that this system is a subset of~$\hlambda;$ in part (B) it is shown that the system is complete
in ~$\hlambda.$

(A) 
By definition (\ref{F: def of hlambda}) of~$\hlambda,$ it is generated by \ $\set{\phi_1(\lambda),\phi_2(\lambda), \ldots},$
or, which is the same, by $\set{\eta_1(\lambda),\eta_2(\lambda), \ldots}.$
For a non-zero type index~$j,$ one can take the limit $y \to 0^+$ in (\ref{F: e j (l+i eps)=alpha j sum ... phi k})
to get
$$
  e_j(\lambda) = \alpha_j(\lambda)^{-1} \sum_{k=0}^\infty e_{kj}(\lambda) \eta_k(\lambda).
$$
It follows that $\set{e_j(\lambda) \colon j \ \text{is of non-zero type}} \subset \hlambda.$

(B) For any index $i$ the following formula holds
\begin{equation} \label{F: eta j(l) = sum k e k(l)}
  \eta_i(\lambda+iy) = \sum_{k=1}^\infty \alpha_k(\lambda+iy) e_{ik}(\lambda+iy)e_{k}(\lambda+iy).
\end{equation}
Indeed, this equality is equivalent to the following one
$$
  \scal{\eta_i(\lambda+iy)}{e_j(\lambda+iy)} = \alpha_j(\lambda+iy) e_{ij}(\lambda+iy).
$$
This equality follows from (\ref{F: eta e j = alpha j eta j}).
Passing to the limit in (\ref{F: eta j(l) = sum k e k(l)}), one gets
\begin{equation*}
  \eta_i(\lambda) = \sum_{k=1, k \in \clZ_\lambda}^\infty \alpha_k(\lambda) e_{ik}(\lambda)e_{k}(\lambda).
\end{equation*}
It follows that the system $\set{e_{j}(\lambda), j \ \text{is of non-zero type}}$ is complete in~$\hlambda.$
\end{proof}

This lemma implies that $\set{e_j(\lambda)}$ is an orthonormal measurability base for the direct integral $\euH.$

Let $P_\lambda \in \clB(\ells)$ be the projection onto~$\hlambda.$

\begin{lemma} \label{L: left inverse of phi(lambda) exists}
There exists a measurable operator-valued function
$\LambHF{H_0}\ni \lambda \mapsto \psi(\lambda) \in \clC(\ells)$
such that $\psi(\lambda)$ is a self-adjoint operator and
$$
  \psi(\lambda)\phi(\lambda) = P_\lambda.
$$
\end{lemma}
\begin{proof} Since $\phi(\lambda)$ is a non-negative compact
operator, this follows from the spectral theorem. We just set $\psi(\lambda) = 0$
on $\ker \phi(\lambda)$ and $\psi(\lambda) = \phi(\lambda)^{-1}$ on $\ker \phi(\lambda)^\perp.$
\end{proof}
\begin{cor} The family of orthogonal projections
  $P_\lambda \colon \ells \to \hlambda$
  is weakly measurable.
\end{cor}
\begin{proof} This follows from Lemmas \ref{L: phi(lambda) is meas}
  and \ref{L: left inverse of phi(lambda) exists}.
\end{proof}
\begin{lemma} \label{L: f is H meas iff it is meas}
A function $f \colon \LambHF{H_0}\ni \lambda \mapsto f(\lambda) \in \hlambda$
is $\euH$-measurable if and only if it
is measurable as a function $\LambHF{H_0}\to \ells$ and is square summable.
\end{lemma}
\begin{proof}
(If)\,  Since the functions $\phi_j(\lambda)$ are measurable, if a function $f \colon \LambHF{H_0}\to \ells$
is measurable and $f(\lambda) \in \hlambda,$ then all the functions
$\scal{f(\lambda)}{\phi_j(\lambda)}_{\hlambda} = \scal{f(\lambda)}{\phi_j(\lambda)}_{\ells}$
are measurable. Hence, $f$ is $\euH$-measurable.

(Only if)\, Let $f(\lambda) \in \hlambda$ be $\euH$-measurable, i.e. be such that for any~$j$
$$
  \scal{\phi_j(\lambda)}{f(\lambda)}
$$
is measurable and $\norm{f(\lambda)}_\hlambda \in L_2(\Lambda,d\lambda).$ This implies that the vector
$$
  (\kappa_j\scal{\phi_j(\lambda)}{f(\lambda)}) = (\scal{\eta_j(\lambda)}{f(\lambda)}) = \eta(\lambda) f(\lambda)
$$
is measurable. 
So, the function $\eta^2(\lambda) f(\lambda) = \phi(\lambda) f(\lambda)$
is also measurable. Since by Lemma \ref{L: left inverse of phi(lambda) exists}
there exists a measurable function $\psi(\lambda),$
such that $\psi(\lambda)\phi(\lambda) = P_\lambda,$
the function $f(\lambda)$ is also measurable.
\end{proof}

\begin{prop} \label{P: linear combinations of eta j}
Let $\chi_\Delta(\cdot)$ be the characteristic function of $\Delta.$
The set of finite linear combinations of functions
$$
  \LambHF{H_0}\ni \lambda \mapsto \chi_\Delta(\lambda) \phi_j(\lambda) \in \ells,
$$
where $\Delta$ is an arbitrary Borel subset of $\Lambda$ and~$j = 1,2,\ldots,$
is dense in~$\euH.$
\end{prop}
\begin{proof}
This follows from Lemma \ref{L: BS Lemma 7.1.5}.
\end{proof}


\subsection{$\euE$ is an isometry}

Note that the system $\set{\phia_j}$ is complete in~$\hilb^{(a)},$ though it is, in general, not linearly independent.
\begin{prop} \label{P: euE is an isometry}
Let~$H_0$ be a self-adjoint operator on a framed Hilbert space $(\hilb,F).$
  The operator $\euE \colon \hilb_1 \to \euH,$ defined by {\rm~(\ref{F: def of euE}),} is bounded
  as an operator from~$\hilb$ to $\euH,$ so that one can define
 ~$\euE$ on the whole~$\hilb$ by continuity.
  The operator $\euE \colon \hilb \to \euH,$ thus defined, vanishes on~$\hilbs$
  and is isometric on~$\hilba.$
\end{prop}
\begin{proof}
Firstly, we show that~$\euE$ is bounded.
It follows from the item \ref{F: (phi(l) j,phi(l) k) = (phi j,Im R phi k)}(viii) that
\begin{equation*}
 \begin{split}
   \scal{\euE\phi_j}{\euE\phi_k}_\euH
     & = \int_{\Lambda} \scal{\euE_\lambda \phi_j}{\euE_\lambda {\phi_k}}_\hlambda\,d\lambda
   \\ & = \int_{\Lambda} \scal{\phi_j(\lambda)}{\phi_k(\lambda)}_\hlambda\,d\lambda
       = \frac 1\pi \int_{\Lambda} \la \phi_j, \Im R_{\lambda+i0}(H_0)\, \phi_k \ra \,d\lambda.
 \end{split}
\end{equation*}
Since by Theorem~\ref{T: Fatou}
\begin{equation} \label{F: (phi i,Im R phi j) = der (phi i,E(l)phi j)}
  \frac 1\pi \la \phi_j, \Im R_{\lambda+i0}(H_0)\phi_k \ra = \frac d{d\lambda} \la \phi_j, E_{\lambda}\phi_k \ra
     \quad \text{for a.e.} \  \lambda \in \Lambda,
\end{equation}
it follows that
$$  \scal{\euE\phi_j}{\euE\phi_k}_\euH
      = \int_{\Lambda} \frac {d\la \phi_j, E_\lambda {\phi_k} \ra}{d\lambda}\,d\lambda.
$$
This implies that 
\begin{equation} \label{F: (euE phi j,euE phi k)}
 \begin{split}
  \scal{\euE\phi_j}{\euE\phi_k}_\euH
     = \int_{\Lambda} \frac {d\la \phi_j, E^{(a)}_\lambda {\phi_k} \ra}{d\lambda}\,d\lambda
     = \la \phi_j, E_\Lambda^{(a)} {\phi_k} \ra
     = \la \phia_j, \phia_k \ra.
 \end{split}
\end{equation}
This equality implies that for any $f \in \euD$ \ (see (\ref{F: euD - def-n}) for the definition of $\euD$)
$
  \norm{\euE  f} = \norm{ f^{(a)}} \leq  \norm{ f},
$ \
and so, $\euE$ is bounded. Since also $\norm{\euE  f} = \norm{P^{(a)} f}$
for all $ f$ from the dense set~$\euD,$ it follows that for any $ f \in \hilb$ \
$
  \norm{\euE  f} = \norm{P^{(a)} f}.
$
This implies that~$\euE$ vanishes on~$\hilbs$ and it is an isometry on~$\hilba.$
The proof is complete.
\end{proof}
This Proposition implies that for any $f \in \hilb$ we have a vector-function $ f(\lambda) = \euE_\lambda(f)$
as an element of the direct integral (\ref{F: direct integral}).
The function $ f(\lambda)$ is defined
for a.e. $\lambda \in \Lambda,$ while for regular vectors $f \in \hilb_1$ \ $f(\lambda)$
is defined for all $\lambda \in \LambHF{H_0}.$
\begin{lemma} \label{L: (xi,eta)=int (xi(l),eta(l))dl}
For any $ f,g \in \hilba$ the equality
$$
  \scal{ f}{g} = \int_{\Lambda} \scal{ f(\lambda)}{g(\lambda)}\,d\lambda
$$
holds.
\end{lemma}
\begin{proof} Indeed, the right hand side of this equality is, by definition, $\scal{\euE f}{\euE g}_\euH,$
  which by (\ref{F: (euE phi j,euE phi k)}) is equal to~$\scal{f}{g}_\hilb.$
\end{proof}

\subsection{$\euE$ is a unitary}
The aim of this subsection is to show that the restriction of the operator $\euE \colon \hilb \to \euH$
to~$\hilba$ is unitary.

\begin{lemma} Let $\Delta$ be a Borel subset of $\Lambda = \LambHF{H_0}.$
If $ f \in E_{\Lambda \setminus \Delta} \hilb,$ then $ f(\lambda)$
is equal to zero on $\Delta$ for a.e. $\lambda \in \Delta.$
\end{lemma}
\begin{proof}
(A) If $g = \sum\limits_{j=1}^N \alpha_j\phi_j \in \euD$ (see (\ref{F: euD - def-n})), then $\norm{E_\Delta g}^2 = \int_\Delta \scal{g(\lambda)}{g(\lambda)}\,d\lambda.$
\\ Proof of (A).
\begin{equation*}
 \begin{split}
   \int_\Delta \scal{g(\lambda)}{g(\lambda)}\,d\lambda
     & = \sum\limits_{j=1}^N \sum\limits_{k=1}^N \bar \alpha_j\alpha_k \int_\Delta \scal{\phi_j(\lambda)}{\phi_k(\lambda)}\,d\lambda \\
     & = \sum\limits_{j=1}^N \sum\limits_{k=1}^N \bar \alpha_j\alpha_k \int_\Delta \frac d{d\lambda} \scal{\phi_j}{E_\lambda\phi_k}\,d\lambda
               \mathcomm{by Thm. \ref{T: Ya 1.2.5}}\\
     & = \sum\limits_{j=1}^N \sum\limits_{k=1}^N \bar \alpha_j\alpha_k \int_\Delta \frac d{d\lambda} \scal{\phi_j}{E^{(a)}_\lambda\phi_k}\,d\lambda
                                                      \mathcomm{by Cor. \ref{C: int Delta dF = int Delta F'dl}}\\
     & = \sum\limits_{j=1}^N \sum\limits_{k=1}^N \bar \alpha_j\alpha_k \scal{\phi_j}{E^{(a)}_\Delta\phi_k}\\
     & = \norm{E^{(a)}_\Delta g}^2.
\end{split}
\end{equation*}
Since $\Delta \subset \LambHF{H_0},$ it follows from Corollary \ref{C: Lambda is a core of sing. sp} that $E_\Delta^{(a)} = E_\Delta.$
It follows that $\norm{E^{(a)}_\Delta g}^2 = \norm{E_\Delta g}^2.$

(B) Proof of the lemma.
Note that $f \in E_{\Lambda \setminus \Delta} \hilb$ implies that $f$ is an absolutely continuous vector for~$H_0.$
Consequently, there exists a sequence $f_1, f_2, \ldots $ of vectors from $P^{(a)}\euD$ converging to $f$ (in~$\hilb$).
Then by Lemma \ref{L: (xi,eta)=int (xi(l),eta(l))dl}
$$
  \int_{\LambHF{H_0}} \scal{ f(\lambda)-f_n(\lambda)}{ f(\lambda)-f_n(\lambda)}\,d\lambda = \norm{ f-f_n}^2  \to 0.
$$
Since by (A)
$$
  \int_{\Delta} \scal{f_n(\lambda)}{f_n(\lambda)}\,d\lambda = \norm{E_\Delta f_n}^2 = \norm{E_\Delta ( f-f_n)}^2 \leq \norm{ f-f_n}^2 \to 0,
$$
it follows that
$$
  \int_{\Delta} \scal{ f(\lambda)}{ f(\lambda)}\,d\lambda = 0.
$$
So, $ f(\lambda) = 0$ for a.e. $\lambda \in \Delta.$
\end{proof}
\begin{cor} \label{C: if ED xi=ED eta then xi(l)=eta(l)}
Let $\Delta$ be a Borel subset of $\Lambda(H_0,F)$ and let $f,g \in \hilb.$
If $E_\Delta  f = E_\Delta g,$ then $ f(\lambda) = g(\lambda)$ for a.e. $\lambda \in \Delta.$
\end{cor}
\begin{cor} \label{C: euE(eDelta f)=chi Delta(l)f(l)}
For any Borel subset $\Delta$ of $\LambHF{H_0}$ and any $f \in \hilb$
$$
  \euE(E_\Delta f)(\lambda) = \chi_\Delta(\lambda)f(\lambda) \qquad \text{a.e.} \ \lambda \in \mbR.
$$
\end{cor}
\begin{cor} \label{C: (ED xi,ED eta) = int D (xi l,eta l)}
Let $\Delta$ be a Borel subset of $\LambHF{H_0}.$
For any $ f, g \in \hilb,$
$$
  \scal{E_\Delta  f}{E_\Delta g} = \int_\Delta \scal{ f(\lambda)}{g(\lambda)}\,d\lambda.
$$
\end{cor}
\begin{prop} \label{P: euE is unitary}
  The map $\euE \colon \hilba \to \euH$ is unitary.
\end{prop}
\begin{proof}
It has already been proven (Proposition \ref{P: euE is an isometry})
that~$\euE$ is an isometry with initial space~$\hilba.$
So, it is enough to show that the range of~$\euE$ coincides with~$\euH.$
Corollary \ref{C: euE(eDelta f)=chi Delta(l)f(l)} implies
that the range of~$\euE$ contains all functions of the form $\chi_\Delta(\cdot) \phi_j(\cdot),$
where $\Delta$ is an arbitrary Borel subset of $\LambHF{H_0}$ and~$j = 1,2,3\ldots$
Consequently, Proposition \ref{P: linear combinations of eta j} completes the proof.
\end{proof}

\subsection{Diagonality of~$H_0$ in $\euH$}
The aim of this subsection is to prove Theorem \ref{T: H0 is diagonal},
which asserts that the direct integral $\euH$ is a spectral representation of~$\hilb$ for the operator~$H^{(a)}_0.$

Using standard step-function approximation argument, it follows from Corollary \ref{C: euE(eDelta f)=chi Delta(l)f(l)} that
\begin{thm} \label{T: euE h(H0)f=...} For any bounded Borel function $h$ on $\LambHF{H_0}$
and any $f \in \hilb$
$$
  \euE_\lambda(h(H_0)f) = h(\lambda) \euE_\lambda f \qquad \text{for a.e.} \ \lambda \in \Lambda.
$$
\end{thm}
%
%
\noindent This theorem implies the following result.
\begin{thm} \label{T: H0 is diagonal}
$H_0^{(a)}$ is naturally isomorphic to the operator of multiplication by $\lambda$ on~$\euH$
via the unitary mapping $\euE \colon \hilba \to \euH:$
$$
  \euE_\lambda(H_0  f) = \lambda \euE_\lambda  f \qquad \text{for a.e.} \ \lambda \in \mbR.
$$
\end{thm}
\noindent Nonetheless, we give another proof of this theorem.
\begin{lemma} \label{L: l(f,Eg)'=(Hf,Eg)'}~\cite[(1.3.12)]{Ya}
  Let~$H$ be a self-adjoint operator on Hilbert space~$\hilb,$ and let $f,g \in \hilb.$
  Then for a.e. $\lambda \in \mbR$
  $$
    \lambda \frac d{d\lambda} \la f, E_\lambda g \ra = \frac d{d\lambda} \la H_0 f, E_\lambda g \ra,
  $$
\end{lemma}

\noindent {\it Proof of Theorem \ref{T: H0 is diagonal}.} \
It is enough to show that for any $ f \in E_\Delta \hilb,$ and for a.e. $\lambda \in \Delta$
the equality $\euE_\lambda (H_0 f) = \lambda f(\lambda)$ holds, where $\Delta$ is any bounded Borel subset of $\Lambda.$

This is equivalent to the statement: for any $g \in E_\Delta \hilb$
$$
  \int_\Delta \scal{\euE_\lambda(H_0 f)}{g(\lambda)}\,d\lambda = \int_\Delta \lambda\scal{ f(\lambda)}{g(\lambda)}\,d\lambda.
$$
By continuity of~$H_0E_\Delta^{H_0}$ and of the multiplicator $\lambda \chi_\Delta(\lambda),$
it is enough to consider the case of $ f = E_\Delta \phi_j \in \hilba$ and $g = E_\Delta \phi_k \in \hilba.$
Then, by (\ref{F: (phi i,Im R phi j) = der (phi i,E(l)phi j)}) and Corollary \ref{C: if ED xi=ED eta then xi(l)=eta(l)},
the right hand side of the previous formula is
\begin{equation*}
 \begin{split}
  \int_\Delta \lambda \frac d{d\lambda}\scal{\phi_j}{E_\lambda\phi_k}\,d\lambda
    = \int_\Delta \frac d{d\lambda}\scal{H_0\phi_j}{E_\lambda\phi_k}\,d\lambda
    = \scal{H_0\phi_j}{E_\Delta \phi_k},
 \end{split}
\end{equation*}
where Lemma \ref{L: l(f,Eg)'=(Hf,Eg)'} has been used. Now,
Corollary \ref{C: (ED xi,ED eta) = int D (xi l,eta l)} completes the proof.
$\Box$

\bigskip

A complete set of unitary invariants of the absolutely continuous part~$H_0^{(a)}$
of the operator~$H_0$ is given by the sequence $(\Lambda_0, \Lambda_1, \Lambda_2, \ldots),$
where $$\Lambda_n = \set{\lambda \in \LambHF{H_0} \colon \dim \hlambda = n}.$$

One of the versions of the spectral theorem says that there exists
a direct integral representation
$$
  \hilba \cong \int_{\hat \sigma}^\oplus \hlambda\,\rho(d\lambda),
$$
of the Hilbert space~$\hilba,$ which diagonalizes~$H_0^{(a)},$ where $\hat \sigma$
is a core of the spectrum of~$H_0,$ and $\rho$ is a measure from the spectral type of~$H_0.$
Actually, instead of changing the measure $\rho$ in its spectral type, it is possible to
change (renormalize) the scalar product of the fiber Hilbert spaces~$\hlambda.$
In the construction of the direct integral, given in this section, a frame in~$\hilb$
in particular fixes a renormalization of scalar products in fiber Hilbert spaces.

The operator $\euE_\lambda$ is the evaluation operator which answers the question $(\ref{F: What is f(l)}).$
As we have seen, for any vector $f \in \hilb_1$ and any point $\lambda$ of the set of full Lebesgue measure $\LambHF{H_0},$
one can define the value of the vector $f$ at $\lambda$ by the formula
$$
  f(\lambda) = \euE_\lambda f.
$$
Vectors $f,$ which belong to $\hilb_1,$ can be defined at every point of the set $\LambHF{H_0},$ since they are regular;
or, rather the contrary, vectors of $\hilb_1$ should be considered regular, since they can be defined at every point of $\LambHF{H_0}.$
If a vector $f$ is not regular, that is, if $f \notin \hilb_1,$ then one can define its value only at almost every
point of $\LambHF{H_0}.$
Results of this section fully justify this interpretation of the operator $\euE_\lambda.$

\begin{rems*} \rm Recall that a vector $f$ is called cyclic for a self-adjoint operator $H_0,$
if vectors $H_0^k f,$ $k=0,1,2,\ldots$ generate the whole Hilbert space $\hilb.$
The construction of the direct integral obviously implies that if $H_0$ has a cyclic vector then $\dim \mathfrak h_\lambda \leq 1$
for all $\lambda \in \LambHF{H_0}.$
\end{rems*}

\begin{rems*}
  Clearly, the family $\Omega_1 :=\set{e_j(\lambda)}$ is a measurability base
and it generates the same set of measurable vector-functions as the measurability base $\Omega_0:=\set{\phi_j(\lambda)};$
that is $\hat \Omega_0 = \hat \Omega_1.$ The family $\Omega_1$ is an orthonormal measurability base.
\end{rems*}

\section{The resonance set $R(\lambda; \set{H_r}, F)$}
\label{S: Tr(z)}

In the previous section we have defined the evaluation operator $\euE_\lambda.$
The evaluation operator is defined on the set $\LambHF{H_0}.$ Since eventually
the operator $H_0$ is going to be perturbed, 
one needs to investigate what happens to the set $\LambHF{H_0}$ when $H_0$ is perturbed.
Clearly, the complement of $\LambHF{H_0}$ consists of points where the operator $H_0$ behaves in some sense badly.
Indeed, by Corollary \ref{C: Lambda is a core of sing. sp} the set $\mbR\setminus \LambHF{H_0}$ is a core of the singular spectrum of $H_0.$
So, one of the reasons, for which a vector $f \in \hilb$ cannot be defined at some point $\lambda\in\mbR$
is that $\lambda$ can be an eigenvalue of $H_0.$

Many results of this section are generally well-known (for rank-one
perturbations), cf. e.g.~\cite{Ar57,Agm,SW,SimTrId2}. I do not
claim any originality for them.

So far we have considered a single fixed self-adjoint operator~$H_0$ on a Hilbert space~$\hilb$
with a frame~$F.$ Now we are going to perturb~$H_0$ by self-adjoint trace-class operators.

We say that an operator-function $\mbR \ni r \mapsto A(r)$ is piecewise analytic in appropriate norm,
if there is a finite or infinite increasing sequence of numbers $r_j, \ j \in \mbZ$ with no finite accumulation points,
such that the restriction of $A(r)$ to any interval $[r_{j-1},r_j]$ has analytic continuation in the norm to a neighbourhood of
that interval. We do not assume continuity of a piecewise-analytic path.

Given a frame $F \in \clL_2(\hilb,\clK)$ in a Hilbert space $\hilb,$
we introduce a vector space $\clA(F)$ of trace-class operators by
\begin{equation} \label{F: clA(F)}
  \clA(F) = \set{ FJF^* \colon J \in \clB(\clK)}.
\end{equation}
For an operator $FJF^* \in \clA(F)$ we define its norm by
$$
  \norm{FJF^*}_{\clA(F)} = \norm{J}.
$$
Obviously, the vector space $\clA(F)$ with such a norm is a Banach space.

\begin{assump}\label{A: assumption on Hr}
Let $F \colon \hilb \to \clK$ be a frame operator in a Hilbert space $\hilb.$
We assume that the path
$$
  \mbR \ni r \mapsto H_r
$$
of self-adjoint operators in $\hilb$ satisfies the following conditions:
\\ (i) \ $H_r = H_0 + V_r,$
\\ (ii) \ $V_r = F^* J_r F,$ where $J_r$ is a bounded self-adjoint operator on the Hilbert space $\clK,$
\\ (iii) \ the path $$\mbR \ni r \mapsto J_r \in \clB(\clK)$$ is continuous and piecewise real-analytic.
\end{assump}
In other words, $H_r \in H_0 + \clA(F)$ and the path $\set{H_r}$ is $\clA(F)$-analytic.

Clearly, $V_0 = 0.$ Obviously, the path
$\set{V_r}$ is continuous and piecewise real-analytic with values in $\clL_1(\clK),$ so that the trace-class derivative
$$
  \dot V_r = F^* \dot J_r F
$$
exists and it is trace-class. Since the derivative $\dot V_r$ belongs to $\clA(F),$ it can be considered as an operator
$\hilb_{-1} \to \hilb_1.$ Clearly, $\dot V_r$ satisfies the following condition:
\begin{equation} \label{F: V hilb(-1) to hilb(1)}
  \dot V_r \colon \hilb_{-1} \to \hilb_1 \quad \text{is a bounded operator}.
\end{equation}

Assumption \ref{A: assumption on Hr} is not too restrictive, as the following lemma shows.
\begin{lemma} \label{L: exists F for H0+rV} Let $H$ be a self-adjoint operator in $\hilb$ and let $V$ be a self-adjoint trace-class operator in $\hilb.$
There exists a frame $F \in \clL_2(\hilb,\clK)$ and a path $\set{H_r}$ which satisfies Assumption \ref{A: assumption on Hr},
such that $H_0=H$ and $H_1 = H+V.$
\end{lemma}
Proof is obvious.

Let
$$
  T_z(H_r) = F R_z(H_r) F^*.
$$

\begin{lemma} \label{L: An to A then An(-1) to A(-1)}
  If operators $A_\alpha,A \in \clBH$ are invertible and $A_\alpha \to A$ uniformly,
  then $A_\alpha^{-1} \to A^{-1}$ uniformly.
\end{lemma}

The following lemma and its proof are well-known (cf. e.g.~\cite[Theorem 4.2]{Agm},~\cite[Lemma 4.7.8]{Ya}).
They are given for completeness. 
\begin{lemma} \label{L: 1+rJT(z) is invertible}
The operator $1+J_r T_z(H_0)$ is invertible for all $r \in \mbR$
and all $z \in \mbC \setminus \mbR.$
\end{lemma}
\begin{proof} The second resolvent equality implies that (Aronszajn's equation~\cite{Ar57}, cf. also~\cite{SW,SimTrId2})
\begin{equation} \label{F: T(z)(1+rJT0(z)) = T0(z)}
  T_z(H_r)(1+J_r T_z(H_0)) = T_z(H_0).
\end{equation}

Since $T_z(H_0)$ is compact, if $1+J_r T_z(H_0)$ is not invertible,
then there exists a non-zero $\psi \in \clK,$ such that
\begin{equation} \label{F: (1+rJT0(z))psi=0}
 (1+J_r T_z(H_0))\psi = 0.
\end{equation}
Combining this equality with~(\ref{F: T(z)(1+rJT0(z)) = T0(z)}) gives
$T_z(H_0)\psi = 0.$ Combining this equality with~(\ref{F: (1+rJT0(z))psi=0})
gives $\psi = 0.$ This contradiction shows that $1+J_rT_0(z)$ is invertible.
\end{proof}
While the operator $1+J_r T_z(H_0)$ is invertible for all non-real values of $z,$
the operator $1+J_r T_{\lambda+i0}(H_0)$ may not be invertible at some points.
The set of points where $1+J_r T_{\lambda+i0}(H_0)$ is not invertible
is of special importance.

\begin{prop} \label{P: def of resonance set}
Let~$\set{H_r, r \in [a,b]}$ be a path of self-adjoint operators on~$\hilb$ with frame~$F,$
which satisfies Assumption \ref{A: assumption on Hr}.
Let $\lambda \in \LambHF{H_0}.$ For any $s \in [a,b]$ the following assertions are equivalent:
  \begin{enumerate}
    \item[(1$_\pm$)] the operator $1 + J_s T_{\lambda\pm i0}(H_0)$ is not invertible;
    \item[(2$_\pm$)] the operator $1 + T_{\lambda\pm i0}(H_0)J_s $ is not invertible;
    \item[(3$_\pm$)] the operator $1 + V_s R_{\lambda\pm i0}(H_0)$ is not invertible in $\hilb_1;$
    \item[(4$_\pm$)] the operator $1 + R_{\lambda\pm i0}(H_0)V_s $ is not invertible in $\hilb_{-1}.$
  \end{enumerate}
\end{prop}
\begin{proof}
The condition (1$_\pm$) is equivalent to (2$_\pm$) by (\ref{F: spec mes(AB)=spec mes(BA)}).
The condition (1$_\pm$) is equivalent to (2$_\mp$) since a bounded operator $T$ is invertible if and only if $T^*$ is invertible.
Equivalence of (1$_\pm$) to (3$_\pm$) and equivalence of (2$_\pm$) to (4$_\pm$) follow from the fact that $F^*$ is an isomorphism
of $\clK$ and $\hilb_1$ and $F$ is an isomorphism of $\hilb_{-1}$ and $\clK.$
\end{proof}

\begin{defn} \label{D: R(lambda)}
We denote by
\begin{equation} \label{F: R(lambda)}
  R(\lambda; \set{H_r}, F) 
\end{equation}
the set of all those real numbers $s$ for which any (and hence all) of the conditions of Proposition \ref{P: def of resonance set} holds.
We call this set the \emph{resonance set} at $\lambda.$
\end{defn}


\begin{lemma} \label{L: R is discrete} The set $R(\lambda; \set{H_r}, F)$ is discrete,
i.e. it has no finite accumulation points.
\end{lemma}
\begin{proof} Since $V_r$ is a piecewise analytic function, this directly follows
from the analytic Fredholm alternative (Theorem \ref{T: analytic Fredholm alternative}).
\end{proof}

\begin{lemma} \label{L: Br exists iff ...}
  Let $\lambda \in \mbR$ be such that the limit $T_{\lambda+i0}(H_0)$ exists in the norm topology.
  Then the limit $T_{\lambda+i0}(H_r)$ exists in the norm topology if and only if
  $r \notin R(\lambda; \set{H_r}, F).$
\end{lemma}
\begin{proof} 
(Only if) \ Assume that $T_{\lambda+i0}(H_r)$ exists. Taking the norm
limit $y = \Im z \to 0$ in (\ref{F: T(z)(1+rJT0(z)) = T0(z)}), one gets
  \begin{gather} \label{F: T(1+T0)=T0}
    T_{\lambda+i0}(H_r)(1+J_r T_{\lambda+i0}(H_0)) = T_{\lambda+i0}(H_0).
  \end{gather}
Since $T_{\lambda+i0}(H_0)$ is compact, $1+J_r T_{\lambda+i0}(H_0)$ is not invertible if and only if
there exists a non-zero $\psi \in \hilb,$ such that $(1+J_r T_{\lambda+i0}(H_0))\psi = 0.$
This and (\ref{F: T(1+T0)=T0}) imply that $T_{\lambda+i0}(H_0)\psi = 0.$ Hence $\psi = 0.$
This contradiction shows that $1 + J_r T_{\lambda+i0}(H_0)$ is invertible.

(If) \ By (\ref{F: T(z)(1+rJT0(z)) = T0(z)}) and Lemma \ref{L: 1+rJT(z) is invertible},
\begin{equation} \label{F: Tr=T0(1+JT0)(-1)}
  T_{\lambda+iy}(H_r) = T_{\lambda+iy}(H_0) \SqBrs{1+J_rT_0(\lambda+iy)}^{-1}.
\end{equation}
If $1+J_rT_0(\lambda+i0)$ is invertible, then by Lemma \ref{L: An to A then An(-1) to A(-1)} the limit of the right hand side as $y \to 0^+$
exists in the norm topology.
\end{proof}

%
%

\begin{thm} \label{T: R(H0,G) is discrete}
Let~$\set{H_r}$ be a path of self-adjoint operators on~$\hilb$ with frame~$F,$
which satisfies Assumption \ref{A: assumption on Hr}.
Let $\lambda \in \LambHF{H_0}.$
For all $r \notin R(\lambda; \set{H_r}, F)$ the inclusion
$\lambda \in \LambHF{H_r} $ holds, where
$R(\lambda; \set{H_r}, F)$ is a discrete subset of~$\mbR,$ defined in (\ref{F: R(lambda)}).
\end{thm}
\begin{proof} 
(A) Since $\lambda \in \Lambda(H_0,F),$ the limit $T_{\lambda+i0}(H_0)$ exists in the norm topology.
It follows from Lemma \ref{L: Br exists iff ...}, that the norm limit of
$$
  T_{\lambda+iy}(H_r) = FR_{\lambda+iy}(H_r)F^*
$$
also exists.

Now, in order to prove that $\lambda \in \Lambda(H_r,F),$
we need to show that the limit of
$F \Im R_{\lambda+iy}(H_r)F^*$ exists in~$\clL_1$-norm.

(B) The formula
\begin{equation} \label{F: Im Tr = (...)Im T0 (...)}
    \Im T_z(H_r) = \brs{1 +T_{\bar z}(H_0)J_r}^{-1} \Im T_z(H_0) \brs{1 +J_r T_z(H_0)}^{-1}
\end{equation}
holds. This follows from (\ref{F: Tr=T0(1+JT0)(-1)}).


(C)
Since $r \notin R(\lambda; \set{H_r}, F),$ it follows from Lemmas \ref{L: An to A then An(-1) to A(-1)}
and \ref{L: 1+rJT(z) is invertible} that
$$
  \brs{1 +T_{\bar z}(H_0)J_r}^{-1} \quad \text{and} \quad \brs{1 +J_r T_z(H_0)}^{-1}
$$ converge in $\norm{\cdot}$ as $y = \Im z \to 0^+.$
Since, by definition of $\LambHF{H_0},$
$\Im T_z(H_0)$ converges to $\Im T_{\lambda+i0}(H_0)$ in~$\clL_1(\clK),$ it follows from (\ref{F: Im Tr = (...)Im T0 (...)}) that
$\Im T_z(H_r)$ also converges in~$\clL_1(\clK)$ as $\Im z \to 0^+.$
Hence, $\lambda \in \LambHF{H_r}.$

That $R(\lambda; \set{H_r}, F)$ is a discrete subset of~$\mbR$ follows from Lemma~\ref{L: R is discrete}.
\end{proof}

\smallskip

Theorem \ref{T: R(H0,G) is discrete} shows that the resonance subset of the plane $(\lambda,r)$
behaves differently with respect to the spectral parameter $\lambda$ and with respect
to the coupling constant $r.$ While for a fixed $r$ the resonance set is a more or less arbitrary null set,
and, consequently, can be very bad, for a fixed $\lambda$ the resonance set is a discrete subset of~$\mbR.$

The discreteness property of the resonance set $R(\lambda; \set{H_r}, F)$
for a.e. $\lambda$ is used in an essential way in subsection \ref{SS: InfinScatM}.

\begin{prop} If $\lambda \in \LambHF{H_0}$ is an eigenvalue of~$H_r,$ then $r \in R(\lambda; \set{H_r}, F).$
\end{prop}
\begin{proof}
Since, by Corollary \ref{C: Lambda is a core of sing. sp}, the complement of $\LambHF{H_r} $
is a support of the singular spectrum of~$H_r,$ which includes all eigenvalues of~$H_r,$
it follows that if $\lambda \in \LambHF{H_0}$ is an eigenvalue of~$H_r,$ then
$\lambda \notin \LambHF{H_r} ,$ so that by Theorem \ref{T: R(H0,G) is discrete}
$r \in R(\lambda; \set{H_r}, F).$
\end{proof}
This proposition partly explains why elements of $R(\lambda; \set{H_r}, F)$ are called resonance points.
Note that the inclusion $r \in R(\lambda; \set{H_r}, F)$
does not necessarily imply that $\lambda$ is an eigenvalue of~$H_r.$

\begin{thm} \label{T: l not in L iff r in R} Let $\lambda \in \LambHF{H_0}.$ Then $\lambda \notin \LambHF{H_r} $ if and only if $r \in R(\lambda; \set{H_r}, F).$
\end{thm}
\begin{proof} The only if part has been established in Theorem \ref{T: R(H0,G) is discrete}.
The if part says that $\lambda \in \LambHF{H_r} $ implies $r \notin R(\lambda; \set{H_r}, F).$
This follows from Lemma \ref{L: Br exists iff ...}.
\end{proof}

\begin{rems*} \rm As can be seen from the proofs, existence of $T_{\lambda+i0}(H_0)$
in $\clL_\infty(\clK)$ or existence of $\Im T_{\lambda+i0}(H_0)$ in $\clL_1(\clK)$
is not essential for the above theorem. In the definition of $\Lambda(H_0,F)$ the ideals $\clL_1(\clK)$ and $\clL_\infty(\clK)$
can be replaced by any $\clL_p(\clK),$ $p\in [1,\infty],$ or even by any pair of invariant operator ideals $\mathfrak S_1$ and $\mathfrak S_2.$
That is, one can consider sets
$$
  \Lambda(H_0,F;\mathfrak S_1, \mathfrak S_2) = \set{\lambda \in \mbR: F \Im R_{\lambda+i0}(H_0)F^* \ \text{exists in} \ \mathfrak S_1 \& F R_{\lambda+i0}(H_0)F^* \ \text{exists in} \ \mathfrak S_2},
$$
so that, in particular, $\Lambda(H_0,F) = \Lambda(H_0,F; \clL_1, \clL_\infty).$
What the last theorem is saying is that, as long as $r_0$ is not a resonance point, the regularity of $\lambda$
is the same for $r=0$ and $r=r_0.$
\end{rems*}

\subsection{Essentially regular points}
Let $\clA = H_0 + \clA(F)$ be the affine space of self-adjoint operators associated with a pair $(H_0,F).$
Theorem \ref{T: l not in L iff r in R} shows that regularity of a point $\lambda \in \mbR$ with respect to an operator $H \in \clA$
does not depend on the path $\set{H_r}.$ This observation suggests the following definition.

Let us fix a frame operator $F$ on a Hilbert space $\hilb$ and an affine space $\clA = H_0+\clA(F)$ of self-adjoint operators.
\begin{defn}
We say that a real number $\lambda$ is \emph{essentially regular},
if there exists an operator $H \in \clA$ such that $\lambda \in \Lambda(H,F).$
\end{defn}
The set of essentially regular numbers we shall denote by $\Lambda(\clA,F).$ So, by definition,
$$
  \Lambda(\clA,F) = \bigcup _{H \in \clA} \Lambda(H,F).
$$
We say that a real number $\lambda$ is \emph{essentially singular}, if it is not
essentially regular. Obviously, the set $\Lambda(\clA,F)$ of essentially regular points has full Lebesgue measure.
By definition, \emph{essentially singular spectrum} of a pair $(\clA,F)$ is the set of all essentially singular points.
Essentially singular spectrum is a null set.

\begin{defn}
If a real number $\lambda$ is essentially regular, then an operator $H \in \clA$ will be called \emph{resonant} at $\lambda,$
if $\lambda \notin \Lambda(H,F).$ Otherwise, we say that $H$ is \emph{regular} or \emph{non-resonant} at $\lambda$ operator.
\end{defn}
We denote the set of regular at $\lambda$ operators by $$\Gamma(\lambda; \clA,F) = \Gamma(\lambda).$$
The complement of $\Gamma(\lambda; \clA,F)$ in $\clA$ will be called the \emph{resonance set} and will be denoted by $R(\lambda; \clA,F).$

Note that if $\lambda$ is essentially singular, then
every operator $H \in \clA$ is resonant at $\lambda,$ though formally in this case the notion of a resonant at $\lambda$ operator does not make sense.

The following reformulation of Theorem \ref{T: l not in L iff r in R} will be useful.
\begin{thm} Let $\lambda$ be an essentially regular point, and let $H_0$ be an operator regular at $\lambda.$
Let $V = F^*JF$ and let $H = H_0+V.$ The operator $H$ is regular at $\lambda$ if and only if the operator
$$
  1 + JT_{\lambda+i0}(H_0)
$$
is invertible.
\end{thm}

\begin{defn} Let $\lambda$ be an essentially regular point, and let $H_0$ be an operator resonant at $\lambda.$
An operator $V \in \clA(F)$ is \emph{regularizing}, if the operator $H_0+V$ is regular at $\lambda.$ An operator $V$
is a \emph{regularizing direction}, if the $H_0+rV$ is regular at $\lambda$ for some $r \in \mbR.$
\end{defn}

\begin{thm} \label{T: Gamma(l) is massive} For every essentially regular point $\lambda \in \mbR,$ the resonance set $R(\lambda; \clA,F)$ is a closed nowhere dense subset of $\clA.$
Moreover, intersection of any real-analytic path (in particular, a
straight line) in $\clA$ with $R(\lambda; \clA,F)$ is either a
discrete set or coincides with the path itself.
\end{thm}
\begin{proof} 
Since $\lambda$ is an essentially regular point, there exists an operator $H_0 \in \clA$ regular at $\lambda.$
If $H$ is another operator regular at $\lambda$ and if $F^*JF = H-H_0,$
then it follows from Theorem \ref{T: l not in L iff r in R} that the operator
$$
  1 + J T_{\lambda+i0}(H_0)
$$
is invertible.
Since for small norm-perturbations of $J$ the latter operator remains to be invertible, it follows from
Theorem \ref{T: l not in L iff r in R} that some neighborhood of $H$ in $\clA$
also belongs to $\Gamma(\lambda;
\clA,F).$ It follows that $\Gamma(\lambda; \clA,F)$ is an open
set, that is, $R(\lambda; \clA,F)$ is closed.

Now, assume that $R(\lambda; \clA,F)$ contains an open ball $U.$
Since $\lambda$ is essentially regular, there exists $H_0 \in
\Gamma(\lambda; \clA,F).$ Let $l$ be the straight line which
passes through $H_0$ and the center of the ball $U.$ By Theorem \ref{T: l not in L iff r in R}, the intersection $U \cap l$ must be a
discrete set, which is is clearly impossible. This proves that
$R(\lambda; \clA,F)$ has empty interior and hence it is nowhere
dense (since it is closed).

Let $l$ be a real-analytic path in $\clA.$ That $R(\lambda;
\clA,F)$ either contains $l$ or intersects $l$ at a discrete set
follows from Theorem \ref{T: l not in L iff r in R}.
\end{proof}

\section{Wave matrix $w_\pm(\lambda; H_r,H_0)$}
\label{S: wave matrix}
In the main setting of the abstract scattering theory, which considers
trace-class perturbations $V$ of arbitrary self-adjoint operators~$H_0,$ one first shows
existence of the wave operators (Kato-Rosenblum theorem,~\cite{KaPJA57,RoPJM57}, cf. also~\cite[\S 6.2]{Ya})
$$
  W_\pm(H_1,H_0) \colon \hilb^{(a)}(H_0) \to \hilb^{(a)}(H_1),
$$
where~$H_1 = H_0+V,$ and after that one shows existence of the wave matrices
\begin{equation} \label{F: formal def of w(pm)}
  w_\pm(\lambda; H_1,H_0) \colon \hlambda(H_0) \to \hlambda(H_1)
\end{equation}
for almost every $\lambda \in \mbR,$ where~$\hlambda(H_j)$ is a fiber Hilbert space
from a direct integral, diagonalizing the absolutely continuous parts~$H^{(a)}_j,$~$j=1,2,$
of the operators~$H_j.$ A drawback of this definition is that,
for a given point $\lambda \in \mbR,$ it is not possible to say whether
$w_\pm(\lambda; H_1,H_0)$ is defined or not. This is because the fiber Hilbert spaces
$\hlambda(H_j)$ are not explicitly defined: they exist for almost every $\lambda,$
but for a fixed $\lambda$ the space~$\hlambda(H_j)$ is not defined.

But if we fix a frame~$F$ in the Hilbert space~$\hilb,$ then for $\lambda \in \LambHF{H_0} \cap \LambHF{H_1}$
it becomes possible
to define the wave matrices $w_\pm(\lambda; H_1,H_0)$ as operators
(\ref{F: formal def of w(pm)}), where~$\hlambda(H_j),$~$j=1,2,$ are the fiber Hilbert spaces
associated with the fixed frame by (\ref{F: def of hlambda}).

While the original proof of Kato and Rosenblum used time-dependent methods,
the method of this paper is based on the stationary approach to abstract scattering theory from~\cite{BE,Ya}.
Combination of ideas from~\cite{BE,Ya} with the construction of the direct integral, given in section \ref{S: direct integral},
allows to define wave matrices $w_\pm(\lambda; H_r,H_0)$ for all $\lambda$ from the set of full Lebesgue measure
$\LambHF{H_0} \cap \LambHF{H_r} $ and prove all their main properties, including the multiplicative property.

In this section~$H_0$ is a self-adjoint operator on~$\hilb$ with frame~$F,$
$V$ is a trace-class self-adjoint operator, for which the path $V_r=rV$ satisfies the condition (\ref{F: V hilb(-1) to hilb(1)}).
We note again, that for any trace-class self-adjoint operator $V$ there exists a frame~$F,$
such that (\ref{F: V hilb(-1) to hilb(1)}) holds for $V_r = rV.$ Consequently, the condition
(\ref{F: V hilb(-1) to hilb(1)}) does not impose any additional restrictions
on the perturbation $V,$ except the trace-class condition.

\subsection{Operators $\mathfrak a_\pm(\lambda;H_r,H_0)$}
In~\cite{Agm}, instead of sandwiching the resolvent, it is considered
as acting on appropriately defined Hilbert spaces. Following this idea,
we consider the limit value $R_{\lambda+i0}(H_0)$ of the resolvent as an operator
$$
 R_{\lambda+i0}(H_0) \colon \hilb_1 \to \hilb_{-1}.
$$
Recall that all Hilbert spaces $\hilb_\alpha, \ \alpha \in \mbR,$ are naturally isomorphic with the isomorphism being
$$
  \abs{F}^{\beta-\alpha} \colon \hilb_\alpha \to \hilb_\beta.
$$
So, if we have an operator-function $A(y), y>0,$ with values in some subclass of $\clBH,$ such that the limit
$$
  \lim_{y \to 0} \abs{F}^\alpha A(y) \abs{F}^\beta
$$
exists in the topology of that class, then the limit
$$
  \lim_{y \to 0} A(y)
$$
exists in the topology of the corresponding subclass of $\clB(\hilb_{\beta},\hilb_{-\alpha}).$
In this way we write $A(0),$ meaning by this an operator from $\hilb_\beta$ to $\hilb_{-\alpha}.$
It is not necessary to use this convention, but otherwise we would need to write a lot of $F$'s in the subsequent formulas,
thus making them cumbersome.

Thus, in an expression such as
$$
  R_{\lambda\mp i\yy}(H_0)V_r
$$
with $y>0,$ both operators $R_{\lambda\mp i\yy}(H_0)$ and $V_r$ can be understood as operators from $\hilb$ to $\hilb,$
\emph{or} the operator $V_r$ can be understood as an operator from $\hilb_{-1}$ to $\hilb_1$ and the operator
$R_{\lambda\mp i\yy}(H_0)$ can be understood as an operator from $\hilb_1$ to $\hilb_{-1}.$ But when we take the limit $y \to 0$
and write
$$
  R_{\lambda\mp i0}(H_0)V_r
$$
both operators should be understood in the second sense, so that the product above is an operator from $\hilb_{-1}$ to
$\hilb_{-1}.$ That is, in the product the operator $V_r\colon \hilb_{-1} \to \hilb_1$ means actually the operator $\abs{F}V_r\abs{F},$
acting in the following way:
\smallskip
$$
  \hilb_{1} \stackrel {\abs{F}} {\mbox {$\leftarrow\!\!\!-\!\!\!-\!\!\!-\!\!\!-\!\!\!-$}} \hilb \stackrel {V_r} {\mbox {$\leftarrow\!\!\!-\!\!\!-\!\!\!-\!\!\!-\!\!\!-$}}
  \hilb \stackrel {\abs{F}} {\mbox {{$\leftarrow\!\!\!-\!\!\!-\!\!\!-\!\!\!-\!\!\!-$}}} \hilb_{-1}.
$$
\smallskip
In the Hilbert space $\hilb$ the operator $R_{\lambda\mp i0}(H_0)V_r$ (if one wishes) should be written as
$$
  \abs{F}R_{\lambda\mp i0}(H_0)\abs{F}V_r,
$$
where $V_r$ is understood as acting from $\hilb$ to $\hilb.$

In the sequel we constantly use this convention without further reference.

\begin{lemma} \label{L: R(lambda+i0) is H.-S.}
If $\lambda \in \LambHF{H_r},$ then
$$
  R_{\lambda\pm i\yy}(H_r) \to R_{\lambda\pm i0}(H_r)
$$
in~$\clL_\infty(\hilb_1,\hilb_{-1})$ as $\yy \to 0^+.$
\end{lemma}
\begin{proof} This follows from Proposition \ref{P: if GAnG converges ...} and definition of $\Lambda(H_0,F).$
\end{proof}

\begin{lemma} \label{L: Im R(lambda+i0) is Trace-class}
If $\lambda \in \LambHF{H_r},$ then
$$
  \Im R_{\lambda+i\yy}(H_r) \to \Im R_{\lambda+i0}(H_r)
$$
in~$\clL_1(\hilb_1,\hilb_{-1})$ as $\yy \to 0^+.$
\end{lemma}
\begin{proof} This follows from Theorem \ref{T: Ya thm 6.1.5} and Proposition \ref{P: if GAnG converges ...}.
\end{proof}

We now investigate the forms (cf.~\cite[Definition 2.7.2]{Ya})
$$
  \mathfrak a_\pm(H_r, H_0; f,g; \lambda) := \lim_{y \to 0^+} \frac y\pi \scal{R_{\lambda\pm iy}(H_r)f}{R_{\lambda\pm iy}(H_0)g}.
$$
Unlike~\cite[Definition 2.7.2]{Ya}, we treat $\mathfrak a_\pm(H_r, H_0; \lambda)$ not as a form, but as an operator from
$\hilb_1$ to~$\hilb_{-1}.$
In~\cite[\S 5.2]{Ya} it is proved that this form is well-defined for a.e. $\lambda \in \mbR.$
In the next proposition we give an explicit set of full measure
on which $\mathfrak a_\pm(H_r, H_0; \lambda)$ exists.

\begin{prop} \label{P: limit of RR}
If $\lambda \in \LambHF{H_r} \cap \LambHF{H_0},$ then the limit
\begin{equation} \label{F: limit of RR}
  \lim\limits_{\yy \to 0^+} \frac \yy \pi R_{\lambda \mp i\yy}(H_r)R_{\lambda \pm i\yy}(H_0)
\end{equation}
exists in~$\clL_1(\hilb_1, \hilb_{-1}).$
\end{prop}
\begin{proof}
We have (cf. e.g.~\cite[(2.7.10)]{Ya})
\begin{equation} \label{F: GRRG exists in L1}
 \begin{split}
  \frac \yy \pi R_{\lambda\mp i\yy}(H_r) & R_{\lambda\pm i\yy}(H_0)
     \\ & = \frac 1 \pi \Im R_{\lambda+i\yy}(H_r) \SqBrs{1 + V_r R_{\lambda\pm i\yy}(H_0)}
     \\ & = \SqBrs{1 - R_{\lambda\mp i\yy}(H_r)V_r} \cdot \frac 1 \pi \Im R_{\lambda+i\yy}(H_0).
 \end{split}
\end{equation}
Since $\lambda \in \LambHF{H_0} \cap \LambHF{H_r},$
by Lemma \ref{L: Im R(lambda+i0) is Trace-class},
the limits of $\Im R_{\lambda+i\yy}(H_0)$
and $\Im R_{\lambda+i\yy}(H_r)$ exist in~$\clL_1(\hilb_1,\hilb_{-1}).$ Also,
by Lemma \ref{L: R(lambda+i0) is H.-S.},
the limits of $R_{\lambda\pm i\yy}(H_0)$ and $R_{\lambda\pm i\yy}(H_r)$
exist in~$\clL_\infty(\hilb_1,\hilb_{-1}),$
while $V_r \colon \hilb_{-1} \to \hilb_1$ is a bounded operator (see (\ref{F: V hilb(-1) to hilb(1)})).
It follows that the limit (\ref{F: limit of RR}) exists in~$\clL_1(\hilb_1,\hilb_{-1}).$
\end{proof}

\begin{defn} \label{D: mathfrak a}
Let $\lambda \in \LambHF{H_r} \cap \LambHF{H_0}.$ The operators
$$
  \mathfrak a_\pm(\lambda;H_r,H_0) \colon \hilb_1 \to \hilb_{-1}
$$
are the limits (\ref{F: limit of RR}) taken in~$\clL_1(\hilb_1, \hilb_{-1})$ topology.
\end{defn}

\begin{prop} \label{P: mathfrak a=(1-RV)Im R} If $\lambda \in \LambHF{H_r} \cap \LambHF{H_0},$ then,
in~$\clL_1(\hilb_{1},\hilb_{-1}),$ the equalities
\begin{equation} \label{F: mathfrak a =(1-RV) Im R}
 \begin{split}
   \mathfrak a_\pm(\lambda;H_r,H_0) & = \SqBrs{1- R_{\lambda \mp i0}(H_r)V_r} \cdot \frac 1\pi \Im R_{\lambda + i0}(H_0) \\
     & = \frac 1\pi \Im R_{\lambda + i0}(H_r)\SqBrs{1+ V_r R_{\lambda \pm i0}(H_0)}
 \end{split}
\end{equation}
hold.
\end{prop}
\begin{proof} This follows from (\ref{F: GRRG exists in L1}), Lemmas \ref{L: R(lambda+i0) is H.-S.}, \ref{L: Im R(lambda+i0) is Trace-class},
  Proposition \ref{P: limit of RR} and (\ref{F: V hilb(-1) to hilb(1)}).
\end{proof}
Note that products such as $R_{\lambda \mp i0}(H_r)V_r \cdot \frac 1\pi \Im R_{\lambda + i0}(H_0)$ should be and are understood as acting in the following way:
\smallskip
$$
  \hilb_{-1} \stackrel {R_{\lambda \mp i0}(H_r)} {\mbox {$\leftarrow\!\!\!-\!\!\!-\!\!\!-\!\!\!-\!\!\!-\!\!\!-\!\!\!-$}} \hilb_1 \stackrel {V_r} {\mbox {$\leftarrow\!\!\!-\!\!\!-\!\!\!-\!\!\!-$}}
  \hilb_{-1} \stackrel {\frac 1\pi \Im R_{\lambda + i0}(H_0)} {\mbox {{$\leftarrow\!\!\!-\!\!\!-\!\!\!-\!\!\!-\!\!\!-\!\!\!-\!\!\!-\!\!\!-\!\!\!-\!\!\!-$}}} \hilb_1.
$$
\smallskip
\begin{lemma} \label{L: if euE f=0 then af=0}
Let $\lambda \in \LambHF{H_r} \cap \LambHF{H_0}$
and let $f \in \hilb_1.$
If $\euE_\lambda(H_0) f = 0,$ then $\mathfrak a_\pm(\lambda;H_r,H_0) f = 0.$
\end{lemma}
\begin{proof} This follows from (see \ref{SS: euE's}(vii) and (\ref{F: (f,A*Ag)=(Af,Ag)}))
\begin{equation} \label{F: euE diam euE = Im R}
  \euE^\diamondsuit_\lambda(H_0) \euE_\lambda(H_0) = \frac 1\pi \Im R_{\lambda + i0}(H_0)
\end{equation}
(as equality in~$\clL_1(\hilb_1,\hilb_{-1})$) and Proposition \ref{P: mathfrak a=(1-RV)Im R}.
\end{proof}

\subsection{Definition of the wave matrix $w_\pm(\lambda; H_r,H_0)$}
Since from now on we need direct integral representations (\ref{F: direct integral})
for different operators~$H_r = H_0 + V_r,$
we denote the fiber Hilbert space, corresponding to~$H_r$ by~$\hlambdar$ or by~$\hlambda(H_r).$


In this section we define the wave matrix $w_\pm(\lambda; H_r,H_0)$
as a form and prove that it is well-defined and bounded, so that it defines an operator.
\begin{defn}  Let $\lambda \in \LambHF{H_r} \cap \LambHF{H_0}.$
The wave matrix $w_\pm(\lambda; H_r,H_0)$
is a densely defined form 
\begin{equation*} 
  w_\pm(\lambda; H_r,H_0) \colon \hlambdar \times \hlambdao \to \mbC,
\end{equation*}
defined by the formula
\begin{equation} \label{F: def of w +-}
  w_\pm(\lambda; H_r,H_0) \brs{\euE_\lambda(H_r)f, \euE_\lambda(H_0) g} =  \scal{f}{\mathfrak a_\pm(\lambda;H_r,H_0)g}_{1,-1},
\end{equation}
where $f,g \in \hilb_1.$
\end{defn}
It is worth to note that this definition depends on endpoint operators $H_0$ and $H_r,$ but it does not depend on the path $\set{H_s}_{s \in [0,r]}$
connecting the endpoints.

One needs to show that the wave matrix is well-defined.
\begin{prop} \label{P: W pm is well defined} For any $\lambda \in \LambHF{H_r} \cap \LambHF{H_0}$ the form $w_\pm(\lambda; H_r,H_0)$
  is well-defined, and it is bounded with norm $\leq 1.$
\end{prop}
\begin{proof} That $w_\pm(\lambda; H_r,H_0)$ is well-defined follows
from Lemma \ref{L: if euE f=0 then af=0}.

Further, by Schwarz inequality, for any $f,g \in \hilb_1,$
\begin{equation} \label{F: norm of mathfrak a}
 \begin{split}
  \frac \yy \pi & \abs{\scal{f}{ R_{\lambda-i\yy}(H_r)R_{\lambda+i\yy}(H_0)g}}
    \\ & \qquad = \frac \yy \pi \abs{\scal{R_{\lambda+i\yy}(H_r)f}{ R_{\lambda+i\yy}(H_0)g}}
    \\ & \qquad \leq \frac \yy \pi \abs{\scal{R_{\lambda+i\yy}(H_r)f}{ R_{\lambda+i\yy}(H_r)f}}^{1/2} \abs{\scal{R_{\lambda+i\yy}(H_0)g}{R_{\lambda+i\yy}(H_0)g}}^{1/2}
    \\ & \qquad = \frac 1\pi \abs{\scal{f}{\Im R_{\lambda+i\yy}(H_r)f}}^{1/2} \cdot \abs{\scal{g}{\Im R_{\lambda+i\yy}(H_0)g}}^{1/2}.
 \end{split}
\end{equation}
Taking the limit $y\to 0^+,$ one gets, using Lemma \ref{L: Im R(lambda+i0) is Trace-class},
Proposition \ref{P: limit of RR} and (\ref{F: euE diam euE = Im R}),
$$
  \abs{\scal{f}{\mathfrak a_\pm(\lambda;H_r,H_0)g}_{1,-1}} \leq \norm{\euE_\lambda (H_r)f}_\hlambdar \cdot \norm{\euE_\lambda (H_0)g}_\hlambdao.
$$
It follows that the wave matrix is bounded with bound less or equal to $1.$
\end{proof}

So, the form $w_\pm(\lambda; H_r,H_0)$ is defined on~$\hlambdar \times \hlambdao.$
We will identify the form $w_\pm(\lambda)$ with the corresponding operator from~$\hlambdao$ to~$\hlambdar,$
so that
$$
  w_\pm(\lambda; H_r,H_0)(\euE_\lambda(H_r) f,\euE_\lambda(H_0) g) = \scal{\euE_\lambda(H_r) f}{w_\pm(\lambda; H_r,H_0)\euE_\lambda(H_0) g},
$$
where $f,g \in \hilb_1.$ 
Note that it follows from the definition of $w_\pm(\lambda; H_r,H_0)$ that
\begin{equation} \label{F: euE w pm euE = mathfrak a}
  \euE^\diamondsuit_\lambda(H_r) w_\pm(\lambda; H_r,H_0) \euE_\lambda(H_0) = \mathfrak a_\pm(\lambda;H_r,H_0).
\end{equation}

The following proposition follows immediately from the definition of $w_\pm(\lambda;H_r,H_0).$
\begin{prop} \label{P: w(H0,H0;lambda)=1 and ...}
1. Let $\lambda \in \LambHF{H_0}.$ Then
$$
  w_\pm(\lambda; H_0,H_0) = \id.
$$
2. Let $\lambda \in \LambHF{H_r} \cap \LambHF{H_0}.$ Then
\begin{equation} \label{F: w* = ...}
  w_\pm^*(\lambda; H_{r},H_{0}) = w_\pm(\lambda; H_0,H_r).
\end{equation}
\end{prop}
\begin{proof} 1. For any $f,g \in \hilb_1,$ one has
\begin{equation*}
 \begin{split}
  \scal{\euE_\lambda(H_0)f}{w_\pm(\lambda; H_0,H_0) \euE_\lambda(H_0) g})_{\hlambdao}
      & = \scal{f}{\mathfrak a_\pm(\lambda;H_0,H_0)g}_{1,-1}         \mathcomment{\ref{F: def of w +-}}
   \\ & = \frac 1\pi \scal{f}{\Im R_{\lambda+i0}(H_0)g}_{1,-1}       \mathcomment{\ref{F: mathfrak a =(1-RV) Im R}}
   \\ & = \scal{f}{\euE^\diamondsuit_\lambda(H_0) \euE_\lambda(H_0)g}_{1,-1}   \mathcomment{\ref{F: euE diam euE = Im R}}
   \\ & = \scal{\euE_\lambda(H_0) f}{\euE_\lambda(H_0)g}_{\hlambdao},          \mathcomment{\ref{F: (f,A*Ag)=(Af,Ag)}}
 \end{split}
\end{equation*}
where (\ref{F: euE diam euE = Im R}) has been used.
Since $\euE_\lambda \hilb_1$ is, by definition, dense in $\hlambda$ (see (\ref{F: def of hlambda})\,) and since,
by Proposition \ref{P: W pm is well defined}, the wave matrix $w_\pm(\lambda; H_r,H_0)$
is bounded, it follows from the last equality that $w_\pm(\lambda; H_0,H_0) = 1.$

2. This follows directly from the definition of $w_\pm(\lambda; H_r,H_0).$ The details are omitted since
later we derive this property of the wave matrix from the multiplicative property.
\end{proof}

\subsection{Multiplicative property of the wave matrix}

We have shown that the wave matrix is a bounded operator from~$\hlambdao$ to~$\hlambdar.$
The next thing to do is to show that it is a unitary operator. Unitary property of the wave matrix
is a consequence of the multiplicative property and the norm bound $\norm{w_\pm}\leq 1.$

In this subsection we establish the multiplicative property of the wave matrix.
We shall intensively use objects such as $\phi_j(\lambda+iy),$ $b_j(\lambda+iy)$ and so on,
associated to a self-adjoint operator~$H_r$ on a fixed framed Hilbert space $(\hilb,F).$
Which self-adjoint operator these objects are associated with will be clear from the context.
For example, if one meets an expression $R_{\lambda+iy}(H_r)b_j(\lambda+iy),$ then this means that
$b_j(\lambda+iy)$ is associated with~$H_r.$

%
%
%
%
%

%
\begin{lemma} \label{L: aaa}
Let $\lambda \in \LambHF{H_0}.$ If $f = \sum\limits_{k=1}^\infty \beta_k \kappa_k\phi_k \in \hilb_1$
(so that $(\beta_j) \in \ell_2$), then
$$
  \scal{\euE_{\lambda+i\yy} (H_0) f}{e_j(\lambda+i\yy)}_{\ell_2} = \alpha_j(\lambda+i\yy) \scal{\beta}{e_j(\lambda+i\yy)}_{\ell_2}.
$$
\end{lemma}
\begin{proof} One has
 \begin{equation*}
   \begin{split}
     \scal{\euE_{\lambda+i\yy}(H_0) f}{e_j(\lambda+i\yy)}
       & = \scal{\euE_{\lambda+i\yy}(H_0)\sum\limits_{k=1}^\infty \beta_k \kappa_k\phi_k }{e_j(\lambda+i\yy)}
     \\ & = \scal{\sum\limits_{k=1}^\infty \beta_k \kappa_k \euE_{\lambda+i\yy}(H_0)\phi_k }{e_j(\lambda+i\yy)}
     \\ & = \scal{\sum\limits_{k=1}^\infty \beta_k \eta_k(\lambda+i\yy) }{e_j(\lambda+i\yy)}                       \mathcomment{\ref{F: phi j(l+iy)=k(-1)eta j(l+iy)}}
     \\ & = \sum\limits_{k=1}^\infty \bar\beta_k \scal{\eta_k(\lambda+i\yy) }{e_j(\lambda+i\yy)}
     \\ & = \sum\limits_{k=1}^\infty \bar\beta_k \alpha_j(\lambda+i\yy) e_{kj}(\lambda+i\yy)
     \\ & = \alpha_j(\lambda+i\yy) \scal{\beta}{e_j(\lambda+i\yy)}_{\ell_2}.
   \end{split}
 \end{equation*}
The second equality holds, since $\euE_{\lambda+i\yy}$ is a bounded operator from~$\hilb_1$ to $\ell_2.$
The fourth equality holds, since the series $\sum\limits_{k=1}^\infty \beta_k \eta_k$ is absolutely convergent.
The fifth equality holds, since $e_j(\lambda+i\yy)$ is an eigenvector of the matrix $\eta(\lambda+i\yy)$
with the eigenvalue $\alpha_j(\lambda+i\yy).$
\end{proof}

For definition of an index of type zero see subsection \ref{SS: index of non-zero type}.
\begin{lemma} \label{L: summands of zero-type}
Let $\lambda \in \LambHF{H_0}$ and $f \in \hilb_1.$
If~$j$ is an index of zero-type, then
$$
  \scal{\euE_{\lambda+i\yy} (H_0) f}{e_j(\lambda+i\yy)} \to 0,
$$
as $\yy \to 0.$
\end{lemma}
\begin{proof}
Using Lemma \ref{L: aaa} (and its representation for $f$) and the definition of $e_j(\lambda+iy)$ we have
\begin{equation*}
 \begin{split}
  \abs{ \scal{\euE_{\lambda+i\yy} (H_0) f}{e_j(\lambda+i\yy)} } & = \alpha_j(\lambda+iy) \abs{ \scal{\beta}{e_j(\lambda+i\yy)}_{\ell_2} }
    \\ & \leq \alpha_j(\lambda+iy) \norm{\beta}\norm{e_j(\lambda+i\yy)} = \alpha_j(\lambda+iy) \norm{\beta}.
 \end{split}
\end{equation*}
If $j$ is an index of zero type, then, by definition, $\alpha_j(\lambda+iy) \to 0$
as $y \to 0.$ The proof is complete.
\end{proof}

\begin{lemma} \label{L: summands of zero-type II}
Let $\lambda \in \LambHF{H_r} \cap \LambHF{H_0}.$
  If~$j$ is of zero-type, then for any $f \in \hilb_1,$
  \begin{equation} \label{F: summands of zero-type II}
      \frac \yy \pi \la{R_{\lambda\pm i\yy}(H_r)f}, {R_{\lambda\pm i\yy}(H_0)b_j(\lambda+ i\yy)}\ra \to 0,
  \end{equation}
as $\yy \to 0.$
\end{lemma}
\begin{proof}
The first equality in (\ref{F: GRRG exists in L1}) and \ref{SS: euE's}(i) 
imply that
  \begin{equation} \label{F: beautiful formula}
    \begin{split}
      \frac \yy \pi \la{R_{\lambda\pm i\yy}(H_r)f}, \,& {R_{\lambda\pm i\yy}(H_0)b_j(\lambda+i\yy)}\ra
      \\ & = \la{\euE_{\lambda+i\yy}(H_0)[1+V_r R_{\lambda\pm i\yy}(H_r)] f}, {\euE_{\lambda+i\yy}(H_0)b_j(\lambda+i\yy)}\ra
      \\ & = \la{\euE_{\lambda+i\yy}(H_0)[1+V_r R_{\lambda\pm i\yy}(H_r)]f}, {e_j(\lambda+i\yy)}\ra,
    \end{split}
  \end{equation}
where the second equality follows from the definition (\ref{F: def of bj}) of $b_j(\lambda+i\yy).$
Since by Lemma \ref{L: R(lambda+i0) is H.-S.} the resolvent $R_{\lambda\pm i\yy}(H_r)$
converges as an operator from $\hilb_1$ to $\hilb_{-1},$ and since $V$ maps $\hilb_{-1}$ to
$\hilb_{1}$ (see (\ref{F: V hilb(-1) to hilb(1)})\,), it follows that the vector $VR_{\lambda\pm i\yy}(H_r) f$ converges in~$\hilb_1$ as $\yy \to 0.$
Now, applying Lemma \ref{L: aaa} and using the fact that for indices of zero type $j$ the eigenvalues $\alpha_j(\lambda+iy)$ converge to $0,$
we conclude that the expression in (\ref{F: summands of zero-type II}) converges to $0$ as $\yy \to 0.$
\end{proof}

The following lemma is well-known and therefore its proof is omitted.
\begin{lemma} \label{L: obvious lemma} If a non-increasing sequence $f_1, f_2, \ldots$ of continuous functions
on $[0,1]$ converges pointwise to $0,$ then it also converges to $0$ uniformly.
\end{lemma}

\begin{lemma} \label{L: sums of alpha's converges to 0 unif...}
   Let $\lambda \in \Lambda(H_0;F).$
   The sum
   $$
     \sum\limits_{j=N}^\infty \alpha_j(\lambda+iy)^2
   $$
   converges to $0$ as $N \to \infty$ uniformly with respect to $y \in [0,1].$
\end{lemma}
\begin{proof} Let $f_N(y)$ be this sum. Since $f_1(y) = \norm{\eta(\lambda+iy)}_2^2,$
it follows from \ref{SS: matrix eta}(iii) and (iv), that $f_1(y),$ and, consequently, all $f_N(y)$ are continuous functions of $y$
in $[0,1].$ So, we have a non-increasing sequence $f_N(y)$ of continuous non-negative functions, converging
pointwise to $0$ as $N\to\infty.$ It follows from Lemma \ref{L: obvious lemma} that the sequence $f_N(y)$ converges
to $0$ as $N\to\infty$ uniformly with respect to $y \in [0,1].$
\end{proof}

\begin{lemma} \label{L: re mult property: uniform convergence of ...}
Let $\lambda \in \LambHF{H_r} \cap \LambHF{H_0}.$ If $f \in \hilb_1,$ then
the sequence
\begin{equation*}
   \brs{\frac \yy \pi}^2\sum\limits_{j=N}^\infty \abs{\scal{R_{\lambda\pm i\yy}(H_r)f}{R_{\lambda\pm i\yy}(H_0)b_j(\lambda+i\yy)}}^2,
           \ \ N = 1,2,\ldots
\end{equation*}
converges to $0$ as $N \to \infty,$ uniformly with respect to $\yy>0.$
\end{lemma}
\begin{proof} We prove the lemma for the plus sign.
The formula (\ref{F: beautiful formula}) and Lemma \ref{L: aaa} imply that
\begin{equation*}
  \begin{split}
     (E) & := \brs{\frac \yy \pi}^2 \sum\limits_{j=N}^\infty \abs{\scal{R_{\lambda+i\yy}(H_r)f}{R_{\lambda+i\yy}(H_0)b_j(\lambda+i\yy)}}^2
        \\ & = \sum\limits_{j=N}^\infty  \abs{\la{\euE_{\lambda+i\yy}(H_0)[1+V_r R_{\lambda+i\yy}(H_r)]f}, {e_j(\lambda+i\yy)}\ra}^2
        \\ & = \sum\limits_{j=N}^\infty  \abs{\alpha_j\brs{\lambda+i\yy} \la \beta(\lambda+iy), {e_j(\lambda+i\yy)}\ra}^2,
  \end{split}
\end{equation*}
where $\beta(\lambda+iy) = (\beta_k(\lambda+iy)) \in \ell_2,$ and
$$
  [1+V_r R_{\lambda+i\yy}(H_r)]f = \sum\limits_{k=1}^\infty \beta_k(\lambda+iy)\kappa_k\phi_k \in \hilb_1.
$$
Since $[1+V_rR_{\lambda+i\yy}(H_r)]f$ converges in~$\hilb_1$ as $y\to 0,$ the sequence $(\beta_k(\lambda+iy))$
converges in $\ell_2$ as $y\to 0.$ It follows that $\norm{\beta(\lambda+iy)}_{\ell_2} \leq C$ for all $y \in [0,1].$
Hence,
$$
  (E) \leq C^2 \sum\limits_{j=N}^\infty \alpha_j(\lambda+iy)^2.
$$
By Lemma \ref{L: sums of alpha's converges to 0 unif...}, the last expression converges to $0$ uniformly.
\end{proof}

In the following theorem we prove the multiplicative property of the wave matrix.
This is a well-known property~\cite{Ya}, but the novelty is that we give an explicit set of full
measure, such that for all $\lambda$ from that set the wave matrices are explicitly defined
and the multiplicative property holds.
\begin{thm} \label{T: mult-ive property of w pm}
Let $\set{H_r}$ be a path satisfying Assumption \ref{A: assumption on Hr}.
If $\lambda \in \LambHF{H_0}$ and if $r_0, r_1, r_2$ are not resonance points of the path $\set{H_r}$
for this $\lambda$ (that is, if $r_0, r_1, r_2 \notin R(\lambda; \set{H_r}, F)$),
then
$$
  w_\pm(\lambda; H_{r_2},H_{r_0}) = w_\pm(\lambda; H_{r_2},H_{r_1})w_\pm(\lambda; H_{r_1},H_{r_0}).
$$
\end{thm}
\begin{proof}
We prove this equality for $+$ sign.
Let $f,g \in \hilb_1.$ 
It follows from \ref{SS: b j}(vi) that
\begin{equation} \label{F: mult. property for a (for eps>0)}
 \begin{split}
    \frac \yy \pi \la{R_{\lambda+i\yy}(H_{r_2})f}, & {R_{\lambda+i\yy}(H_{r_0})g}\ra
   \\ & = \brs{\frac \yy \pi}^2 \sum\limits_{j=1}^\infty
  \scal{R_{\lambda+i\yy}(H_{r_2})f}{R_{\lambda+i\yy}(H_{r_1})b_j(\lambda+i\yy)}
  \\ & \mbox{ } \ \ \ \qquad\qquad \cdot \scal{R_{\lambda+i\yy}(H_{r_1})b_j(\lambda+i\yy)}{R_{\lambda+i\yy}(H_{r_0})g},
 \end{split}
\end{equation}
where the series converges absolutely, since the set of vectors $\set{\sqrt{\frac \yy\pi}R_{\lambda+i\yy}(H_{r_1})b_j(\lambda+i\yy)}$
is orthonormal and complete (see \ref{SS: b j}(vi)).
Applying Schwarz inequality to (\ref{F: mult. property for a (for eps>0)})
and using Lemma \ref{L: re mult property: uniform convergence of ...}, 
one can take the limit $\yy \to 0$ in this formula.
By Lemma \ref{L: summands of zero-type II},
the summands with zero-type~$j$ disappear after taking the limit $\yy \to 0.$

It follows from this and Definition \ref{D: mathfrak a} of
$\mathfrak a_\pm,$ that
\begin{equation} \label{F: nearly mult. property (for eps>0)}
 \begin{split}
    \la f, & \mathfrak a_+(\lambda;H_{r_2},H_{r_0})g\ra_{1,-1}
   \\ & = \sum\limits_{j=1}^\infty
      \la f, \mathfrak a_+(\lambda;H_{r_2},H_{r_1}) b_j(\lambda+i0) \ra_{1,-1}
       \,   \la b_j(\lambda+i0),\mathfrak a_+(\lambda;H_{r_1},H_{r_0})g \ra_{1,-1},
 \end{split}
\end{equation}
where the summation is over indices of non-zero type.
By definition (\ref{F: def of w +-}) of $w_\pm,$ it follows
from (\ref{F: nearly mult. property (for eps>0)}) that
\begin{equation} 
 \begin{split}
    \la & \euE^{(r_2)}_\lambda f, w_\pm(\lambda; H_{r_2},H_{r_0}) \euE_\lambda^{(r_0)} g\ra
   \\ & = \sum\limits_{j=1}^\infty
      \la \euE_\lambda^{(r_2)} f,  w_\pm(\lambda; H_{r_2},H_{r_1})\euE_\lambda^{(r_1)} b_j(\lambda+i0) \ra
       \,   \la \euE_\lambda^{(r_1)} b_j(\lambda+i0),w_\pm(\lambda; H_{r_1},H_{r_0})\euE_\lambda^{(r_0)} g \ra
   \\ & = \sum\limits_{j=1}^\infty
      \la \euE_\lambda^{(r_2)} f,  w_\pm(\lambda; H_{r_2},H_{r_1})e_j(\lambda+i0) \ra
       \,   \la e_j(\lambda+i0),w_\pm(\lambda; H_{r_1},H_{r_0})\euE_\lambda^{(r_0)} g \ra
   \\ & = \la \euE_\lambda^{(r_2)} f,  w_\pm(\lambda; H_{r_2},H_{r_1}) w_\pm(\lambda; H_{r_1},H_{r_0})\euE_\lambda^{(r_0)} g \ra,
 \end{split}
\end{equation}
where in the last equality Lemma \ref{L: e j(l) is basis} 
was used.
Since the set $\euE_\lambda \hilb_1$ is dense in~$\hlambda,$ the proof is complete.
\end{proof}

\begin{cor} \label{C: wave matrix is unitary}
Let $\lambda \in \LambHF{H_r} \cap \LambHF{H_0}.$ Then $w_\pm(\lambda; H_r,H_0)$
  is a unitary operator from~$\hlambdao$ to~$\hlambdar$ and (\ref{F: w* = ...}) holds.
\end{cor}
\begin{proof} Indeed, using the first part of Proposition \ref{P: w(H0,H0;lambda)=1 and ...} and the multiplicative property
of the wave matrix (Theorem \ref{T: mult-ive property of w pm}), one infers that
$$
  w_\pm(\lambda; H_0,H_r)w_\pm(\lambda; H_r,H_0) = w_\pm(\lambda; H_0,H_0) = 1
$$
and
$$
  w_\pm(\lambda; H_r,H_0)w_\pm(\lambda; H_0,H_r) = w_\pm(\lambda; H_r,H_r) = 1.
$$
Since by Proposition \ref{P: W pm is well defined} \ $\norm{w_\pm(\lambda; H_r,H_0)} \leq 1,$
it follows that $w_\pm(\lambda; H_r,H_0)$ is a unitary operator and
$$
  w^*_\pm(\lambda;H_r,H_0) = w_\pm(\lambda; H_0,H_r).
$$
\end{proof}

%
\begin{rems} \rm There is an essential difference between
$\sqrt{\frac \yy\pi} R_{\lambda+i\yy}(H_0)$ (or $\sqrt{\frac \yy\pi} R_{\lambda-i\yy}(H_0)$)
and
$\euE_{\lambda+i\yy}(H_0).$ While they have some common features
(see formulae \ref{SS: b j}(iv) and \ref{SS: b j}(v)), the second operator is better than the first one.
Actually, as it can be seen from the definitions of $\sqrt{\frac \yy\pi} R_{\lambda+i\yy}(H_0)$
and $\euE_{\lambda+i\yy}(H_0),$ these operators ``differ'' by the phase part.
This statement is enforced by the fact that the~$\clL_2(\hilb_1,\hilb)$ norm of the difference
$$
  \sqrt{\frac \yy\pi} R_{\lambda+i\yy}(H_0) - \sqrt{\frac {\yy_1}\pi} R_{\lambda+i\yy_1}(H_0)
$$
remains bounded as $\yy, \yy_1 \to 0,$ even though it does not converge to $0.$ Convergence is hindered
by the non-convergent phase part, which is absent in $\euE_{\lambda+i\yy}(H_0).$
\end{rems}

\subsection{The wave operator}

Recall that a family of operators $A_\lambda \colon \hlambda(H_0) \to \hlambda(H_1)$
is measurable, if it maps measurable vector-functions to measurable vector-functions.
Recall that if
$$
  A = \int_\Lambda^\oplus A(\lambda)\,d\lambda \ \ \text{and} \ \ B = \int_\Lambda^\oplus B(\lambda)\,d\lambda,
$$
then
$$
  AB = \int_\Lambda^\oplus A(\lambda)B(\lambda)\,d\lambda.
$$

We define the wave operator $W_\pm$ as the direct integral of wave matrices:
\begin{equation} \label{F: def of WO pm}
  W_\pm(H_r,H_0) := \int^\oplus_{\LambHF{H_r} \cap \LambHF{H_0}} w_\pm(\lambda; H_r,H_0)\,d\lambda.
\end{equation}
It is clear from (\ref{F: def of w +-}) that the operator-function
$$
  \LambHF{H_r} \cap \LambHF{H_0} \ni \lambda \mapsto w_\pm(\lambda; H_r,H_0)
$$
is measurable, so that the integral above makes sense.

The following well-known theorem (cf.~\cite[Chapter 2]{Ya})
is a direct consequence of the definition (\ref{F: def of WO pm}) of the wave operator $W_\pm,$
Theorem \ref{T: mult-ive property of w pm} and Corollary \ref{C: wave matrix is unitary}.
\begin{thm} \label{T: properties of WO} Let~$\set{H_r}$ be a path of self-adjoint operators
which satisfies Assumption \ref{A: assumption on Hr}.
The wave operator $W_\pm(H_r,H_0) \colon \hilba(H_0) \to \hilba(H_r)$
possesses the following properties.
  \begin{enumerate}
     \item[(i)] $W_\pm(H_r,H_0)$ is a unitary operator.
     \item[(ii)] $W_\pm(H_r,H_0) = W_\pm(H_r,H_s)W_\pm(H_s,H_0).$
     \item[(iii)] $W_\pm^*(H_r,H_0) = W_\pm(H_0,H_r).$
     \item[(iv)] $W_\pm(H_0,H_0) = 1.$
  \end{enumerate}
\end{thm}
If we define $W_\pm(H_r,H_0)$ to be zero on the singular subspace $\hilb^{(s)}(H_0),$ then the part (iv) becomes
$$
  W_\pm(H_0,H_0) = P^{(a)}(H_0).
$$
That is, $W_\pm(H_r,H_0)$ becomes a partial isometry with initial space $\hilb^{(a)}(H_0)$
and final space $\hilb^{(a)}(H_r).$ So,
$$
  W_\pm(H_r,H_0) = W_\pm(H_r,H_0)P^{(a)}(H_0) = P^{(a)}(H_r)W_\pm(H_r,H_0).
$$

\begin{thm} (cf.~\cite[Theorem 2.1.4]{Ya})
For any 
bounded measurable function~$h$ on~$\mbR$
  \begin{equation}
     h(H_r)W_\pm(H_r,H_0) = W_\pm(H_r,H_0) h(H_0).
  \end{equation}
Also,
  \begin{equation} \label{F: Hr W= W H0}
     H_rW_\pm(H_r,H_0) = W_\pm(H_r,H_0)H_0.
  \end{equation}
\end{thm}
\begin{proof} This follows from the definition (\ref{F: def of WO pm}) of $W_\pm$ and Theorem \ref{T: euE h(H0)f=...}. 
\end{proof}
As a consequence, we also get the Kato-Rosenblum theorem.
\begin{cor}~The operators $H_0^{(a)}$ and~$H_1^{(a)},$ considered as operators on absolutely continuous subspaces
$\hilba(H_0)$ and $\hilba(H_1)$ respectively, are unitarily equivalent.
\end{cor}
This follows from (\ref{F: Hr W= W H0}).

\section{Connection with time-dependent definition of the wave operator}
\label{S: connection with time...}
\ndef{\wW}{\stackrel{\mathrm w}W\!\!}
\ndef{\sW}{\stackrel{\mathrm s}W\!\!}

In this section we show that the wave operator defined by (\ref{F: def of WO pm})
coincides with the classical time-dependent definition.
In this subsection I follow \cite{Ya}. Though the proofs follow almost verbatim
those in \cite{Ya} (in \cite{Ya} the proofs are given in a more general setting),
they are given here for reader's convenience and completeness. On the other hand,
availability of the evaluation operator $\euE_\lambda$ allows to simplify the proofs slightly.

In abstract scattering theory the wave operator is usually defined by the formula (cf. e.g.~\cite[(2.1.1)]{Ya}) 
\begin{equation} \label{F: time depend formula for WO}
  W_\pm(H_r,H_0) = \lim\limits_{t \to \pm \infty} e^{itH_r}e^{-itH_0}P^{(a)}(H_0) =: \sW_\pm(H_r,H_0),
\end{equation}
where the limit is taken in the strong operator topology.
Since we define the wave operator in a different way, this formula becomes a theorem.

We denote by $P_r^{(a)}$ the projection $P^{(a)}(H_r).$

The weak wave operators $\wW_\pm$ are defined, if they exist, by the formula
\begin{equation} \label{F: def of weak WO}
  \wW_\pm(H_r,H_0) := \lim_{t \to \pm \infty} P^{(a)}_re^{itH_r}e^{-itH_0}P^{(a)}_0,
\end{equation}
where the limit is taken in the weak operator topology.

Proof of the existence of the wave operator in the strong operator topology uses the existence of the weak
wave operator and the multiplicative property of it. The proof of the latter constitutes the main difficulty
of the stationary approach.


The following lemma is taken from \cite[Lemma 5.3.1]{Ya}.
\begin{lemma} \label{L: int F U0(t)...} If $g \in \hilb$ is such that $\norm{\euE_\lambda g}_\hlambda \leq N$ for a.e. $\lambda \in \LambHF{H_0},$
then
$$
  \int_{-\infty}^\infty \norm{F e^{-itH_0} P_0^{(a)}g}^2\,dt \leq  2 \pi N^2\norm{F}_2^2.
$$
\end{lemma}
\begin{proof} (A) For any frame vector $\phi_j$ the following estimate holds:
$$
  \int_{-\infty}^\infty \abs{\scal{e^{-itH_0}P_0^{(a)} g}{\phi_j}}^2\,dt \leq 2 \pi N^2.
$$
Proof. Note that $g(\lambda)$ is defined for a.e. $\lambda \in
\LambHF{H_0}$ as an element of the direct integral~$\euH.$
It follows from Theorem \ref{T: euE h(H0)f=...} and Lemma \ref{L: (xi,eta)=int (xi(l),eta(l))dl} that
\begin{equation*}
  \begin{split}
  \scal{e^{-itH_0}P_0^{(a)} g}{\phi_j} & = \int_{\LambHF{H_0}} e^{-i\lambda t} \scal{g(\lambda)}{\phi_j(\lambda)}\,d\lambda
   \\ & = \sqrt{2\pi} \hat f_{j}(t),
  \end{split}
\end{equation*}
where $f_j(\lambda) = \scal{g(\lambda)}{\phi_j(\lambda)}$ and $\hat f_j$ is the Fourier transform of $f_j.$
It follows that
\begin{equation*}
  \begin{split}
    (E) & := \int_{-\infty}^\infty \abs{\scal{e^{-itH_0}P_0^{(a)} g}{\phi_j}}^2\,dt
     \\ & = 2\pi \int_{-\infty}^\infty \abs{\hat f_{j}(t)}^2\,dt
      = 2\pi \int_{\LambHF{H_0}} \abs{f_{j}(\lambda)}^2\,d\lambda
     \\ & = 2\pi \int_{\LambHF{H_0}} \abs{\scal{g(\lambda)}{\phi_j(\lambda)}}^2\,d\lambda
         \leq 2\pi N^2 \int_{\LambHF{H_0}} \norm{\phi_j(\lambda)}^2\,d\lambda
     \\ & \leq 2\pi N^2.
  \end{split}
\end{equation*}
We write here $\LambHF{H_0}$ instead of $\mbR,$ but since
$\LambHF{H_0}$ has full Lebesgue measure, it makes no difference.
The proof of (A) is complete.

(B) Using the Parseval equality one has (recall that $(\psi_j)$ is the orthonormal basis from the definition (\ref{F: Frame=...}) of the frame operator
$F$)
  \begin{equation*}
    \begin{split}
       (E) := \int_{-\infty}^\infty \norm{F e^{-itH_0} P_0^{(a)} g}^2\,dt
         & = \int_{-\infty}^\infty \sum_{j=1}^\infty\abs{\scal{F e^{-itH_0} P_0^{(a)}g}{\psi_j}}^2\,dt
      \\ & = \int_{-\infty}^\infty \sum_{j=1}^\infty \kappa_j^2\abs{\scal{e^{-itH_0}P_0^{(a)} g}{\phi_j}}^2\,dt
      \\ & = \sum_{j=1}^\infty \kappa_j^2 \int_{-\infty}^\infty \abs{\scal{e^{-itH_0}P_0^{(a)} g}{\phi_j}}^2\,dt.
    \end{split}
  \end{equation*}
Now, it follows from (A) that
$$
  (E) \leq \sum_{j=1}^\infty \kappa_j^2 \cdot 2\pi N^2 = 2\pi N^2 \norm{F}_2^2.
$$ The proof is complete.
\end{proof}
For the following theorem, see e.g. \cite[Theorem 5.3.2]{Ya}
\begin{thm} \label{T: weak WO exists}
The weak wave operators (\ref{F: def of weak WO}) exist.
\end{thm}
\begin{proof}
(A) For any $f, f_0 \in \hilb$ the estimate
\begin{multline*}
     \Big| \scal{e^{-i{t_2}H_r} e^{-i{t_2}H_0}f_0}{f} - \scal{e^{-i{t_1}H_r} e^{-i{t_1}H_0}f_0}{f}\Big|
      \\ \leq \norm{J_r} \brs{\int_{t_1}^{t_2} \norm{F e^{-i{t_2}H_0}f_0}^2\,dt}^{1/2} \brs{\int_{t_1}^{t_2} \norm{Fe^{-i{t_2}H_r}f}^2\,dt}^{1/2}.
\end{multline*}
holds.

Proof. For any $f, f_0 \in \hilb,$
\begin{equation*}
  \begin{split}
    \frac d{dt} \scal{e^{-itH_0}f_0}{e^{-itH_r}f} & = \scal{(-iH_0)e^{-itH_0}f_0}{e^{-itH_r}f} + \scal{e^{-itH_0}f_0}{(-iH_r)e^{-itH_r}f}
    \\ & = - i \scal{(H_r-H_0) e^{-itH_0}f_0}{e^{-itH_r}f}
    \\ & = - i \scal{V_r e^{-itH_0}f_0}{e^{-itH_r}f}
    \\ & = - i \scal{J_r F e^{-itH_0}f_0}{Fe^{-itH_r}f},
  \end{split}
\end{equation*}
where in the last equality the decomposition $V_r = F^* J_r F$ was used.
It follows that
\begin{equation*}
  \begin{split}
     \scal{e^{-i{t_2}H_r} e^{-i{t_2}H_0}f_0}{f} & - \scal{e^{-i{t_1}H_r} e^{-i{t_1}H_0}f_0}{f}
      \\ &  = - i \int_{t_1}^{t_2} \scal{J_r F e^{-i{t_2}H_0}f_0}{Fe^{-i{t_2}H_r}f}\,dt.
  \end{split}
\end{equation*}
Using the Schwarz inequality,
this implies that
\begin{equation*}
  \begin{split}
     \Big| & \scal{e^{-i{t_2}H_r} e^{-i{t_2}H_0}f_0}{f} - \scal{e^{-i{t_1}H_r} e^{-i{t_1}H_0}f_0}{f}\Big|
      \\ &  \qquad \leq \norm{J_r} \int_{t_1}^{t_2} \norm{F e^{-i{t_2}H_0}f_0}\norm{Fe^{-i{t_2}H_r}f}\,dt
      \\ &  \qquad \leq \norm{J_r} \brs{\int_{t_1}^{t_2} \norm{F e^{-i{t_2}H_0}f_0}^2\,dt}^{1/2} \brs{\int_{t_1}^{t_2} \norm{Fe^{-i{t_2}H_r}f}^2\,dt}^{1/2}.
  \end{split}
\end{equation*}

(B)
 Let $N \in \mbR.$ Let $g, g_0 \in \hilb$ be such that
for a.e. $\lambda \in \LambHF{H_0}$
\begin{equation} \label{F: |f_0| leq N ...}
  \norm{\euE_\lambda(H_0)P_0^{(a)}g_0}_{\hlambdao} \leq N \ \text{and} \  \norm{\euE_\lambda(H_r)P_r^{(a)}g}_{\hlambdar} \leq N.
\end{equation}
Applying the estimate (A) to the pair of vectors $f = P^{(a)}(H_r)g$ and $f_0 = P^{(a)}(H_0)g_0,$
it now follows from the estimates (\ref{F: |f_0| leq N ...}) and Lemma \ref{L: int F U0(t)...}, that
\begin{multline*}
  \Big| \scal{e^{-i{t_2}H_r} e^{-i{t_2}H_0}P^{(a)}(H_0)g_0}{P^{(a)}(H_r)g} - \scal{e^{-i{t_1}H_r} e^{-i{t_1}H_0}P^{(a)}(H_0)g_0}{P^{(a)}(H_r)g}\Big|
  \\ \leq \norm{J_r} \cdot 2\pi N^2 \norm{F}_2^2.
\end{multline*}
Consequently, the right hand side vanishes, when $t_1, t_2 \to \pm\infty.$
It follows that the limits
$$
  \lim_{t \to \pm\infty} \scal{P_r^{(a)}e^{-itH_r} e^{-itH_0}P_0^{(a)}g_0}{g}
$$
exist.
Since the set of vectors $g_0,g,$ which satisfy the estimate (\ref{F: |f_0| leq N ...}) for some $N,$ is dense in~$\hilb,$ it follows from the last
estimate that the weak wave operators (\ref{F: def of weak WO}) exist.
\end{proof}

The following theorem and its proof follow verbatim \cite[Theorem 2.2.1]{Ya}
\begin{thm} \label{T: if w-WO exists and ... then s-WO exists}
  If the weak wave operators $\wW_\pm(H_r,H_0)$ exist and
  \begin{equation} \label{F: wW* wW = Pa}
    \wW_\pm(H_r,H_0)^* \wW_\pm(H_r,H_0) = P_0^{(a)},
  \end{equation}
  then the strong wave operators $\sW_\pm(H_r,H_0)$ exist and coincide with the weak wave operators $\wW_\pm(H_r,H_0).$
\end{thm}
\begin{proof} We have
 \begin{equation*}
   \begin{split}
   E_\pm(t) & := \norm{ e^{itH_r} e^{-itH_0}P_0^{(a)} f - \wW_\pm f }^2
    \\ & = \scal{ e^{itH_r} e^{-itH_0}P_0^{(a)} f - \wW_\pm f }{ e^{itH_r} e^{-itH_0}P_0^{(a)} f - \wW_\pm f }
    \\ & = \scal{P_0^{(a)} f}{f} - 2\Re \scal{ e^{itH_r} e^{-itH_0}P_0^{(a)} f}{\wW_\pm f }
    + \scal{\wW_\pm f}{\wW_\pm f}.
   \end{split}
 \end{equation*}
Since $\wW_\pm = P_r^{(a)} \wW_\pm,$ it follows from (\ref{F: def of weak WO}) that
the second term on the right-hand side of this equality converges to $-2\scal{\wW_\pm f}{\wW_\pm f}$
as $t \to \pm\infty.$
It follows from this and (\ref{F: wW* wW = Pa}) that
$$
  \lim_{t \to \pm\infty} E_\pm(t) = \scal{P_0^{(a)} f}{f} - \scal{\wW_\pm ^*\wW_\pm f}{f} = 0.
$$
That is, the strong wave operators $\sW_\pm$ exist and are equal to $\wW_\pm.$
\end{proof}

The next theorem is taken from \cite[Chapter 2]{Ya}.
\begin{thm} \label{T: time dep-t WO} 
  The strong wave operators $\sW_\pm$ exist and coincide with $W_\pm.$
\end{thm}
\begin{proof} (A) Let $f,g \in \hilb_1$ and let $\Lambda = \LambHF{H_r} \cap \LambHF{H_0}.$  For every $\lambda \in \Lambda$
the vectors $f^{(r)}(\lambda) = \euE_\lambda(H_r) f$ and $g^{(0)}(\lambda) = \euE_\lambda(H_0) g$ are well-defined and the functions
$f^{(r)}(\cdot)$ and $g^{(0)}(\cdot)$  are $\euH$-measurable in the corresponding direct integrals, so that
$$
  \tilde f := P_r^{(a)}f = \int_\Lambda^\oplus f^{(r)}(\lambda)\,d\lambda, \qquad \tilde g := P_0^{(a)}g = \int_\Lambda^\oplus g^{(0)}(\lambda)\,d\lambda.
$$
It follows from the definitions (\ref{F: def of WO pm}) and (\ref{F: def of w +-}) of the wave operator $W_\pm$ and
the wave matrix $w_\pm(\lambda)$ that
\begin{equation*} 
  \begin{split}
    \scal{\tilde f}{W_\pm(H_r,H_0) \tilde g} & = \int_\Lambda \scal{f^{(r)}(\lambda)}{w_\pm(\lambda;H_r,H_0)g^{(0)}(\lambda)}_\hlambdar\,d\lambda
     \\ & = \int_\Lambda \scal{\tilde f}{\mathfrak a_\pm(\lambda;H_r,H_0)\tilde g}_{1,-1}\,d\lambda.
  \end{split}
\end{equation*}
By definition (\ref{D: mathfrak a}) of the operators $\mathfrak a_\pm(\lambda),$ it follows from the last equality that
\begin{equation} \label{F: (f,Wg)=int lim...}
  \scal{\tilde f}{W_\pm(H_r,H_0) \tilde g} = \int_\Lambda \lim\limits_{y\to 0} \frac y\pi \scal{R_{\lambda\pm iy}(H_r)\tilde f}{R_{\lambda\pm iy}(H_0)\tilde g}\,d\lambda.
\end{equation}

(B) Claim: the limit and the integral can be interchanged.

Let $Y$ be a Borel subset of $\Lambda$ and let
$$
  f_y = \frac y\pi \scal{R_{\lambda\pm iy}(H_r)\tilde f}{R_{\lambda\pm iy}(H_0)\tilde g}.
$$
The Schwarz inequality implies
\begin{equation*}
  \begin{split}
    \int_{Y} \abs{f_y}\,d\lambda & \leq \frac y\pi \int_Y \norm{R_{\lambda\pm iy}(H_r)\tilde f} \norm{R_{\lambda\pm iy}(H_0)\tilde g}\,d\lambda
     \\ & \leq \brs{\frac y\pi \int_Y \norm{R_{\lambda\pm iy}(H_r)\tilde f}^2\,d\lambda}^{1/2} \brs{\frac y\pi \int_Y \norm{R_{\lambda\pm iy}(H_0)\tilde g}^2\,d\lambda}^{1/2}
     \\ & \leq \brs{\frac 1\pi \int_Y \scal{\Im R_{\lambda + iy}(H_r)\tilde f}{\tilde f}\,d\lambda}^{1/2} \brs{\frac 1\pi \int_Y \scal{\Im R_{\lambda + iy}(H_0)\tilde g}{\tilde g}\,d\lambda}^{1/2}
  \end{split}
\end{equation*}
Since $\tilde f$ is an absolutely continuous vector for $H_r$ and since $\tilde g$ is an absolutely continuous vector for $H_0,$
the functions $\frac 1\pi \scal{\Im R_{\lambda + iy}(H_r)\tilde f}{\tilde f}$ and $\frac 1\pi \scal{\Im R_{\lambda + iy}(H_0)\tilde g}{\tilde g}$
are Poisson integrals of summable functions $\frac d{d\lambda}\scal{E^{H_r}_\lambda \tilde f}{\tilde f}$
and $\frac d{d\lambda}\scal{E^{H_0}_\lambda \tilde g}{\tilde g}$ respectively.
From Lemma \ref{L: Zet} and from the above estimate it now follows that for $f_y$ the conditions of
Vitali's Theorem \ref{T: Vitali} hold. Hence, Vitali's theorem completes the proof of (A).

(C) 
Claim: $\wW_\pm(H_r,H_0) = W_\pm(H_r,H_0).$

Proof. Using~\cite[(2.7.2)]{Ya}, it follows from (\ref{F: (f,Wg)=int lim...}) and (B) that
$$
  \scal{\tilde f}{W_\pm (H_r,H_0) \tilde g} = \lim_{\eps \to 0} 2\eps \int_0^\infty e^{-2\eps t} \scal{e^{\mp itH_r}\tilde f}{e^{\mp itH_0}\tilde g}\,dt.
$$
Since, by Theorem \ref{T: weak WO exists}, the function $t \mapsto \scal{e^{\mp itH_r} \tilde f}{e^{\mp itH_0}\tilde g}$
has a limit, as $t \to \infty,$ equal to $\scal{\tilde f}{\wW_\pm(H_r,H_0) \tilde g},$ it follows that
the right hand side of the last equality is also equal to $\scal{\tilde f}{\wW_\pm(H_r,H_0) \tilde g}.$
Hence,
$$
  \scal{\tilde f}{W_\pm(H_r,H_0) \tilde g} = \scal{\tilde f}{\wW_\pm(H_r,H_0) \tilde g}.
$$
Since for any self-adjoint operator $H$ the set $P^{(a)}(H)\hilb_1$ is dense in $\hilba(H)$ and since both operators $\wW_\pm(H_r,H_0)$ and $W_\pm(H_r,H_0)$
vanish on singular subspace $\hilbs(H_0)$ of $H_0,$
it follows that $W_\pm(H_r,H_0) = \wW_\pm(H_r,H_0).$

(D) Since for $W_\pm$ the multiplicative property holds (Theorem \ref{T: properties of WO}(ii)),
it follows from (C) that the multiplicative property holds also for the weak wave operator~$\wW_\pm.$
Further, by Theorem \ref{T: if w-WO exists and ... then s-WO exists} existence of the weak wave operator and the multiplicative property
imply that the strong wave operator $\sW_\pm$ exists
and coincides with the wave operator as defined in (\ref{F: def of WO pm}).
\end{proof}

\begin{rems} \label{R: W pm is independent} \rm The operator $\sW_\pm(H_r,H_0)$ acts on $\hilb,$ while the operator $W_\pm(H_r,H_0)$
acts on the direct integral $\euH.$
In Theorem \ref{T: time dep-t WO} by $W_\pm(H_r,H_0)$
one, of course, means the operator
$$
  \euE^*(H_r)W_\pm(H_r,H_0)\euE(H_0) \colon \hilb \to \hilb.
$$
Theorem \ref{T: time dep-t WO}, in particular, shows that
the operators $W_\pm(H_r,H_0)$ are independent from the choice of the frame $F$
in the sense that the operators $\euE^*(H_r)W_\pm(H_r,H_0)\euE(H_0)$ are
independent from $F.$
\end{rems}

\section{The scattering matrix}

In~\cite{Ya} the scattering matrix~$S(\lambda;H_1,H_0)$ is defined
via a direct integral decomposition of the scattering operator $\mathbf S (H_1,H_0).$
In our approach, we first define~$S(\lambda;H_1,H_0),$
while the scattering operator $\mathbf S(H_1,H_0)$ is defined
as a direct integral of~$S(\lambda;H_1,H_0).$
\begin{defn}
  For $\lambda \in \LambHF{H_r} \cap \LambHF{H_0}$ we define the scattering matrix~$S(\lambda;H_r,H_0)$
  by the formula
   \begin{equation} \label{F: def-n of SM}
      S(\lambda;H_r,H_0) := w_+^*(\lambda;H_r,H_0) w_-(\lambda;H_r,H_0).
   \end{equation}
\end{defn}

We list some properties of the scattering matrix which immediately follow from this definition
(cf.~\cite[Chapter 7]{Ya}).
\begin{thm} \label{T: properties of SM}
Let $\set{H_r}$ be a path of operators which satisfy Assumption \ref{A: assumption on Hr}.
Let $\lambda \in \LambHF{H_0}$ and $r \notin R(\lambda; \set{H_r}, F).$
The scattering matrix~$S(\lambda;H_r,H_0)$
possesses the following properties.
  \begin{enumerate}
     \item[(i)] $S(\lambda;H_r,H_0) \colon \hlambdao \to \hlambdao$ is a unitary operator.
     \item[(ii)] For any $h$ such that $r+h \notin R(\lambda; \set{H_r}, F)$ the equality
     $$
       S(\lambda; H_{r+h},H_0) = w_+^*(\lambda;H_r,H_0)S(\lambda; H_{r+h},H_r)w_-(\lambda;H_r,H_0)
     $$
     holds.
     \item[(iii)]
     For any $h$ such that $r+h \notin R(\lambda; \set{H_r}, F)$ the equality
     \begin{equation*}
       S(\lambda; H_{r+h},H_0) = w_+^*(\lambda;H_r,H_0)S(\lambda; H_{r+h},H_r)w_+(\lambda;H_r,H_0) S(\lambda;H_r,H_0)
     \end{equation*}
     holds.
  \end{enumerate}
\end{thm}
\begin{proof}
  (i) By Corollary \ref{C: wave matrix is unitary} the operators $w_+^*(\lambda;H_r,H_0)$ and $w_-(\lambda;H_r,H_0)$ are unitary.
  It follows that their product $S(\lambda;H_r,H_0) = w_+^*(\lambda;H_r,H_0) w_-(\lambda;H_r,H_0)$ is also unitary. \\
  (ii) From the definition of the scattering matrix (\ref{F: def-n of SM}) and multiplicative property of the wave matrix
    (Theorem \ref{T: mult-ive property of w pm}) it follows that
  \begin{equation*}
    \begin{split}
    S(\lambda; H_{r+h},H_0) & = w_+^*(\lambda;H_{r+h},H_0) w_-(\lambda;H_{r+h},H_0)
     \\ & = (w_+(\lambda;H_{r+h},H_r)w_+(\lambda;H_{r},H_0))^* w_-(\lambda;H_{r+h},H_r)w_-(\lambda;H_{r},H_0)
     \\ & = w_+(\lambda;H_{r},H_0)^*w_+(\lambda;H_{r+h},H_r)^* w_-(\lambda;H_{r+h},H_r)w_-(\lambda;H_{r},H_0)
     \\ & = w_+(\lambda;H_{r},H_0)^*S(\lambda; H_{r+h},H_r)w_-(\lambda;H_{r},H_0).
    \end{split}
  \end{equation*}
Note that since $r, r+h \notin R(\lambda; \set{H_r}, F),$ all the operators above make sense.
\\  (iii) It follows from (ii) and unitarity of the wave matrix (Corollary \ref{C: wave matrix is unitary}), that
  \begin{equation*}
    \begin{split}
    S(\lambda; H_{r+h},H_0) & = w_+(\lambda;H_{r},H_0)^*S(\lambda; H_{r+h},H_r)w_-(\lambda;H_{r},H_0)
      \\ & = w_+(\lambda;H_{r},H_0)^*S(\lambda; H_{r+h},H_r)w_+(\lambda;H_{r},H_0) (w_+^*(\lambda;H_{r},H_0) w_-(\lambda;H_{r},H_0))
      \\ & = w_+^*(\lambda;H_r,H_0)S(\lambda; H_{r+h},H_r)w_+(\lambda;H_r,H_0) S(\lambda;H_r,H_0).
    \end{split}
  \end{equation*}
The proof is complete.
\end{proof}
We define the scattering operator by the formula
\begin{equation} \label{F: def-n of SO}
  \mathbf S(H_r,H_0) := \int^\oplus_{\LambHF{H_r} \cap \LambHF{H_0}} S(\lambda;H_r,H_0)\,d\lambda.
\end{equation}
It follows from the definition of the wave operator (\ref{F: def of WO pm}) and
the definition of the scattering matrix that
$$
  \mathbf S(H_r,H_0) = W_+^*(H_r,H_0)W_-(H_r,H_0),
$$
which is a usual definition of the scattering operator.

By Remark \ref{R: W pm is independent}, the definition of the scattering operator (\ref{F: def-n of SO})
is independent from the choice of the frame operator $F.$
\begin{thm}~\cite[Chapter 7]{Ya}
The scattering operator (\ref{F: def-n of SO}) has the following properties
\begin{enumerate}
  \item[(i)] The scattering operator $\mathbf S(H_r,H_0) \colon \hilba(H_0) \to \hilba(H_0)$ is unitary.
  \item[(ii)] The equality
     $$
       \mathbf S(H_{r+h},H_0) = W_+(H_0,H_r)\mathbf S(H_{r+h},H_r)W_-(H_r,H_0)
     $$
     holds.
  \item[(iii)] The equality
  \begin{equation*}
    \mathbf S(H_{r+h},H_0) = W_+(H_0,H_r)\mathbf S(H_{r+h},H_r)W_+(H_r,H_0) \mathbf S(H_r,H_0)
  \end{equation*}
  holds.
  \item[(iv)]  The equality
  $$
    \mathbf S(H_r,H_0)H_0 = H_0 \mathbf S(H_r,H_0)
  $$
  holds.
\end{enumerate}
\end{thm}
\begin{proof} (i) This follows from Theorem \ref{T: properties of SM}(i). \\
(ii) This follows from Theorem \ref{T: properties of SM}(ii). \\
(iii) This follows from Theorem \ref{T: properties of SM}(iii). \\
(iv) follows from the definition of the scattering operator (\ref{F: def-n of SM})
and Theorem \ref{T: euE h(H0)f=...}.
\end{proof}

\subsection{Stationary formula for the scattering matrix}
The aim of this subsection is to prove the stationary formula for the scattering matrix.

\begin{lemma} \label{L: for stat. formula} If $\lambda \in \LambHF{H_r} \cap \LambHF{H_0},$ then
\begin{equation} 
  \begin{split}
    \brs{1 + R_{\lambda - i0}(H_0)V_r} & \cdot \Im R_{\lambda + i0}(H_r) \cdot \brs{1+ V_rR_{\lambda - i0}(H_0)}
    \\ & = \Im R_{\lambda + i0}(H_0)\SqBrs{(1 - 2i V_r[1 - R_{\lambda+i0}(H_r)V_r]) \Im R_{\lambda + i0}(H_0)}
  \end{split}
\end{equation}
as equality in~$\clL_1(\hilb_1,\hilb_{-1}).$
\end{lemma}
\begin{proof}
We write
$$
  R_0 = R_{\lambda + i0}(H_0), \quad R^*_0 = R_{\lambda - i0}(H_0), \quad R_r = R_{\lambda + i0}(H_r), \quad R_r^* = R_{\lambda - i0}(H_r).
$$
Then the last formula becomes
\begin{equation} \label{F: formula X}
    \brs{1 + R_0^*V_r} \cdot \Im R_r \cdot \brs{1+ V_rR_0^*}
         = \Im R_0 \SqBrs{1- 2iV_r(1 - R_r V_r) \Im R_0}.
\end{equation}
Note that by the second resolvent identity
\begin{gather} \label{F: Rr=(1-rVRr)R0}
  R_r = (1-R_rV_r)R_0.
\end{gather}
Using (\ref{F: mathfrak a =(1-RV) Im R}), one has 
\begin{equation*}
    \brs{1 + R_0^*V_r} \Im R_r
    = \Im R_0(1 - V_r R_r).
\end{equation*}
Further, using (\ref{F: Rr=(1-rVRr)R0}),
\begin{equation*}
  \begin{split}
     1 - 2iV_r(1-R_rV_r)\Im R_0 & = 1 - V_r(1-R_r V_r)(R_0-R_0^*)
      \\ & = 1 - V_r(1-R_r V_r)R_0+V_r(1-R_r V_r)R_0^*
      \\ & = 1 - V_r R_r+V_r(1-R_rV_r)R_0^*
      \\ & = (1 - V_r R_r) (1+ V_r R_0^*).
  \end{split}
\end{equation*}
Combining the last two formulae completes the proof.
\end{proof}

In the following theorem, we establish for trace-class perturbations well-known stationary formula for the scattering matrix
(cf.~\cite[Theorems 5.5.3, 5.5.4, 5.7.1]{Ya}).

\begin{thm} \label{T: stationary rep-n for SM}
For any $\lambda \in \LambHF{H_r} \cap \LambHF{H_0}$ the stationary formula for the scattering matrix
\begin{equation} \label{F: stationary rep-n for SM}
  S(\lambda;H_r,H_0) = 1_\lambda - 2\pi i \euE_\lambda(H_0) V_r(1 + R_{\lambda+i0}(H_0)V_r)^{-1} \euE^\diamondsuit_\lambda(H_0)
\end{equation}
holds.
\end{thm}
(The meaning of notation $1_\lambda$ is clear, though the subscript $\lambda$ will be often omitted).
\begin{proof} 
For $\lambda \in \LambHF{H_r} \cap \LambHF{H_0},$ the second resolvent identity
$$
  R_z(H_r) - R_z(H_0) = -R_z(H_r)V_r R_z(H_0) = -R_z(H_0)V_r R_z(H_r),
$$
implies that the stationary formula can be written as
\begin{equation} \label{F: stat. formula 3}
    S(\lambda;H_r,H_0) 
       = 1 - 2\pi i \euE_\lambda(H_0) V_r(1 - R_{\lambda+i0}(H_r)V_r) \euE^\diamondsuit_\lambda(H_0).
\end{equation}   
It follows that it is enough to prove the equality
$$
  w_+^*(\lambda;H_r,H_0)w_-(\lambda;H_r,H_0) = 1 - 2\pi i \euE_\lambda(H_0) V_r(1 - R_{\lambda+i0}(H_r)V_r) \euE^\diamondsuit_\lambda(H_0).
$$

Since the set $\euE_\lambda(H_0) \hilb_1$ is dense in~$\hlambdao=\hlambda(H_0),$ it is enough to show that for any $f,g \in \hilb_1$
\begin{equation*}
  \begin{split}
    & \scal{ \euE_\lambda(H_0) f}{w_+^*(\lambda;H_r,H_0)w_-(\lambda;H_r,H_0)\euE_\lambda(H_0) g}_\hlambdao
    \\ & \mbox{ } \qquad = \scal{\euE_\lambda f}{\brs{1 - 2\pi i \euE_\lambda V_r(1 - R_{\lambda+i0}(H_r)V_r) \euE^\diamondsuit_\lambda}\euE_\lambda g}_\hlambdao.
  \end{split}
\end{equation*}
In other words, using Lemma \ref{L: for stat. formula} and (\ref{F: euE diam euE = Im R}), it is enough to show that
\begin{equation} \label{F: for Lippmann-Schwinger}
 \begin{split}
  (E) & := \scal{w_+(\lambda;H_r,H_0)\euE_\lambda(H_0) f}{w_-(\lambda;H_r,H_0)\euE_\lambda(H_0) g}_\hlambdar
      \\ & = \scal{f}{\brs{1 + R_{\lambda-i0}(H_0)V_r} \frac 1\pi \Im R_{\lambda+i0}(H_r) \brs{1+ V_r R_{\lambda-i0}(H_0)}g}_{1,-1}.
 \end{split}
\end{equation}
Let $\eps > 0.$
Since 
$\euE_\lambda(H_r)\hilb_1$ is dense in~$\hlambdar$ (see (\ref{F: def of hlambda})),
there exists $h \in \hilb_1$ such that the vector
\begin{equation} \label{F: a := w+...}
  a := w_+(\lambda;H_r,H_0)\euE_\lambda(H_0) f - \euE_\lambda(H_r) h \ \in \ \hlambdar
\end{equation}
has norm less than $\eps.$
Definition (\ref{F: def of w +-}) of $w_-(\lambda;H_r,H_0)$ implies that
\begin{equation*}
 \begin{split}
  (E)   & =  \scal{w_+(\lambda;H_r,H_0)\euE_\lambda(H_0) f}{w_-(\lambda;H_r,H_0) \euE_\lambda(H_0) g}_\hlambdar
    \\  & =  \scal{\euE_\lambda(H_r) h+a}{w_-(\lambda;H_r,H_0) \euE_\lambda(H_0) g}_\hlambdar
    \\  & =  \scal{h}{\mathfrak a_-(\lambda;H_r,H_0)(\lambda)g}_{1,-1}   +   \scal{a}{w_-(\lambda;H_r,H_0) \euE_\lambda(H_0) g}_\hlambdar.
 \end{split}
\end{equation*}
So, by the second equality of (\ref{F: mathfrak a =(1-RV) Im R})
$$
  (E) = \scal{h}{\frac 1\pi \Im R_{\lambda+i0}(H_r)[1+V_r R_{\lambda-i0}(H_0)]g}   +   \scal{a}{w_-(\lambda;H_r,H_0) \euE_\lambda(H_0) g}.
$$
Further, by (\ref{F: euE diam euE = Im R}) and (\ref{F: a := w+...}),
\begin{equation*}
 \begin{split}
  (E) & = \scal{\euE_\lambda(H_r) h}{\euE_\lambda(H_r)[1+V_rR_{\lambda-i0}(H_0)]g}
      \\ & \qquad +   \scal{a}{w_-(\lambda;H_r,H_0) \euE_\lambda(H_0) g}
      \\ & = \scal{w_+(\lambda;H_r,H_0)\euE_\lambda(H_0) f - a}{\euE_\lambda(H_r)[1+V_rR_{\lambda-i0}(H_0)]g}
      \\ & \qquad +   \scal{a}{w_-(\lambda;H_r,H_0) \euE_\lambda(H_0) g}
      \\ & = \scal{\euE_\lambda(H_0) f}{w_+(\lambda;H_0,H_r)\euE_\lambda(H_r)[1+V_rR_{\lambda-i0}(H_0)]g}
      \\ & \qquad - \scal{a}{\euE_\lambda(H_r)[1+V_rR_{\lambda-i0}(H_0)]g}
      +   \scal{a}{w_-(\lambda;H_r,H_0) \euE_\lambda(H_0) g}.
 \end{split}
\end{equation*}
By definition (\ref{F: def of w +-}) of $w_+(\lambda;H_r,H_0),$ it follows that
$$
  (E) = \scal{f}{ \mathfrak a_+(\lambda;H_0,H_r) [1+V_rR_{\lambda-i0}(H_0)]g} + \text{remainder},
$$
where
$$
  \text{remainder} := \scal{a}{w_-(\lambda;H_r,H_0) \euE_\lambda(H_0) g - \euE_\lambda(H_r)[1+V_rR_{\lambda-i0}(H_0)]g}.
$$
By the first equality of (\ref{F: mathfrak a =(1-RV) Im R}),
$$
  (E) = \scal{f}{[1 + R_{\lambda-i0}(H_0)V_r] \frac 1\pi \Im R_{\lambda+i0}(H_r) [1+V_rR_{\lambda-i0}(H_0)]g} + \text{remainder}.
$$
Since the norm of the remainder term can be made arbitrarily small, it follows that
$$
  (E) = \scal{f}{[1 + R_{\lambda-i0}(H_0)V_r] \frac 1\pi \Im R_{\lambda+i0}(H_r) [1+V_rR_{\lambda-i0}(H_0)]g}.
$$
The proof is complete.
\end{proof}

As it can be seen from the proof, the remainder term in the proof of the last theorem
is actually equal to zero and so it does not depend on a choice of the vector $h \in \hilb_1;$
that is, for any $f,g \in \hilb$
$$
  \scal{w_+(\lambda;H_r,H_0)\euE_\lambda(H_0) f}{w_-(\lambda;H_r,H_0) \euE_\lambda(H_0) g - \euE_\lambda(H_r)[1+V_rR_{\lambda-i0}(H_0)]g} = 0.
$$
Since the set $w_+(\lambda;H_r,H_0)\euE_\lambda(H_0) \hilb_1$ is dense in $\hlambda(H_r),$ it follows that
\begin{cor}
For any $\lambda \in \LambHF{H_r} \cap \LambHF{H_0}$ the following equality holds:
$$
  w_-(\lambda;H_r,H_0) \euE_\lambda(H_0) = \euE_\lambda(H_r)[1+V_rR_{\lambda-i0}(H_0)].
$$
\end{cor}
From this equality and (\ref{F: for Lippmann-Schwinger}) it also follows that
$$
  w_+(\lambda;H_r,H_0) \euE_\lambda(H_0) = \euE_\lambda(H_r)[1+V_rR_{\lambda+i0}(H_0)].
$$
These equalities are analogues of Lippmann-Schwinger equations for scattering states (see e.g. \cite{TayST}).

\begin{cor} \label{C: S(lambda) in 1+L1}
  If $\lambda \in \LambHF{H_r} \cap \LambHF{H_0},$ then~$S(\lambda;H_r,H_0) \in 1+\clL_1(\hlambdao).$
\end{cor}
\begin{proof} Since $\euE_\lambda^\diamondsuit \in \clL_2(\hlambdao,\hilb_{-1}),$
$V \in \clB(\hilb_{-1},\hilb_{1}),$ $R_{\lambda+i0}(H_0) \in \clL_\infty(\hilb_{1},\hilb_{-1})$
and $\euE_\lambda \in \clL_2(\hilb_1,\hlambdao),$
this follows from (\ref{F: stat. formula 3}).
\end{proof}

Physicists (see e.g. \cite{TayST}) write the stationary formula in a form as it looks in (\ref{F: stationary rep-n for SM}).
As it is often the case with notation used by physicists, the stationary formula, as it is written by physicists,
does not have a rigorous mathematical sense.

In the paper \cite{BE}, \Birman\ and \Entina\ created a mathematically rigorous stationary scattering theory
for trace-class perturbations. In order to give the stationary formula a rigorous meaning, they factorized the perturbation
$V$ in a form $G^*JG$ with Hilbert-Schmidt operator $G \colon \hilb \to \clK$
and a bounded operator $J\colon \clK \to \clK,$ and rewrote the stationary formula for the scattering matrix in the form (see also \cite{Ya})
\begin{equation} \label{F: stationary formula in BE form}
  S(\lambda;H_0+V,H_0) = 1 - 2\pi i Z(\lambda; G) (1 + J T_{\lambda+i0}(H_0))^{-1}J_r Z^*(\lambda; G),
\end{equation}
where
$$
  Z(\lambda;G)f = \euF(G^* f)(\lambda), \ \ T_z(H_0) = G R_{z}(H_0) G^*,
$$
and where $\euF$ is an isomorphism of the absolutely continuous (with respect to $H_0$) subspace of $\hilb$
to a direct integral of Hilbert spaces
$$
  \int_{\hat \sigma}^\oplus \hlambda\,d\rho(\lambda),
$$
such that
$$
  \euF(H_0f)(\lambda) = \lambda \euF(f)(\lambda) \ \ \text{a.e.} \ \lambda \in \mbR.
$$
Existence of such an isomorphism is a consequence of the spectral theorem. The stationary formula has been rewritten in the form
(\ref{F: stationary formula in BE form}), since all ingredients of this formula can be given a rigorous sense (see \cite{BE,Ya}):
$T_{\lambda+i0}(H_0)$ exists for a.e. $\lambda$ by Theorem \ref{T: Ya thm 6.1.9}, and, while the operator $\euF$ does not make sense at a given point $\lambda$ of the spectral line,
combined with a Hilbert-Schmidt operator $G^*$ in $Z(\lambda; G),$ it defines a bounded operator for a.e. $\lambda.$
In this way, the stationary formula is given a rigorous sense for a.e. $\lambda.$ One drawback of the classical approach of \cite{BE}
to stationary scattering theory is that it is impossible to keep track of the set of full Lebesgue measure for which the stationary formula holds
(already because the isomorphism $\clF$ is intrinsically defined for a.e. $\lambda,$ but it cannot be defined at a given point $\lambda$).

In the approach to stationary scattering theory given in this paper, all the ingredients of the stationary formula
--- as it is written by physicists --- are given a rigorous sense; as a consequence, there is no need
to consider operators such as $Z(\lambda; G).$ The function of these operators is performed by the frame operator $F.$

One can also note here that, while factorizations such as $V = G^*JG$ have no physical meaning (at least I am not aware of them),
the background frame operator $F$ may have a physical meaning, since, as it was discussed in subsection \ref{SS: def of frame}, in the case of the Hilbert space
$L_2(M,g)$ with $(M,g)$ a Riemannian manifold, $F$ bears the same information as the Laplace operator $\Delta,$ which in its turn is determined by the metric,
that is, by gravitation\footnote{I am not a physicist, and this is a purely speculative remark}.

\begin{prop} \label{P: S(r) is continuous}
  The scattering matrix~$S(\lambda;H_r,H_0)$ is a meromorphic function of~$r$
  with values in $1+\clL_1(\hlambdao),$
  which admits analytic continuation to all resonance points of the path $\set{H_r}.$ 
\end{prop}
\begin{proof}
Since $R_{\lambda+i0}(H_0)$ is compact, the function $$\mbR \ni r \mapsto S(\lambda; H_r, H_0) \in 1+\clL_1(\hlambdao)$$
admits meromorphic continuation to~$\mbC$ by (\ref{F: stationary rep-n for SM})
and the analytic Fredholm alternative (see Theorem \ref{T: analytic Fredholm alternative}).
Since~$S(\lambda; H_r, H_0)$ is also bounded (unitary-valued)
on the set $\set{r \in \mbR \colon \lambda \in \LambHF{H_r}},$ which by Theorem \ref{T: R(H0,G) is discrete}
has discrete complement in~$\mbR,$ it follows that~$S(\lambda; H_r, H_0)$ has analytic continuation to~$\mbR \subset \mbC,$
that is, the Laurent expansion of $S(\lambda; H_r,H_0)$ (as a function of the coupling constant $r$)
in a neighbourhood of any resonance point $r_0 \in R(\lambda; \set{H_r}, F)$ does not have negative powers of $r - r_0.$
\end{proof}
Though this proposition is quite straightforward it seems to be new (to the best knowledge of the author).
Proposition \ref{P: S(r) is continuous} asserts that the scattering matrix does not notice, in a certain sense, resonance points.
There is a modified ``scattering matrix''
$$
  \widetilde S(\lambda+i0; H_r,H_0) = 1 - 2i r \sqrt{\Im T_{\lambda+i0}(H_0)} J (1 + r T_{\lambda+i0}(H_0)J)^{-1} \sqrt{\Im T_{\lambda+i0}(H_0)},
$$
introduced in \cite{Pu01FA}, which, unlike the
scattering matrix, does notice the resonance points.
This has some implications which have been discussed in \cite{Az2}
and in the setting of this paper they will be discussed in section \ref{S: Push...}.

\subsection{Infinitesimal scattering matrix}
\label{SS: InfinScatM}
Let $\set{H_r}$ be a path of operators which satisfies Assumption \ref{A: assumption on Hr}.

If $\lambda \in \LambHF{H_0},$ then, by \ref{SS: euE's}(vi), the Hilbert-Schmidt operator
$\euE_\lambda \colon \hilb_1 \to \hlambda$ is well defined.
Hence, for any $\lambda \in \LambHF{H_0},$
it is possible to introduce the \emph{infinitesimal scattering matrix}
$$
  \Pi_{H_0}(\dot H_0)(\lambda) \colon \hlambdao \to \hlambdao
$$
by the formula
\begin{equation} \label{F: Pi = ZJZ*}
  \Pi_{H_0}(\dot H_0)(\lambda) 
                        = \euE_\lambda(H_0) \dot H_0 \euE_\lambda^\diamondsuit(H_0),
\end{equation}
where $\euE^\diamondsuit_\lambda \colon \hlambda \to \hilb_{-1}$ is a Hilbert-Schmidt operator as well
(see Subsection \ref{SSS: diamond conjugate}).
Here by $\dot H_0$ we mean the value of the trace-class derivative $\dot H_r$ at $r = 0.$
Since $\euE_\lambda(H_0)$ and $\euE_\lambda^\diamondsuit(H_0)$ are Hilbert-Schmidt operators,
and $\dot H_0 \colon \hilb_{-1} \to \hilb_1$ is bounded, it follows
that $\Pi_{H_0}(\dot H_0)(\lambda)$ is a self-adjoint trace-class operator on the fiber Hilbert space~$\hlambdao.$

The notion of infinitesimal scattering matrix was introduced in~\cite{Az}.

We note the following simple property of $\Pi_{H}(V)(\lambda).$
\begin{lemma}
The operator (transformator)
$$
  \clA(F) \ni V \mapsto \Pi_{H}(V)(\lambda) \in \clL_1\brs{\hlambda(H_0)}
$$
is bounded.
\end{lemma}
\begin{proof} This follows from the estimate
\begin{equation*}
  \norm{\euE_\lambda(H_0) V \euE_\lambda^\diamondsuit(H_0)}_{\clL_1\brs{\hlambda}}
  \leq \norm{\euE_\lambda}_{\clL_2(\hilb_{1},\hlambda)} \, \norm{V}_{\clB(\hilb_{-1},\hilb_{1})} \, \norm{\euE^\diamondsuit_\lambda}_{\clL_2(\hlambda,\hilb_{-1})}.
\end{equation*}
\end{proof}

Dependence of $\Pi_{H}(V)(\lambda)$ on $H$ does not make an exact sense,
since for different $H$ the infinitesimal scattering matrix
acts in different Hilbert spaces $\hlambda(H).$
But given an analytic path $\set{H_r}$ of operators,
we can identify Hilbert spaces $\hlambda(H_r)$ and $\hlambda(H_s)$ in a natural way via the unitary operator $w_\pm(\lambda; H_r,H_s).$
So, one can ask how the operator-function
$$
  \mbR \ni r \mapsto w_\pm(\lambda; H_0,H_r)\Pi_{H_r}(V)(\lambda)w_\pm(\lambda; H_r,H_0) \in \clL_1\brs{\hlambda(H_0)}
$$
depends on $r.$ It turns out that this function is very regular, as we shall see.

As for dependence on $\lambda,$ in the context of arbitrary self-adjoint operators, the dependence of $\Pi_{H}(V)(\lambda)$
on $\lambda$ has to be very bad.


\begin{lemma} \label{L: dot S = -2 pi i ...}
Let $\set{H_r}$ be a path as above. Let $r_0$ be a point of analyticity of $H_r.$
If $\lambda \in \LambHF{H_{r_0}},$ then $\lambda \in \LambHF{H_r}$
for all $r$ close enough to $r_0$ and
$$
  \frac d{dr} S(\lambda; H_r,H_{r_0}) \big|_{r=r_0} = -2 \pi i \Pi_{H_{r_0}}(\dot H_{r_0})(\lambda),
$$
where the derivative is taken in~$\clL_1(\hlambdao)$-topology.
\end{lemma}
\begin{proof} By Theorem \ref{T: R(H0,G) is discrete},
if $\lambda \in \LambHF{H_{r_0}},$ then $\lambda \in \LambHF{H_r}$    
for all $r$ from some neighbourhood of $r_0.$ Without loss of generality we can assume that $r_0 = 0.$
We have
\begin{equation} \label{F: dot Sr}
 \begin{split}
  \frac d{dr} V_r (1 + R_{\lambda+i0}(H_{0})V_r)^{-1} & = \dot V_r (1 + R_{\lambda+i0}(H_{0})V_r)^{-1}
   \\ & \qquad - V_r(1 + R_{\lambda+i0}(H_{0})V_r)^{-1} R_{\lambda+i0}(H_{0})\dot V_r (1 + R_{\lambda+i0}(H_{0})V_r)^{-1},
 \end{split}
\end{equation}
where the derivative is taken in $\clB(\hilb_{-1},\hilb_1).$
Since $V_0 = 0$ and $\dot H_r = \dot V_r,$
this and Theorem \ref{T: stationary rep-n for SM} imply that
\begin{equation} \label{F: S = 1 - 2 pi i ...}
  \begin{split}
    \frac d{dr} S(\lambda; H_r,H_{0}) \big|_{r=r_0} & = \frac d{dr} \brs{1_\lambda - 2\pi i \euE_\lambda(H_{0}) V_r (1 + R_{\lambda+i0}(H_{0})V_r)^{-1}\euE_\lambda^\diamondsuit(H_{0})}\big|_{r=0}
    \\ & = -2\pi i \euE_\lambda(H_{0}) \cdot \frac d{dr} \brs{V_r (1 + R_{\lambda+i0}(H_{0})V_r)^{-1}}\big|_{r=0} \cdot \euE_\lambda^\diamondsuit(H_{0})
    \\ & = -2\pi i \euE_\lambda(H_{0}) \dot H(0) \euE_\lambda^\diamondsuit(H_{0}).
  \end{split}
\end{equation}

This and (\ref{F: Pi = ZJZ*}) complete the proof.
\end{proof}

\begin{thm} If $\lambda \in \LambHF{H_r} \cap \LambHF{H_0},$ then
\begin{gather} \label{F: ddr S(r) = ...}
  \frac d{dr} S(\lambda;H_r,H_0) = -2\pi i w_+(\lambda;H_0,H_r)\Pi_{H_r}(\dot H_r)(\lambda)w_+(\lambda;H_r,H_0)S(\lambda;H_r,H_0),
\end{gather}
where the derivative is taken in the trace-class norm.
\end{thm}
\begin{proof} By Theorem \ref{T: R(H0,G) is discrete}, for all small enough $h$ the inclusion
$\lambda \in \LambHF{H_{r+h}}$ holds. It follows from Theorem \ref{T: properties of SM}(iii) and
unitarity of $w_\pm(\lambda; H_r,H_0)$ (Corollary \ref{C: wave matrix is unitary}) that
\begin{equation*}
 \begin{split}
   S(\lambda; & H_{r+h},H_0) - S(\lambda;H_r,H_0)
    \\ & = w_+(\lambda;H_0,H_r) \SqBrs{S(\lambda; H_{r+h},H_r) - 1_\lambda}w_+(\lambda;H_r,H_0) S(\lambda; H_r,H_0).
 \end{split}
\end{equation*}
Dividing this equality by $h$ and
taking the trace-class limit $h \to 0$ in it we get
\begin{equation*}
 \begin{split}
   \frac d{dh}S(\lambda; & H_{r+h},H_0)\big|_{h=0}
    \\ & = w_+(\lambda;H_0,H_r) \frac d{dh} S(\lambda; H_{r+h},H_r)\big|_{h=0} w_+(\lambda;H_r,H_0) S(\lambda; H_r,H_0).
 \end{split}
\end{equation*}
So, Lemma \ref{L: dot S = -2 pi i ...} completes the proof.
\end{proof}

Definition of the chronological exponential $\Texp,$ used in the next theorem, is given in Appendix \ref{A: Texp}.
\begin{thm} \label{T: S = T exp...}
Let $\set{H_r}$ be a path of operators which satisfies Assumption \ref{A: assumption on Hr}.
If $\lambda \in \LambHF{H_r} \cap \LambHF{H_0},$ then
\begin{equation} \label{F: S = T exp...}
  S(\lambda;H_r,H_0) = \Texp\brs{-2\pi i \int_0^r w_+(\lambda;H_0,H_s)\Pi_{H_s}(\dot H_s)(\lambda)w_+(\lambda;H_s,H_0)\,ds},
\end{equation}
where the chronological exponential is taken in the trace-class norm.
\end{thm}
\begin{proof}
By Theorem \ref{T: R(H0,G) is discrete}, by the definition
(\ref{F: Pi = ZJZ*}) of the infinitesimal scattering matrix
and by Proposition \ref{P: W pm is well defined}, the expression under the integral in (\ref{F: S = T exp...})
makes sense for all~$s$ except the discrete resonance set $R(\lambda; \set{H_r}, F).$
By Proposition \ref{P: S(r) is continuous} and (\ref{F: dot Sr}),
the derivative $\frac d{dr} S(\lambda; H_r,H_0)$ is piecewise $\clL_1(\hlambdao)$-analytic.
Since~$\mbR \ni r \mapsto S(\lambda; H_r,H_0)$ is also piecewise $\clL_1(\hlambdao)$-analytic,
by the formula (\ref{F: ddr S(r) = ...}) the function
\begin{equation} \label{F: w Pi w = S'(r)S(-1)}
  r \mapsto w_+(\lambda; H_0, H_r) \Pi_{H_r}(\dot H_r)(\lambda) w_+(\lambda; H_r, H_0)
   = - \frac 1{2\pi i} \SqBrs{\frac d{dr} S(\lambda;H_r,H_0)} S(\lambda;H_r,H_0)^{-1}
\end{equation}
is also piecewise $\clL_1(\hlambdao)$-analytic. Hence, integration of the equation (\ref{F: ddr S(r) = ...})
by Lemma \ref{L: T exp} gives (\ref{F: S = T exp...}).
\end{proof}
\begin{cor} \label{C: Tr Pi is analytic}
  Let $\set{H_r}$ be a path of operators which satisfies Assumption \ref{A: assumption on Hr} and
  let $\lambda \in \LambHF{H_0}.$ The function
  $$
    \mbR \ni r \mapsto \Tr\brs{\Pi_{H_r}(\dot H_r)(\lambda)}
  $$
  is piecewise analytic (not necessarily continuous). Analyticity of this path may fail only at those points where $r \mapsto S(\lambda; H_r,H_0)$ is not analytic.
\end{cor}
\begin{proof} This follows from (\ref{F: w Pi w = S'(r)S(-1)}), unitarity of the wave matrix $w_+(\lambda; H_0, H_r)$
(Corollary \ref{C: wave matrix is unitary}) and unitarity and analyticity of the scattering matrix $S(\lambda;H_r,H_0)$ as a function of $r$
(Proposition \ref{P: S(r) is continuous}).
\end{proof}

One can also prove the following formula
\begin{equation} \label{F: S = T exp... (2)}
  S(\lambda;H_r,H_0) = \overrightarrow{\exp}\brs{-2\pi i \int_0^r w_-(\lambda;H_0,H_s)\Pi_{H_s}(\dot H_s)(\lambda)w_-(\lambda;H_s,H_0)\,ds}.
\end{equation}
where the right chronological exponential $\overrightarrow{\exp}$ is defined by
\begin{gather*} \label{F: def of right exp}
  \overrightarrow{\exp}\brs{\frac 1i\int_a^t A(s)\,ds} =
     1 + \sum_{k=1}^\infty \frac 1{i^k} \int_a^{t} dt_k \int_a^{t_k} dt_{k-1} \ldots \int_a^{t_{2}} dt_1 A(t_1)\ldots A(t_k).
\end{gather*}

\section{Absolutely continuous and singular spectral shift functions}

\subsection{Infinitesimal spectral flow}
In this subsection we prove a theorem,
which asserts that the trace of the infinitesimal scattering matrix is
a density of the absolutely continuous part of the infinitesimal
spectral flow. 

We recall that if $A \colon \hilb \to \clK$ and $B\colon \clK \to \hilb$
are two bounded operators, such that $AB$ and $BA$ are trace-class operators in Hilbert spaces~$\clK$
and~$\hilb$ respectively, then
\begin{equation} \label{F: Tr(AB)=Tr(BA)}
  \Tr_\clK(AB) = \Tr_\hilb(BA).
\end{equation}

Let~$\set{H_r}$ be a path of self-adjoint operators which satisfies Assumption \ref{A: assumption on Hr}.
In addition to this assumption, the condition
\begin{equation}  \label{A: double sum is conv-t}
  \sum_{j,k=1}^\infty \kappa_j\kappa_k J^{r}_{jk} \quad \text{is absolutely convergent}
\end{equation}
will be used, where $V_r = F^*J_rF,$ and $\brs{J^{r}_{jk}}$ is the matrix of $J_r$ in the basis $(\psi_k),$
that is, $J^{r}_{jk} = \scal{\psi_j}{J_r \psi_k}.$

Obviously, for a straight line path $H_r=H_0+rV,$ there exists a frame $F$ such that this additional condition holds.

%

\begin{rems*} \rm 
V.\,V.\,Peller constructed an example of a trace-class operator $A=(a_{ij})$ and a bounded operator $B=(b_{ij})$ on $\ell_2,$
such that the double series
$$
  \sum_{i,j=1}^\infty \abs{a_{ij}b_{ij}}
$$
diverges\footnote{Private communication}.
\end{rems*}

\begin{lemma} \label{L: double sum is summable}
Let $F$ be a frame operator on a Hilbert space~$\hilb.$ Let $\set{H_r}$
be a path of operators on $\hilb,$ such that Assumption \ref{A: assumption on Hr} and (\ref{A: double sum is conv-t}) hold.
For any $r$ the double series
$$
  \sum_{j=1}^\infty \sum_{k=1}^\infty \kappa_j \kappa_k J^r_{jk} \scal{\phi_{j}(\lambda)}{\phi_{k}(\lambda)}
$$
is absolutely convergent for a.e. $\lambda \in \LambHF{H_0}$ and the function
$$
  \LambHF{H_0} \ni \lambda \mapsto \sum_{j=1}^\infty \sum_{k=1}^\infty \kappa_j \kappa_k \abs{J^r_{jk} \scal{\phi_{j}(\lambda)}{\phi_{k}(\lambda)}}
$$
is integrable.
\end{lemma}
\begin{proof} It follows from the assumption (\ref{A: double sum is conv-t})
and Corollary \ref{C: phi j(l),phi k(l) is summable} that
it is enough to prove the following assertion.

If a non-negative series $\sum_{j=1}^\infty a_n$ is convergent (to $A$) and if
a sequence of integrable functions $f_1, f_2, \ldots$ is such that $\norm{f_j}_{L_1}\leq 1$
for all~$j=1,2,\ldots,$ then the series
$$
  \sum_{j=1}^\infty a_j f_j
$$
is absolutely convergent a.e. and its sum is integrable.

The series $g(x) := \sum_{j=1}^\infty a_j \abs{f_j}(x)$ is convergent (so far, possibly to $+\infty$) a.e.
Since
$$
  \int \sum_{j=1}^N a_j \abs{f_j}(x)\,dx \leq A
$$
for all $N,$ the series $\sum_{j=1}^\infty a_j \abs{f_j}$ is absolutely convergent a.e. and its sum $g(x)$ is integrable.
Since
$$
  \sum_{j=1}^N a_j \abs{f_j}(x) \leq g(x),
$$
it follows that, by the Lebesgue Dominated Convergence theorem, the series above is absolutely convergent and its sum is integrable.
\end{proof}

\begin{prop} \label{P: Tr(F* Im R F) is L1} Let $H_0$ be a s.a. operator on a framed Hilbert space $(\hilb,F).$
The non-negative function
$$
  \LambHF{H_0} \ni \lambda \mapsto \Tr \phi(\lambda) = \frac 1\pi \Tr \brs{F \Im R_{\lambda+i0}(H_0) F^*}
$$
  is summable and
$$
  \int_{\LambHF{H_0}} \Tr \phi(\lambda)\,d\lambda = \Tr(F E^{(a)} F^*).
$$
\end{prop}
\begin{proof} Indeed,
\begin{equation*}
 \begin{split}
    \int \Tr \phi(\lambda)\,d\lambda & = \frac 1\pi \int \sum_j \kappa_j^2 \scal{\phi_j}{\Im R_{\lambda+i0}(H_0)\phi_j}\,d\lambda
    \\ & = \frac 1\pi \sum_j \int \kappa_j^2 \scal{\phi_j}{\Im R_{\lambda+i0}(H_0)\phi_j}\,d\lambda
    \\ & = \frac 1\pi \sum_j \int \scal{F^*\psi_j}{\Im R_{\lambda+i0}(H_0)F^*\psi_j}\,d\lambda
    \\ & = \frac 1\pi \sum_j \scal{\psi_j}{F E^{(a)} F^*\psi_j} = \Tr(F E^{(a)} F^*).
 \end{split}
\end{equation*}
\end{proof}

\begin{thm} \label{T: Tr Pi = Phi} Let~$H_0$ be a self-adjoint operator on a Hilbert space
with a frame~$F.$ Let $V$ be a trace-class operator such that for $V_r = rV$ the condition (\ref{F: V hilb(-1) to hilb(1)}) holds.
Then for any bounded measurable function $h$ the equality
$$
  \Tr(Vh(H_0^{(a)})) = \int_{\LambHF{H_0}} h(\lambda) \Tr_{\hlambda} (\Pi_{H_0}(V)(\lambda))\,d\lambda
$$
holds. 
\end{thm}
\begin{proof} (A) Here we prove the claim assuming (\ref{A: double sum is conv-t}).

Since $V$ satisfies (\ref{F: V hilb(-1) to hilb(1)}), it has the representation
\begin{equation} \label{F: V=F*JF}
  V = F^* J F,
\end{equation}
where $J \colon \clK \to \clK$ is a bounded self-adjoint operator
(not necessarily invertible). We recall that the frame operator~$F$ is given by (\ref{F: Frame=...}).
Let $(J_{jk})$ be the matrix of $J$ in the basis $(\psi_j)$ (see (\ref{F: Frame=...})), i.e.
\begin{equation} \label{F: J(ij)}
  J\psi_j = \sum\limits_{k=1}^\infty J_{jk} \psi_k.
\end{equation}
Using (\ref{F: Tr(AB)=Tr(BA)}) and (\ref{F: V=F*JF}), we have
$$
  \Tr_\hilb(Vh(H_0^{(a)})) = \Tr_\clK(JFh(H_0^{(a)})F^*).
$$
Calculation of the trace in the right hand side of this formula in the orthonormal basis $(\psi_j)$ of~$\clK,$
together with (\ref{F: J(ij)}) and (\ref{F: F phi j=...}) give
\begin{equation*}
 \begin{split}
   \Tr_\hilb(Vh(H_0^{(a)})) & = \sum_{j=1}^\infty \scal{JFh(H_0^{(a)})F^*\psi_j}{\psi_j}
     \\ & = \sum_{j=1}^\infty \scal{h(H_0^{(a)})F^*\psi_j}{F^*J\psi_j}
     \\ & = \sum_{j=1}^\infty \scal{h(H_0^{(a)})F^*\psi_j}{F^* \sum_{k=1}^\infty J_{jk}\psi_k}
     \\ & = \sum_{j=1}^\infty \kappa_j \scal{h(H_0^{(a)})\phi_j}{\sum_{k=1}^\infty J_{jk} \kappa_k \phi_k}
     \\ & = \sum_{j=1}^\infty \sum_{k=1}^\infty \kappa_j \kappa_k J_{jk} \scal{h(H_0^{(a)})\phi_j}{\phi_k}.
 \end{split}
\end{equation*}
This double sum is absolutely convergent by the assumption (\ref{A: double sum is conv-t})
and the estimate \ $\abs{\scal{h(H_0^{(a)})\phi_j}{\phi_k}} \leq \abs{h}_\infty.$

Now, combining the last equality with Theorem \ref{T: euE h(H0)f=...}
and Corollary \ref{C: (ED xi,ED eta) = int D (xi l,eta l)} implies
\begin{equation*}
  \Tr_\hilb(Vh(H_0^{(a)}))
     = \sum_{j=1}^\infty \sum_{k=1}^\infty \kappa_j \kappa_k J_{jk}
         \int_{\LambHF{H_0}} h(\lambda)\scal{\phi_j(\lambda)}{\phi_k(\lambda)}_\hlambda\,d\lambda.
\end{equation*}
It follows from Lemma \ref{L: double sum is summable}, that
the integral and summations in the last equality can be interchanged:
\begin{equation} \label{F: some formula 1}
  \Tr_\hilb(Vh(H_0^{(a)}))
     = \int_{\LambHF{H_0}} h(\lambda)\sum_{j=1}^\infty \sum_{k=1}^\infty \kappa_j \kappa_k J_{jk}
         \scal{\phi_j(\lambda)}{\phi_k(\lambda)}_\hlambda\,d\lambda.
\end{equation}

On the other hand, by (\ref{F: Tr(AB)=Tr(BA)}) and (\ref{F: V=F*JF}), for any $\lambda \in \LambHF{H_0}$
\begin{equation*}
   \Tr_{\hlambda} (\euE_\lambda V \euE_\lambda^\diamondsuit)
      = \Tr_{\clK} (JF \euE_\lambda^\diamondsuit\euE_\lambda F^*).
\end{equation*}

Similarly, evaluation of the last trace in the orthonormal basis $(\psi_j)$ of~$\clK$
gives
\begin{equation*}
 \begin{split}
  \Tr_{\hlambda} (\Pi_{H_0}(V)(\lambda)) =
  \Tr_{\hlambda} (\euE_\lambda V \euE_\lambda^\diamondsuit)
       & = \sum_{j=1}^\infty \scal{\euE_\lambda^\diamondsuit\euE_\lambda F^*\psi_j}{F^*J \psi_j}_{-1,1}
    \\ & = \sum_{j=1}^\infty \scal{\euE_\lambda F^*\psi_j}{\euE_\lambda F^*J \psi_j}_{\hlambda}
    \\ & = \sum_{j=1}^\infty \kappa_j \scal{ \euE_\lambda\phi_j}{\euE_\lambda \sum_{k=1}^\infty J_{jk} \kappa_k \phi_k}_\hlambda
    \\ & = \sum_{j=1}^\infty \sum_{k=1}^\infty \kappa_j \kappa_k J_{jk} \scal{ \euE_\lambda\phi_j}{\euE_\lambda\phi_k}_\hlambda.
 \end{split}
\end{equation*}
(The last equality here holds, since $\sum_{k=1}^\infty$ converges absolutely).
Combining this equality with (\ref{F: some formula 1})
completes the proof of (A).

(B) Plainly, the set of operators $J$ which satisfy the condition (\ref{A: double sum is conv-t}) is dense in $\clB(\clK)$
in the strong operator topology. So, let $J \in \clB(\clK)$ and let $J_1, J_2, \ldots$ be a sequence of operators converging
to $J$ in the strong operator topology and such that all $J_n$ satisfy (\ref{A: double sum is conv-t}). Convergence in strong operator topology
implies that operators $J_1, J_2, \ldots$ are uniformly bounded.

We have
\begin{multline*}
  \Tr_{\hlambda} (\Pi_{H_0}(V)(\lambda)) = \Tr_{\hilb_1} (V \euE^\diamondsuit_\lambda\euE_\lambda) = \Tr_{\hilb_1} \brs{V \frac 1\pi \Im R_{\lambda+i0}(H_0)}
    \\ = \Tr_{\hilb} \brs{J \frac 1\pi F\Im R_{\lambda+i0}(H_0)F^*}.
\end{multline*}
It follows from this and Lemma \ref{L: An to A so then AnV to AV}
that for all $\lambda \in \Lambda(H_0,F)$
$$\Tr_{\hlambda} (\Pi_{H_0}(V_n)(\lambda)) \to \Tr_{\hlambda} (\Pi_{H_0}(V)(\lambda))$$
as $n \to \infty,$ where $V_n = F^*J_nF.$
Since norms of operators $J_1, J_2, \ldots$ are uniformly bounded (by, say, $C>0$) , the summable functions $\lambda \mapsto \Tr_{\hlambda} (\Pi_{H_0}(V_n)(\lambda))$
are dominated by a single summable (by Proposition \ref{P: Tr(F* Im R F) is L1}) function $C \Tr_{\hilb} \brs{\frac 1\pi F\Im R_{\lambda+i0}(H_0)F^*}.$

It follows from this and the Lebesgue dominated convergence theorem that
$$
  \norm{\Tr_{\hlambda} (\Pi_{H_0}(V_n)(\cdot)) - \Tr_{\hlambda} (\Pi_{H_0}(V)(\cdot))}_1  \to 0.
$$

(C) Combining (A), (B) and Lemma \ref{L: An to A so then AnV to AV} completes the proof.
\end{proof}

The \emph{infinitesimal spectral flow} $\Phi_{H_0}(V)$ is a distribution on~$\mbR,$
defined by the formula
$$
  \Phi_{H_0}(V)(\phi) = \Tr(V\phi(H_0)).
$$
This notion was introduced in~\cite{ACS} and developed in~\cite{AS2,Az}.
The terminology ``infinitesimal spectral flow'' is justified by the following classical formula
from formal perturbation theory (see e.g. \cite[(38.6)]{LL3})
$$
  E_n^{(1)} = V_{nn},
$$
where $V_{nn} = \la n | V | n \ra$ is the matrix element of the perturbation $V,$
$E_n^{(1)}$ denotes the first correction term for the $n$-th eigenvalue $E_n^{(0)}$ (corresponding to $|n\ra$)
of the unperturbed operator~$H_0$ perturbed by $V.$ If the support of $\phi$ contains only the
eigenvalue $E_n^{(0)}$ and $\phi\brs{E_n^{(0)}} = 1,$ then $\Tr(V\phi(H_0)) = V_{nn}.$
So, $\Phi_{H_0}(V)(\phi)$ measures the shift of eigenvalues of~$H_0.$ Another justification is that,
according to the Birman-Solomyak formula (\ref{Int: BS formula}), the spectral shift function is the integral of infinitesimal spectral flow
$\Phi_{H_r}(V)(\delta).$


\begin{rems} \label{R: mu(phi)}
From now on, for an absolutely continuous measure $\mu$
we denote its density by the same symbol. So, in $\mu(\phi), \ \phi \in C_c(\mbR),$
$\mu$ is a measure, while in $\mu(\lambda), \ \lambda \in \mbR,$ $\mu$ is a function.
\end{rems}

Actually, $\Phi_{H_0}(V)$ is a measure~\cite{AS2}. So, one can introduce the absolutely continuous
and singular parts of the infinitesimal spectral flow:
$$
  \Phia_{H_0}(V)(\phi) = \Tr(V\phi(H^{(a)}_0)),
$$
and
$$
  \Phis_{H_0}(V)(\phi) = \Tr(V\phi(H^{(s)}_0)).
$$

Recall (see (\ref{F: Pi = ZJZ*})) that for every $\lambda \in \Lambda(H_0;F)$ and any $V \in \clA(F)$ (see (\ref{F: clA(F)})
for the definition of $\clA(F)$), we have a trace class operator
$$
  \Pi_{H_0}(V)(\lambda) \colon \hlambdao \to \hlambdao.
$$
We define the \emph{standard density} function of the absolutely continuous infinitesimal
spectral flow by the formula
\begin{equation} \label{F: def of Phia(l)}
  \Phia_{H_0}(V)(\lambda) := \Tr(\Pi_{H_0}(V)(\lambda)) \quad \text{for all} \ \ \lambda \in \LambHF{H_0},
\end{equation}
where $V \in \clA(F),$ and where, allowing a little abuse of notation\footnote{See Remark \ref{R: mu(phi)}}, we denote the density of
the infinitesimal spectral flow $\Phia_{H_0}(V)$ by the same symbol $\Phia_{H_0}(V)(\cdot).$
Since $\Phia_{H_0}(V)$ is absolutely continuous, the usage of this notation should not cause any problems.
This terminology and notation are justified by Theorem \ref{T: Tr Pi = Phi}.

The absolutely continuous part $\Phia_{H_0}(\cdot)(\lambda)$ of infinitesimal spectral flow can be looked at as a one-form on the affine space of operators
$$
  H_0 + \clA(F),
$$
where $\clA(F)$ is defined by (\ref{F: clA(F)}).

The standard density $\Phia_{H_0}(V)(\cdot)$ of the absolutely continuous part of the infinitesimal spectral
flow may depend on a frame operator~$F.$ But, as Theorem \ref{T: Tr Pi = Phi} shows,
for any two frames the corresponding standard densities are equal a.e.
\begin{cor} Let $H_0$ be a self-adjoint operator and let $V$ be a trace-class self-adjoint operator.
For any two frame operators $F_1$ and~$F_2,$ such that $V \in \clA(F_1) \cap \clA(F_2),$
the standard densities (\ref{F: def of Phia(l)}) of the absolutely continuous part
of the infinitesimal spectral flow coincide a.e.
\end{cor}
\begin{proof} Indeed, the left hand side of the formula in Theorem \ref{T: Tr Pi = Phi} does not depend on $F.$
\end{proof}

Recall that $\clA(F)$ is a vector space of trace-class self-adjoint operators, associated with a given frame $F,$ (see (\ref{F: clA(F)})).
\begin{lemma} \label{L: L1 norm of Phia <L1 norm of V} Let $H$ be a self-adjoint operator on a Hilbert space with frame $F$
and let $V \in \clA(F).$ The function $\Lambda(H,F) \ni \lambda \mapsto \Phia_{H}(V)(\lambda)$  is summable
and its $L_1$-norm is $\leq \norm{V}_1.$
\end{lemma}
\begin{proof} By Theorem \ref{T: Tr Pi = Phi}, this function is a density of an absolutely continuous finite measure $\phi \mapsto \Phia_H(V)(\phi).$
By the same theorem, $L_1$-norm of this function is $\leq \norm{V}_1.$
\end{proof}

One can consider the resonance set as a set-function of two variables $r$ and $\lambda:$
$$
  \gamma(\set{H_r}; F) := \set{(\lambda,r) \in \mbR^2 \colon \lambda \in \LambHF{H_r}}.
$$
\begin{lemma} \label{L: Phia(r,lambda) is meas-ble}
Let $\set{H_r}$ be a path of self-adjoint operators, which satisfy Assumption \ref{A: assumption on Hr}.
The set $\gamma(\set{H_r}; F) \subset \mbR^2$ is Borel measurable and
the function (see (\ref{F: def of Phia(l)}))
\begin{equation} \label{F: Phia(r,lambda)}
  \gamma(\set{H_r}; F) \ni (\lambda,r) \mapsto \Phia_{H_r}(\dot H_r)(\lambda)
\end{equation}
is also measurable. Moreover, the complement of $\gamma(\set{H_r}; F)$ is a null set in~$\mbR^2.$
\end{lemma}
\begin{proof} 
The set $\gamma(\set{H_r}; F)$ is Borel measurable since it is the (intersection of two) set of points of convergence of
a family of continuous functions
$$
  F R_z(H_r)F^*
$$
of two variables $r$ and $z = \lambda+iy$ (see Definition \ref{D: Lambda(H0,F)}), as $y \to 0^+.$

The function $(\lambda,r) \mapsto \Phia_{H_r}(\dot H_r)(\lambda)$ is measurable since
\begin{equation*}
 \begin{split}
  \Phia_{H_r}(\dot H_r)(\lambda) = \Tr(\Pi_{H_r}(\dot H_r)(\lambda)) & = \Tr \brs{\euE_\lambda(H_r) \dot H_r \euE_\lambda^\diamondsuit(H_r)}
    \\ & = \lim_{y \to 0^+} \Tr \brs{\euE_{\lambda+iy}(H_r) \dot H_r \euE_{\lambda+iy}^\diamondsuit(H_r)},
 \end{split}
\end{equation*}
where the last equality follows from the fact that $\euE_{\lambda+iy}^\diamondsuit(H_r) \colon \hlambdar \to \hilb_{-1}$
is Hilbert-Schmidt (see subsection \ref{SS: euE's}), $\dot H_r \colon \hilb_{-1}\to \hilb_{1}$ is bounded
and $\euE_{\lambda+iy}(H_r) \colon \hilb_{1}\to \hlambdar$
is also Hilbert-Schmidt, and the operators $\euE_{\lambda+iy}^\diamondsuit(H_r),$ $\euE_{\lambda+iy}(H_r)$
converge to $\euE_{\lambda+i0}^\diamondsuit(H_r),$ $\euE_{\lambda+i0}(H_r)$ in the Hilbert-Schmidt norm, so that the product
$\euE_{\lambda+iy}(H_r) \dot H_r \euE_{\lambda+iy}^\diamondsuit(H_r)$ converges to $\euE_\lambda(H_r) \dot H_r \euE_\lambda^\diamondsuit(H_r)$
in the trace-class norm, as $y \to 0^+.$

That the complement of $\gamma(\set{H_r}; F)$ is a null set in~$\mbR^2$ now follows from Fubini's Theorem, from the discreteness property
of the resonance set with respect to $r$ (Theorem \ref{T: R(H0,G) is discrete}) and from the fact that $\LambHF{H_r}$ is a full set
(Proposition \ref{P: Lambda(H0) has full meas}).
\end{proof}

\subsection{Absolutely continuous and singular spectral shift functions}
Let $$\gamma = \set{H_r\colon r \in [0,1]}$$ be a continuous piecewise real-analytic path of operators.

For the given path $\gamma,$ we define the spectral shift function $\xi$ and its absolutely continuous $\xia$
and singular $\xis$ parts as distributions by the formulae
\begin{equation} \label{F: def of xi(phi)}
  \xi_{\gamma}(\phi; H_1, H_0) = \xi(\phi; H_1, H_0) = \int_0^1 \Phi_{H_r}(\dot H_r)(\phi)\,dr, \quad \phi \in \Cc,
\end{equation}
\begin{equation} \label{F: def of xia(phi)}
  \xia_{\gamma}(\phi; H_1, H_0) = \int_0^1 \Phia_{H_r}(\dot H_r)(\phi)\,dr, \quad \phi \in \Cc,
\end{equation}
\begin{equation} \label{F: def of xis(phi)}
  \xis_{\gamma}(\phi; H_1, H_0) = \int_0^1 \Phis_{H_r}(\dot H_r)(\phi)\,dr, \quad \phi \in \Cc.
\end{equation}
For the straight path $\set{H_r = H_0+rV},$
the first of these formulae is the Birman-Solomyak spectral averaging formula~\cite{BS75SM},
which shows that the definition of the spectral shift function, given above, coincides
with the classical definition of \KreinMG\ ~\cite{Kr53MS}. It was shown in \cite{AS2} that the integral in (\ref{F: def of xi(phi)})
is the same for all continuous piecewise analytic paths connecting $H_0$ and $H_1.$

While $\xi$ does not depend on the path $\gamma$ connecting $H_0$ and $H_1,$ the distributions $\xia$ and $\xis$
depend on the path (see subsection \ref{SS: xis is non-additive}).

\begin{lemma} Let $\gamma = \set{H_r}$ be a path which satisfies Assumption \ref{A: assumption on Hr}.
The distribution $\xia_{\gamma}$ is a finite absolutely continuous measure
with density\footnote{See Remark \ref{R: mu(phi)}}
\begin{equation} \label{F: def of xia}
  \xia_{\gamma}(\lambda; H_1,H_0) := \int_0^1 \Phia_{H_r}(\dot H_r)(\lambda)\,dr, \quad \ \ \lambda \in \LambHF{H_0}.
\end{equation}
\end{lemma}
\begin{proof} (A) The function \ $\Phia_{H_r}(\dot H_r)(\lambda)$ \ 
  is summable on $[0,1]\times \mbR.$ Indeed, by
  Lemma \ref{L: Phia(r,lambda) is meas-ble} this function is measurable and
  by Lemma \ref{L: L1 norm of Phia <L1 norm of V}
  the $L_1$-norm of $\Phia_{H_r}(V)(\lambda)$ is uniformly bounded (by $\norm{V}_1$) with respect to $r \in [0,1].$

  (B) It follows from (A) and Fubini's theorem,
  that for any bounded measurable function~$h,$ in the iterated integral
  $$
    \int_0^1 \int _\mbR h(\lambda) \Phia_{H_r}(\dot H_r)(\lambda)\,d\lambda\,dr
  $$
  one can interchange the order of integrals. It follows from this and Theorem \ref{T: Tr Pi = Phi} that
  for any $\phi \in C_c(\mbR)$
  \begin{equation*}
    \begin{split}
       \xia_{\gamma}(\phi) & = \int_0^1 \Phia_{H_r}(\dot H_r)(\phi)\,dr                      \mathcomment{\ref{F: def of xia(phi)}}
       \\ & = \int_0^1 \int_\mbR \Phia_{H_r}(\dot H_r)(\lambda)\phi(\lambda)\,d\lambda\,dr    \mathcomm{by Thm. \ref{T: Tr Pi = Phi}}
       \\ & = \int_\mbR \phi(\lambda) \int_0^1 \Phia_{H_r}(\dot H_r)(\lambda)\,dr\,d\lambda
       \\ & = \int_\mbR \phi(\lambda) \xia_{\gamma}(\lambda)\,d\lambda.
    \end{split}
  \end{equation*}
  It follows that $\xia_{\gamma}$ is absolutely continuous. Lemma \ref{L: L1 norm of Phia <L1 norm of V} implies that $\xia_{\gamma}$ is a finite measure.
\end{proof}
\begin{cor} The measure $\xis$ is also absolutely continuous and finite.
\end{cor}
\begin{proof} Since $\xi$ and $\xia$ are finite and absolutely continuous, the claim follows from $\xis = \xi - \xia.$
\end{proof}
In the last lemma we again denote by the same symbol $\xia_{\gamma}$
an absolutely continuous measure and its density.
We call the function $\xia_{\gamma}(\lambda)$ the standard density of $\xia_{\gamma}$ with respect to the frame $F.$
Note that $\xia_{\gamma}$ is explicitly defined for all $\lambda \in \LambHF{H_0}.$
It is not difficult to see that, for the straight line path $H_r = H_0+rV,$ the function $\xia_{\gamma}(\lambda),$ thus defined, coincides a.e. with the right hand side
of the formula (\ref{Int: def of xia}).

For further use, we note the following
\begin{prop} \label{P: xia(r) is analytic}
Let $\lambda$ be an essentially regular point. Let $\gamma = \set{H_r, r \in [a,b]}$ be an analytic path.
The function $[a,b] \ni r \mapsto \xia_\gamma(\lambda; H_r, H_0)$ is analytic in a neighbourhood of $[a,b].$
\end{prop}
\begin{proof} The function $r \mapsto \Phia_{H_r}(\dot H_r)(\lambda)$ is analytic as the trace of the analytic function
$r \mapsto \Pi_{H_r}(\dot H_r)(\lambda)$ (see Corollary \ref{C: Tr Pi is analytic}).
Therefore, the function $r \mapsto \xia_\gamma(\lambda; H_r, H_0)$ is analytic as the definite integral of an analytic function
$r \mapsto \Phia_{H_r}(H_r)(\lambda).$
\end{proof}

\bigskip

If $T$ is a trace-class operator, then by \ $\det(1+T)$ \ we denote the classical Fredholm determinant of $1+T$ (see subsection \ref{SS: Fredholm det}).
Since, by Corollary \ref{C: S(lambda) in 1+L1}, the scattering matrix $S(\lambda;H_r,H_0)$ takes values in $1 + \clL_1(\hlambdao),$
the determinant $\det S(\lambda;H_r,H_0)$ makes sense.

Let $\lambda \in \LambHF{H_0}.$ Note that,
by Proposition \ref{P: S(r) is continuous}, the function
$$
  \mbR \ni r \mapsto S(\lambda; H_r,H_0) \in 1 + \clL_1(\hlambdao)
$$
is continuous in~$\clL_1(\hlambdao).$ Hence, the function
$$
  \mbR \ni r \mapsto \det S(\lambda; H_r,H_0) \in \mbT
$$
is also continuous (see (\ref{F: det is cont-s})), where~$\mbT = \set{z \in \mbC\colon \abs{z} = 1}.$
So, it is possible to define a continuous function
$$
  \mbR \ni r \mapsto -\frac 1{2\pi i} \log \det S(\lambda; H_r,H_0) \in \mbR
$$
with zero value at $0.$

\begin{thm} \label{T: xia = -2pi i log det S}
Let $F$ be a frame operator on $\hilb$ and let $\gamma = \set{H_r}_{r \in [0,1]}$ be a path of operators,
which satisfies the Assumption \ref{A: assumption on Hr}.
For all $\lambda \in \LambHF{H_r} \cap \LambHF{H_0}$ the equality
\begin{equation} \label{F: xia = -2pi i log det S}
  \xia_{\gamma}(\lambda; H_r, H_0) = -\frac 1{2\pi i} \log \det S(\lambda;H_r,H_0)
\end{equation}
holds, where the logarithm is defined in such a way that the function
$$
  [0,r] \ni s \mapsto \log \det S(\lambda;H_s,H_0)
$$
is continuous.
\end{thm}
\begin{proof} By definitions (\ref{F: def of xia}) and (\ref{F: def of Phia(l)}) of $\xia$ and $\Phia$ we have
\begin{equation} \label{F: xia(l)=int Phia}
  \xia_{\gamma}(\lambda; H_r, H_0) = \int_0^r \Phia_{H_s}(\dot H_r)(\lambda)\,ds = \int_0^r \Tr_{\mathfrak h^{(s)}_\lambda} (\Pi_{H_s}(\dot H_r)(\lambda))\,ds.
\end{equation}
By Theorem \ref{T: R(H0,G) is discrete} and by the definition (\ref{F: Pi = ZJZ*})
of the infinitesimal scattering matrix $\Pi_{H_s}(V)(\lambda),$
the integrand of the last integral is defined for all $s \in [0,r]$
except the \emph{discrete} resonance set $R(\lambda; \set{H_r}, F),$ (see (\ref{F: R(lambda)})).
Moreover, by Corollary \ref{C: Tr Pi is analytic}, the function
$$
  \mbR \ni s \mapsto \Tr_{\mathfrak h^{(s)}_\lambda} (\Pi_{H_s}(V)(\lambda))
$$
is piecewise analytic. Consequently, the integral (\ref{F: xia(l)=int Phia}) is well defined.

Since, by Corollary \ref{C: wave matrix is unitary},
the operator $w_+(\lambda;H_s,H_0) \colon \mathfrak h^{(0)}_\lambda \to \mathfrak h^{(s)}_\lambda$
is unitary for all $s \notin R(\lambda; \set{H_r}, F),$ 
it follows from (\ref{F: xia(l)=int Phia}) that
$$
  \xia_{\gamma}(\lambda; H_r, H_0)
     = \int_0^r \Tr_{\mathfrak h^{(0)}_\lambda} (w_+(\lambda;H_0,H_s)\Pi_{H_s}(V)(\lambda)w_+(\lambda;H_s,H_0))\,ds.
$$
Theorem \ref{T: S = T exp...} and Lemma \ref{L: det Texp = exp Tr} now imply
$$
  -2\pi i \xia_{\gamma}(\lambda; H_r, H_0)
     = \log \det S(\lambda;H_r,H_0),
$$
where the branch of the logarithm is chosen as in the statement of the theorem.
\end{proof}

\begin{cor} \label{C: exp(-2pi i xia)=det S}
Let $F$ be a frame operator on $\hilb,$ let $\gamma = \set{H_r}_{r \in [0,1]}$ be a path of operators,
which satisfies Assumption \ref{A: assumption on Hr}.
If $\lambda \in \LambHF{H_r} \cap \LambHF{H_0},$ then
$$
  e^{-2\pi i \xia_\gamma(\lambda; H_r, H_0)} = \det S(\lambda;H_r,H_0).
$$
\end{cor}

Let $\xis_\gamma(\lambda; H_r, H_0)$ (respectively, $\xi(\lambda; H_r, H_0)$) be the density of
the absolutely continuous measure\footnote{See Remark \ref{R: mu(phi)}} $\xis_\gamma(\phi; H_r, H_0)$
(respectively, $\xi(\phi; H_r, H_0)$).
Since $V$ is trace-class, Corollary \ref{C: exp(-2pi i xia)=det S}
and the Birman-Krein formula (\ref{Int: BK formula})
$$
  e^{-2\pi i \xi(\lambda)} = \det S(\lambda;H_r,H_0) \quad \text{a.e.} \ \ \lambda \in \mbR
$$
imply the following result.
\begin{thm} \label{T: xis is int valued}
For any path of operators $\gamma = \set{H_r}_{r \in [0,1]},$
which satisfies Assumption \ref{A: assumption on Hr},
the singular part $\xis_\gamma(\lambda;H_0+V,H_0)$ of the spectral shift function
is an a.e. integer-valued function.
\end{thm}
Theorem \ref{T: xis is int valued} suggests that the singular part of the spectral shift function measures
the ``spectral flow'' of the singular spectrum regardless of its position with respect
to absolutely continuous spectrum.

The following corollary is the result mentioned in the introduction.
\begin{cor} Let $H_0$ be a self-adjoin operator and let $V$ be a self-adjoint trace-class operator.
Let $H_r = H_0 + rV.$
The density $\xis(\lambda; H_1, H_0)$ of the absolutely continuous measure
$$
  \xis_{H_1,H_0}(\phi) = \int_0^1 \Tr(V \phi(H_r^{(s)}))\,dr
$$
is a.e. integer-valued function.
\end{cor}
\begin{proof} By Lemma \ref{L: exists F for H0+rV}, for the straight line path $\set{H_r = H_0+rV, r \in [0,1]},$
which connects $H_0$ and $H_0+V,$ there exists a frame $F,$
such that Assumption \ref{A: assumption on Hr} holds.
So, Theorem \ref{T: xis is int valued} completes the proof.
\end{proof}

\subsection{Non-additivity of the singular spectral shift function}
\label{SS: xis is non-additive}
In the previous version of this paper (in arXiv) and in \cite{Az5} I mistakenly claimed that the singular part of the spectral
shift function was additive. In development of an example of a non-trivial singular spectral shift function given in \cite{Az6}
I found a contradiction. Looking for source of this contradiction I found a gap in the proof of additivity of the singular part of the spectral
shift function. Since the example was based on additivity of singular spectral shift function, the example is also wrong,
but it allows to give a counter-example to additivity of the singular spectral shift function. An example of a non-trivial singular spectral shift function
will be given in \cite{Az2011B}.

\begin{thm} \label{T: s.SSF is not additive} The singular part of the spectral shift function is not additive. That is,
there exist self-adjoint operators $H_0,H_1,H_2$ with trace class differences such that
$$
  \xis_{H_2,H_0} \neq \xis_{H_2,H_1} + \xis_{H_1,H_0},
$$
where the paths connecting the operators are assumed to be straight lines.
As a consequence, the absolutely continuous part of the spectral shift function is also not additive:
$$
  \xia_{H_2,H_0} \neq \xia_{H_2,H_1} + \xia_{H_1,H_0}.
$$
\end{thm}
The rest of this subsection is devoted to the proof of this theorem.

\begin{lemma} \label{L: Rz(D,v)} Let $\hilb$ a Hilbert space, let $v \in \hilb$ and let $D$ be a self-adjoint operator on~$\hilb.$
If
$$
  H = \left( \begin{array}{cc} D & v \\ \scal{v}{\cdot} & \alpha \end{array} \right)
$$
is a self-adjoint operator on $\hilb \oplus \mbC,$ where $\alpha \in \mbR,$ then the resolvent of $H$ is
$$
  R_z(H) = (H-z)^{-1} = \left( \begin{array}{ll} R_z(D) + A\scal{R_{\bar z}(D)v}{\cdot}R_z(D)v & \ \ -A R_z(D)v \\ -A\scal{R_{\bar z}(D)v}{\cdot} & \ \ A \end{array} \right),
$$
where $A = \brs{\alpha - z - \scal{v}{R_z(D)v}}^{-1}.$
\end{lemma}
\begin{proof} Direct calculation. \end{proof}

Note also that if $V = \scal{v}{\cdot}v,$ then (see e.g. \cite[(6.7.3)]{Ya})
\begin{equation} \label{F: resolvent for 1-dim perturbation}
  R_z(D+rV) = R_z(D) - \frac r{1+r\scal{v}{R_z(D)v}} \scal{R_{\bar z}(D)v}{\cdot}R_{z}(D)v.
\end{equation}

\begin{lemma}\label{L: pure point is empty} Let
$$
  H_{r,\alpha} := \left(\begin{array}{cc} D + r\scal{v}{\cdot}v & rv \\ r\scal{v}{\cdot} & \alpha
  \end{array}\right)
$$
be a self-adjoint operator on the Hilbert space $L_2(\mbR) \oplus \mbC,$
where $r, \alpha \in \mbR$ and
$$
  D = \frac 1i \frac d{dx} \ \ \text{and} \ \ v = \frac 1{\sqrt[4]{\pi}}e^{-\frac{x^2}{2}}.
$$
If $r\neq 0,$ then the operator $H_{r,\alpha}$ is absolutely continuous.
\end{lemma}
\begin{proof} (A) Here we show that the pure point part of $H_{r,\alpha}$ is zero.

Assume that there is a non-zero vector $\mathbf f = { f \choose f_0 } \in \hilb$
such that $H_{r,\alpha} \mathbf f = \lambda \mathbf f$ for some $\lambda \in \mbR.$
This implies that $f$ belongs to the domain of $D.$ Further, we have
$$
  H_{r,\alpha}\mathbf f = \left(\begin{array}{c} D f + r\scal{v}{f}v + rf_0 v \\ r\scal{v}{f} + \alpha f_0
  \end{array}\right) = \left(\begin{array}{c} \lambda f \\ \lambda f_0 \end{array}\right).
$$
This implies that \
$
  D f = \lambda f - r\scal{v}{f} v - r f_0 v,
$ \
so that $f' \in L_2(\mbR).$ Taking the Fourier transform of the last equality gives
$$
  \xi \hat f(\xi) = \lambda \hat f(\xi) - r\scal{v}{f} \hat v(\xi) - r f_0 \hat v(\xi).
$$
Since $\hat v = v,$ it follows that
$$
  \hat f(\xi) = -r\brs{\scal{v}{f} + f_0} \cdot \frac{v(\xi)}{\xi - \lambda}.
$$
Since $\frac{v(\xi)}{\xi - \lambda}$ is not $L_2,$ it follows that $f = 0$ and $f_0=0;$ that is, $\mathbf f=0.$

This contradiction completes the proof of (A).

(B) It is left to show that the singular continuous part of $H_{r,\alpha}$ is also empty.

Let $(\phi_j,\kappa_j)$ be a frame in $L_2(\mbR)$ consisting of, say, Hermite polynomials $\phi_j;$
numbers $\kappa_1, \kappa_2, \ldots$ can be chosen arbitrarily as long as they satisfy definition of frame.
Using Lemma \ref{L: Rz(D,v)} and (\ref{F: resolvent for 1-dim perturbation}) one can show that the set $\Lambda_0(H_{r,\alpha},F),$ given by (\ref{F: def of Lambda0}), is finite.
By Proposition \ref{P: Lambda 0 is a core of sing. sp}, it follows that the singular continuous part of the operator $H_{r,\alpha}$ is also zero.
This completes the proof.
\end{proof}

{\it Proof of Theorem \ref{T: s.SSF is not additive}.} \
In notation of Lemma \ref{L: pure point is empty}, let $H_0 = H_{0,-1}, \ H_1 = H_{1,0} \ \text{and} \ H_2 = H_{0,1}.$
It is easy to see that \ $\xis_{H_2,H_0} = \chi_{[-1,1]}.$
At the same time, by Lemma \ref{L: pure point is empty}, operators, which connect $H_0$ with $H_1$ and $H_1$ with $H_2,$
have zero singular parts. Hence, $\xis_{H_2,H_1} = \xis_{H_1,H_0} = 0.$
$\Box$

This proof also shows that the pure point part of the spectral shift function (defined in an obvious way) is also not additive.

\section{Pushnitski $\mu$-invariant and singular spectral shift function}
\label{S: Push...}
Though Theorem \ref{T: xis is int valued} shows that $\xis(\lambda)$ is an a.e. integer-valued,
it leaves a feeling of dissatisfaction, since the set of full measure, on which $\xis$ is defined,
is not explicitly indicated.

In fact, it is possible to give another proof of Theorem \ref{T: xis is int valued}, which uses
a natural decomposition of Pushnitski $\mu$-invariant $\mu(\theta,\lambda)$ (cf.~\cite{Pu01FA}, cf. also~\cite{Az2})
into absolutely continuous $\mua(\theta,\lambda)$ and singular $\mus(\theta,\lambda)$ parts,
so that the Birman-Krein formula becomes a corollary of this result and Theorem \ref{T: xia = -2pi i log det S}, rather than the other way.

In this section it will be shown that $\mus(\theta,\lambda)$ does not depend on the angle variable $\theta$
and coincides with $-\xis(\lambda).$ Since the $\mu$-invariant is integer-valued (it measures the spectral flow
of partial scattering phase shifts), it follows that $\xis(\lambda)$ is integer-valued.
The invariants $\mu(\theta,\lambda),$ $\mua(\theta,\lambda)$ and $\mus(\theta,\lambda)$
can be explicitly defined on $\LambHF{H_r} \cap \LambHF{H_0}.$

\subsection{Spectral flow for unitary operators}
Spectral flow for unitary operators of the class $1 + \clL_\infty(\hilb)$
was studied in \cite{Pu01FA}. Here we suggest a different approach to the definition of spectral flow
for unitary operators. It is based on the following intuitively obvious theorem, proof of which is nevertheless
lengthy and tedious.

We denote by $\set{a,b,\ldots}^*$ sets in which elements may appear more than once, and these multiple appearances are counted,
so that, say, $\set{7,7}^*\neq \set{7}^*,$ unlike usual sets. We call such sets rigged sets.

For $p \in [1,\infty],$ let
$$
  \clU_p(\hilb) = \set{U \in 1 + \clL_p(\hilb) \colon U \ \text{is unitary}}
$$
with the topology of convergence in $\clL_p(\hilb)$-norm.

\begin{thm} \label{T: selection thm} Let $-\infty \leq a < b \leq +\infty.$ Let
$$
  U \colon [a,b] \to \clU_1(\hilb)
$$
be a continuous path of unitary operators. There exists a sequence of continuous functions
$$
  \theta_j \colon [a,b] \to \mbR, \ j = 1,2,\ldots,
$$
such that for any $r \in [a,b]$ the rigged set
\begin{equation} \label{F: theta1...}
  \set{e^{i\theta_1(r)}, \ e^{i\theta_2(r)}, \ \ldots}^*
\end{equation}
coincides with the spectrum of $U(r)$ (counting multiplicities), excepting possibly the point~$1.$
\end{thm}
\begin{defn}
The \emph{$\mu$-invariant} of the path $U$ is a function
$$
  \mu(\theta; U) \colon (0,2\pi) \to \mbZ \cup \frac 12 \mbZ
$$
defined by the formula
$$
  \mu(\theta; U) = \sum_{j=1}^\infty [\theta; \theta_j(a),\theta_j(b)],
$$
where
$$
  [\theta; \theta_1, \theta_2] = \frac 12 \Brs{\#\set{k \in \mbZ \colon \theta_1 < \theta+2\pi k < \theta_2} + \#\set{k \in \mbZ \colon \theta_1 \leq \theta+2\pi k \leq \theta_2}}.
$$
\end{defn}
When $U(a) = 1,$ so that $\theta_j(a) = 0$ for all $j=1,2,\ldots,$ the last formula can be written as
$$
  \mu(\theta; U) = - \sum_{j=1}^\infty \SqBrs{\frac{\theta - \theta_j(b)}{2\pi}}.
$$

The $\mu$-invariant $\mu(\theta, U)$ counts the number of times eigenvalues of $U(r)$
cross a point $e^{i\theta} \in \mbT$ in anticlockwise direction
as $r$ moves from $a$ to $b.$ In other words, $\mu$-invariant is the spectral flow
of a path of unitary operators.

\begin{thm} \label{T: properties of mu} (1) The $\mu$-invariant is correctly defined, that is, it does not depend on the choice of continuous enumeration from Theorem \ref{T: selection thm}.
\\ (2) The $\mu$-invariant is homotopically invariant: if two continuous paths $U_1,U_2 \colon [a,b] \to 1+\clL_1(\hilb)$
of unitary operators with the same end-points
are homotopic, then $\mu(\theta; U_1) = \mu(\theta; U_2)$ for all $\theta \in (0,2\pi).$
\\ (3) If two continuous paths $U_1 \colon [a,b] \to \clU_1(\hilb)$
and $U_2 \colon [a,b] \to \clU_1(\clK)$
are such that $U_1(a) = 1_\hilb$ and $U_2(a) = 1_\clK$ and spectra of $U_1(b)$ and $U_2(b)$
coincide (counting multiplicities, and possibly excepting $1$),
then the difference
$$
  \mu(\theta; U_1) - \mu(\theta; U_2)
$$
is constant (does not depend on $\theta$).
\\ (4) If $U(a) = 1,$ then the equality
$$
  \int_0^{2\pi} \mu(\theta; U)\,d\theta = \sum_{j=1}^\infty \theta_j(b).
$$
holds. In particular, the right hand side of this equality does not depend on the choice
of continuous enumeration (\ref{F: theta1...}).
\end{thm}

We introduce the following definition.
\begin{defn} \label{D: def of xi-invariant} Let
$$
  U \colon [a,b] \to \clU_1(\hilb)
$$
be a continuous path of unitary operators, such that U(a) = 1.
The $\xi$-invariant of this path is the number
$$
  \xi(U) = - \frac 1 {2\pi} \int_0^{2\pi} \mu(\theta; U)\,d\theta = - \frac 1 {2\pi} \sum_{j=1}^\infty \theta_j(b).
$$
\end{defn}
It follows from Theorem \ref{T: properties of mu} that $\xi$-invariant is homotopically invariant.
\begin{prop} \label{P: xi(U) is cont-s} Let $U$ be as in Theorem \ref{T: selection thm}. The function
$$
  [a,b] \ni r \mapsto \xi(U_r) \in \clU_1(\hilb)
$$
is continuous, where $U_r$ is the restriction of the path $U$ to the interval $[a,r].$
\end{prop}

Proofs of Theorems \ref{T: selection thm} and \ref{T: properties of mu} and of Proposition \ref{P: xi(U) is cont-s} can be found in \cite{Az3}.

\subsection{Absolutely continuous part of Pushnitski $\mu$-invariant}
Let $\lambda \in \LambHF{H_0}.$ Let $\gamma = \set{H_r}$ be a path of operators which satisfies Assumption \ref{A: assumption on Hr}.

   We denote by
   $$e^{i\theta^*_1(\lambda,r)}, \ e^{i\theta^*_2(\lambda,r)}, \ e^{i\theta^*_3(\lambda,r)}, \ \ldots \in \mbT, $$ the eigenvalues of the scattering matrix
  ~$S(\lambda;H_r,H_0).$ Since, by Proposition \ref{P: S(r) is continuous},
   the scattering matrix~$S(\lambda;H_r,H_0)$ is a meromorphic function, which is analytic for real $r$'s,
   by Theorem \ref{T: selection thm} for a given path $\set{H_r}$
   the arguments
   $$
     \theta^*_1(\lambda,r),\theta^*_2(\lambda,r), \theta^*_3(\lambda,r), \ldots
   $$ may and will be chosen to be continuous functions of $r,$
   such that $\theta^*_j(\lambda,0) = 0.$

   \begin{defn} \label{D: a.c. Push. inv-t}
   Let $\lambda \in \Lambda(H_0,F).$ Let $\gamma = \set{H_r}$ be a path of operators which satisfies Assumption \ref{A: assumption on Hr}.
   The absolutely continuous part $\mua(\theta,\lambda; H_r, H_0)$ of Pushnitski $\mu$-invariant
   is the $\mu$-invariant of the path
   \begin{equation} \label{F: r to S(r)}
     [0,1] \ni r \mapsto S(\lambda; H_r,H_0).
   \end{equation}
   \end{defn}
   In other words,
   \begin{equation} \label{F: def of a.c. mu}
     [0,2\pi)\times \LambHF{H_0} \ni (\theta,\lambda) \mapsto \mua(\theta,\lambda; H_r, H_0)
       = - \sum_{j=1}^\infty \SqBrs{\frac{\theta - \theta^*_j(\lambda,r)}{2\pi}}.
   \end{equation}

Recall that the vector space $\clA(F)$ is defined in (\ref{F: clA(F)}).
\begin{thm}
Let $\gamma = \set{H_r}$ be a path of operators which satisfies Assumption \ref{A: assumption on Hr}.
For every $\lambda \in \LambHF{H_0}$ the equality
\begin{equation} \label{F: xia = 1/2pi int}
  \xia_\gamma(\lambda; H_1,H_0) = - \frac 1{2\pi} \int_0^{2\pi} \mua(\theta,\lambda; H_1,H_0)\,d\theta
\end{equation}
holds, where $H_1=H_0+V.$ That is, $\xia_\gamma(\lambda; H_1,H_0)$ is equal to the $\xi$-invariant of the path (\ref{F: r to S(r)}).
\end{thm}
Recall that $\xia_\gamma(\lambda; H_1,H_0)$ is defined by the formula (\ref{F: def of xia}).
\begin{proof} By the Lidskii theorem (see (\ref{F: Lidskii for det}))
\begin{equation*}
  \det S(\lambda;H_r,H_0) = \prod\limits_{j=1}^\infty e^{i\theta^*_j(\lambda,r)} = \exp\brs{i\sum\limits_{j=1}^\infty \theta_j^*(\lambda,r)}.
\end{equation*}
It follows from this, Proposition \ref{P: xi(U) is cont-s} and Theorem \ref{T: properties of mu}(4) that
$$
  -\frac 1{2\pi i } \log \det S(\lambda; H_r, H_0) = - \frac 1{2\pi} \sum\limits_{j=1}^\infty \theta_j^*(\lambda,r) = - \frac 1{2\pi} \int_0^{2\pi} \mua(\theta,\lambda; H_r,H_0)\,d\theta,
$$
where all functions of $r$ are continuous.
Now, Theorem \ref{T: xia = -2pi i log det S} completes the proof.
\end{proof}

\subsection{Pushnitski $\mu$-invariant}
Let $H_0$ be a self-adjoint operator on $\hilb,$
let $F$ be a frame operator on $\hilb,$ let
$$
  V_r \in F^*J_rF,
$$
where $J_r \in \clB(\clK),$
and let $\set{H_r=H_0+V_r}$ satisfy Assumption \ref{A: assumption on Hr}. 

Let $z \in \mbC,$ $\Im z > 0.$ Following \cite{Pu01FA}, we define the
$\tilde S$-function by the formula
\begin{equation} \label{F: def of tilde S(z,r)}
 \begin{split}
   \tilde S(z,r) & = \tilde S(z; H_r,H_0;F)
     \\ & = 1 - 2i \sqrt{\Im T_z(H_0)} J_r (1 + T_z(H_0)J_r)^{-1} \sqrt{\Im T_z(H_0)} \ \ \in \ 1 + \clL_1(\clK),
 \end{split}
\end{equation}
where
$$
  T_z(H_0) = F R_z(H_0) F^*.
$$
It is not difficult to verify that $\tilde S(z;
H_r,H_0;F)$ is a unitary operator, so that
$$
  \tilde S(z; H_r,H_0;F) \in \clU_1(\clK).
$$

\begin{lemma} \label{L: tilde S(z,r) is L1 cont-s}
Let $\set{H_r}$ be a path which satisfies Assumption \ref{A: assumption on Hr}
The function
$$
  (z,r) \mapsto \tilde S(z; H_r,H_0;F)
$$
is $\clL_1$-continuous on $\mbC_+\times \mbR.$
\end{lemma}
\begin{proof}
   By Lemma \ref{L: 1+rJT(z) is invertible}, 
   the operator $1+J_r T_z(H_0)$ is invertible on $\mbC_+\times \mbR.$
   Since, by Lemma \ref{L: lemma Y}, 
   $\sqrt{\Im T_z(H_0)}$ is $\clL_2$-continuous,
   it follows from H\"older's inequality (\ref{F: Holder inequality}) that $\tilde S(z; H_r,H_0;F)$ is $\clL_1$-continuous $\mbC_+\times \mbR.$
\end{proof}
\begin{lemma} \label{L: tilde S(l+i0,r) exists}
If $\lambda \in \Lambda(H_0,F) \cap \Lambda(H_r,F),$ then the limit
\begin{equation} \label{F: tilde S(l+i0,r)}
 \begin{split}
   \tilde S(\lambda+i0,r) & = \tilde S(\lambda+i0; H_r,H_0;F)
     \\ & = 1 - 2i \sqrt{\Im T_{\lambda+i0}(H_0)} J_r (1 + T_{\lambda+i0}(H_0)J_r)^{-1} \sqrt{\Im T_{\lambda+i0}(H_0)} \ \ \in \ \clU_1(\clK)
 \end{split}
\end{equation}
exists in $\clL_1(\clK)$-norm.
\end{lemma}
\begin{proof} Since $\lambda \in \Lambda(H_0,F),$ the limit $\Im T_{\lambda+i0}(H_0)$ exists in $\clL_1(\hilb).$
By Lemma \ref{L: lemma Y}, the limit $\sqrt{\Im T_{\lambda+i0}(H_0)}$ exists
in $\clL_2(\clK)$-norm and from $\lambda \in \Lambda(H_r,F)$ it follows that the operator $(1 + T_{\lambda+i0}(H_0)J_r)^{-1}$
is invertible. So, again H\"older's inequality (\ref{F: Holder inequality}) completes the proof.
\end{proof}

When $y \to +\infty,$ the operator $\tilde S(\lambda+iy,r)$ goes to $1.$
So, we have a continuous (in fact, real-analytic) path of unitary operators in $\clU_1(\clK):$
\begin{equation} \label{F: tilde S on [0,infty]}
  [-\infty,0] \ni y \to \tilde S(\lambda-iy; H_r,H_0) \in \clU_1(\clK).
\end{equation}
\begin{defn} \label{D: Push. inv-t}
Let $\lambda \in \LambHF{H_0}\cap \LambHF{H_r}.$ \emph{Pushnitski $\mu$-invariant}
of the pair $(H_0,H_r)$ is the $\mu$-invariant of the continuous path (\ref{F: tilde S on [0,infty]}).
\end{defn}
Pushnitski $\mu$-invariant will be denoted by $\mu(\theta,\lambda; H_r,H_0).$

\subsection{$M$-function}
Let $z \in \mbC \setminus \mbR$ and let $H_0,H_1$ be two self-adjoint operators on $\hilb$
with bounded difference $V = H_1-H_0.$
Following \cite[(4.1)]{Pu01FA}, we define the
$M$-function by the formula
\begin{equation} 
    M(z; H_1, H_0) = \brs{H_1 - \bar z} R_z(H_1) \brs{H_0 - z} R_{\bar z}(H_0) \in \clB(\hilb).
\end{equation}
The $M$-function can be considered as a product of the Cayley
transforms of operators $H_1$ and $H_0,$ and its values are
unitary operators.

Let $\gamma = \set{H_r},$ $H_r = H_0 + V_r,$ be a continuous piecewise real-analytic path.

Evidently, the multiplicative property
\begin{equation} \label{F: M is multiplicative}
  M(z; H_{r_2}, H_{r_0}) = M(z; H_{r_2}, H_{r_1})M(z; H_{r_1}, H_{r_0})
\end{equation}
holds.

One can also easily check that (see \cite[(4.4)]{Pu01FA})
\begin{equation} \label{F: M(z)=1-2pi iyrRVR}
  M(z; H_r, H_0) = 1 -2iy R_z(H_r)V_rR_{\bar z}(H_0).
\end{equation}
This equality, the estimate $\norm{R_z(H)} \leq \frac 1{\abs{\Im
z}}$ and the norm continuity of the function $\mbC_+\times \mbR
\ni (z,r) \mapsto R_z(H_r)$ imply the following lemma.
\begin{lemma} \label{L: M(z) is continuous}
(i) The function
$$
  (z,r) \in \mbC_+ \times \mbR \mapsto M(z; H_r, H_0)
$$
takes values in $1+\LpH{1}$ and is continuous in $\LpH{1}$\tire norm.

(ii) 
When $y \to +\infty$
$$
  \norm{M(\lambda+iy,H_r,H_0) - 1}_1 \to 0
$$
locally uniformly with respect to $r \in \mbR.$
\end{lemma}
Indeed, it follows from (\ref{F: M(z)=1-2pi iyrRVR}) that
$$
  \norm{M(\lambda+iy,H_r,H_0) - 1}_1 \leq \frac {2\norm{V_r}_1}y.
$$
Since $\norm{V_r}_1$ is locally bounded, the claim follows.

\begin{thm} \label{T: Push. Thm 4.1}
\cite[Theorem 4.1]{Pu01FA}
Spectral measures of operators $M(z;H_r,H_0)$ and
$\tS(z; H_r,H_0;F)$ coincide.
\end{thm}
This proposition means that in the definition of Pushnitski $\mu$-invariant one can replace $\tilde S$-function
by $M$-function.
\begin{cor} \label{C: mu-inv-t is the same} The $\mu$-invariant (and, consequently, the $\xi$-invariant as well) of the path
\begin{equation} \label{F: M on [0,infty]}
  [-\infty,0] \ni y \to M(\lambda-iy; H_r,H_0) \in \clU_1(\clK)
\end{equation}
coincides with the $\mu$-invariant (respectively, the $\xi$-invariant) of the path (\ref{F: tilde S on [0,infty]}).
The assertion holds also, if the interval $[-\infty,0]$ is replaced by $[-\infty,y_0]$ with $y_0>0.$
\end{cor}
\begin{proof} This follows immediately from Theorem \ref{T: Push. Thm 4.1} and Lemmas \ref{L: M(z) is continuous} and \ref{L: tilde S(z,r) is L1 cont-s}.
\end{proof}

The following formula is taken from \cite{Az2}.
\begin{thm} \label{T: M(z,r)=Texp} Let $\set{H_r}$ be a continuous piecewise analytic path of operators and let $z \in \mbC_+.$
The formula
\begin{equation} \label{F: M(z,r)=Texp}
  M(z; H_r, H_{r_0}) = \Texp\brs{-2iy \int_{r_0}^r R_z(H_s) \dot V_r R_{\bar z} (H_s)\,ds},
\end{equation}
holds.
\end{thm}
\begin{proof}
It follows from  (\ref{F: M(z)=1-2pi iyrRVR}) that in $\LpH{1}$
$$
  \frac d{dr} M(z; H_r, H_s)\Big|_{r=s} = -2iy R_z(H_s) \dot V_r R_{\bar z} (H_s).
$$
This equality and the multiplicative property (\ref{F: M is multiplicative}) of the $M$-function
imply that
$$
  \frac d{dr} M(z; H_r, H_s) = -2iy R_z(H_r) \dot V_r R_{\bar z} (H_r) M(z; H_r, H_s).
$$
Combining this with Lemma \ref{L: T exp} we obtain (\ref{F:
M(z,r)=Texp}).
\end{proof}

\subsection{Smoothed spectral shift function}
We define the smoothed spectral shift function $\xi(\lambda+iy; H_1, H_0)$ of a pair of operators $H_1$ and $H_0$
as the $\xi$-invariant of the path
\begin{equation} \label{F: tilde S on [y0,infty]}
  [-\infty,-y] \ni \tilde y \mapsto \tilde S(\lambda-i\tilde y; H_1, H_0) \in \clU_1(\clK).
\end{equation}
This means by definition that ($z \in \mbC_+$)
\begin{equation} \label{F: def of xi(z)}
  \xi(z; H_1, H_0) = - \frac 1{2\pi}\sum_{j=1}^\infty \theta_j(z) = - \frac 1{2\pi i} \log \det \tilde S(z; H_1,H_0),
\end{equation}
where functions
$$
  \theta_1(z,r), \ \theta_2(z,r), \ \theta_3(z,r), \ \ldots
$$ are chosen as in Theorem \ref{T: selection thm} for the continuous path (\ref{F: tilde S on [y0,infty]}).

\begin{prop} \label{P: unused prop} Let $\set{H_r}$ be a continuous path which connects $H_0$ and $H_1.$
The smoothed spectral shift function $\xi(z; H_1, H_0)$ is equal to the $\xi$-invariant of the path
$$
   [0,1] \ni r \mapsto M(z; H_r,H_0).
$$
\end{prop}
\begin{proof} Let $z_0 = \lambda + iy_0$ and let $y_0<y_1.$ Consider a path which connects $M(\lambda+iy_0; H_r,H_0)$ with 1 and
which consists of two arcs: the first arc connects $M(\lambda+iy_0; H_r,H_0)$ with $M(\lambda+iy_1; H_r,H_0)$ as $y$ changes from $y_0$ to $y_1,$
and the second arc connects $M(\lambda+iy_1; H_r,H_0)$ with $1$ as $r$ changes from $1$ to $0.$
Now we let $y_1$ to move from $y_0$ to $+\infty.$ It follows from Lemma \ref{L: M(z) is continuous} that this gives a homotopy of the two paths
connecting $M(\lambda+iy_0; H_r,H_0)$ with the identity operator, where in the first path $y$ goes from $y_0$ to $+\infty$ and in the second path
$r$ goes from $1$ to $0.$ It follows from Theorem \ref{T: properties of mu} that the $\xi$-invariants of these two paths coincide.
\end{proof}
\begin{lemma} \label{L: xi(l+i0) exists} If $\lambda \in \Lambda(H_0,F) \cap \Lambda(H_1,F),$ then the limit
$$
  \xi(\lambda+i0; H_1, H_0) := \lim_{y \to 0^+} \xi(\lambda+iy; H_1, H_0)
$$
exists and
\begin{equation} \label{F: xi(l+i0) is int mu}
  \xi(\lambda+i0; H_1, H_0) = -\frac 1{2\pi} \int_0^{2\pi} \mu(\theta, \lambda; H_1,H_0)\,d\theta.
\end{equation}
\end{lemma}
\begin{proof} Existence of $\xi(\lambda+i0)$ follows from Lemma \ref{L: tilde S(l+i0,r) exists} and Proposition \ref{P: xi(U) is cont-s}.
Proposition \ref{P: xi(U) is cont-s} also implies that $\xi(\lambda+i0)$ is the $\xi$-invariant of the path (\ref{F: tilde S on [0,infty]}),
so that equality (\ref{F: xi(l+i0) is int mu}) holds by Definition \ref{D: Push. inv-t} of the $\mu$-invariant and Definition \ref{D: def of xi-invariant}
of the $\xi$-invariant.
\end{proof}

Theorem \ref{T: Push. Thm 4.1} and the proof of Proposition \ref{P: unused prop} imply the following
\begin{cor} \label{C: xi-inv-t for two piece path} Let $\lambda \in \Lambda(H_0,F) \cap \Lambda(H_1,F).$ The number $\xi(\lambda+i0; H_1,H_0)$ is the $\xi$-invariant
of a continuous path of unitary operators which consists of the following two pieces:
$$
  [0,1] \ni r \mapsto \tilde S(\lambda+iy_0; H_r,H_0)
$$
and
$$[-y_0,0] \ni y \mapsto \tilde S(\lambda-iy; H_1,H_0).$$
\end{cor}
Meaning of this corollary is simple: we cannot directly connect the unitary operator $\tilde S(\lambda+i0; H_1,H_0)$ with the identity
operator by sending $r$ from $1$ to $0$ because of possible resonance points in $[0,1],$
but we can do this after shifting the point $\lambda+i0$ out of the real axis.

\begin{prop} \label{P: xi(l+iy)=int Fz(s)ds} The formula
\begin{equation} \label{F: xi(l+iy)=int Fz(s)ds}
    \xi(\lambda+iy; H_1, H_0) = \int_{0}^1 \Tr\SqBrs{V \frac 1 \pi \Im R_{\lambda+iy}(H_s)} \,ds
\end{equation}
holds.
\end{prop}
\begin{proof}
It follows from (\ref{F: def of xi(z)}), Corollary \ref{C: mu-inv-t is the same}, Theorem \ref{T: M(z,r)=Texp} and Lemma \ref{L: det Texp = exp Tr} that
\begin{equation*} 
  \begin{split}
  \xi(\lambda+iy; H_1, H_0) & = - \frac 1{2\pi i} \log \det M(\lambda+iy; H_1, H_0)
  \\ & = \frac y\pi \int_{0}^1 \Tr(R_{\lambda+iy}(H_s) V R_{\lambda-iy}(H_s)) \,ds
  \\ & = \int_{0}^1 \Tr\SqBrs{V \frac 1 \pi \Im R_{\lambda+iy}(H_s)} \,ds.
  \end{split}
\end{equation*}
\end{proof}

\subsection{Pushnitski formula}
The following theorem was proved in \cite{Pu01FA}. Here we give another simpler proof of this formula.
This proof follows that of from \cite{Az2}.
\begin{thm} (Pushnitski formula) \label{T: xi = -average of mu} For a.e. $\lambda \in \LambHF{H_0}\cap \LambHF{H_r}$ the equality
$$
  \xi(\lambda; H_r,H_0) = - \frac 1{2\pi} \int_0^{2\pi} \mu(\theta,\lambda; H_r,H_0)\,d\theta
$$
holds.
\end{thm}
\begin{proof} By Lemma \ref{L: xi(l+i0) exists}, it is enough to show that for a.e. $\lambda \in \LambHF{H_0}\cap \LambHF{H_r}$
\begin{equation} \label{F: xi(l+i0)=xi(l)}
  \xi(\lambda+i0; H_1, H_0) = \xi(\lambda; H_1, H_0).
\end{equation}

The trace in the right hand side of (\ref{F: xi(l+iy)=int Fz(s)ds}) is the Poisson integral of the measure $\Delta \mapsto \Tr\brs{VE^{H_s}_{\Delta}}.$
It follows from (\ref{F: xi(l+iy)=int Fz(s)ds}) and Fubini's theorem (see e.g. \cite[VI.2]{Ja} or \cite[Lemma 2.4]{ACS}) that
$\xi(\lambda+iy; H_1, H_0)$ is the Poisson integral of the measure
$$
  \Delta \mapsto \int_0^1 \Tr(V E^{H_s}_\Delta)\,ds,
$$
which is the absolutely continuous spectral shift measure $\xi$ (see (\ref{F: def of xi(phi)})).
Hence, by Theorem \ref{T: Fatou}, for a.e. $\lambda \in \mbR$ (\ref{F: xi(l+i0)=xi(l)}) holds.
\end{proof}

This theorem allows to define explicitly the spectral shift function on the full set $\LambHF{H_0}.$
   \begin{defn} \label{D: def of xi} Let $\lambda \in \LambHF{H_0} \cap \LambHF{H_1}.$
     The Lifshits-Krein spectral shift function $\xi(\lambda)$ is by definition
     $$
       \xi(\lambda; H_1,H_0) = - \frac 1{2\pi} \int_0^{2\pi} \mu(\theta,\lambda; H_1,H_0)\,d\theta.
     $$
   \end{defn}
   In other words, $\xi(\lambda)$ is the $\xi$-invariant of the path (\ref{F: tilde S on [0,infty]}).
   The advantage of this definition of the spectral shift function
   is that it gives explicit values of $\xi$ on an explicit set of full Lebesgue measure.

\begin{rems} \rm The functions $\xi$ and $\xia$ are summable. As such one can consider full sets $\Lambda(\xi)$
and $\Lambda(\xia)$ and standard values of $\xi$ and $\xia$ on these sets. However, the above definitions of $\xi(\lambda)$
and $\xia(\lambda)$ and of the corresponding sets of full Lebesgue measure differ from the standard definition of $f(\lambda)$
for a general summable function $f.$ In particular, it may be that $\xi(\lambda) \neq 0$ at some regular point $\lambda,$ while $\xi = 0$
as an element of $L_1(\mbR).$
\end{rems}

It is known that $\xi$ is additive in the sense that $\xi(\lambda; H_2,H_0) = \xi(\lambda; H_2,H_1) + \xi(\lambda; H_1,H_0)$
for a.e. $\lambda \in \mbR.$ Definition \ref{D: def of xi} poses a question of whether this equality holds for \emph{every}
$\lambda$ from the full set $\LambHF{H_0} \cap \LambHF{H_1} \cap \LambHF{H_2}.$ The answer is affirmative.

\begin{lemma} \label{L: xi(VU)=xi(U)xi(V)} If $U,V \colon [a,b] \to \clU_1(\hilb)$ are two continuous paths such that $U(a) = V(a) = 1,$
then
$$
  \xi(UV) = \xi(U) + \xi(V).
$$
\end{lemma}
\begin{proof} By (\ref{F: det(AB)=det(A)det(B)}), for every $r \in [a,b]$
$$
  \det(U(r) V(r)) = \det(U(r))\det(V(r)).
$$
Also, by Proposition \ref{P: xi(U) is cont-s} and (\ref{F: det is cont-s}), the equality
$$
  \det(U(r)) = e^{-2\pi i \xi(U(r))}
$$
holds. It follows that
$$
  \xi(U(r) V(r)) = \xi(U(r)) + \xi(V(r)) \mod \mbZ.
$$
Since both sides are continuous function of $r$ and since $\xi(U(0) V(0)) = \xi(U(0)) + \xi(V(0)) = 0,$ the claim follows.
\end{proof}
\begin{thm} \label{T: xi(l) is pointwise additive} Let $H_1, H_2 \in H_0 + \clA(F).$ For every $\lambda \in \LambHF{H_0} \cap \LambHF{H_1} \cap \LambHF{H_2}$ the equality
$$\xi(\lambda; H_2,H_0) = \xi(\lambda; H_2,H_1) + \xi(\lambda; H_1,H_0)$$ holds.
\end{thm}
\begin{proof} 
Since the $M$-function is multiplicative (\ref{F: M is multiplicative}), the claim follows from Lemma \ref{L: xi(VU)=xi(U)xi(V)},
and Corollary \ref{C: mu-inv-t is the same}.
\end{proof}

\subsection{Singular part of $\mu$-invariant}
Let $\gamma = \set{H_r, r \in [0,1]}$ be a continuous piecewise analytic path of operators which satisfy Assumption \ref{A: assumption on Hr}.
Let $\lambda \in \Lambda(H_0,F) \cap \Lambda(H_r,F).$
\begin{defn} The singular part of Pushnitski $\mu$-invariant is the function
$$
  \mus_\gamma(\theta, \lambda; H_r, H_0) := \mu(\theta, \lambda; H_r, H_0) - \mua_\gamma(\theta, \lambda; H_r, H_0).
$$
\end{defn}
Note that while $\mu(\theta, \lambda; H_r, H_0)$ and $\mua_\gamma(\theta, \lambda; H_r, H_0)$ are $\mu$-invariants of some paths of unitary operators,
the singular part $\mus_\gamma(\theta, \lambda; H_r, H_0)$ of $\mu$-invariant is not.

Also, for every $\lambda \in \Lambda(H_0,F) \cap \Lambda(H_r,F)$ we define the standard density $\xis_\gamma(\lambda)$
of the singular part of the spectral shift function $\xis_\gamma$ by the formula
$$
  \xis_\gamma(\lambda; H_r,H_0) = \xi(\lambda; H_r,H_0) - \xia_\gamma(\lambda; H_r,H_0),
$$
where $\xi(\lambda; H_r,H_0)$ is defined by (\ref{D: def of xi}) and $\xia_\gamma(\lambda; H_r,H_0)$ is defined by (\ref{F: xia(l)=int Phia}).

\begin{lemma} \label{L: eig of tS = eig of S} \cite{Pu01FA}
Let $\lambda \in \Lambda(H_0,F) \cap \Lambda(H_r,F).$
The eigenvalues of $\tilde S(\lambda+i0;H_r,H_0)$ coincide
with the eigenvalues of the scattering matrix $S(\lambda;H_r,H_0)$
(counting multiplicities); that is, spectral measures of these operators coincide.
\end{lemma}
\begin{proof}
The stationary formula for the scattering matrix (Theorem \ref{T: stationary rep-n for SM}),
definition (\ref{F: def of tilde S(z,r)}) of $\tilde S(\lambda+i0; H_1,H_0)$
and the equality (\ref{F: euE diam euE = Im R}), combined with (\ref{F: spec mes(AB)=spec mes(BA)}),
imply that the spectra of operators $S(\lambda; H_r,H_0)$ and $\tilde S (\lambda; H_r,H_0;F)$ coincide
counting multiplicities.
\end{proof}
%
\begin{thm} \label{T: mus(theta)=const}
Let $\gamma = \set{H_r, r \in [0,1]}$ be a continuous piecewise analytic path of operators which satisfy Assumption \ref{A: assumption on Hr}.
  The singular part of Pushnitski $\mu$-invariant $\mus_\gamma(\theta,\lambda)$
  does not depend on the angle variable $\theta.$ Thus defined function of the variable $\lambda$
  is equal to minus the density $\xis_\gamma(\lambda)$ of the singular part of the spectral shift function. That is,
  for all $\lambda \in \Lambda(H_0,F)$ and for all $r \notin R(\lambda;\set{H_r},F),$
\begin{equation} \label{F: xis=-mus}
  \xis_\gamma(\lambda; H_r, H_0) = - \mus_\gamma(\lambda; H_r, H_0).
\end{equation}
Consequently, the singular part of the spectral shift function $\xis_\gamma(\lambda)$ is integer-valued.
\end{thm}
\begin{proof}
It follows from Lemma \ref{L: eig of tS = eig of S} and Theorem \ref{T: properties of mu}(3) that the singular part of the $\mu$-invariant does not depend on $\theta.$
The equality (\ref{F: xis=-mus}) now follows from (\ref{F: xia = 1/2pi int}) and Definition \ref{D: def of xi}.
\end{proof}

The last theorem deserves some comment.
The scattering matrix $S(\lambda; H_r,H_0)$ for $\lambda \in \LambHF{H_0} \cap \LambHF{H_r}$ is a unitary operator of the class
$1 + \clL_1(\hlambda).$ So, the spectrum of $S(\lambda; H_r,H_0)$ is a discrete subset of the unit circle $\mbT$ with only one possible accumulation point at $1.$
The eigenvalues of $S(\lambda; H_r,H_0)$ (called scattering phases)
can be send to $1$ in two essentially different ways. The first way is to connect $S(\lambda; H_r,H_0)$  with the identity operator
by letting the coupling constant $r$ move from $1$ to $0.$ This is possible to do, since $S(\lambda; H_r,H_0)$
is continuous for all $r \in \mbR.$ Now, the operators $S(\lambda; H_r,H_0)$ and $\tilde S(\lambda+i0; H_r,H_0)$ have the same eigenvalues
(counting multiplicities).
So, the second way to send scattering phases to $1$ is to move $y$ from $0$ to $+\infty$ in $\tilde S.$ In both ways,
scattering phases go to $1$ continuously. Nevertheless, it is possible that these two ways
are not homotopic; that is, some eigenvalue can make a different number of windings around the unit circle
as it is sent to $1.$ Pushnitski $\mu$-invariant $\mu(\theta,\lambda;H_r,H_0)$ and its absolutely continuous part
$\mua_\gamma(\theta,\lambda;H_r,H_0)$ measure the spectral flow of the scattering phases through $e^{i\theta}$
in two different ways, corresponding to the above mentioned two ways of connecting the scattering phases
with $1,$ and the difference $\mu(\theta,\lambda;H_r,H_0) - \mua_\gamma(\theta,\lambda;H_r,H_0)$ does not depend on $\theta.$
This difference measures the difference of winding numbers.

Combined with Corollary \ref{C: exp(-2pi i xia)=det S}, Theorem \ref{T: mus(theta)=const} gives a proof of
\begin{thm} {\rm (Birman-\Krein\ formula)} \ Let~$H_0$ be a self-adjoint operator and $V$ be a trace-class self-adjoint operator.
Then for a.e. $\lambda \in \mbR$
$$
  e^{-2\pi i \xi(\lambda;H_1,H_0)} = \det S(\lambda;H_1,H_0),
$$
where~$H_1 = H_0 + V.$
\end{thm}
This theorem holds for all $\lambda$ from the set of full Lebesgue measure $\Lambda(H_0,F) \cap \Lambda(H_1,F),$ provided that there is some fixed frame $F$ such that
$V \in \clA(F).$ By Lemma \ref{L: exists F for H0+rV}, for every trace-class operator $V$ such a frame exists, and consequently, the Birman-\Krein\ formula
holds for a.e. $\lambda \in \mbR.$ Also, in this theorem the scattering matrix $S(\lambda; H_1, H_0)$ is defined by the formula (\ref{F: def-n of SM}),
but as is shown in Section \ref{S: connection with time...}, this definition coincides with classical definition of the scattering matrix via the direct integral decomposition
of the scattering operator.

\bigskip

Let $\lambda$ be a fixed essentially regular point. We consider the singular spectral shift function
$\xis(r) = \xis(\lambda; H_r, H_0)$ as a function of $r.$
Theorem \ref{T: mus(theta)=const} tells us that $\xis(r)$ is an integer number. It turns out that $\xis(r)$
is a locally constant function, and it can jump only at resonance points of the path $\set{H_r}.$
In the rest of this section we prove this assertion.

\begin{lemma} \label{L: tilde S converge uniformly}
If $\lambda \in \LambHF{H_0},$ then $\tilde S(\lambda+iy; H_r,H_0)$
converges to $\tilde S(\lambda+i0; H_r,H_0)$ in $\LpH{1}$ locally
uniformly with respect to $r$ outside of the resonance set $R(\lambda; H_0,V;F)$ as $y
\to 0.$
\end{lemma}
\begin{proof} (A) If $I$ is a closed interval which does not
contain resonance points of the path $\set{H_r},$ then the
function
$$
  [0,1] \times I \ni (y,r) \mapsto (1+T_{\lambda+iy}(H_0)J_r)^{-1}
$$
is bounded.

Proof. Since $\lambda$ is regular, $T_{\lambda+iy}(H_0)$ is continuous
on $[0,1]$ and so $1+T_{\lambda+iy}(H_0)J_r$ is continuous on $[0,1]
\times I.$ Since the map $A \mapsto A^{-1}$ is also
continuous, the image of the function
$(1+T_{\lambda+iy}(H_0)J_r)^{-1}$ on the compact rectangle $[0,1] \times I$ is
bounded.

(B) We have
\begin{multline*}
  (1+T_{\lambda+iy}(H_0)J_r)^{-1} - (1+T_{\lambda+i0}(H_0)J_r)^{-1}
   \\ = (1+T_{\lambda+iy}(H_0)J_r)^{-1} \cdot \sqbrs{T_{\lambda+i0}(H_0) - T_{\lambda+iy}(H_0)} J_r \cdot (1+T_{\lambda+i0}(H_0)J_r)^{-1}.
\end{multline*}
Since, by (A), $(1+T_{\lambda+iy}(H_0)J_r)^{-1}$ is locally uniformly
bounded outside of the resonance set $R(\lambda, H_0,V;F)$ times
$\set{y \in [0,1]},$ it follows from the last equality that
$$
  (1+T_{\lambda+iy}(H_0)J_r)^{-1} \to (1+T_{\lambda+i0}(H_0)J_r)^{-1} \ \text{as} \ y \to 0
$$
in $\norm{\cdot}$ locally uniformly
with respect to $r \notin R(\lambda; H_0,V;F).$
Since by Lemma \ref{L: lemma Y}
$\sqrt{\Im T_{\lambda+iy}(H_0)}$ converges to $\sqrt{\Im
T_{\lambda+i0}(H_0)}$ in $\LpH{2},$ the claim follows from the
definition (\ref{F: def of tilde S(z,r)}) of $\tilde S(\lambda+iy,r)$ and
the H\"older inequality (\ref{F: Holder inequality}). 
\end{proof}

\begin{thm} \label{T: xis(r) is loc const} Let $\set{H_r}$ be a path which satisfies Assumption \ref{A: assumption on Hr}.
Let $\lambda$ be a fixed essentially regular point.
The singular spectral shift function $\xis(\lambda; H_r, H_0)$ is a locally constant function of $r$
and discontinuity points of this function of $r$ are resonance points of the path $\set{H_r}.$
\end{thm}
\begin{proof} Since both $\xi$ and $\xia$ are path additive, it is enough to show that if there are no resonance points,
then $\xi = \xia$ as function of $r$ under fixed $\lambda.$
In this case, it follows from Lemma \ref{L: tilde S converge uniformly}, that the function
$[0,\infty) \times [0,1] \ni (y,r) \mapsto \tilde S(\lambda+iy; H_r,H_0)$ is continuous.
It follows from Corollary \ref{C: xi-inv-t for two piece path} and Theorem \ref{T: properties of mu}(2)
that $\xi(\lambda)$ is equal to the $\xi$-invariant of the continuous path
$$
  [0,1] \ni r \mapsto \tilde S(\lambda+i0,r).
$$
By Lemma \ref{L: eig of tS = eig of S}, this path and the continuous path
$$[0,1] \ni r \mapsto S(\lambda; H_r, H_0)$$
have the same spectral measures. It follows that they have the same $\xi$-invariants.
\end{proof}

\begin{cor} \label{C: xi(r) is loc analyt} Let $\gamma = \set{H_r}$ be a real-analytic path which satisfies Assumption \ref{A: assumption on Hr}.
Let $\lambda$ be a fixed essentially regular point. The value $\xi(\lambda; H_r,H_0)$ of the spectral shift function at $\lambda$
as a function of $r \in \mbR$ is a locally analytic function, with (necessarily integer) jumps only at resonance points of the path $\gamma.$
\end{cor}

\begin{cor} Let $\lambda$ be an essentially regular point.
If a path $\gamma = \set{H_r}$ which satisfies Assumption \ref{A: assumption on Hr}
does not intersect the resonance set $R(\lambda; \clA,F),$ then $\xi(\lambda; H_1, H_0) = \xia_\gamma(\lambda; H_1, H_0).$
\end{cor}

Using point-wise additivity of $\xi(\lambda; H_1, H_0)$ (Theorem \ref{T: xi(l) is pointwise additive})
and the last corollary, it can be shown that
for a fixed essentially regular point $\lambda$
one-forms $\Phi_H(\cdot)(\lambda)$ and $\Phia_H(\cdot)(\lambda)$ are locally exact and, as a consequence, are also closed on the manifold $\Gamma(\lambda; \clA, F).$

\appendix

\section{Chronological exponential} \label{A: Texp}
In this appendix an exposition of the chronological exponential is given. See e.g.
\cite{AG78,Gamk} and~\cite[Chapter 4]{BSh}.

Let $p \in [1,\infty]$ and let $a < b.$
Let ${A(\cdot) \colon [a,b] \to \LpH{p}}$ be a piecewise continuous path of self-adjoint operators from $\LpH{p}.$
Consider the equation
\begin{gather} \label{F: dot Xt = At Xt}
  \frac {d X(t)}{dt} = \frac 1i A(t) X(t), \ \ X(a) = 1,
\end{gather}
where the derivative is taken in $\LpH{p}.$
Let $0 \leq t_1 \leq t_2 \leq \ldots \leq t_k \leq t.$
By definition, the left chronological exponent $\Texp=\overleftarrow{\exp}$ is
\begin{gather} \label{F: def of Texp}
  \Texp\brs{\frac 1i\int_a^t A(s)\,ds} =
     1 + \sum_{k=1}^\infty \frac 1{i^k} \int_a^{t} dt_k \int_a^{t_k} dt_{k-1} \ldots \int_a^{t_{2}} dt_1 A(t_k)\ldots A(t_1),
\end{gather}
where the series converges in $\LpH{p}$\tire norm.
\begin{lemma} \label{L: T exp}
The equation (\ref{F: dot Xt = At Xt}) has a unique continuous solution $X(t),$
given by formula
$$
  X(t) = \Texp\brs{\frac 1i\int_a^t A(s)\,ds}.
$$
\end{lemma}
\begin{proof} Substitution shows that (\ref{F: def of Texp}) is a continuous solution of (\ref{F: dot Xt = At Xt}).
Let $Y(t)$ be another continuous solution of (\ref{F: dot Xt = At Xt}).
Taking the integral of (\ref{F: dot Xt = At Xt}) in $\LpH{p},$ one gets
$$
  Y(t) = 1 + \frac{1}{i} \int _a^t A(s) Y(s)\,ds.
$$
Iteration of this integral and the bound $\sup_{t \in [a,b]} \norm{A(t)}_p \leq \const$
show that $Y(t)$ coincides with $(\ref{F: def of Texp}).$
\end{proof}
A similar argument shows that $\Texp\brs{\frac 1i\int_a^t A(s)\,ds}X_0$ is the unique solution
of the equation
$$
  \frac {d X(t)}{dt} = \frac 1i A(t) X(t), \ \ X(a) = X_0 \in 1 + \LpH{p}.
$$
\begin{lemma} \label{L: Texp su = Texp st + Texp tu}
The following equality holds
$$
  \Texp\brs{\int_s^u A(s)\,ds} = \Texp\brs{\int_t^u A(s)\,ds} \Texp\brs{\int_s^t A(s)\,ds}.
$$
\end{lemma}
\begin{proof} Using (\ref{F: def of Texp}), it is easy to check that
both sides of this equality are solutions of the equation (in~$\clL_p(\hilb)$)
$$\frac {dX(u)}{du} = \frac 1i A(u)X(u)$$ with the initial condition $X(t) = \Texp\brs{\int_s^t A(s)\,ds}.$
So, Lemma \ref{L: T exp} completes the proof.
\end{proof}
By $\det$ we denote the classical Fredholm determinant (cf. e.g.~\cite{GK,SimTrId2,Ya}).
\begin{lemma} \label{L: det Texp = exp Tr}
If $p = 1$ then the following equality holds
$$
  \det\, \Texp\brs{\frac 1i \int_a^t A(s)\,ds} = \exp\brs{\frac 1i \int_a^t \Tr(A(s))\,ds}.
$$
\end{lemma}
\begin{proof} Let~$F(t)$ and~$G(t)$ be the left and the right hand sides of this equality respectively.
Then $\frac {d}{dt}G(t) = \frac 1i \Tr(A(t)) G(t),$ $G(a) = 1.$
Further, by Lemma \ref{L: Texp su = Texp st + Texp tu} and the product property of $\det$
$$
  \frac d{dt} F(t) = \lim_{h\to 0} \frac 1h \brs{\det \, \Texp \brs{\frac 1i \int_t^{t+h} A(s)\,ds} - 1} F(t)
  = \frac 1i \Tr(A(t)) F(t),
$$
where the last equality follows from definitions of determinant~\cite[(3.5)]{SimTrId2},
$\Texp$ and piecewise continuity of $A(s).$
\end{proof}

\section*{Acknowledgements} \ I would like to thank  \,P.\,G.\,Dodds for useful discussions, and especially for
the proof of Lemma \ref{L: lemma Y}. I also thank D.\,Zanin for indicating that Theorem \ref{T: complement of Lambda(F)}
follows from~\cite[Theorem IV.9.6]{Saks}. 
I also would like to thank K.\,A.\,Makarov and A.\,B.\,Pushnitski for useful discussions.

Finally, I would like to thank the referee for his/her numerous helpful remarks.

%
%
%
\mathsurround 0pt
\ndef{\AndSoOn}{$\dots$}

\end{document}